%% file: quasi-rectif-fin.tex
\documentclass[10pt,leqno]{amsart}
\usepackage{amsthm,amsfonts,amssymb,amsmath,oldgerm,xcolor,bbm,upgreek,dsfont}
\usepackage{tikz}
\newcommand*\circled[1]{\tikz[baseline=(char.base)]{
            \node[shape=circle,draw,inner sep=2pt] (char) {#1};}}
\usepackage{bm,upgreek}
\usepackage{booktabs}
\usepackage{array, makecell} 

\expandafter\newcommand\csname r@tocindent4\endcsname{4in}
\numberwithin{equation}{section}
\usepackage{setspace}
\usepackage{stmaryrd,pifont}
\usepackage{mathrsfs}
\usepackage{fancyhdr}
\usepackage{graphicx}
\usepackage[colorlinks = true, linkcolor=blue,
filecolor=darkgreen,
urlcolor=cadmiumgreen,
citecolor=red,pdftex]{hyperref}
\usepackage{psfrag}
\usepackage{listings}

\newtheorem{princ}{Fact}
\newtheorem{remark}{Remark}
\newtheorem{example}{Example}

\definecolor{airforceblue}{rgb}{0.36, 0.54, 0.66}
\definecolor{darkgreen}{rgb}{0.0, 0.2, 0.13}
\definecolor{darkslategray}{rgb}{0.18, 0.31, 0.31}
\definecolor{mediumjunglegreen}{rgb}{0.11, 0.21, 0.18}
\definecolor{prussianblue}{rgb}{0.0, 0.19, 0.33}
\definecolor{warmblack}{rgb}{0.0, 0.26, 0.26}
\definecolor{cadmiumgreen}{rgb}{0.0, 0.42, 0.24}
\definecolor{mauvetaupe}{rgb}{0.57, 0.37, 0.43}
\definecolor{maroon(x11)}{rgb}{0.69, 0.19, 0.38}

\include{macros}

\begin{document}
\title[Quasi-rectification]{Turbulent effects 
through quasi-rectification}

\author[R. Carles]{R\'emi Carles}
\email{Remi.Carles@math.cnrs.fr}
 \urladdr{http://carles.perso.math.cnrs.fr/}
\author[Ch. Cheverry]{Christophe Cheverry}
\email{christophe.cheverry@univ-rennes1.fr}
\urladdr{https://perso.univ-rennes1.fr/christophe.cheverry}
\address{Univ Rennes, CNRS\\ IRMAR - UMR 6625\\ F-35000 Rennes\\ France}

\begin{abstract}  This article introduces a physically realistic model for explaining how electromagnetic 
waves can be internally generated, propagate and interact in strongly magnetized plasmas or in nuclear 
magnetic resonance experiments. It studies high frequency solutions of nonlinear hyperbolic equations 
for time scales at which dispersive and nonlinear effects can be present in the leading term of the 
solutions. It explains how the produced waves can accumulate during long times to produce constructive 
and destructive interferences which, in the above contexts, are part of turbulent effects. 

\bigskip

 \end{abstract}

\maketitle

{\small \parskip=1pt
\tableofcontents
}

 \break

 
 \section{Introduction}  \label{sec:preamble} 
 In this introduction, we present the main aspects of our text.  In Subsection~\ref{subsec:toymodel}, we introduce a 
 simple ODE model that is intended to  serve as a guideline.  In Subsection~\ref{subsec:morerealistic}, we extend this model 
 to better incorporate important specificities of two realistic situations which are related to strongly magnetized plasmas 
 (SMP) and nuclear magnetic resonance (NMR). In Subsection~\ref{subsec:statemainresult}, we state under simplified 
 assumptions our two main results, Theorems~\ref{theo:resumeNL} and \ref{theo:resumeNLbis}. We also give an overview 
 of our article.


 \subsection{A toy model}  \label{subsec:toymodel} Introduce the {\it phase} $ \varphi : \RR \rightarrow \RR $ given by
  \begin{equation}\label{third-sipourphipha}
\varphi(t) := t + \gamma (\cos t -1) , \qquad \gamma \in \, ]0,1/4[ .
\end{equation} 
Let $ \eps \in \, ]0,1] $ be a small parameter, and $ \lambda \in \CC $. Fix numbers $ (j_1,j_2,\nu) \in \NN^2 \times \RR $ 
such that $ j_1+j_2\ge 2$. Select $ n \in \ZZ $ and $ \omega \in \RR $. Then, define
\begin{equation}\label{ajoutFLNL}
  F_{\! L} (\eps,t) := \eps^{3/2} e^{in \varphi(t) / \eps} , \qquad F_{\! N \! L} (\eps,t,u) := \lambda \eps^\nu 
e^{i \omega t/ \eps} u^{j_1} \bar u^{j_2} . 
\end{equation} 

\begin{defi} \label{gaugeparadefi}
The number $ \textfrak {g} := \omega+j_1-j_2 \in \RR $ is called the {\it gauge parameter} associated with $ F_{\! N \! L} $. 
\end{defi}

\noindent Consider the ordinary differential equation on the complex plane $ \CC $ given by
\begin{equation}\label{third-si}
\qquad \frac{d}{dt} u - \frac{i}{\eps} u =   F (\eps,t,u) := F_{\! L} (\eps,t) + F_{\! N \! L} (\eps,t,u) , \qquad 
u_{\mid t=0} = 0 .
\end{equation} 
We can study the equation (\ref{third-si}) on three different time scales: 

\smallskip

\noindent $ \bullet $ {\it Fast}, when $ t \sim \eps $, that is when $ F $ undergoes a few number of oscillations;

\smallskip

\noindent $ \bullet $ {\it Normal}, when $ t \sim 1 $, that is when $ F $ generates $ \cO (\eps^{-1} ) $ oscillations, 
whereas the periodic part ($ \cos t $) inside $ \varphi $ sees a few number of oscillations;

\smallskip

\noindent $ \bullet $ {\it Slow}, when $ t \sim \eps^{-1} $ or $ T := \eps t  \sim 1 $, that is when $ F $ involves 
$ \cO ( \eps^{-2} ) $ oscillations. 

\medskip

\noindent In this subsection, we analyze (\ref{third-si}) during long times $ t \sim \eps^{-1} $ or $ T \sim 1 $. 
With this in mind, we can change $ u $ according to
\begin{equation}\label{changeofuu}
  u(t) = \eps e^{it/\eps} \cU(\eps t) , \qquad \cU (T) := \eps^{-1} e^{-iT/\eps^2} u (\eps^{-1} T) . 
\end{equation}

\noindent Expressed 
in terms of $ \cU  $, the equation (\ref{third-si}) becomes
\begin{equation}\label{eq:EDONLajout}
\quad \frac{d}{dT} \cU =  \frac{1}{\sqrt \eps} e^{i (n-1) T / \eps^2 + i n \gamma (\cos (T/\eps) -1)/ \eps} + \lambda 
\eps^{\nu+j_1+j_2-2} e^{i (\textfrak {g} -1) T/ \eps^2} \cU^{j_1} \bar \cU^{j_2} .
\end{equation}
The initial data is still zero. Denote by $\cU_{\rm lin}$ the solution corresponding to the linear evolution, that is the solution 
obtained from \eqref{eq:EDONLajout} when $\lambda=0$. When $\lambda \not = 0 $ and when $\nu +j_1+j_2 > 2$, the 
solution to \eqref{eq:EDONLajout} looks like $\cU_{\rm lin}$. Our aim is to first study the expression $\cU_{\rm lin}$. Then,
we incorporate nonlinear effects by looking at a critical size for the nonlinearity, corresponding to the special case 
$ \lambda \not = 0 $ and $ \nu+j_1+j_2=2 $. This means to single out the following equation
 \begin{equation}\label{eq:EDONL}
\ \frac{d}{dT} \cU =  \frac{1}{\sqrt \eps} e^{i (n-1) T / \eps^2+ i n \gamma (\cos (T/\eps) -1)/ \eps} + \lambda 
e^{i (\textfrak {g} -1) T/ \eps^2} \cU^{j_1} \bar \cU^{j_2} , \quad \cU_{\mid T=0} = 0 .
\end{equation}
The integral formulation of (\ref{eq:EDONL}) 
 reads
\begin{equation}\label{integralformu}
  \cU (T) = \cU_{\rm lin}(T) + \lambda \int_0^T e^{i (\textfrak {g} -1) s / \eps^2}  \cU(s)^{j_1}  \bar \cU(s)^{j_2} d s .
\end{equation}
In Paragraph \ref{linintrocase}, we first show that $\cU_{\rm lin} (T) =\cO(1)$, an estimate which is sharp when $n=1$. 
As a consequence, the nonlinear contribution brought by the integral term inside (\ref{integralformu}) is likely to be of 
the same order of magnitude as the linear one. It can be expected that $ \cU (T) \not \equiv  \cU_{\rm lin}(T) + o(1) $.
In Paragraph \ref{nonlinintrocase}, we prove that this is indeed the case if and only if $\textfrak {g}=1$. 


\subsubsection{The linear case}\label{linintrocase} By construction, we have
\begin{equation}\label{third-expliu}
u_{\rm lin} (t) := \eps e^{it/\eps}\cU_{\rm lin}(\eps t) = \eps^{3/2} e^{i t / \eps} \int_0^t e^{i \, \lbrack n \varphi (s) -s \rbrack / \eps } ds . 
\end{equation}
We start the analysis of (\ref{integralformu}) by looking at the part $ \cU_{\rm lin} $ through the expression $ u_{\rm lin} $ of 
(\ref{third-expliu}). Examine the right hand side of (\ref{third-expliu}). For harmonics $ n \in \ZZ $ with $ n \not = 1 $, since 
$ 0 < \gamma < 1/4 $, remark that
 \begin{equation}\label{nonstaphi}
\forall s \in \RR \, , \qquad 1/2 \leq \vert n \, \varphi'(s) -1 \vert = \vert n -1 - \gamma \, n \, \sin s \vert . 
\end{equation}
 Exploiting (\ref{nonstaphi}), a single integration by parts yields
 \begin{equation*}
\forall t \ge0 \, , \qquad u_{\rm lin} (t) = \cO \bigl( \eps^{5/2} (1+t) \bigr) . 
\end{equation*} 
In other words, assuming that $ n \not = 1 $, we find
\begin{equation}\label{devharmmnotmlong}
\forall T \ge 0 , \qquad \cU_{\rm lin}(T) = \cO (\eps^{3/2} + \sqrt\eps T)  . 
\end{equation} 
The situation is completely different when $ n = 1 $. Fix an integer $ K\ge 1 $. The solution $ u_{\rm lin}$  computed at the 
time $ t = 2 K \pi $ can be viewed as a sum of contributions produced over time by the source term, namely
 \begin{equation}\label{third-si-devprime}
\qquad u_{\rm lin} (2 K \pi) = \sum_{k=0}^{K-1} u_k , \qquad u_k := \eps^{3/2} e^{i 2 K \pi / \eps} \int_{2 k \pi}^{2 (k+1) \pi} 
e^{i \lbrack \varphi (s) -s \rbrack / \eps} ds .
\end{equation} 
Since the function $ s \mapsto \varphi(s)-s = \gamma (\cos s-1) $ is periodic of period $ 2\pi $, the wave packets $ u_k $ can be 
interpreted according to $ u_k = \eps^{3/2} e^{i 2 K \pi / \eps} {\rm v}_k $ with
 \begin{equation}\label{third-si-devprimeinter}
{\rm v}_k = \int_{2 k \pi - \pi / 2}^{2 k \pi+ 3 \pi / 2} e^{i \gamma  (\cos s -1) / \eps} ds = {\rm v} := \int_{- \pi / 2}^{3 \, \pi / 2} 
e^{i \gamma  (\cos s -1) / \eps} ds . 
  \end{equation} 
The function $ s \mapsto \gamma \, (\cos s -1) $ has exactly two non-degenerate stationary points  
in  the interval $ [ 2 k \pi - \pi/2, 2 k \pi + 3 \, \pi / 2 ] $, at the positions $ s = 2 k \pi $ and $ s = 2 k \pi + \pi $. Using the 
periodicity to get rid of the boundary terms and applying stationary phase formula, it follows that
 \begin{equation}\label{third-si-dev}
{\rm v} = \sqrt{\frac{2 \pi\eps }{\gamma}} e^{- i \frac{\gamma}{\eps}} \( e^{i ( \frac{\gamma}{\eps} - \frac{\pi}{4})} + e^{- i (\frac{\gamma}{\eps} 
- \frac{\pi}{4})} \)  + \cO \bigl( \eps^{3/2} \bigr) .  
\end{equation} 
Let $ A_\eps \in \CC $ be such that
\begin{equation}\label{defcoefaeps}
A_\eps^2 = \sqrt{\frac{2}{\pi \gamma}} e^{- i \frac{\gamma}{\eps}} \cos \Bigl( \frac{\gamma}{\eps} - \frac{\pi}{4} \Bigr) ,
\qquad \limsup_{\eps \rightarrow 0} \vert A_\eps^2 \vert = \sqrt{\frac{2}{\pi \gamma}} \not = 0 .
\end{equation}
Observe that
\begin{equation}\label{observevandu}
{\rm v} = 2 \pi A_\eps^2 \sqrt \eps +  \cO \bigl( \eps^{3/2} \bigr) , \qquad \vert u_k \vert = 2 \pi \vert A_\eps^2 \vert \eps^2 +  
\cO \bigl( \eps^3 \bigr) . 
\end{equation} 
The combination of (\ref{third-si-devprime}), (\ref{defcoefaeps}) and (\ref{observevandu}) indicates that, when $ n=1 $, 
wave packets $ u_k $ of amplitude $ \eps^2 $ are repeatedly created over time when solving (\ref{third-si}) in the case 
$ \lambda = 0 $. 

\smallskip

\noindent Look at (\ref{third-si-devprime}). The emitted signals $ u_k $ (one per period $ 2\pi $) have cumulative effects
up to the stopping time $ 2 K \pi $. They give rise to a growth rate with respect to the time variable $ t $. For long times 
$ T \sim 1 $, assuming that $ n = 1 $, we can assert that
 \begin{equation}\label{third-si-devtt}
\qquad \cU_{\rm lin}(T) = A^2_\eps T + \cO ( \eps ) = A^2_\eps \int_0^{+\infty} \mathds{1}_{[0,T]} (s) ds + \cO ( \eps ) = \cO (1) . 
\end{equation}
This short discussion about the linear situation ($ \lambda = 0 $) highlights a difference between the cases $ n \not = 1 $
- see \eqref{devharmmnotmlong} - and $ n = 1 $ - see \eqref{third-si-devtt}. This observation is important in the perspective of nonlinear 
effects. As a matter of fact, it allows a first selection between the different modes $ n \in \ZZ $. 


\subsubsection{Nonlinear effects} \label{nonlinintrocase} Here, we consider the nonlinear framework, when $ \lambda \not = 0 $
and $ \nu+j_1+j_2=2 $.
The difference $ \cW := \cU - \cU_{\rm lin}$ is subject to
 \begin{equation}\label{eq:EDONLD}
\cW (T) = \lambda \int_0^T e^{i (\textfrak {g} -1) s / \eps^2}  (\cU_{\rm lin} + \cW)(s)^{j_1} (\bar  \cU_{\rm lin}  + \bar \cW)(s)^{j_2}  d s .
\end{equation}
Using a Picard scheme, it is easy to infer that the life span of the solution $ \cW $ to the integral equation (\ref{eq:EDONLD}), 
and therefore of the solution $ \cU $ to (\ref{eq:EDONL}), can be bounded below by a positive constant not depending on 
$ \eps \in \, ]0,1] $. Knowing (\ref{devharmmnotmlong}) and (\ref{third-si-devtt}), it is also possible to deduce that $ \cW (T) $ 
is of size $ \cO(\eps^{(j_1+j_2)/2}) = \cO(\eps) $ when $ n \not = 1 $, and of size $ \cO(1) $ when $ n = 1 $. This means that 
the preceding dichotomy between the two cases $ n \not = 1 $ and $ n = 1 $ remains when $ \lambda \not = 0 $.

\begin{princ} \label{fact1} When solving \eqref{eq:EDONL}, the harmonic $ n = 1 $ stands out from the others. Given $ T >0$, 
we find $ \cU (T) = \cO( \sqrt \eps ) $ when $ n \not = 1 $, and $ \cU (T) = \cO( 1) $ when $ n= 1 $.
\end{princ}

\noindent Assume that $ \textfrak {g} \not = 1 $. The identity (\ref{integralformu}) becomes after an integration by parts
\begin{align}
\cU (T) = \cU_{\rm lin} (T) & - \frac{i \lambda \eps^2}{\textfrak {g}
                              -1} e^{i (\textfrak {g} -1) T / \eps^2}
                              \cU(T)^{j_1} \bar \cU(T)^{j_2}  \label{sisiilfaut} \\ 
\ & + \frac{i \lambda \eps^2}{\textfrak {g} -1} \int_0^T e^{i (\textfrak {g} -1) s / \eps^2}  \part_s \bigl( \cU(s)^{j_1} \bar \cU(s)^{j_2} 
\bigr) d s .  \nonumber
    \end{align}
From the equation \eqref{eq:EDONL}, since we have seen that the solution $ \cU $ is (at least) bounded, we know that $ \part_s
\cU (s) = \cO(\eps^{-1/2}) $. From (\ref{sisiilfaut}), it follows that
 \begin{equation*}
 \forall T \in \RR , \qquad \cU (T) = \cU_{\rm lin}(T) + \cO( \eps^{3/2} ) . 
 \end{equation*} 

\noindent Now, assume that $ n = 1 $ and moreover that $ \textfrak {g} = 1 $. To show that, in this situation, nonlinear effects actually 
occur, it suffices to produce an example. To this end, take $ (j_1,j_2,\nu)= (2,0,0) $ and $ \omega = -1 $, so that $ \textfrak {g} = 1 $. 
Choose $ \lambda = 1 $. Then, using (\ref{third-si-devtt}), the identity (\ref{integralformu}) becomes
\begin{equation}\label{integralformuspecialcase}
  \cU (T) =  A^2_\eps T  + \cO (\eps) + \int_0^T \cU(s)^2 d s .
\end{equation}
This implies that $ \cU(T) = A_\eps \tan (A_\eps T) + \cO(\eps) $, and therefore
\begin{equation*}
\cU(T) - \cU_{\rm lin} (T) = A_\eps \tan (A_\eps T) - A^2_\eps T + \cO(\eps) \not = o(1) .
\end{equation*}
In view of the above formula, the asymptotic behavior of the nonlinear solution $ \cU $ can strongly differ from the one of the 
linear solution $ \cU_{\rm lin} $.

\begin{princ} \label{fact2} When solving \eqref{eq:EDONL}, the gauge
  parameter $ \textfrak {g} = 1 $ stands out from the others. When  
$ \textfrak {g} \not = 1 $, the asymptotic behaviors of $ \cU $ and $ \cU_{\rm lin} $ when $ \eps $ goes to $ 0 $ are the same.
On the contrary, when $ \textfrak {g} = 1 $, nonlinear effects can be expected at leading order.  
\end{princ}


 \subsection{A more realistic model}  \label{subsec:morerealistic} The preceding features, Facts 1 and 2, which have been 
 emphasized in the case of ODEs, are still present when dealing with partial differential equations arising in strongly magnetized 
 plasmas (SMP) or in  nuclear magnetic resonance experiments (NMR). But, there are two emerging issues: the first is due to
 dispersive effects which are completely absent in the ODE case; the second comes from the occurrence of non-trivial spatial 
 variations when dealing with the phase $ \varphi  $. At all events, the discussion becomes much more subtle, and new 
 important phenomena can and do occur. 
 
 \smallskip
 
 \noindent In order to investigate SMP or NMR, we must consider the PDE counterpart of \eqref{third-si}, which is
 \begin{equation}\label{third-si-gene}
 \part_t u - \frac{i}{\eps} p (\eps D_x) \, u =  F = F_{\! L} + F_{\! N \! L} , \qquad u_{\mid t=0} = 0 ,
 \qquad 0 < \eps \ll 1 ,  
\end{equation} 
where $ t \in \RR $ and $ x \in \RR $. The state variable is $ u \in \RR $ and $ D_x := - i \, \part_x $. The action of the  
pseudo-differential operator $ p (\eps D_x) $ is given on the Fourier side by the multiplier $ p (\eps \xi) $. 

\smallskip

\noindent In what follows, we will focus on the scalar wave equation (\ref{third-si-gene}). The origin of equation (\ref{third-si-gene}), 
its physical significances and the reasons why it may be seen as a universal problem (when dealing with systems of hyperbolic 
equations) will be clearly explained in Sections \ref{sec:model} and \ref{sec:setting}. We will work in space dimension one. The 
possible multidimensional effects will not be investigated here.

\smallskip

\noindent We now fix some notations and we introduce simplified assumptions intended to facilitate the presentation of our main 
results. We suppose that the symbol $ p  $ is smooth, say $p\in \cC^\infty (\RR) $. The function $ p  $ is even. It is such that 
$ p_{\mid [-\xi_c,\xi_c]}\equiv 0 $ for some $\xi_c\ge 0$. It is strictly increasing on $(\xi_c,\infty)$. Moreover,  for large values of 
$ \xi $, it is subject to
\begin{equation}\label{hyppreliintrod}
\qquad \lim_{\xi \rightarrow + \infty} \ p (\xi) = 1 , \quad \ \lim_{\xi \rightarrow + \infty} \ p' (\xi) = 0 , \quad \ \exists \ell<0, \quad 
\lim_{\xi \rightarrow + \infty} \ \xi^{4} \ p'' (\xi) = \ell , 
\end{equation}
as well as
\begin{equation}\label{hyppreliintrodetoui}
\exists D\ge 4 ; \qquad \forall \, n \in \{ 2, \cdots, D \} \, , \quad \limsup_{\xi \rightarrow + \infty} \ \frac{\vert p^{(n)}(\xi) \vert}{p'(\xi)} < + \infty .
\end{equation}
Fix some $M \in \NN^* $. The source term $ F_L $ is defined by 
\begin{equation}\label{sssooour}
  F_L (\eps,t,x) =\, - \, \eps^{3/2} \! \sum_{m\in [-M,M]\setminus\{0\}}\ a_m(\eps \, t,t,x) \ e^{i \,m \varphi(t,x) / \eps} .
\end{equation}
In (\ref{sssooour}), the amplitudes $ a_m (T,t,x) $ are chosen in the set $ \cC_b^\infty (\RR^3) $ of smooth
functions whose all derivatives are bounded. They are selected in such a way that, for some $\cT >0$ and some 
$ r \in \RR_+^* $ with $ r<\gamma/2 $, we have
\begin{equation}\label{asymptoticallypietoui}
\forall m \in [-M,M]\setminus\{0\} \, , \qquad \mathrm{supp} \,  a_m \subset \, ]-\infty,\cT] \times [1,+\infty[ \, \times [-r,r] .
\end{equation}
The amplitude $ a_1 (T,t,x) $ is chosen periodic for large times in the second variable. In other words, 
there exists $ {\rm t}_s  \in \RR_+^* $ and a smooth function $ \aaa(T,t,x) $ such that 
\begin{equation} \label{asymptoticallypi}
 \forall t \geq {\rm t}_s , \quad \forall n \in \NN , \quad a_1 (\cdot,t+ n \pi,\cdot) \equiv \aaa (\cdot,t+  n \pi,\cdot) \equiv \aaa
(\cdot,t,\cdot) .  
\end{equation}
The phase $ \varphi $ arising in (\ref{sssooour}) is more general than in (\ref{third-sipourphipha}). It does depend on the spatial 
variable $ x \in \RR $. It is the sum of a quadratic part (in $ t $ and $ x $) and a periodic part (in $ t $).

\begin{ass}[Selection of a relevant phase $ \varphi $] \label{choiofa}The function $ \varphi $ is
\begin{equation} \label{pharetain}
 \varphi(t,x) =  t - x \ t + \gamma \ ( \cos \, t -1) \, , \qquad 0<\gamma <1/4 .  
\end{equation}
\end{ass}

\noindent In Section~\ref{sec:model}, the above assumptions on $ p $ and $ \varphi $ will be motivated by the study of two 
 realistic situations which are related to strongly magnetized plasmas (SMP) and nuclear magnetic resonance (NMR).  In 
 Section \ref{sec:setting}, to better incorporate important specificities of SMP and NMR, they will be somewhat generalized. 
 
 \smallskip
 
 \noindent In the right hand side of (\ref{third-si-gene}), the nonlinear part $F_{\! N\! L}$ is, up to some localization in time and space, 
 of the same form as in the previous subsection. Select a nonnegative cut-off function $\chi$ which is equal to $1$ in a neighborhood 
 of the origin and which is such that  $\operatorname{supp}\chi\subset [-1,1]$. Fix some parameter $\iota\in [0,1]$ which is aimed to 
 measure the strength of the spatial localization. We impose
\begin{equation}\label{eq:NLedp}
  F_{\! N\! L}  (\eps,t,x,u)=\lambda \eps^{\nu}\chi \Bigl( 3-2\frac{\eps t}{\cT}\Bigr) \chi\(\frac{x}{r\eps^\iota}\) e^{i\omega t/\eps} u^{j_1}\bar u^{j_2} .
  \end{equation}
Taking into account the conditions on the support of the $a_m$'s and  $\chi$, the term $F_{\! N\! L}$ becomes effective only 
for $t\ge \cT/\eps$, that is after the term $F_{\! L}$ has played its part. So we observe successively two distinct  phenomena: 
a possible linear amplification, and then nonlinear interactions.

 \smallskip
 
 \noindent The solution $ u $ to (\ref{third-si-gene}) exists on a time interval $ [0,\tilde \cT / \eps] $ with $ \cT < \tilde \cT $. The 
 argument is similar to the one given for the toy model. Through the change (\ref{changeofuu}), we can reformulate the equation
 (\ref{third-si-gene}) in terms of $ \cW = \cU - \cU_{\rm lin} $, see (\ref{eq:detercD}) and (\ref{eq:defcG}). Since  $ \nu+j_1+j_2 > 2 $, 
 the lifespan expressed in terms of $ T = \eps t $ does not shrink to $ \cT $ when $ \eps $ goes to zero. Note however that, due to 
 the quadratic nonlinearity, the global-in-time existence is not at all guaranteed concerning (\ref{third-si-gene}).
 
 \smallskip
 
 \noindent We still denote by $ u_{\rm lin} $ the linear solution 
obtained from (\ref{third-si-gene}) when $\lambda=0$. One point should be underlined here. Our discussion of the linear situation 
is based on the analysis in $L^\infty$ of oscillatory integrals appearing in a suitable wave packet decomposition of $ u_{\rm lin} $. 
The precise structure of these wave packets is lost under the influence of nonlinearities. It follows that our key argument cannot 
be iterated to obtain the existence and the asymptotic behavior of the solution to the full nonlinear equation \eqref{third-si-gene}. 
For this reason, we do not work with \eqref{third-si-gene}. Instead, we look at the first two iterates of an associated Picard iterative 
scheme, which are
\begin{subequations}\label{eq:Picard} 
\begin{eqnarray} 
& &  \displaystyle \part_t u^{(0)} - \frac{i}{\eps} p (\eps D_x) \, u^{(0)} = F_{\! L} , \qquad u_{\mid t=0}^{(0)} = 0 , \label{eq:Picard0} \\
& & \part_t u^{(1)} - \frac{i}{\eps} p (\eps D_x) \, u^{(1)} = F_{\! L}+F_{\! N\! L} \bigl(u^{(0)}\bigr)  , \qquad u_{\mid t=0}^{(1)} = 0 . 
\label{eq:Picard1NL} 
\end{eqnarray}
\end{subequations}
Generalizing (\ref{changeofuu}), we can define 
 \begin{equation} \label{eq:u/U}
\qquad \ \cU^{(j)}(T,z) := \frac{1}{\eps} e^{-iT/\eps^2} u^{(j)} \Bigl( \frac{T}{\eps},\eps z \Bigr) , \quad \ \ u^{(j)} (t,x) := \eps  
e^{i t/\eps} \, \cU^{(j)} \Bigl( \eps t,\frac{x}{\eps} \Bigr) .  
 \end{equation}
 The expression $ \cU^{(0)} $ is the solution to the linear equation ($ \lambda = 0 $). Thus, we have
 $$ \cU^{(0)} (T,z) = \cU_{\rm lin} (T,z) := \frac{1}{\eps} e^{-i T/ \eps^2} u_{\rm lin} \Bigl( \frac{T}{\eps},\eps z \Bigr) . $$
 Symbols like $ p $ appear when looking at special branches $ \cV $ of {\it characteristic varieties} describing the propagation 
 of electromagnetic waves
\begin{equation}\label{graphs-gradV}
\cV := \bigl \lbrace ( t,x, \tau , \xi ) \, ; \, \tau = p(\xi) \, , \, (t,x,\xi) \in \RR^3 \bigr \rbrace \subset T^*(\RR^2) \equiv \RR^2 \times \RR^2 .
\end{equation} 
On the other hand, the phase $ \varphi $ may reflect the transport properties of particles. The graph $ \cG $ of the gradient of 
$ \varphi $ is associated with the {\it Lagrangian manifold}
 \begin{equation}\label{graphs-grad}
\qquad \cG := \bigl \lbrace \bigl( t , x,\part_t \varphi (t,x) , \part_x \varphi (t,x) \bigr) \, ; \, (t,x) \in \RR^2 \bigr \rbrace \subset 
T^*(\RR^2) \equiv \RR^2 \times \RR^2 .
\end{equation} 
In the ODE framework of Paragraph \ref{subsec:toymodel}, we simply find
$$ \quad \cV_{\rm ode} = \bigl \lbrace ( t,x, 1 , \xi ) \, ; \, (t,x,\xi) \in \RR^3 \bigr \rbrace  , \qquad
 \cG_{\rm ode}  = \bigl \lbrace (t,x,1 - \gamma \sin t, 0) \, ; \, (t,x) \in \RR^2 \bigr \rbrace ,$$
so that
 \begin{equation}\label{intersecentrevg}
\cV_{\rm ode} \cap \cG_{\rm ode} = \bigl \lbrace (k \pi ,x,1 , 0) \, ; \, (k,x) \in \ZZ \times \RR \bigr \rbrace . 
\end{equation} 
Thus, the production at the successive times $ k \pi $ with $ k \in \NN $ of the wave packets $ u_k $ which appear at the level 
of (\ref{third-si-devprime}) can be interpreted as coming from positions which are inside $ \cV_{\rm ode} \cap \cG_{\rm ode} $. 
This principle is illustrated in Figure~\ref{figurecross} below, given at $ x $ fixed and $ \xi = 0 $, with $ t $ in abscissa and 
the time frequency $ \tau $ in ordinate. 
\begin{figure}[h]\label{figurecross}
\begin{minipage}[b]{.95\linewidth}
\includegraphics[scale=0.235,width=12.1cm,height=3cm]{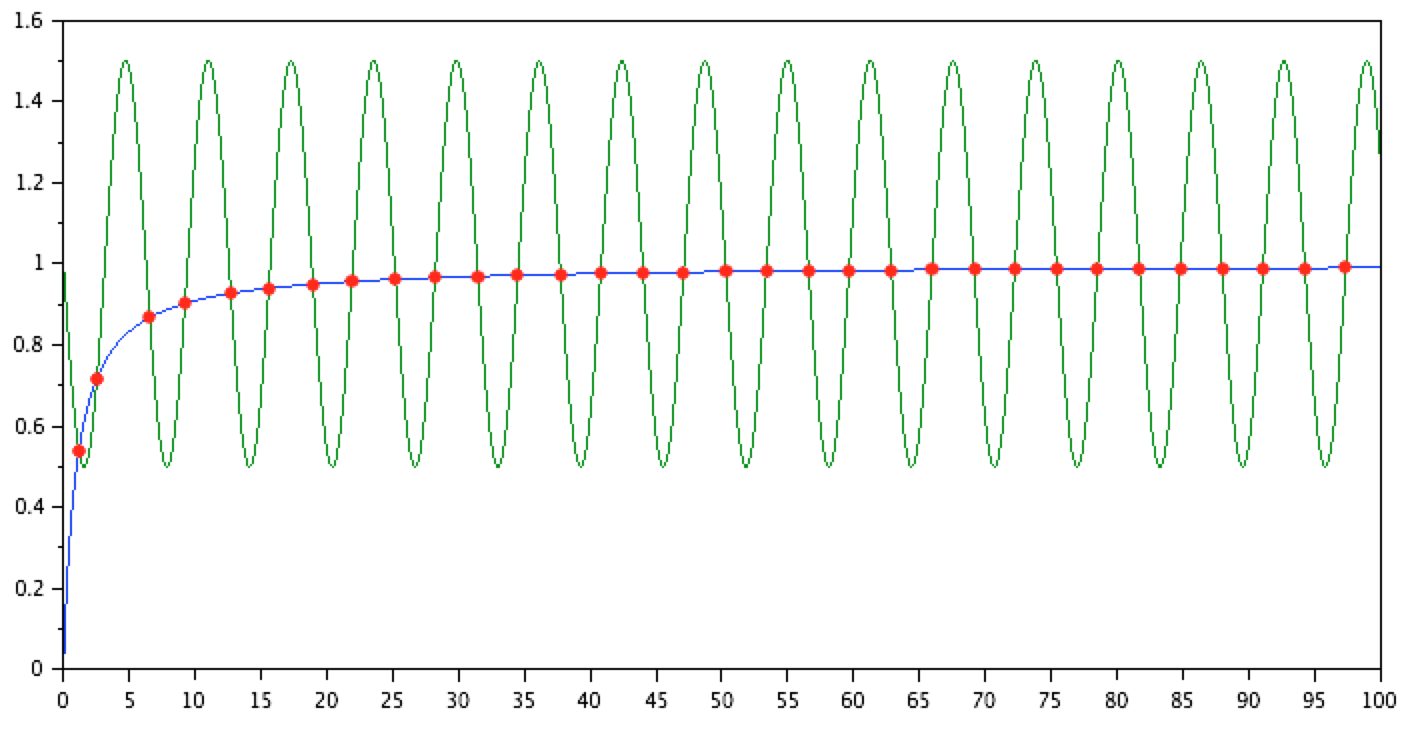} 
\caption{
Intersection (in {\color{red}{red}}) of $ \cV_{\rm ode} $ (in {\color{airforceblue}{blue}}) and 
$ \cG_{\rm ode}  $ (in {\color{green}{green}})
}
\end{minipage}
\end{figure}

\noindent Similarly, in the general framework (\ref{third-si-gene}), \underline{two-dimensional} oscillating waves $ u_k $ can 
emanate from the more complicated intersection
$$ \cV \cap \cG = \bigl \lbrace \bigl( t,x,p(-t),-t \bigr) \, ; \, (t,x) \in \RR^2 \ \text{and} \ p(-t)=p(t) =1-x - \gamma \sin t \bigr \rbrace . $$
In view of (\ref{hyppreliintrod}), for large values of $ \vert \xi \vert $, the dispersion relation $ p(\xi) = \tau $ mimics the choice 
$ p \equiv 1 $ of (\ref{third-si}). As in (\ref{intersecentrevg}), the set $ \cV \cap \cG $ contains (near $ x = 0 $ and for $ t $  large
enough) an infinite number of curve portions (in $ \RR^2 $) which appear repeatedly in time, and from which oscillating waves 
$ u_k $ may be triggered. 

\smallskip
 
\noindent In the framework of SMP and NMR, the symbol $ p $ and the phase $ \varphi $ are issued from different physical laws. 
They are originally unrelated, see Section \ref{sec:model}. But they are connected when solving the equation (\ref{third-si-gene}). 
The interactions between``waves" (associated with $ p $) and ``particles" (described by $ \varphi $) may be revealed through the 
intersection between the two geometrical objects $ \cV $ and $ \cG $, from which waves $ u_k $ can be emitted. 

\smallskip
 
\noindent The amplification mechanism that may arise after summing the $ u_k $'s can be viewed as a {\it resonance}. But now, 
the waves $ u_k $ are no more sure to overlap. In contrast to the toy model, since $ \part_x  \varphi \not \equiv 0 $ and $ p' \not 
\equiv 0 $, the waves $ {\rm u}_k $ do propagate in $ \RR^2 $. They propagate in different directions and with various group 
velocities. They can mix before reaching the long times $ t \sim \eps^{-1} $. 

 \begin{princ}\label{fact3} In the the PDE framework of equation \eqref{eq:Picard}, the accumulation  of the emitted oscillating waves 
 $ u_k $ can produce during long times $ T \sim 1 $ both
\href{https://en.wikipedia.org/wiki/Wave_interference}{constructive and destructive interferences}.
 \end{princ}
 

 \subsection{Statement of main results}  \label{subsec:statemainresult} 
 
\noindent The analysis of the creation, the propagation, the linear superposition, and the nonlinear interaction of the $ u_k $'s is a 
manner to approach some kind of {\it turbulence}. We start with
situations where the linear aspects are predominant. A standard Picard
scheme can be used to approximate the nonlinear equation
\eqref{third-si-gene}. The corresponding first two iterates yield the 
Cauchy problems \eqref{eq:Picard0} and \eqref{eq:Picard1NL}. 

 \begin{theo}[Linear behavior]\label{theo:resumeNL} Select a source term $ F_L $ as indicated in \eqref{sssooour} with a phase $ \varphi $ depending 
 on $ \gamma $ according to \eqref{pharetain}. Take profiles $ a_m $ satisfying both \eqref{asymptoticallypietoui} and \eqref{asymptoticallypi}.

 
 \smallskip
 
\noindent Consider the equation \eqref{third-si-gene} with a symbol $ p $ subject to \eqref{hyppreliintrod} and \eqref{hyppreliintrodetoui}.
Let $ \cU^{(j)}$, with $ j \in \{0,1\} $, issued from \eqref{eq:u/U} 
after solving \eqref{eq:Picard}. Select some $ \cT >0 $. 
\begin{enumerate} 
\item \emph{Linear case.} Concerning $ \cU^{(0)} \equiv \cU_{\rm lin} $, we can produce the following distinct 
asymptotic behaviors when $\eps $ goes to zero.
\begin{itemize}
\item \emph{Constructive interferences.} For all $j\in \ZZ$ and $T\in [\cT,2\cT]$,
  \begin{align}
 \cU_{\rm lin} (T,2j) = \cO(1) = A^2_\eps \int_0^{+\infty} e^{-i\frac{\ell}{6}\(\frac{1}{s}-\frac{T}{s^2}\)} \aaa(s,0,0)ds+o(1) ,
 \label{consintertheoprin}
  \end{align}
  where $ \displaystyle A_\eps^2= \sqrt{\frac{2}{\pi \gamma}} e^{- i \frac{\gamma}{\eps}} \cos \( \frac{\gamma}{\eps} - \frac{\pi}{4} \) $ is as 
  in \eqref{defcoefaeps}. 
\item \emph{Destructive interferences.}   By contrast, for all $ z \in \RR\setminus 2\ZZ$ and for all $T\in [\cT,2\cT]$, we find that 
  \begin{equation}
    |\cU_{\rm lin} (T,z )| =o(1) . \label{desnonlinintertheoprin}
  \end{equation}
  \end{itemize}
\item \emph{Nonlinear case.}
   Adjust the nonlinearity $ F_{\! N\! L}  $  as in \eqref{eq:NLedp}, with real parameters $ \nu $, $ j_1 $, $ j_2 $ $ \omega $ and $ \iota $. 
 Assume that either $\nu+j_1+j_2>2$,  or $\nu+j_1+j_2=2$ with $\omega+j_1-j_2\not = 1$. 
 Fix some $\iota\in [0,1]$. In the case
 $\nu+j_1+j_2-2=\omega+j_1-j_2=0$, set $\iota=1$. Then the nonlinearity plays no role at leading order in the sense that
  \begin{equation} \label{comparethetwo}
  \forall (T,z) \in [0, 2\cT] \times \RR , \qquad \cU^{(1)} (T,z) = \cU_{\rm lin}(T,z) + o(1) .  
   \end{equation}
   \end{enumerate}
 \end{theo}

 \noindent Interpreted in the setting of SMP, Theorem \ref{theo:resumeNL} shows, as forecast in \cite{editorial}, that small plasma waves 
 (the $ u_k $'s) driven by microscopic instabilities can accumulate over long times to furnish nontrivial effects. In turn, this phenomenon 
 participate in some anomalous transport \cite{MR3582249} and can trigger instabilities which may act as obstructions to the confinement 
 of magnetized plasmas \cite{MR4084146}. Applied in the context of NMR, our result investigates the processes whereby human tissues could 
 be heated during \href{https://en.wikipedia.org/wiki/Magnetic_resonance_imaging}{magnetic resonance imaging} \cite{MR2127862}. 
 
 \smallskip
 
 \noindent It is worth noting that the turbulent aspects which are revealed by Theorem \ref{theo:resumeNL} are inherently linked to spatial heterogeneity. 
 They are caused by the impact of the inhomogeneous source term $ F_{\! L} $, which involves special oscillating wave front sets. Both in SMP 
 and NMR, the input of energy is due to a strong external magnetic field $ \bf B  $, whose directions vary with the spatial positions, see 
 Section \ref{sec:model}.
 
 \smallskip
 
 \noindent Theorem \ref{theo:resumeNL} indicates that Facts \ref{fact1}, \ref{fact2} and \ref{fact3} indeed prevail. We still have two notions 
 of criticality as far as nonlinear effects are concerned: the size of the nonlinearity (through the choice of $ \nu+j_1+j_2 $) and the 
 nature of oscillations (involving the gauge parameter $ \textfrak {g} = \omega+j_1-j_2 $). 
 
\smallskip
 
\noindent The case $\nu+j_1+j_2>2$ corresponds 
 to a nonlinearity whose amplitude is too weak to have effects at leading order, regardless of the gauge. The case $ \nu+j_1+j_2=2 $ 
 corresponds to a nonlinearity with a critical size, for which we have to further investigate the content of the oscillations. 
 
 \smallskip
 
 \noindent It turns out that,  for $ \textfrak {g} \not = 1$, that is for $ \omega+j_1-j_2\not = 1 $, the oscillations in the nonlinear term are 
 not \emph{resonant}. They prevent the nonlinearity from having a leading order contribution. This is why we have \eqref{comparethetwo}.
 
 \smallskip
 
 \noindent In practice, the expression \eqref{consintertheoprin} is built as a sum of wave packets, which may be viewed as corresponding 
 to the terms $ u_k $ of \eqref{third-si-devprimeinter}. But now, the wave packets accumulate only at special positions which, in the space  
 variable $ x $, are located on a moving lattice of size $ \eps $. The complete statement is Proposition~\ref{supnormamplifi}, which takes 
 into account the general choices of $ p$ and $F_{\! L}$ introduced in Section~\ref{sec:model}.

 \smallskip
 
 \noindent By contrast, at all other positions, as indicated in \eqref{desnonlinintertheoprin}, the wave  packets $u_k $ compensate 
 to furnish asymptotic disappearance. This is due to mixing properties induced by the variations of the phase ($ \part_x \varphi \not 
 \equiv 0 $) and dispersive effects ($ p' \not \equiv 0 $), mixing properties which are recorded in the arithmetic properties of a phase 
 shift. This is a feature of the PDE \eqref{third-si-gene}, which is completely absent from the ODE \eqref{third-si}. 
 The full statement can be found in Proposition~\ref{supnormnonamplifi2}. 

\smallskip

\noindent Compare \eqref{third-si-devtt} and \eqref{consintertheoprin}. The characteristic function $ \mathds{1}_{[0,T]} (s) $ 
of \eqref{third-si-devtt} plays the role of $ \aaa(s,0,0) $ inside \eqref{consintertheoprin}. Observe however that the formula 
\eqref{consintertheoprin} differs from \eqref{third-si-devtt}, due to the factor $\exp \bigl( {-i\frac{\ell}{6} (\frac{1}{s}-
\frac{T}{s^2} )} \bigr) $ in front of $ \aaa $. This additional factor is induced by the rate of convergence of $ p'' (\xi) $ 
towards $ 1 $, which appears at the end of line \eqref{hyppreliintrod}. It is absent when $ p \equiv 1 $. In comparison to 
\eqref{third-si-devtt}, due to the presence of an oscillating factor, it can reduce the amplification phenomenon which is 
revealed by \eqref{consintertheoprin}. It reflects some microlocal effect, which is encoded in the behavior of $ p  $, 
on the asymptotic behavior of the solution $  \cU_{\rm lin}  $.

\smallskip

\noindent Remark that the constructive interferences \eqref{consintertheoprin} would be very difficult to detect in Lebesgue 
norms other than $L^\infty$, like $ L^2 $. This is because the asymptotic profile of $ \cU_{\rm lin} $ is nontrivial only on a 
set of Lebesgue measure zero (the lattice $ \ZZ $). To some extent, we can say that the underlying mechanisms rely on 
the recombination of small scales (rapid oscillations) into larger scales, which produces (asymptotically) a very weak solution.

\smallskip

\noindent As already explained, the linear part (1) of Theorem~\ref{theo:resumeNL} is a direct consequence of Propositions~\ref{supnormamplifi}
and~\ref{supnormnonamplifi2}. The proof relies basically on classical stationary and non-stationary phase arguments to precisely describe 
the infinite number ($ k \in \NN $) of emitted signals $ u_k $. But the linear superposition of the $ u_k $ is a quite complicated mechanism. This 
requires to sort between dispersive and almost stationary waves, and this means to carefully examine the phase compensation phenomena 
that occur in the summation process. The integral inside \eqref{consintertheoprin} appears ultimately as the limit of a Riemann sum indexed by $ k $.

\smallskip

\noindent The comparison between the linear solution $\cU^{(0)} \equiv \cU_{\rm lin} $ and the expression $\cU^{(1)}$ is 
a nontrivial test to measure whether or not nonlinear effects can alter the solution at leading order. The content of 
\eqref{comparethetwo} is proved in Section~\ref{sec:sorting-gauge}. According to the choice of $ \textfrak {g} $ or $ \iota \in [0,1] $,
the size of the $ o(1) $ inside \eqref{comparethetwo} may be improved, see Propositions~\ref{sortnonlinaer}, 
\ref{sortnonlinearstrictgauge} and \ref{sortnonlinaergaugeenzero}. In view of Theorem \ref{theo:resumeNL}, nonlinear phenomena 
can be expected only under critical nonlinearities ($ \nu+j_1+j_2=2 $) and resonant oscillations ($ \textfrak {g} = 1 $). 

\smallskip

\noindent General nonlinear source terms will be investigated in Subsections~\ref{sec:settingNL} and \ref{sec:sorting-gauge}. 
But, because it is simpler and already quite illustrative, in Subsection~\ref{sec:nonlineareffectg=1}, we only examine the 
case of $ u^2 $. Other quadratic nonlinearities may be more difficult to resolve. Retain also that, higher-order nonlinearities, 
like the cubic choice $ \vert u \vert^2 u $, appear to be not directly manageable through our approach, see Remark~\ref{criticalcubicrem}.

\smallskip

\noindent Recall that $ F_{\! L} $ has been defined at the level of \eqref{sssooour}. The implementation of $ u^2 $ corresponds at the level of \eqref{eq:NLedp} 
to the selection of $ \lambda = 1 $ and $ (\nu, j_1,j_2 ) = (0,2,0) $, so that $ \omega = - 1 $ (since we want to impose $ \textfrak {g} = 1$). 
Thus, we consider the solution $ u^{(0)} =u_{\rm lin}$ to \eqref{eq:Picard0}, as well as the solution $ u^{(1)} $ to $ u^{(1)}_{\mid t=0} =0 $ together with 
\begin{equation}\label{eq:Picard1}
\part_t u^{(1)} - \frac{i}{\eps} \, p ( - i \eps \part_x) u^{(1)} = F_{\! L} +\chi\Bigl( 3-2\frac{\eps t}{\cT}\Bigr) \chi\(\frac{x}{r\eps^\iota}\)
e^{-it/\eps} \bigl( u^{(0)} \bigr)^2 .
\end{equation}

\begin{theo}[Nontrivial nonlinear effects in the presence of resonances]\label{theo:resumeNLbis} The general context is as in Theorem 
\ref{theo:resumeNL}. We fix $ \nu=1$, $ j_1 =2 $,  $ j_2 = 0 $ and $ \omega = -1 $ to deal with the quadratic source term $ u^2 $ of 
\eqref{eq:Picard1}. It follows that the gauge parameter $ \textfrak {g} = \omega + j_1- j_2 = 1 $ is resonant.
Select some $\iota \in ]\iota_-,1[ $ with $ \iota_- := (13 - \sqrt{89}) / 8 $. Then, for all time $T\in [\cT,2\cT]$ and for all position $ z \in \RR $, 
the expressions $ \cU^{(0)} (T,\cdot) $ and $ \cU^{(1)} (T,\cdot) $ which are issued from \eqref{eq:u/U} after solving \eqref{eq:Picard0} and 
\eqref{eq:Picard1}  have the following asymptotic behaviors when $\eps $ goes to zero. 

\smallskip

\begin{itemize}
\item \emph{Constructive interferences}. When $ z = 2 j $ for some $ j\in \ZZ $, the nonlinear interactions have some effect at 
leading order. As a matter of fact, we find 
\vskip -6mm
\renewcommand\arraystretch{2.5}
\begin{equation} \label{desintertheoprinbis}
\quad \  \begin{array}{ll}
\displaystyle \cW^{(1)}(T,2j) \! \! \! & \displaystyle := \cU^{(1)}(T,2j) - \cU^{(0)}(T,2j) \\
\ & \displaystyle \, = o(1) + A_\eps^4 \int_0^T \chi \bigl( 3-2\frac{s}{\cT} \bigr) \\
& \displaystyle \quad \ \times \( \int_0^{+\infty}\!\!\!\! \int_0^{+\infty} e^{-i\frac{\ell}{6} \frac{T-s}{(\sigma_1+\sigma_2)^2}} 
b(\sigma_1,s) b(\sigma_2,s) d\sigma_1 d\sigma_2 \) ds , 
  \end{array}
  \end{equation}
  \renewcommand\arraystretch{1}
where $ A_\eps^2 $ is as in \eqref{defcoefaeps} and $ b(\sigma,s) := e^{-i\frac{\ell}{6}\(\frac{1}{\sigma}-\frac{s}{\sigma^2}\)}
\aaa(\sigma,0,0)$. 
\vskip 1mm
 \item \emph{Destructive interferences.} By contrast, when $ z \in \RR\setminus 2\ZZ $, the nonlinear interactions are still 
 negligible at leading order in the sense that
  \begin{equation}\label{destrasyana}
\forall \, z \in \RR\setminus 2\ZZ , \qquad |\cW^{(1)}(T,z)| =o(1) . 
  \end{equation}
\end{itemize}
 \end{theo}
 \noindent Theorem \ref{theo:resumeNLbis} means that both constructive and destructive interferences persist in 
  the nonlinear framework. 
  
\smallskip

\noindent The different wave packets $ u_k $ composing $ \cU^{(0)} $ interact through the quadratic term of equation 
\eqref{eq:Picard1}. There are consequently additional nonlinear  effects which are reflected in the triple integral appearing 
in the right hand side of \eqref{desintertheoprinbis}. The nonlinear impact is not obtained, as could be expected by extrapolating 
\eqref{integralformuspecialcase}, that is by just multiplying the linear profiles $ b $ inherited from \eqref{consintertheoprin}. It 
also involves the correlation coefficient $ \exp \( \tfrac{-i \ell (T-s)}{ 6 (\sigma_1 + \sigma_2)^2 }\) $.
 
 \smallskip
 
\noindent It should be emphasized that Theorem~\ref{theo:resumeNLbis} cannot be inferred from Theorem~\ref{theo:resumeNL}, 
even on a formal level, due to the fact that nonlinear effects are quite strong. We will discuss more specifically these aspects 
at the end of Section~\ref{sec:nonlinear}, in Subsection~\ref{sec:nonlineareffectg=1}, where Theorem~\ref{theo:resumeNLbis} is proved. 

 \smallskip
 
\noindent It may seem that the assumptions made to state Theorem~\ref{theo:resumeNLbis} are quite restrictive, for instance: 
the space and time localization of the nonlinearity (through the cut-off function $\chi$), a rather strange lower bound on the 
parameter $\iota$ related to the spatial scale, and the fact that we consider only the first two iterates of a Picard's scheme 
(this last point was already motivated above). Nevertheless, to obtain Theorem~\ref{theo:resumeNLbis}, we need already a 
rather involved analysis and careful estimates to deal with the oscillatory integrals coming from Duhamel's formula. 

 \smallskip
 
\noindent Pursuing the  analysis in order to examine the ``complete'' nonlinear  situation \eqref{third-si-gene} is beyond  the 
scope of this article, see Remark \ref{nonlieafinrem}. 
 
 \smallskip

 \noindent In conclusion, the key innovation of the present article is, in the context of SMP and NMR, a refined analysis of resonances, 
 as well as a subsequent study of related interferences and nonlinear interactions.  This will be done first in a linear setting (Section 
 \ref{sec:lineffect}) and then in a nonlinear framework (Section \ref{sec:nonlinear}).

\subsection*{Acknowledgements} The authors wish to thank the referee for 
a very careful reading of the paper and  numerous constructive comments.


\section{The origin of the model}\label{sec:model}



 The equation (\ref{third-si}) with $ \varphi $ as in (\ref{third-sipourphipha}) first appears in \cite{editorial} as a textbook case when it comes 
 to studying plasma turbulence. It is a very  elementary model aimed at explaining wave-particle interaction \cite{Koch}. In (\ref{third-si}), 
 the ``{\it wave}'' is represented by $ u $ while the influence of ``{\it particles}'' is incorporated at the level of the source term, through the 
 special structure of the phase $ \varphi $ inside $ F_{\! L} $ as well as the choice of the nonlinearity $ F_{\! N \! L} $. 
 
 \smallskip
 
\noindent The content of $ \varphi $, of equation \eqref{third-si}, of $ F_{\! L} $ and of $ F_{\! N \! L} $ must be adjusted in connection with 
physics.  In  this section, we examine two frameworks. The first one deals with strongly magnetized plasmas (SMP); the second is about 
nuclear magnetic resonance (NMR). From these perspectives, the properties of $ \varphi $, \eqref{third-si}, $ F_{\! L} $ and $ F_{\! N \! L} $
selected in Subsection~\ref{subsec:toymodel} are far from sufficient.

 \smallskip
 
 \noindent Both SMP and NMR involve a strong varying external magnetic field $ \bf B (\cdot) $, and both imply rapid oscillations around the field 
 lines generated by $ \bf B (\cdot) $ at a Larmor frequency  which, in the time variable $ t $, is $ \eps^{-1} $ with $ \eps \ll 1 $. In SMP, the gyroscopic 
 motion refers to the dynamics of charged particles, and it is  governed by the \href{https://en.wikipedia.org/wiki/Vlasov_equation}{Vlasov equation}.
 In NMR, this motion concerns the magnetic moment $ \bf M $ that is induced by the spin of particles, and it is handled by 
 \href{https://en.wikipedia.org/wiki/Bloch_equations}{Bloch equations}. 
 
 \smallskip
 
 \noindent These two applications share another remarkable feature. They both entail some secondary slower periodic motion. 
 \begin{itemize}
   \item In SMP like coronas, planetary magnetospheres or fusion devices, the latter comes from the bouncing back and forth of charged particles 
   between two mirror points \cite{Whistler,MR3582249}. 
   \item In NMR, it is generated by the repeated action of many radio frequency excitations \cite{MR2127862}. 
\end{itemize}
 This second time periodic motion emerges at the level of the phase $ \varphi $ through the presence of the periodic function ``$ \cos t $" inside 
 \eqref{third-sipourphipha}. It also appears through the two time scales $ T / \eps $ and $ T / \eps^2 $ in the right hand side of \eqref{eq:EDONL}. But there is more: the 
 spatial inhomogeneities of the field $ \bf B $ generate variations of the phase $ \varphi$ with respect to the variable $ x $. The graph 
 $ \cG $ of the gradient of $ \varphi $, which is defined by \eqref{graphs-grad}, is associated with special {\it Lagrangian manifolds}, whose 
 geometries reflect the peculiarities of $ \bf B$. 

\smallskip

\noindent In SMP, classical choices of $ \bf B $ are the dipole model \cite{Whistler} and the axisymmetric field \cite{MR3582249} which 
are respectively adapted to the description of magnetospheres and tokamaks. In both situations, the condition $ \nabla_x \varphi \not = 0 $ 
results from some spreading of the characteristics. The level surfaces of $ \varphi $ involve very specific patterns. They give rise to wavefronts 
that are isolated and studied in \cite{Whistler,MR3582249}, where they are associated with a self-organization into coherent structures.

\smallskip

\noindent  In NMR, the applied field $ \bf B $ is the sum of a background field $ { \bf B }_0 $, plus a gradient field $ \bf G$ of the form 
$ \beta \cdot x $ with $ \beta \in \RR^3 $ and $ x \in \RR^3 $, plus a time dependent periodic field $ {\bf B}_1 $. In other words
\[  {\bf B} (t,x) = { \bf B}_0 + {\bf G} (x) + {\bf B}_1 (t) , \qquad {\bf G} (x) = \beta \cdot x . \]
In the course of an experiment, 
the static field $ \bf G  $ is turned on and off by selecting a collection of data $ \beta \in \RR^3 $ in view of signal processing. On the 
other hand, the radio frequency excitation  $ {\bf B}_1 $  is
triggered again and again to counterbalance  the effects of noise in
the measurements.  
The property $ \nabla_x \varphi \not = 0 $ is due to the gradient fields $ \bf G $. The corresponding structure of $ \varphi $ is identified 
(without exploitation) in the text \cite{MR2127862}. It will be more highlighted in what follows, see Paragraph \ref{subsec:Blocheq}. 
 
\smallskip
 
 \noindent Whether for SMP or NMR, the function $ \varphi $ is the sum of a linear function in $ t $, plus (locally near the origin) a quadratic
 function in $ (t,x) $, plus a periodic function in $ t $. A representative selection of $ \varphi $ is the one given in \eqref{pharetain}.
 More details are given in the course of this section. Subsection \ref{sec:SMP} is devoted to SMP, while Subsection \ref{subsec:Blocheq} 
 deals with NMR.


\subsection{Resonant wave-particle interactions}\label{sec:SMP} What happens inside collisionless plasmas is basically described by the 
Vlasov-Maxwell system, see \cite{MR4084146} for a specific study concerning the strongly magnetized case. Simplified models (of fluid type)
are also available through magnetohydrodynamics, see for instance the PhD thesis \cite[Appendix~A.2]{fontaine} and the numerous 
references therein. In the latter case, the equations take the form
 \begin{equation} \label{magnetohydrodynamics}
\part_t u + \frac{1}{\eps} L (\eps D_x) \, u + F = \, 0 \, , \qquad u_{\mid t=0} = 0 .
\end{equation}
In Paragraph \ref{subsec:RWI}, we exhibit some specificities of the differential operator $ L (\eps D_x) $, which acts on the {\it wave} $ u $. 
In Paragraph \ref{subsec:ICP}, we explain the features of the source term $ F $, which result from the motion of charged {\it particles}
(electrons or protons). The coupling between $ u $ and $ F $ through \eqref{magnetohydrodynamics} is a way to investigate phenomena 
related to wave-particle interactions \cite{Koch}.


\subsubsection{Plasma dispersion relations}\label{subsec:RWI} 
Here, the spatial dimension is $ d = 3 $. The state variable is $ u = {}^t (B,E,\cJ)\in \RR^9 $. It involves the magnetic field $ B \in \RR^3 $, the
electric field  $ E \in \RR^3 $, and the electric current $ \cJ \in \RR^3 $. Unlike the external fixed magnetic field $ \bf B $, the electromagnetic 
field $ (E,B) $ is self-consistent, and unknown. The wave propagation in strongly magnetized plasmas (SMP) is studied in detail in the articles 
\cite{CheFon,CheFon2}. It can be undertaken through the asymptotic analysis (when $ \eps $ goes to zero) of
\vskip -4mm
\begin{equation*} 
\part_t u + \sum_{j=1}^d \, S_j \, \part_{x_j} u + \frac{1}{\eps} A u + F =  0 ,
\qquad u_{\mid t=0} = 0 .
\end{equation*}
In practice, the number $ \eps^{-1} $ is a large parameter ($ \eps^{-1} \simeq 10^5 $) that is issued from a \href{https://en.wikipedia.org/wiki/Plasma_parameters}
{gyrofrequency}. Now, to recover the formulation (\ref{magnetohydrodynamics}), it suffices to define 
\begin{equation} \label{mhdeq} 
 L (\eps D_x) := \sum_{j=1}^3 S_j \, \eps \part_{x_j} + A= \eps 
 \begin{pmatrix}
   0 & + \nabla_x \times & 0 \\
- \nabla_x \times & 0 & 0 \\
0 & 0 & 0 
 \end{pmatrix} + A .
\end{equation}
In \eqref{mhdeq}, the differential operators $ \pm \nabla_x $ come from Maxwell's equations in vacuum. 
The matrix $A $ can be decomposed into $ 9 $ blocks of size $ 3 \times 3 $ given by (for some constant $ {\rm b}_e $
proportional to the strength of the external magnetic field)
\begin{equation} \label{mhdeqcontent} 
\qquad  A = \begin{pmatrix}
   0 & 0 & 0 \\
0 & 0 & + \mathrm{Id} \\
0 & - \mathrm{Id} & {\rm b}_e \, \Lambda 
 \end{pmatrix}
, \quad \ \Lambda = e_3 \times = 
        \begin{pmatrix}
           0 & -1 & 0 \\
+1 & 0 & 0 \\
0 & 0 & 0 
         \end{pmatrix}
, \quad \ e_3 := 
\begin{pmatrix}
  0 \\
0 \\
1
\end{pmatrix} . 
\end{equation}
The skew-symmetric 
matrix $ A $ can be split into two distinct parts involving $ \pm \mathrm{Id} $ and $ {\rm b}_e \, \Lambda $. The two components 
$ \pm \mathrm{Id} $ are due to the coupling between the charged particles and $ E $. They take into account one aspect of 
wave-particle interactions, arising in the electron cyclotron regime when computing the electric current in the Vlasov-Maxwell 
system. On the other hand, the skew-symmetric matrix $ {\rm b}_e \, \Lambda $ captures the influence of the Lorentz force. 
It corresponds to the effects of a strong external magnetic field having (rescaled) amplitude $ {\rm b}_e $ and fixed direction 
$ e_3 $. To underline the dependence of the semi-classical operator $ L (\eps D_x) $ upon $S$ and $A$, we will sometimes 
denote by $L(S,A,\xi)$ the symbol of this operator. Thus
\begin{equation*}
 L(\xi) \equiv L(S,A,\xi) := i \xi_1 S_1  + i \xi_2 S_2 +  i \xi_3 S_3  + A.
\end{equation*}
In vacuum, when $ A = 0 $, the kernel of $ L(S,0,\xi) $ is (for $ \xi \not = 0 $) of dimension $ 5 $. The situation is different in 
magnetized plasmas, when $ A\not = 0 $. When $ A $ is as in (\ref{mhdeqcontent}), the dimension of $ \mathrm{ker} 
L(S,A,\xi) $ may be $ 2 $ or $ 3 $. In any case, it is strictly less than $ 5 $. This means that some nonzero eigenvalue 
$ \tau_j (S,A,\xi) $ of $ i L(S,A,\xi) $  is connected to $ 0 $ when $ A $ goes to $ 0 $, while the corresponding dispersion 
relation $ \tau_j (S,A,\xi) $ remains bounded for large values of $
\xi $. Emphasis will be placed on such eigenvalue.  

\smallskip

\noindent The characteristic variety associated with (\ref{magnetohydrodynamics}) is
\[
  \operatorname{Char} L = \RR_t\times \RR^3_x\times  \left\{ (\tau,\xi) \in \RR \times \RR^3 \, ; \ \det \bigl( i \tau \operatorname{Id} 
  + L(\xi) \bigr) = 0 \right\} .
  \]
  The analysis of $ \operatorname{Char} L $ in the context of \eqref{mhdeq}-\eqref{mhdeqcontent}  is achieved in the article 
  \cite{CheFon}, with explicit algebraic formulas. The general situation is rather complicated. But, for parallel
propagation,  
meaning that $ \xi = {}^t (0,0,\xi_3) \parallel e_3 $, the computations are simplified. With this in mind, we consider solutions $ u $ 
 which depend only on the third coordinate $ x_3 \in \RR $ so that $ x \equiv x_3 \in \RR $ (we work in space dimension $ d = 1 $) 
and $ \xi \equiv \xi_3 \in \RR $. Then, the dispersion relations issued from \eqref{mhdeq} are displayed in this 
\href{https://en.wikipedia.org/wiki/Electromagnetic\_electron\_wave}{link} \cite{linkwave}, which presents basic features of electron 
waves. As usual in physics, in \cite{linkwave}, the functions $ \tau_j $ are available through implicit relations involving the index 
of refraction $ \xi / \tau $. In particular, one can distinguish the right circular polarization corresponding to R-waves (which are 
sometimes also called \href{https://en.wikipedia.org/wiki/Whistler\_(radio)}{whistler modes})
\begin{equation} \label{condtauxiRini}
\frac{c_0^2 \ \xi^2}{\tau^2} = 1 - \frac{\omega_p^2/\tau^2}{1 -  (\omega_c/\tau)} \, .
\end{equation}
There is also the left circular polarization corresponding to L-waves
\begin{equation} \label{condtauxiLini}
\frac{c_0^2 \ \xi^2}{\tau^2} = 1 - \frac{\omega_p^2/\tau^2}{1 +
  (\omega_c/\tau)} \, . 
\end{equation}
In (\ref{condtauxiRini}) and (\ref{condtauxiLini}), the three constants $ c_0 $, $ \omega_p $ and $ \omega_c $ represent respectively 
the speed of light in vacuum, the \href{https://en.wikipedia.org/wiki/Plasma\_oscillation}{plasma frequency}, and the 
\href{https://en.wikipedia.org/wiki/Electron\_cyclotron\_resonance}{electron cyclotron resonance frequency}. The two conditions 
(\ref{condtauxiRini}) and (\ref{condtauxiLini}) correspond to the selection of two important branches inside $ \text{\rm Char} \, 
L $. The first is issued from (\ref{condtauxiRini}); it is valid only for $ 0 < \tau < \omega_c \, $; and it becomes physically relevant 
when $ \tau $ becomes close to the resonance frequency $ \omega_c $. The second branch comes from (\ref{condtauxiLini}); it 
operates when $ \tau $ is above a cutoff frequency. 

\smallskip

\noindent The two conditions (\ref{condtauxiRini}) and (\ref{condtauxiLini}) can be written in dimensionless form (implying that 
$c_0 = 1 $ and $ \omega_p = \omega_c = 1 $). Concerning 
the relation (\ref{condtauxiRini}), this yields 
\begin{equation} \label{condtauxiR}
\frac{1}{\xi^2} = G_- (\tau) \, , \qquad G_- (\tau) := \frac{\tau - 1}{\tau^2 \, (\tau-1) -\tau} , 
\qquad 0 < \tau < 1 . \quad
\end{equation}
From (\ref{condtauxiLini}), we can extract
\begin{equation} \label{condtauxiL}
\! \frac{1}{\xi^2} = G_+ (\tau) \, , \qquad G_+ (\tau) := \frac{\tau + 1}{\tau^2 \, (\tau + 1) -\tau} , 
\qquad \frac{\sqrt 5 -1}{2}  < \tau . 
\end{equation}
A simple calculation shows that
\begin{equation} \label{Gprime}
\quad  \forall \, \tau \in \, ]0,1[ \, , \quad \ G_-'(\tau) = \Bigl( \frac{- \tau+2}{\tau^2 -\tau-1} + \frac{1}{\tau} \Bigr)' = \frac{(\tau-1) \, 
(\tau-3)}{(\tau^2 -\tau-1)^2} - \frac{1}{\tau^2} \leq -1 ,  
\end{equation}
and that
\begin{equation} \label{Gprimeenc}
\lim_{\tau \rightarrow 0+} G_- (\tau) = + \infty \, , \qquad G_- (1) = 0 \, , \qquad G'_- (1) = -1 . 
\end{equation}
The function $ G_- $ is continuous and strictly decreasing from $ ]0,1] $ onto $ [0 , + \infty [ $. Therefore, it gives rise 
to a diffeomorphism between these two intervals, with inverse function $ G_-^{-1}  $. The whistler dispersion relation 
expresses $ \tau $ as a function of $ \xi $, through $ \tau \equiv \tau_w (\xi) := G_-^{-1} \bigl( \xi^{-2} \bigr) $. This function 
$ \tau_w (\cdot) $ is even. This property does not come from the general condition (\ref{symhyp}), but from other specificities
related to (\ref{mhdeq}). By construction, we have
\begin{equation} \label{doublelimG}
\quad \ \ \lim_{\xi \rightarrow 0 \pm} G_-^{-1} \bigl( \xi^{-2} \bigr) = \tau_w (0) = 0 \, , \qquad \lim_{\xi \rightarrow \pm \infty} 
G_-^{-1} \bigl( \xi^{-2} \bigr) = \lim_{\xi \rightarrow \pm \infty} \tau_w (\xi) = 1 . 
\end{equation}
In (\ref{doublelimG}), the first limit means that the whistler dispersion relation is linked to some zero eigenvalue 
of $ L (S,A,0) \equiv A $. The second limit indicates, as noted before, that it  appears as a perturbation (in terms of $ A $) 
of some zero eigenvalue of $ L(S,0,\xi) $. This is consistent with a bounded behavior of $ \tau $ when $ \vert \xi \vert $ goes 
to infinity. 

\begin{figure}[h]
\begin{minipage}[b]{.95\linewidth}
\includegraphics[scale=0.235,width=12cm,height=4cm]{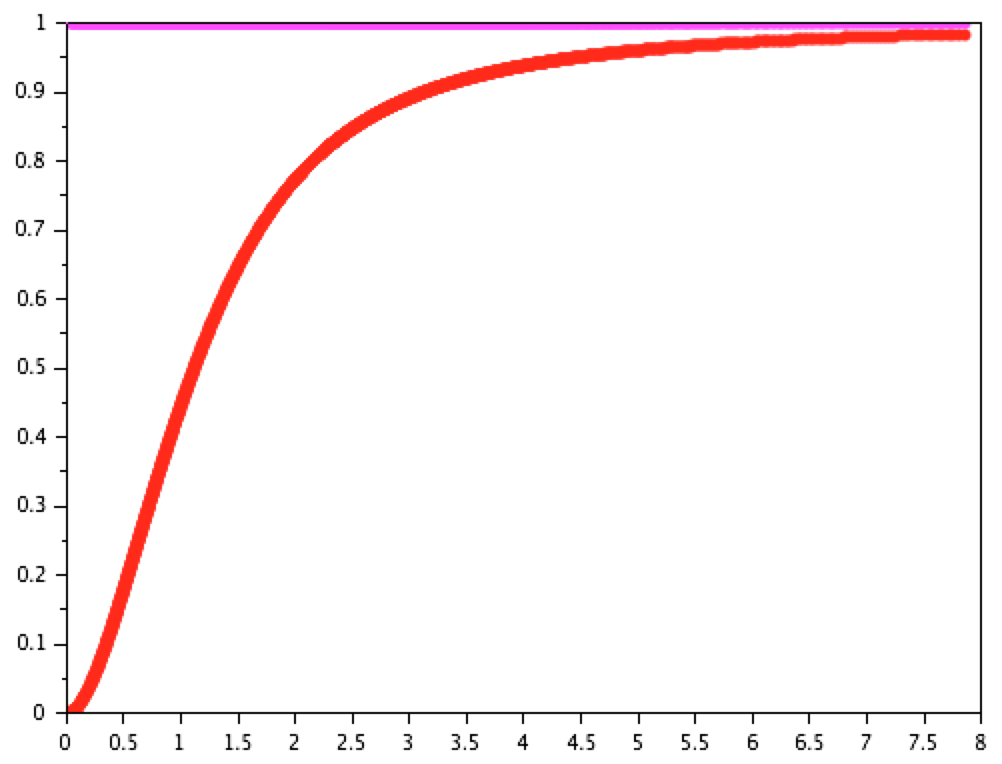} 
\caption{
\color{airforceblue}{\small Graph of the function $ \xi \longmapsto \tau_w (\xi) $ on $ \RR_+ $ in {\color{red}{red}}. 
$ \qquad \qquad \qquad \qquad \qquad \qquad \qquad \qquad \qquad $}
{\color{airforceblue}{\small Asymptotic direction of the dispersion relation in {\color{magenta}{magenta}}.}}
}
\label{graofg}
\end{minipage}
\end{figure}

\noindent The function $ \tau_w  $ connects $ 0 $ (for $ \xi = 0 $) to $ \omega_c \equiv 1 $ (for $ \xi = \pm \infty $), 
in a one-to-one smooth relation. Moreover, we can see on Figure \ref{graofg} that the value $ \tau $ becomes close 
to $ \omega_c \equiv 1 $ on condition that $ \vert \xi \vert $ goes to infinity. In view of applications (see e.g. \cite{ElPa17,SrSh11}), 
the whistler dispersion relation $ \tau_w  $ has more impact near the resonance, that is when $ \vert \xi \vert $ becomes
large enough. 

\smallskip

\noindent On the other hand, the regime is semiclassical. This means that the value $ \xi = 0 $ corresponds 
to a transition zone between spatial frequencies of size $ 1 $ and $ \eps^{-1} $. What happens near $ \xi = 0 $
is therefore physically less significant. For this reason and also to avoid a possible singularity at $ \xi = 0 $, we 
can skip what occurs near $ \xi = 0 $. Thus, we can multiply $ \tau_w $ by $ 1 - \chi $ where $ \chi $ is an even, 
smooth cut-off function which, for instance, is such that 
\begin{subequations}\label{mollichi} 
\begin{eqnarray} 
\qquad &  \displaystyle \forall \, \vert s \vert < 5/8 , \quad \chi(s) = 1 , & \qquad \forall \, s \in \, ] 5/8 , 1[ , \quad 
0 < \chi(s) = \chi(-s) , \label{mollichi1} \\
\qquad &  \displaystyle \forall \, \vert s \vert > 1 , \qquad \chi (s) = 0 , & \qquad \forall \, s \in \, ] 5/8 , 1[ , \quad 
\chi' (s) < 0 . \label{mollichi2} 
\end{eqnarray}
\end{subequations}
 
\begin{exam}[The physical model of R-waves] \label{phyexRwaves} Take $ p \in \cC^\infty ( \RR ; \RR ) $ with 
\begin{equation} \label{reldispright}
\qquad p(\xi) := \bigl( 1 - \chi (\xi) \bigr) \ G_-^{-1} \bigl( \xi^{-2} \bigr) = \bigl( 1 - \chi (\xi) \bigr) \ \tau_w (\xi) .
\end{equation}
\end{exam}

\noindent The function $ p $ inside (\ref{reldispright}) is even; it is equal to $ 0 $ in a neighborhood of $ \xi = 0 $; it 
coincides  with the function $ \tau_w  $ for $ 1 \leq \vert \xi \vert $. Applying Fa\`a di Bruno's formula, we can 
also see that
\begin{equation} \label{asymppon}
\forall \, n \in \NN^* \, , \qquad \lim_{\xi \rightarrow + \infty} \ \xi^{2+n} \ p^{(n)} (\xi) = (-1)^{n+1} \ (n+1) ! 
\end{equation}

\noindent This article is a first mathematical approach of the subject. Thus, we will only consider a scalar wave equation 
in one space dimension ($ d = 1 $), like (\ref{third-si-gene}). In what follows, the special choice (\ref{reldispright}) of 
$ p  $ will serve to guide the discussion.


\subsubsection{The impact of charged particles}\label{subsec:ICP} The source term $ F $ inside (\ref{magnetohydrodynamics}) 
is aimed to collect extra contributions appearing when passing from the Vlasov-Maxwell system to MHD equations. Typically, the 
function $F$ is built with moments  
\[
  \cM_n (f) := \int \underbrace{v \otimes \cdots \otimes v}_{n \
    \text{times}} \ f(t,x,v) \ dv \, , \qquad n \in \NN^*
  \]
of the distribution function $ f(t,x,v) $ satisfying the Vlasov equation. As explained in \cite[Appendix~A.2]{fontaine}, the content
of $ F$ must take into account the underlying physics. In the context of confined magnetized plasmas, the function $ F $ inherits 
from the computation of $ \cM_n (f) $ a special set of characteristics.

\smallskip

\noindent We consider as a first approximation that the expression $ F$  takes the following form
\begin{equation} \label{nonlinform}
F \equiv F \Bigl( \eps, \eps \, t,t,x,\frac{\varphi (t,x)}{\eps},u \Bigr) . 
\end{equation}
We must specify briefly the three-scale, oscillating and nonlinear structure of $ F $. The function $ F (\eps,T, t,x,\theta,u) $ depends 
on the parameter $ \eps \in [0,1] $, on the long time variable $ T := \eps \, t $ with $ T \in [0,\cT] $ for some $ \cT >0$, on the time 
variable $ t \ge 0 $, on the spatial position $ x \in \RR $, on the periodic variable $ \theta \in \TT := \RR / (2 \, \pi \, \ZZ) $, and on the 
state variable $ u \in \RR $. It is a smooth function of class $ \cC^\infty $ of all these variables, on the domain $ [0,1] \times [0,\cT] 
\times \RR^2 \times \TT \times \RR $. In Paragraph \ref{monoaspects}, we explain the origin of $ \varphi $. In 
Paragraph \ref{Osciaspects}, we describe the dependence of $ F(\cdot) $ on $ \theta $ and $ u $.

\medskip

\paragraph {\underline{The} \underline{mono}p\underline{hase} \underline{context}.} \label{monoaspects} Under the influence of a 
strong external magnetic field, the collective motion of charged particles creates {\it coherent  structures} which involve mesoscopic 
oscillations \cite{Whistler,MR3582249}. Through a procedure detailed in \cite{editorial}, when computing the moments $ \cM_n (f) $, 
this furnishes macroscopic oscillations involving a specific phase $ \varphi (t,x) $. As outlined in \cite{editorial}, see the lines (2.7) 
and (3.7) there, the relevant function $ \varphi$ is issued from a mesoscopic gyrophase after freezing the momentum $ v $ at mirror 
points. It can be determined through
\begin{equation} \label{detervarphi}
\varphi(t,x) = \int_0^t {\rm b}_e \bigl( X_r (s,x) \bigr) \ ds ,
\end{equation}

\smallskip

\noindent where the function $ X_r $ can be deduced from gyrokinetic equations or, as in \cite{Whistler,MR3582249}, from a notion 
of {\it reduced} Hamiltonian (the subscript $ r $ in $ X_r $ stands for \underline{r}educed). The function $ X_r $ is smooth, and it can 
be viewed as a flow on $ \RR^d $, with $ d=3 $ in the case of applications. We refer to the articles \cite{Whistler,MR3582249} for more 
details concerning the properties of $ X_r $ in connection with plasma physics, and to \cite{editorial} for a short presentation. 

\smallskip

\noindent In what follows, we will just retain the basic representative features of $ X_r$, and therefore of $ \varphi  $. There is a remarkable 
fact concerning $ X_r$, which is due to underlying integrability conditions. For all $ x $, the function $ X_r (\cdot,x) $ is periodic with respect 
to the first variable $ s$. To simplify the discussion (or after reductions), we can even suppose that the  period of $ X_r (\cdot,x) $ is uniform 
with respect to all positions $ x $, say equal to $ 2 \pi $, so that
\begin{equation} \label{deterXrsati}
\forall \, (s,x) \in \RR \times \RR^3 , \qquad  X_r (0 , x) = x , \qquad X_r (s+2 \pi , x) = X_r (s,x)  .
\end{equation}
The periodic function $ X_r (\cdot,x) $ produces a spatial periodic trajectory, starting  from $ x $ at time $ s=0 $. From 
(\ref{detervarphi}), we can deduce a decomposition of $ \varphi $ separating average and oscillatory parts. The average part is
\[ \langle {\rm b}_e \circ X_r \rangle (x) := \frac{1}{2  \pi} \int_0^{2  \pi} {\rm b}_e \bigl( X_r (s,x) \bigr) \, ds  .  \]
The oscillatory part $ ({\rm b}_e \circ X_r)^* (t,x) $ is $ 2 \pi $-periodic with zero average. It may be defined according to
\begin{equation} \label{decompovarphi}
\begin{array}{rl}
({\rm b}_e \circ X_r)^* (t,x) \! \! \! & \displaystyle = \int_0^t \bigl \lbrack {\rm b}_e \circ X_r (s,x) - \langle {\rm b}_e \circ X_r \rangle (x) 
\bigr \rbrack \, ds \\
\ & = \varphi(t,x) -  \langle {\rm b}_e \circ X_r \rangle (x) \ t .
\end{array}
\end{equation}
Recall that the quantity $ {\rm b}_e (x) $ represents the strictly positive amplitude of the external magnetic field computed at $ x $.
In (\ref{decompovarphi}), the linear part $ \langle {\rm b}_e \circ X_r \rangle (x) \, t $ is produced by the mean effect of the bouncing 
back and forth of charged particles between the mirror points, whereas the oscillating part $ ({\rm b}_e \circ X_r)^* (t,x) \not \equiv 0 $ 
takes into account the variations around this mean value. By definition, given $ x \in \RR^3 $, the latter term is of mean value zero 
with respect to $ t \in \TT $.  Observe that
\begin{equation} \label{decompovarphiderive}
\part_t \varphi(t,x) = {\rm b}_e \circ X_r (t,x)= \langle {\rm b}_e \circ X_r \rangle (x) + \part_t ({\rm b}_e \circ X_r)^* (t,x) > 0 .
\end{equation}
Working in the vicinity of a fixed position, say near the origin $ x = 0 $, we can roughly replace the part $ \langle {\rm b}_e \circ X_r 
\rangle (\cdot) $ by
\begin{equation} \label{pha}
\langle {\rm b}_e \circ X_r \rangle (x) \simeq \alpha +  \beta \cdot x \, , \qquad (\alpha, \beta) \in \RR \times \RR^3 ,
\end{equation}
with the following identifications
\begin{equation} \label{identi-fication}
\alpha = \langle {\rm b}_e \circ X_r \rangle (0) >0\, , \qquad  \beta = \nabla_x \langle {\rm b}_e \circ X_r \rangle (0) 
\in \RR^3 \setminus \{ 0 \} . 
\end{equation}
In fact, for the purpose of our analysis, the choice of $ \varphi $ or $ - \varphi $ does not matter. The remark (\ref{decompovarphiderive})
just says that $ \part_t \varphi $ has a constant sign which is positive under the convention (\ref{detervarphi}) since the function 
$ {\rm b}_e  $ is positive. As a consequence, we have to retain that $ \alpha >0 $ at the level of (\ref{identi-fication}). 

\smallskip

\noindent The dimensionless quantity $ \alpha $ has a physical meaning. It is a measure of the ratio between the size of the magnetic field 
and the cyclotron resonance frequency $ \omega_c $. Thus, when $ \omega_c = 1 $, as will be assumed later after rescaling 
(the aim of this arbitrary choice is just to simplify notations), the number $ \alpha $ indicates the average amplitude of 
the (rescaled) external magnetic field. Now, resonances arise when $ \alpha \sim \omega_c $. For this reason, we
select the value $ \alpha = 1 $.

\smallskip

\noindent In practice, both functions $ {\rm b}_e  $ and $ X_r $ are nontrivial functions of $x$. This is why we set $  \beta 
\not = 0 $ at the level of (\ref{identi-fication}). In fact, the inhomogeneities of the external magnetic field induce some spreading 
of the integral curves which are associated to the Vlasov equation. This is reflected in the term $  \beta \cdot x \not \equiv 0 $ 
of (\ref{pha}). Without loss of generality, after spatial rotations and rescalings, we can always adjust $  \beta $ so that $  
\beta = (0,0,-1) $. For solutions which depend only on the direction $ x_3 $, as it was supposed before, we just find 
$ \beta = -1 $. Finally, as a prototype of a nontrivial periodic function with zero mean, we can take
\begin{equation} \label{deritofvarphi}
\part_t ({\rm b}_e \circ X_r)^* (t,x) = \, - \, \gamma \, \sin \, t .
\end{equation}
Combining (\ref{pha}), (\ref{identi-fication}) and (\ref{deritofvarphi}), the positivity condition (\ref{decompovarphiderive}) is satisfied, 
at least for small enough positions $ x $, on condition that $ 0 < \gamma < \alpha = 1 $. To work on a spatial domain where the 
amplitude $ {\rm b}_e (\cdot) $ is expected to remain of magnitude comparable to the mean value $ \alpha $, we fix $ \gamma $ 
in the interval $ ]0,1/4[ $. This assumption turns out to be rather convenient for the forthcoming computations. It ensures that only 
one harmonic is resonant. The more general case $ \gamma >0$ would be more complicated. It may lead to supplementary 
dynamics compared to the one described in this paper. Since $ \varphi (0,\cdot) \equiv 0 $, the preceding discussion indicates that a choice 
of $ \varphi $ which should be relevant from the viewpoint of applications is given by (\ref{pharetain}).

\medskip

\paragraph {\underline{Nonlinear} \underline{as}p\underline{ects}.}\label{Osciaspects}
In the articles \cite{Whistler,MR3582249}, the function $ f(t,x,v) $ is obtained as the composition of a  
localized initial data $ f_0 (x,v) $ with the oscillatory flow that is  issued from the Vlasov equation. It follows that all 
harmonics $ m \varphi $ with $ m \in \ZZ $ are necessarily involved. Accordingly, the periodic function $ F
(\eps,T,t,x,\cdot,u) $ can be decomposed  in Fourier series
\begin{equation*} 
F (\eps,T,t,x,\theta,u) = \sum_{m\in \ZZ} F_m (\eps,T,t,x,u) \ e^{i  m \theta} . 
\end{equation*}
The MHD equations resulting from the Vlasov-Maxwell system are not closed. Some approximations are needed to recover 
self-contained equations. They usually are made in the form of nonlinearities. In the scalar setting (\ref{third-si}), this means to consider 
that the source term $ F$ is semilinear in $ u $ and $ \bar u $. The function $ F $ is made up of a part $ F_{\! L} $ which is 
affine with respect to $( u,\bar u) $, plus some nonlinear part $ F_{\! N \! L} $. We can decompose $ F_{\! L} $ into $ F_{\! L} 
(\cdot,u) = F_{\! L}^0  + F_{\! L}^1  \, u + \underline {\ } \! \! F_{ L}^1  \, \bar u $ with 
\begin{equation} \label{fnlineaajout}
F_{\! L}^0 (\cdot) := F (\cdot,0)  , \qquad F_{\! L}^1 (\cdot) := \part_u F (\cdot,0) , \qquad \underline {\ } \! \! F_{ L}^1 (\cdot) := 
\part_{\bar u} F (\cdot,0) . 
\end{equation}

\begin{rem}[About the elimination of $ F_{\! L}^1 $] \label{eliminationdeF} 
The influence of $ F_{\! L}^1 $ can sometimes be removed. This can be achieved for instance by 
modifying the dependence on $ t $ inside $ F_{\! L}^0 $, $ \underline {\ } \! \! F_{ L}^1 $ and $ F_{\! NL} $. To simplify, assume 
that  $ F_{\! L}^1 $ does not depend on $ (x,\theta) $ but only on $ (T,t) $. Then, define
\begin{equation} \label{changetoeliminate}
 v (t,x) := e^{- i \, \theta (t) / \eps } u (t,x) , \qquad \theta (t) := - \, i \, \eps \int_0^t F_{\! L}^1 (\eps,\eps s ,s) \, ds . 
\end{equation}
The new function $ v  $ solves
\begin{equation*} 
\qquad \part_t v - \frac{i}{\eps} p (\eps D_x) \, v + \, e^{- i \theta (t) / \eps } \, F_{\! L}^0 + e^{- 2i \theta (t) / \eps } \, \underline 
{\ } \! \! F_{ L}^1 \bar v + e^{- i \theta (t) / \eps } \ F_{\! N \! L} \bigl( \eps , \eps \, t , t , e^{+ i \, \theta (t) / \eps } \, v \bigr) .   
\end{equation*}
In particular, when $ F_{\! L}^1 (\eps,T,t) = \eps^{-1} \tilde F_{\! L}^1 (t) $ with a function $ \tilde F_{\! L}^1  $ purely 
imaginary and periodic in $ t $ with mean zero, the above expression $ \theta  $ becomes a periodic real valued 
function. This means that the gauge transformation \eqref{changetoeliminate} can introduce in the source term of \eqref{third-si} 
oscillations with a phase similar to \eqref{third-sipourphipha}. In other words, the oscillations of \eqref{third-sipourphipha} can 
appear after a procedure aimed to absorb the ``potential'' $ F_{\! L}^1 $, even if such oscillations are not visible at first sight. 
This provides another motivation for implementing phases $ \varphi$ like in \eqref{pharetain}.
\end{rem}

\noindent The nonlinear part  $ F_{\! N \! L}$ is chosen of the same form as in the introduction. As will be seen, an 
oscillatory source term such as (\ref{nonlinform}) does generate oscillations of the solution $ u $. By a mechanism 
similar to (\ref{third-si-devprime})-(\ref{third-si-dev}), the function $ u $ can be viewed as a sum of oscillating waves
$ u_k $. Note that, in the present context, nonlinear aspects can be revealed at the level of the source term $ F $ 
(coming from Vlasov) through the harmonics of $ \varphi $ but also inside the wave $ u $ itself (related to Maxwell) 
through the harmonics of the phases involved by the $ u_k $'s.

\smallskip

\noindent In the case of the Earth's magnetosphere, the emission of whistler waves $ u_k $ is a long-standing 
experimental evidence, coming back to works of  
\href{https://fr.wikipedia.org/wiki/Heinrich_Barkhausen}{H. Barkhausen} in 1917,
\href{https://en.wikipedia.org/wiki/Thomas_Eckersley}{T.L. Eckersley}
in 1935, and L.R.O. Storey in 1953 \cite{Storey}.
Thanks to progress in satellite means, like  
\href{https://fr.wikipedia.org/wiki/Van_Allen_Probes}{Van Allen Probes} of NASA or 
 \href{https://fr.wikipedia.org/wiki/Cluster_(satellite)}{Cluster} of ESA, whistler waves $ u_k $ can be today 
observed in detail. The internal mechanisms underlying the production of the $ u_k $'s are clarified in \cite{Whistler}.

\smallskip

\noindent In practice, the whistler waves $ u_k $ accumulate and form a  
\href{https://science.nasa.gov/science-news/science-at-nasa/2012/28sep_earthsong/}{chorus}.
Theorems~\ref{theo:resumeNL} and \ref{theo:resumeNLbis} describe intermittency phenomena that can occur during 
this process. It should be stressed that our results deal with a plasma turbulence which by nature is {\it anisotropic}. The 
   energizing external field $ {\bf B} $ points in special directions, which undergo variations according to
   specific geometries (revealed by the Lagrangian manifold $ \cG $). We will not further investigate here the potential 
   implications of our analysis in terms of plasma physics. We just refer to the text \cite{editorial} for preliminary
   comparisons between mathematical previsions and concrete observations.


\subsection{Nuclear magnetic resonance}\label{subsec:Blocheq}
The general framework concerning  
\href{https://en.wikipedia.org/wiki/Nuclear_magnetic_resonance}{Nuclear Magnetic Resonance} (NMR) is well explained in 
a paper of C.L. Epstein, see \cite{MR2127862} and the references therein. The NMR experiments are intended (for instance) 
to determine the distribution of water molecules in an extended domain. To this end, an external magnetic field $ {\bf B} (t,x) $ 
is repeatedly applied to the object under examination. At a microscopic level, the spins of particles react to $ {\bf B} $ by producing 
a magnetization field $ {\bf M} (t,x) $. 

\smallskip

\noindent The time evolution of $ {\bf M} $ is described by 
\href{https://en.wikipedia.org/wiki/Bloch_equations}{Bloch equations} which, in the absence
of relaxation terms, take the following form
\begin{equation} \label{timevomagreso}
\frac{d \bf M}{dt} =  g {\bf M} \times {\bf B} .
\end{equation}
The function $ {\bf B}$ is usually viewed as a sum $ {\bf B} = {\bf B}_0 + {\bf G} (x) + {\bf B}_1 (t) $, where $ {\bf B}_0 $ 
is a constant background field, $ {\bf G} $ is a gradient field which is collinear with $ {\bf B}_0 $, whereas $ {\bf B}_1$ 
is a time dependent radio frequency field which is orthogonal to $ {\bf B}_0 $ and $ {\bf G} (x)$. Typically \cite{MR2127862}, the 
components $ {\bf B}_0 $ and $ {\bf G}$ can be adjusted according to
$$ {\bf B}_0 = {}^t (0,0, b_0) , \qquad {\bf G} (x) = {}^t (0,0, \beta \cdot x) , $$
where $ b_0 >0 $ is the strength of the external magnet, and where the vector $  \beta \in \RR^3 $ comes from a collection
of static fields that are turned on and off. In hospital magnetic resonance imaging devices, acceptable orders of magnitude 
\cite{MR2127862} are
\begin{equation} \label{characdataNRM}
g \simeq 10^7 - 10^8 \ \text{rad/T\!esla} , \qquad b_0 \simeq 1 - 10 \ \text{T\!eslas} .
\end{equation}
In the absence of $ {\bf B}_1$, at the position $ x $, the
magnetization $ {\bf M} (\cdot,x) $ undergoes a (counterclockwise)  
precession about the $ z $-axis with the angular frequency 
\begin{equation*} 
\omega (x) :=  \omega_0 + g  \beta \cdot x , \qquad \omega_0 := g b_0 \simeq 10^8 .
\end{equation*}
 This underscores the importance of the small parameter $ \eps := \omega_0^{-1} \ll 1 $, which is the inverse of the {\it Larmor frequency}
 $ \omega_0 $. The field $ {\bf B}_1 $  represents the repeated action during the experiments of RF-excitations, which are all adjusted 
near the resonant frequency $ \omega (x) $. As explained in \cite{MR2127862}, this can be modeled by
\renewcommand\arraystretch{1.2}
\begin{equation} \label{RFexciNRM}
{\bf B}_1 (t) = b_1 (t,x) \left( \begin{array}{c}
+ \cos \, \bigl( \omega (x) t + c_1 (t,x) \bigr) \\
- \sin \, \bigl( \omega (x) t + c_1 (t,x) \bigr) \\
0
\end{array} \right) .
\end{equation}
\renewcommand\arraystretch{1}
\vskip -2mm
\noindent The amplitude $ b_1 $ of the field $ {\bf B}_1 $ is a scalar function. To model the repetition of measurements, 
which is aimed to reduce noise effects, the function $ b_1 (\cdot,x) $ is chosen periodic in $ t $ (say of period $ 2 \pi $). 
Define
\begin{equation} \label{defined111}
 d_1 (t,x) := g \int_0^t b_1 (s,x) ds = g t \bar b_1 (x) + g  b_1^* (t,x) , \qquad \quad 
 \end{equation}
where
$$ \bar b_1 (x) := \frac{1}{2 \pi} \int_0^{2 \pi} b_1 (s,x) ds , \qquad b_1^* (t,x) := b_1 (t,x) - \bar b_1 (t,x) = b_1^* (t+2 \pi,x) . $$
In (\ref{RFexciNRM}), the phase shift $ c_1 $ takes into account the small 
variations occurring when calibrating the frequency of the RF-pulse. The equation (\ref{timevomagreso}) can be 
interpreted by following the motion of $ {\bf M} (\cdot,x) $ in a frame rotating at the frequency $ \omega(x) $. This amounts to replacing $ \bf M $ 
by the new unknown  
$$ {\bf N} (t,x) := e^{- \omega(x) t \Lambda } {\bf M} (t,x) , \qquad \Lambda := \left( \begin{array}{ccc}
 0 & 1 & 0 \\
 -1 & 0 & 0 \\
0 & 0 & 0
\end{array} \right) . $$
The equation satisfied by  $ {\bf N} $ is simply
\begin{equation*} 
\frac{d \bf N}{dt} =  g b_1 (t,x) \left( \begin{array}{ccc}
 0 & 0 & \sin c_1(t,x) \\
0 & 0 & \cos c_1(t,x) \\
- \sin c_1(t,x) & - \cos c_1(t,x)  & 0
\end{array} \right)  {\bf N} .
\end{equation*}
When $ c_1$ does not depend on $ t $, for instance when $ c_1\equiv 0 $, the solution is explicit. With $ d_1 $ as in (\ref{defined111}), 
the solution operator $ {\bf U} $ associated with (\ref{timevomagreso}) is given by
\begin{equation} \label{formpourUNRM} 
{\bf U} (t,x) = \left( \begin{array}{ccc}
+ \cos (\omega t) & \sin (\omega t) & 0\\
- \sin (\omega t) & \cos (\omega t)  & 0 \\
0 & 0 & 1
\end{array} \right)
\left( \begin{array}{ccc}
 1 & 0 & 0 \\
0 & + \cos d_1 & \sin d_1 \\
0 & - \sin d_1 & \cos d_1
\end{array} \right) . 
\end{equation}
The formula (\ref{formpourUNRM}) reveals the role of the phase $ \omega (x) t $ and also, after linearization, the presence
in the description of $ {\bf U} (t,x) $ of the two extra phases $ \omega (x) t \pm d_1 (t,x) $. In coherence with (\ref{characdataNRM}), 
we can take $ b_0 = 1 $, so that $ \omega_0 \equiv g \equiv \eps^{-1} $. In one space dimension (when the vectors $  \beta $ 
have a fixed direction), there remains $  \beta \cdot x = \beta x $ with $ \beta \in \RR $ and $ x \in \RR $. Then, for the special 
choices $ \beta = -1 $, $ \bar b_1 \equiv 0 $ and $ b_1^* \equiv \gamma (1-\cos t) $, we just find $ \omega t - d_1 = \varphi / \eps $ 
with $ \varphi $ exactly as in (\ref{pharetain}). The solution $ {\bf M} $ to (\ref{timevomagreso}) does oscillate at the frequency 
$ \eps^{-1} $ according to the phase $ \varphi $. Thus, by plugging $ {\bf M} $ into Maxwell's equations through the magnetization 
current $ {\bf J}_m := \nabla \times {\bf M} $, we end up with a model similar to (\ref{magnetohydrodynamics}). Note also that the 
influence of such oscillating function $ \bf M $ in the source term of Maxwell's equations can also appear in the context of 
Maxwell-Landau-Lifshitz equations \cite{MR3227510}. 

\begin{rem}[Dispersion relations in human tissues] \label{disphumantissues} In the context of NMR, the relevant functions 
$p  $ do not appear to have been modeled precisely. But, like in SMP, the NMR experiments involve a magnetized 
medium. As a consequence, the corresponding dispersion relations should share common characteristics. Be that as it may, 
both situations involve hyperbolic systems to depict wave propagation. And, as will be seen in the next section, the 
properties of $p  $ which have been introduced in Subsection~\ref{sec:SMP} are fairly general in such a framework.
\end{rem}

\noindent  The production of the $ u_k $'s corresponds to some electromagnetic radiation. When dealing with magnetic 
resonance imaging, this may contribute to the heating of human tissues in a way which could be a consequence of 
Theorems~\ref{theo:resumeNL} and \ref{theo:resumeNLbis}.


\section{General setting and assumptions}\label{sec:setting}

In this section, we present a general framework which includes the previous two examples: strongly magnetized plasmas and nuclear
magnetic resonance. We also gather the assumptions that will be retained in the rest of the analysis. 

\subsection{The evolution equation}\label{pseudo-differential}

\noindent In this subsection, we start with $ N \in \NN^* $ state variables, so that $ u \in \RR^N $. The time variable is $ t \in \RR $. The  
spatial dimension is $ d \ge 1 $, so that $ x \in \RR^d $. The general context is based 
on dispersive nonlinear geometric optics \cite{MR1435679,MR1657485}. Consider a system of equations having the form
\begin{equation} \label{ondenonlinsymmetrizable} 
\part_t u + \frac{1}{\eps} L (\eps D_x) \, u + F \, = \, 0  , \qquad u_{\mid t=0} = 0  . 
\end{equation}
The real number $ \eps \in \, ]0,1] $ is a small parameter ($ \eps \ll
1 $). The dual
variables of $ t $ and $ x $ are denoted by  
$ \tau \in \RR $ and $ \xi \in \RR^d $, respectively. 
In Paragraph~\ref{AssumptionsL}, we describe precisely the content of $ L (\eps D_x) $. In Paragraph~\ref{redscalar},  we decompose 
(\ref{ondenonlinsymmetrizable}) into a diagonal system of transport equations, which are coupled through semilinear terms. Then, we 
make a strong decoupling assumption to work with $ N = 1 $, and we restrict our attention to one dimensional effects so that $ d = 1 $.
This reduction process leads to (\ref{third-si-gene}). At the end of this subsection, in Paragraph~\ref{solvescalar}, we explain how to 
express the solution to (\ref{third-si-gene}) as an  oscillatory integral.


\subsubsection{The pseudo-differential operator $L (\eps D_x) $}\label{AssumptionsL} The semiclassical symbol 
that is associated to $L (\eps D_x) $ is a matrix $ L (\xi) $. We suppose that (\ref{ondenonlinsymmetrizable}) is 
symmetrizable, that is, $ L (\xi) $ is antihermitian. Typically, we work with systems that can be reduced to the  
following symmetric form
\begin{equation} \label{ondenonlin} 
\part_t u + \sum_{j=1}^d \,S_j \, \part_{x_j} u + \frac{1}{\eps} A u + F \, = \, 0  ,
\qquad u_{\mid t=0} = 0  .
\end{equation}
In (\ref{ondenonlin}), the letters $ S_j $ represent real-valued {\underline s}ymmetric matrices. On the other hand,
the matrix $ A $ may be complex-valued and is {\underline a}ntihermitian. In other words
\begin{equation} \label{symhyp} 
\forall \, j 
\in \{1,\cdots,d \} \, , \qquad S_j  = {}^t S_j  \, , \qquad A^* = {}^t \bar A = - A .
\end{equation}
The symbol associated to (\ref{ondenonlin}) is 
\begin{equation*} 
L (S,A,\xi) \, := \sum_{j=1}^d \, i \, \xi_j \, S_j + A = - L (S,A,\xi)^* . 
\end{equation*}
The matrix-valued symbol $ i \, L (S,A,\xi) $ is hermitian. It is therefore diagonalizable, with real eigenvalues $ \tau_j (S,A,\xi) $
satisfying
\begin{equation} \label{Lipschitzeigen} 
\tau_j (S,A,\xi) = \tau_j ( S,0 , \xi) + \vert \xi \vert \ \left \lbrace \tau_j \Bigl( S, \frac{A}{\vert \xi \vert} , \frac{\xi}{\vert \xi \vert} \Bigr) 
- \tau_j \Bigl( S, 0 , \frac{\xi}{\vert \xi \vert} \Bigr) \right \rbrace . 
\end{equation}
The function $ \tau_j  $ is Lipschitz continuous on compact sets, including the compact neighborhoods of $ \bigl \lbrace 
(S,0) \bigr \rbrace \times {\mathbb S}^{d-1} $. With this in mind, exploiting (\ref{Lipschitzeigen}) and assuming a little more 
regularity in the variable $ A $ near the position $ (S,0, \sigma) $ with $ \sigma = \xi / \vert \xi \vert $, we can infer that
\begin{equation} \label{Lipschitboundbiss}
\tau_j (S,A,\xi) = \tau_j ( S,0 , \xi) + ( A \cdot \nabla_{\! A} ) \tau_j ( S, 0 , \sigma ) + \cO \bigl( \vert \xi \vert^{-1} \bigr) . 
\end{equation}
By this way, the expression $ \tau_j (S,A,\xi) $ appears for large values of $ \vert \xi \vert $ as a bounded perturbation of the 
eigenvalue $ \tau_j (S,0,\xi) $. Since $ \tau_j (S,0,\cdot) $ is homogeneous of degree one in $ \xi $, the directions $ \sigma 
\in {\mathbb S}^{d-1} $ with $ \tau_j ( S,0 , \sigma) \not = 0 $ give rise to symbols $ \tau_j (S,A,\xi) $ which tend to 
$ \pm \infty $ when $ \xi = \lambda \sigma $ goes to infinity (when $ \lambda \rightarrow + \infty$). On the contrary, the 
directions $ \sigma \in {\mathbb S}^{d-1} $ such that 
\begin{equation} \label{bagagegecz} 
 \tau_j (S,0,\sigma) = 0
\end{equation}
furnish eigenvalues $ \tau_j (S,A,\xi) $ of $  i \, L (S,A,\xi) $ satisfying
\begin{equation} \label{Lipschitbound}
\lim_{\lambda \rightarrow + \infty} \ \tau_j (S,A,\lambda \sigma) = \tau_j^\infty (S,A,\sigma) := ( A \cdot \nabla_{\! A} ) \tau_j ( S, 0 , \sigma ) . 
\end{equation}
The condition (\ref{Lipschitbound}) appears already in Paragraph \ref{subsec:RWI} at the level of (\ref{doublelimG}) when 
studying wave propagation in magnetized plasmas. It is reflected at the level of Figure \ref{graofg} by the horizontal 
asymptotic line. In what follows, the focus will be on such situations, which in fact have a general scope.

\begin{rem}[Omnipresence of a finite limit]\label{zeroeigenvalues} Zero eigenvalues $ \tau_j (S,0,\xi) = 0 $ of $ i \, L (S,0,\xi) $ 
are nearly always present in the evolution equations of mathematical physics. In the most favorable cases, there is a number of indices $ j $ 
such that \eqref{bagagegecz} is verified for all directions $ \sigma \in {\mathbb S}^{d-1} $. Otherwise, fix any $ \sigma \in 
{\mathbb S}^{d-1} $. Then, change $ x $ into $ x- t \tau_j (S,0,\sigma) \sigma $, and modify the solution $ u$ accordingly. By this way,  
it can be ensured that $ \tau_j (S,0,\sigma) \equiv 0 $. Consequently, given $ \sigma \in {\mathbb S}^{d-1} $, the existence 
of a finite limit as in \eqref{Lipschitbound} is (modulo adequate transformations) systematic. 
\end{rem}

\noindent From now on, the matrices $ S_j $ and $ A $ are fixed, with $ A \not = 0 $. The symbol $ L (S,A,\xi) $
is simply denoted by $ L \equiv L(\xi) $. We assume that, for $ \xi \not = 0 $, the matrix $ L(\xi) $ has exactly 
$ \tilde N \ge 1 $ (with $ \tilde N \leq N $) distinct eigenvalues which are of constant multiplicity,   denoted by 
$ - i \, \tau_j (\xi) $ with $ j = 1 , \ldots, \tilde N $. The characteristic variety which is issued from $L (\eps D_x) $ 
is $ \operatorname{Char} (L) := \RR_t \times \RR_x^d \times  \cV $ with 
\begin{equation} \label{caldeter} 
\cV \, := \bigl \lbrace (\tau,\xi) \in \RR \times \RR^d \, ; \ \det \bigl( i  \tau  \operatorname{Id} + L(\xi) \bigr) = 0 \bigr \rbrace .
\end{equation}
By construction, the set $ \cV $ consists of a finite number $ \tilde N $ of smooth sheets, which correspond to 
different branches of $ \operatorname{Char} (L) $, and which are nonintersecting except possibly at the position $ \xi = 0 $. We have 
$ (t,x,\tau,\xi) \in \operatorname{Char} (L ) $ if and only if $ \tau = \tau_j (\xi) $ for some $ j $. The functions $ \tau_j  $ are called 
\emph{dispersion relations}. They are smooth away from $ \xi = 0 $. Retain that $ \tau_j \in \cC^\infty \bigl( \RR^d \setminus 
\{0\} ; \RR \bigr) $. 


\subsubsection{Reduction to a scalar equation}\label{redscalar} The aim here is to explain how to pass from the system 
(\ref{ondenonlinsymmetrizable}) to a finite number of scalar equations.
\begin{lem}\label{lem:passage}
  For all $ \xi \in \RR^d \setminus \{ 0 \}  $, the normal matrix $ L (\xi) $ is unitarily 
similar to a diagonal matrix $ D(\xi) $. In other words
  \begin{equation*} 
\exists \, U (\xi) \in \cU (N) \, ; \qquad U (\xi)^{-1} = U (\xi)^*  \, , \qquad U (\xi) \, L(\xi) \, U (\xi)^* = D(\xi) . 
\end{equation*}
In addition, the function $U  $ can be chosen smooth away from $\{\xi=0\}$, bounded
as well as all its derivatives, and it has a non-zero limit in each direction
\begin{equation}\label{limhomodeU}
  \forall \sigma \in {\mathbb S}^{d-1},\qquad  \exists U_{\sigma} \in \cU (N) \, ; \qquad \lim_{\lambda \to + \infty} U(\lambda 
  \sigma)=U_{\sigma}.  
\end{equation}
\end{lem}
\begin{proof}
  The diagonalisation is straightforward since $L(\xi)$ is antihermitian. The smoothness of $ \xi \mapsto U(\xi) $ follows from the 
  constant multiplicity assumption. It remains to focus on (\ref{limhomodeU}). In the case
  $A=0$, the function $U $ is homogeneous of degree zero, and we have (\ref{limhomodeU}) with $ U_{\sigma} = U 
  (\sigma) $. The lemma then follows from perturbative arguments. For large values of $|\xi|$, the contribution issued from 
 introducing  $ A $ yields only $ \cO (1/ \vert \xi \vert ) $ terms, as in the previous paragraph. 
\end{proof}

\noindent To avoid a possible singularity of $ p  $ at $ \xi = 0 $, we introduce a cut-off function 
$ \underline \chi_c $ near the zero frequency. In the absence of singularity at $ \xi = 0 $, just take $\xi_c=0$ and $\underline 
\chi_c\equiv 0 $. In the presence of a singularity at $ \xi = 0 $, fix some $\xi_c>0$, and select a smooth even cut-off function 
$ \underline \chi_c  $ satisfying
\begin{equation}\label{chicestc}
  \mathbf 1_{[-\xi_c,\xi_c]}\le \underline\chi_c\le  \mathbf    1_{[-2\xi_c,2\xi_c]} .
\end{equation}
The diagonal entries of the matrix $ D(\xi) $ are the imaginary numbers $ - i \, \tau_j (\xi)  $. Apply the operator $ \bigl \lbrack 1 -  
\underline\chi_c( \eps D_x ) \bigr \rbrack U (\eps D_x) $ on the left side of (\ref{ondenonlin}). Accordingly, define the 
modal decomposition 
\begin{equation*} 
{}^t (u_1, \cdots, u_N) := \bigl \lbrack 1 - \underline\chi_c ( \eps D_x ) \bigr \rbrack U (\eps D_x) \ u \in \RR^N . 
\end{equation*}
We emphasize again that this (possible) cut-off near $ \xi = 0 $ is consistent with physical approaches. In practice, the 
dispersion relation near the zero frequency requires often a distinct treatment.  Examples include the Alfv\'en wave regime as 
opposed to the Whistler wave regime, see e.g. \cite{ElPa17,SrSh11}. By this way, the PDE (\ref{ondenonlin}) is reduced to a 
coupled system of $ N $ scalar equations 
\begin{equation} \label{ondenonlinscainterm} 
  \part_t u_n - \frac{i}{\eps} p_n (\eps D_x) \, u_n + \bigl \lbrack 1 - \underline\chi_c (\eps D_x) \bigr \rbrack \, \bigl \lbrack U 
(\eps D_x) \, F \bigr \rbrack_n \, = \, 0 \, , \quad 1 \leq n \leq N . 
\end{equation}
In (\ref{ondenonlinscainterm}),  the pseudo-differential operator $ p_n (\eps D_x) $ involves a symbol $ p_n $ which 
is a smooth function that comes from a dispersion relation $ \tau_j $. More precisely
\begin{equation} \label{multipnparchi} 
\exists \, j \in \{ 1, \cdots , \tilde N \} \, ; \qquad p_n (\xi) = \bigl \lbrack 1 - \underline\chi_c (\xi) \bigr \rbrack \ \tau_j (\xi) .
\end{equation}
The multiplication by $ 1 - \underline\chi_c (\xi) $ inside \eqref{multipnparchi}  allows a smooth connection of the values $ \tau_j (\xi) $ 
for $ \xi \not = 0 $ to the value $ p_n (0)= 0 $ for $ \xi = 0 $. Now, we can remove the index $ n $. The above polarization and 
microlocalization procedure does highlight the important role of scalar equations of the type
\begin{equation} \label{ondenonlinsca} 
\qquad \part_t u - \frac{i}{\eps} p (\eps D_x) \, u + \zeta (\eps D_x)  F  =  0  , \qquad u_{\mid t=0} = 0 , 
\end{equation}
where $\zeta(\xi) := \bigl(1-\underline\chi_c (\xi) \bigr) U(\xi) $. In (\ref{ondenonlinsca}), there is only one mode of propagation.
We have $ u \in \RR $ and $ N = 1$. Moreover, assuming that the function $ F  $ depends only on $ u $, there is no more 
coupling between the different modes $ u_j $'s. 

\smallskip

\noindent Of course, the passage from (\ref{ondenonlinscainterm}) to (\ref{ondenonlinsca}) is a great simplification. Nonetheless, 
the equation (\ref{ondenonlinsca}) is very interesting. It is a simplified model giving a good indication of many mechanisms 
occurring at the level of systems like (\ref{ondenonlin}). Of special interest are the symbols $ p $ which, like in Section~\ref{sec:model}, 
stem from realistic models. Indeed, they allow to identify key physical phenomena.  


\subsubsection{Solving the scalar equation}\label{solvescalar}
In a first stage, we consider \eqref{ondenonlinsca}  when the source term $ F$ reduces to 
\begin{equation} \label{firstchoiceofF} 
F(\eps,t,x) \equiv F^0_{\! L} (\eps,t,x) = \, - \, \eps^{3/2} a_m (\eps \, t , t,x) \ e^{i m\varphi(t,x)/ \eps } .
\end{equation}
The function $ a_m$ is a smooth profile, which is compactly supported
with respect to the spatial variable $ x$; the function  $\varphi$ is
defined 
in \eqref{pharetain}, and $m\in \ZZ$.
More general choices of $ F$ will be presented in Subsection~\ref{sec:assumptions}.
By interpreting the equation (\ref{ondenonlinsca}) on the Fourier side, we can extract
\[ \hat u (t,\xi) = \, - \int_0^t  \! \int e^{i \, \lbrack - (s- t) \, p(\eps \xi) - \eps y \xi \rbrack /\eps} \zeta (-\eps \xi) \ F(\eps,s,y) \ ds \, dy  ,\]
where the Fourier transform is defined as
\begin{equation} \label{defFouriertrans} 
\qquad \hat f (\xi)= \cF f(\xi)=\int_\RR e^{-ix\xi} \ f(x) \, dx, \qquad \cF^{-1}f(x) = \frac{1}{2\pi} \ \int_\RR e^{ix\xi} \ \hat f(\xi) \ d\xi  .
\end{equation}
Since $ \zeta $ is smooth and bounded,  and $ a_m(\eps \, s ,s, \cdot ) $ is in the Schwartz space, the expression  
$ \hat u (t,\cdot) $ is rapidly decreasing and therefore integrable (in $ \xi $). The inverse Fourier transform of $ \hat u 
(t,\cdot) $ furnishes 
\[ u(t,x) = \, - \, \frac{1}{2 \pi} \int \! e^{ i x \, \xi } \left( \int_0^t  \! \! \int e^{i \, \lbrack - (s- t) \, p(\eps \xi) - y (\eps \xi) \rbrack 
 /\eps} \zeta (-\eps \xi) \ F(\eps,s,y) \ ds \, dy \! \right) d \xi . 
\]
Replacing $ F $ as indicated in (\ref{firstchoiceofF}), and changing the variable $ \xi $ into $  -\xi / \eps $ in the resulting
integral yields $ u(t,x) = \cI (\eps,t,x;m\varphi,\zeta,a) $ where the oscillatory integral $ \cI $ is given by
\begin{equation} \label{soloscoif} 
\qquad \cI := \frac{\sqrt \eps}{2\pi} \ \int \left( \int_0^t \! \int e^{- i \, \Phi(t,x;s,y,\xi) /\eps} \zeta(\xi) \ a(\eps \, s, s,y) \ ds \, 
dy \right)  d \xi  ,
\end{equation}
and where, assuming that $ p $ is even, the function $ \Phi $ stands for the phase
\begin{equation} \label{phaseOIFgen} 
\qquad \Phi (t,x;s,y,\xi) := (s - t ) \, p(\xi) + (x-y)  \xi - m\varphi(s,y) .
\end{equation}
The notation $ \Phi (t,x;s,y,\xi) $ emphasizes the dependence of the function of $(s,y,\xi)$ upon the parameters $(t,x)$. 
A large part of our work will focus on the asymptotic behavior when $ \eps $ goes to zero of expressions like (\ref{soloscoif}). 
Note that the application $ a(\eps t,t,\cdot) \mapsto u(t,\cdot) $ belongs to the general category of Fourier integral operators, see
the book \cite{MR2741911}. Non-standard features come here from the large domain of integration in time (of size 
$ s \sim \eps^{-1} $) and from the unusual properties of the phase $ \Phi $ (during large times $ s \sim \eps^{-1} $). Typically, 
the phase $\varphi (s,y) $ is linear in $y$ for fixed time, but with an increasing coefficient $s$, which enhances new phenomena. 


\subsection{Main assumptions} \label{sec:assumptions} Consider (\ref{ondenonlinsca}). By incorporating the action of 
$ \zeta (\eps D_x) $ inside the definition of $ F $, we find a scalar equation in one space dimension like
\begin{equation} \label{eq:main}
  \part_t u -\frac{i}{\eps}p\(\eps D_x\) u + F=0,\quad u_{\mid t=0}=0 .
\end{equation}
As in Paragraph \ref{Osciaspects}, we can decompose the source term $ F $ into $F= F_{\! L} + F_{\! N\! L}$ with $ F_{\! L} $ 
as in (\ref{fnlineaajout}). Below, we state our general assumptions regarding the terms which are present in \eqref{eq:main}.


\subsubsection{Assumptions on the dispersion relation}\label{disprela} We select some $ j $ giving rise to (\ref{bagagegecz}),
and we consider the subset of $ \operatorname{Char} (L) $ which is associated to the choice of the eigenvalue $ \tau_j (S,A,\xi) 
\equiv \tau_j $. For convenience, we will sometimes omit to mention $ j $ when dealing with $ \tau_j $ or related expressions. The 
surface
\begin{equation*} 
\cV \equiv \cV_j \, := \, \bigl \lbrace \bigl(\tau_j (\xi) ,\xi) \, ; \ \xi \in \RR^d \bigr \rbrace \subset \RR^{1+d}
\end{equation*}
does not depend on $ (t,x) $, and it is contained in all sections of $ \operatorname{Char} (L ) $. Our aim is to study the equation 
(\ref{eq:main}) with a pseudo-differential operator $ p (- i  \eps \part_x) $ whose symbol $ p = (1-\underline\chi_c) \tau $ is 
satisfying assumptions which are inspired by (\ref{reldispright}). With this in mind, we impose the following conditions.

\begin{ass}[Existence of a resonance]\label{resoa}  The symbol $ p \in \cC^\infty(\RR ; \RR )$ satisfies:
\begin{enumerate}
  \item [(a)] There exists $\xi_c \in \lbrack 0,1/2 \rbrack $ such that $p_{\mid [0,\xi_c]}\equiv 0 \, $; 
   \item [(b)] The function $p'  $ is positive on the interval $ ] \xi_c, + \infty [ \, $; 
  \item [(c)]  The derivative $ p'(\xi) $ converges to zero when $ \xi $ tends to $ + \infty \, $;
 \item [(d)] The function $p  $ is such that
 \begin{equation} \label{hypp2} 
\exists\, q\ge 2 \ \text{and} \ \ell <0 \, ; \qquad \lim_{\xi \rightarrow + \infty} \ \xi^{q+2} \ p'' (\xi) = \ell \, ;
\end{equation}
  \item [(e)] The function $p  $ is even.
  \end{enumerate}
\end{ass}

\noindent In other words, we have (\ref{hypp2}) together with 
\begin{subequations}\label{increasing} 
\begin{eqnarray} 
\qquad & & \displaystyle  \forall \, \xi \in [0,\xi_c] \, , \qquad \ p(\xi) = 0  , \label{increasing2} \\
\qquad & & \displaystyle  \forall \, \xi \in \ ]\xi_c,+\infty[  , \quad \, 0 < p' (\xi)  , \label{increasing3}  \\
\qquad & &  \displaystyle \forall \, \xi \in [0,+\infty[ \, , \quad \ \, 0 \leq p(\xi) = p(-\xi)  , \label{increasing1} 
\end{eqnarray}
\end{subequations}
as well as
\begin{equation} \label{increasing4} 
\lim_{\xi \rightarrow + \infty} \ p'(\xi) = 0  .
\end{equation}

\noindent As seen in Section~\ref{sec:model}, the formula (\ref{reldispright}) furnishes a typical example of 
symbol $ p  $, which is issued from electromagnetism.

\begin{lem} Assumption \ref{resoa} is satisfied by the function $ p  $ of \eqref{reldispright}. 
\end{lem} 

\begin{proof} Recall that $ G_-^{-1}  $ is strictly decreasing from $ [0,+\infty[ $ to $ ]0,1] $. Taking into
account (\ref{mollichi}) and (\ref{reldispright}), we have (a) with $ \xi_c = 5/8 $. Compute 
$$ p'(\xi) = - \chi' (\xi) \, G_-^{-1} \bigl( \xi^{-2} \bigr) - 2 \, \bigl( 1 - \chi (\xi) \bigr) \, \xi^{-3} \, G'_- \circ G_-^{-1} 
\bigl( \xi^{-2} \bigr)^{-1} . $$
For $ \xi > \xi_c $, in view of (\ref{Gprime})  and (\ref{mollichi2}), $ p'(\xi) $ is the sum of two positive expressions.
This furnishes (b). On the other hand, using (\ref{Gprimeenc}), we find
$$ \lim_{\xi \rightarrow + \infty} \ \xi^3 p'(\xi) = - 2 \, \lim_{\xi \rightarrow + \infty} \ G'_- \circ G_-^{-1} \bigl( \xi^{-2} 
\bigr)^{-1}  = - 2 \, G'_- (1)^{-1} = 2 . $$
As a direct consequence, we have (c). For $ \xi > 1 $, there remains
$$ p''(\xi) = 6 \, \xi^{-4} \, G'_- \circ G_-^{-1} \bigl( \xi^{-2} \bigr)^{-1} - 4 \, \xi^{-6} \, G''_- \circ G_-^{-1} \bigl( 
\xi^{-2} \bigr) G'_- \circ G_-^{-1} \bigl( \xi^{-2} \bigr)^{-3} .$$
Passing to the limit $ \xi \rightarrow + \infty $, we recover (\ref{hypp2}) with $ q = 2 $ and $ \ell = -6 < 0 $. 
Finally, by construction, the function $ p $ is the product of two even functions, and therefore we 
have (e).
\end{proof}

\begin{rem}[About the optimality of Assumption \ref{resoa}]\label{optimalityofAssumption} Most of our results remain valid 
when $ q \geq 1 $, and some of them hold true when \eqref{hypp2} is relaxed according to
\begin{equation} \label{hypp2relax} 
\exists \, (q,\ell) \in \, \rbrack 0,+\infty[ \times \RR_- \, ; \qquad \lim_{\xi \rightarrow + \infty} \ \xi^{q+2} \ p'' (\xi) = \ell .
\end{equation}
For instance, with $ \chi$ as in \eqref{mollichi} and $ p(\xi) = \bigl \lbrack 1 - \chi (\xi) \bigr \rbrack \tau (\xi) $ as in 
\eqref{multipnparchi}, we could also consider the following choices
\[ \begin{array}{lll}
\displaystyle \tau (\xi) = \frac{2}{\pi} \ \arctan \vert \xi \vert \, , & \quad \ q =1 \, , & \displaystyle \quad \ \ell =- \frac{4}{\pi}  ,  \\
\displaystyle \tau ( \xi ) = \int_{0}^{\xi} \frac{ds}{1+|s|^{1+q}} , & \quad \ q > 0 \, , & \displaystyle  \quad \ \ell= -1- q  . 
   \end{array}
   \]
However, the precise description of the large time behavior of $u$ (for $t$ of order $1/\eps$, as in Proposition~\ref{supnormamplifi})
does require $ \ell <0$. As a matter of fact, the case $ \ell = 0 $ is an option which does not allow to quantify the dispersive effects 
(as in Lemma~\ref{localizations}).
\end{rem}

\noindent From (\ref{hypp2}), we obtain that $ p''$ is integrable on $
\RR_+ $. Using (\ref{increasing4}), this yields 
\begin{equation*} 
p'(\xi) = \, - \int_\xi^{+\infty} p''(s) \ ds  .
\end{equation*}
Then, exploiting (\ref{hypp2relax}), we can obtain
\begin{equation} \label{hypp3ded} 
\lim_{\xi \rightarrow + \infty} \ \xi^{q+1} \ p' (\xi) = - \frac{\ell}{q+1} >0  .
\end{equation}
Note that the condition (\ref{increasing3}) implies that the limit in the right hand side of (\ref{hypp3ded}) should be nonnegative. 
This is compatible with the condition $ \ell < 0 $ of (\ref{hypp2}) or with the condition $ \ell \leq 0 $ of (\ref{hypp2relax}). Now,  
from (\ref{hypp3ded}), we know that $ p'$ is integrable on $ \RR_+ $. Thus, we can find a number 
$ \omega^\infty_+ >0$ such that
\begin{equation}\label{reso1} 
\lim_{\xi \rightarrow + \infty} \ p(\xi) = \omega^\infty_+ := \int_0^{+\infty} p'(\xi) \ d \xi > 0 .
\end{equation}
The limit $ \omega^\infty_+ $ has clear physical meaning, in the sense of a resonance frequency. In the context of SMP, 
the number $ \omega^\infty_+ $ coincides with the electron cyclotron resonance frequency $ \omega_c $ introduced at 
the level of (\ref{condtauxiRini}) and (\ref{condtauxiLini}). In the text \cite{editorial}, it is called a {\it resonance of the first 
type}. 

\smallskip

\noindent In the scalar framework (\ref{eq:main}) which involves only one dispersion relation, changing the time scale $ t $ 
into $ \omega^\infty_+ \, t $, the symbol $ p $ and the source term $ F $ are respectively replaced by $ (\omega^\infty_+)^{-1} 
p $ and $ (\omega^\infty_+)^{-1} F $. By this way, we can ensure that $ \omega^\infty_+ = 1 $.

\begin{ass}[Normalization of the resonance]\label{normareso} The resonance frequency, that is the limit 
$ \omega^\infty_+ $, is normalized to unity.
\begin{equation} \label{resobis} 
\lim_{\xi \rightarrow + \infty} \ p(\xi) = 1 .   
\end{equation}
\end{ass}

\noindent Consequently, for $ \xi $ large enough, the dispersion relation $ \tau = p(\xi) $ does mimic the choice $ p \equiv 1 $ 
of \eqref{third-si}. However, in Assumption \ref{resoa}, due to (b), the function $ p  $ is definitely not constant, and therefore the 
variety $\cV$ is curved. Again, this is a hint that some kind of \emph{dispersive effects} are present. Let us clarify this point. 
On the one hand, combining \eqref{hypp3ded} and \eqref{resobis}, we get easily 
\[
  \lim_{\xi \rightarrow + \infty} \ p'(\xi) \ \xi \ p(\xi)^{-1} \, = \, 0 \, \not = \, 1 .
\]
In the vocabulary of geometric optics, this means that the group velocity $p'(\xi)$ and the phase velocity $p(\xi)/\xi$ are 
(asymptotically) different, and hence dispersive effects persist (see e.g. \cite{Rauch}). On the other hand, because the 
symbol $p$ is bounded, there are no (local in time) Strichartz 
estimates which could be associated to the propagator $e^{itp(D_x)}$,
improving Sobolev embeddings, in the sense that if
\begin{equation*}
  \|e^{itp(D_x)} f\|_{L^a([0,T];L^b(\RR))}\le C(T)
  \|f\|_{H^k(\RR)},\quad \forall f\in H^k(\RR),
\end{equation*}
for some $T>0$ and $(a,b)$ satisfying
$\tfrac{2}{a}=\tfrac{1}{2}-\tfrac{1}{b}$ (admissible pair),
then necessarily, $k\ge \tfrac{1}{2}-\tfrac{1}{b}$ (see \cite{CaDisp}). On
the other hand,  \emph{frequency localized} 
 Strichartz estimates are available (see \cite{CDGG06}): since we consider high
 frequency phenomena, these localized estimates are not helpful.  
One thus has to be cautious about the notion of dispersive effects that is involved. In this article, it refers to the first interpretation.

\begin{rem} The equation \eqref{third-si} of the introduction can be  put in the form \eqref{eq:main}  with $ p \equiv 1 $. It is also 
dispersive since, for $ \xi \not = 0 $, the derivative $ p'(\xi) = 0 $
is different from $ \xi^{-1} p(\xi) = \xi^{-1} $. It satisfies
Assumption ~\ref{normareso} but not Assumption~\ref{resoa}. Indeed, in
contradiction with \eqref{increasing3}, the group velocity $ p' (\xi) $ is simply zero.  
\end{rem}

\noindent Now, let us come back to the content of (a), (b), (c), (d) and (e).

\smallskip

\noindent - (a). By multiplying the eigenvalue $ \tau  $ by the cut-off function $ 1 - \underline\chi_c $ with $ \underline\chi_c $ 
as in (\ref{chicestc}), we can always guarantee (a) for $ p = (1- \underline\chi_c) \tau $. 

\smallskip

\noindent - (b). To better understand the origin of (b) and the conditions under which the property (b) is indeed accessible, 
we can examine what happens in the one dimensional framework ($ d = 1 $). Starting from (\ref{ondenonlin}), this means 
to fix $ \sigma \in \mathbb S^{d-1} $, to identify $ L(S,0,\sigma) $ with some symmetric matrix $ S $, and to look at 
\begin{equation} \label{ondenonlinoned} 
\part_t u + S \, \part_{x} u + \eps^{-1} A u + F \, = \, 0  , \qquad u_{\mid t=0} = 0 , \qquad x \in \RR .
\end{equation}
As already explained, to obtain (\ref{bagagegecz}) for some index $ j $, zero must be an eigenvalue of the symmetric 
matrix $ i L(S,0,1) $, which coincides with $ - S $. Let $ P $ be the orthogonal projector onto the kernel of $ S $. If the 
matrix $ S $ commutes with $ A $, that is if $ \lbrack S,A \rbrack = 0 $, the two matrices $ S $ and $ A $ are simultaneously 
diagonalizable, and the system (\ref{ondenonlinoned}) can be decoupled into distinct transport equations. In particular, 
from (\ref{ondenonlinoned}), we can extract
\begin{equation} \label{ondenonlinonedproject} 
\part_t Pu + \eps^{-1} (PAP)  Pu + PF \, = \, 0  , \qquad P u_{\mid t=0} = 0 .
\end{equation}
Among the eigenvalues $ \tau_j (S,A,\xi) $ of $ i L(S,A,\xi) $, we can distinguish those coming from (\ref{ondenonlinonedproject}),
which are simply eigenvalues of $ P A P $, and therefore constant in $ \xi $. Then, the condition (b) is not verified. This means 
that interesting situations may arise only on condition that $ \lbrack S,A \rbrack \not = 0 $. 

\smallskip

\noindent In fact, in order to have the condition (b), the important thing is the presence asymptotically, say for $ \xi \geq \xi_p $ 
with $ \xi_p \in \RR_+ $, of some nontrivial monotone behavior of $ \tau  $. Then, changing $ t $ into $ -t $ if necessary, 
we get the growth criterion $ p'(\xi) \geq 0 $ for $ \xi \geq \xi_p $. Changing $ x $ into $ \mu x $ with $ \mu \in \RR $, we can 
obtain $ 0 \leq \xi_p \leq \xi_c \leq 1/2 $. Then, multiplying $ \tau_j $ by $ 1- \underline\chi_c $ as in (\ref{multipnparchi}), the 
situation does fit in with (b). The supplementary restriction $ p'(\xi) > 0 $ is aimed to guarantee that dispersive effects associated 
with the variations of $ p  $ actually occur. 

\smallskip

\noindent - (c). The condition (c) is useful to infer (\ref{reso1}) from (\ref{hypp2}). Subject to (\ref{bagagegecz}), it is an easy  
consequence of (b). Let us briefly explain why. In view of Remark \ref{zeroeigenvalues}, the condition (\ref{bagagegecz}) 
gives rise to the existence of a finite limit $ \tau_j^\infty \equiv \tau_j^\infty (S,A,\sigma) $. At this stage, there is no sign condition 
on $ \tau_j^\infty \in \RR $. Now, given any $ \lambda \in \RR $, we can change the solution $ u (t,\cdot) $ into $ \tilde u (t,\cdot) := 
\text{exp} (i \lambda t/\eps ) \, u (t,\cdot) $. This gauge transformation is not without consequence on the source term $ F $ 
(see Paragraph~\ref{Osciaspects}), which must be adjusted accordingly. It also modifies the matrix $ A $ into $ \tilde A := A -
i \lambda $, and the symbol $ \tau_j (S,A,\xi) $ into 
$$ \tilde \tau_j (S, \tilde A ,\xi) := \tau_j (S, \tilde A ,\xi) = \tau_j (S, A ,\xi) + \lambda . $$
It follows that $ \tau_j^\infty $ and $ \tau_j (S,A,0) $ are respectively turned into 
$$ \tilde \tau_j^\infty := \tau_j^\infty + \lambda, \qquad \tilde \tau_j (S,\tilde A,0) = \tau_j (S,A,0) + \lambda . $$ 
For the choice of a sufficiently large number $ \lambda $, the new limit $ \tilde \tau_j^\infty $ becomes positive, like 
$ \omega^\infty_+ $ in (\ref{reso1}). Moreover, provided that $ \tau_j (S,A,\cdot) $ is increasing on $ \RR_+ $, which 
means that we can take $ \xi_p = 0 $ as well as $ \xi_c >0$ arbitrarily small, by selecting $ \lambda = - 
\tau_j (S,A,0) $, we can ensure that $ \tilde \tau_j (S,\tilde A,0) = 0 $ and $ \tilde \tau_j^\infty \equiv \omega^\infty_+ 
>0$. Then, as in the whistler case, the function $ \tilde \tau_j  (S,\tilde A,\cdot) $ connects some zero 
eigenvalue of $ S $ (when $ \xi \rightarrow + \infty $) to some zero eigenvalue of $ \tilde A $ (when $ \xi = 0 $). 

\smallskip

\noindent - (d). The condition (\ref{hypp2}) does not only guarantee the existence of a resonance frequency. It is much 
more restrictive. It provides information about the asymptotic behavior inside (\ref{Lipschitbound}). Indeed, from (\ref{hypp3ded}), 
we can extract the rate of convergence
\begin{equation}\label{reso2} 
\lim_{\xi \rightarrow + \infty} \ \xi^q \ \left\lbrack \omega^\infty_+ - p (\xi) \right\rbrack = -\frac{\ell}{q \, (q+1)} >0 .  
\end{equation}

\smallskip

\noindent - (e). The last restriction (e) is not essential. It is inspired by the whistler dispersion relation $ \tau_w  $ 
which is an even function, and for which a global analysis up to the zero frequency $ \xi_p = 0 $ is directly available, without 
involving a cut-off function $ \underline\chi_c $ with $ \xi_c $ large. It is imposed here for the sake of simplicity. It can be 
avoided, albeit with technicalities to distinguish what happens in the two directions $ \pm \sigma $.

\smallskip

\noindent To conclude, let us illustrate the above discussion in the context of equation (\ref{ondenonlinoned}), when $ N = 2 $.
After adequate rescalings, the framework can be reduced to
$$ S = \left( \begin{array}{cc}
0 & 0 \\
0 & - 1
\end{array} \right) , \qquad A = \left( \begin{array}{cc}
- i a & - b - i c \\
b - i c & - i d 
\end{array} \right), \qquad (a,b,c,d) \in \RR^4 . $$
The pertinent eigenvalue $ \underline {\tau} (\xi ) $ of 
$$ i L(s,A,\xi) = - \xi S + i A = \left( \begin{array}{cc}
a & c - i b \\
c + i b & d + \xi 
\end{array} \right) $$ 
is the one issued from the zero eigenvalue of $ S $. Thus, it must be zero when $ A = 0 $. For $ \xi \geq 0 $, this means to select
$$ \underline {\tau} (\xi ) := \frac{1}{2} \bigl( \xi + a + d - \sqrt {(\xi-a +d)^2 + 4 b^2 + 4 c^2} \bigr) , \qquad \lim_{\xi \rightarrow 
+ \infty} \ \underline {\tau} (\xi ) = a . $$
For large values of $ \xi \geq 0 $, the function $ \underline {\tau}  $ is constant (equal to $ a $) if and only if $ b = c = 0 $, 
or equivalently if and only if $ \lbrack S,A \rbrack = 0 $. We have to address the opposite case, when  $ \lbrack S,A \rbrack \not 
= 0 $ or when $ b c \not = 0 $. By adjusting 
the gauge parameter $ \lambda $, we can work with $ a = 1 $. Multiply $ \underline {\tau}  $ by $ 1 - \underline\chi_c $
to form $ p = (1 - \underline\chi_c) \underline {\tau} $. Then, for all choice of $ \xi_c >0 $, we obtain (a), (b) and (c). 
We also find the substitute (\ref{hypp2relax}) for (\ref{hypp2}) - or (d) - with $ q = 1 $. We do not have (e) but, as seen before, 
this is just a simplifying assumption. 

\smallskip

\noindent For technical reasons, we need to highlight the following condition.

\begin{ass}[Control of derivatives of $ p $] \label{strengthenp} 
\begin{equation} \label{controlderip} 
\quad \exists \, D \in \NN \setminus \{0,1\} \, ; \qquad \forall \, n \in \{ 2, \cdots, D \} \, , \qquad \limsup_{\xi \rightarrow + \infty} \
\frac{\vert p^{(n)}(\xi) \vert}{p'(\xi)} < + \infty .
\end{equation}
\end{ass}

\begin{lem}  Assumption~\ref{resoa} implies  Assumption~\ref{strengthenp}.
\end{lem} 

\begin{proof} When $ \ell <0$ as required in (\ref{hypp2}), from (\ref{hypp2}) and (\ref{hypp3ded}), we deduce that
\begin{equation*} 
\lim_{\xi \rightarrow + \infty} \ \xi \ \frac{p''(\xi)}{p'(\xi)} = - q-1  , 
\end{equation*}
and hence 
\begin{equation*} 
\limsup_{\xi \rightarrow + \infty} \ \frac{\vert p''(\xi) 
\vert}{p'(\xi)} = 0  .
\end{equation*}
By this way, we find (\ref{controlderip}) for $ n= D= 2 $.
\end{proof}

\noindent Come back to (\ref{reldispright}) with $ G_- $ as in (\ref{condtauxiR}). In this case, comparing \eqref{asymppon} 
and \eqref{hypp3ded}, we see that the finiteness condition in Assumption~\ref{strengthenp} is verified for all $ D$. However, 
in the general case, when $ D > 2 $, Assumption~\ref{strengthenp} is adding new information, as shown by the example
\[
 p''(\xi) = \frac{\ell}{(1+ \xi^2)^2} \ \( 1 + \frac{1}{1+ \xi^2} \, \cos \, \xi^6 \),\quad  q = 2 \, , \ \ell < 0 \ \text{and} \  D=3 . 
\]
 In order to examine the role of $ D $, we will keep track of $ D $ in the various estimates.


\begin{rem} [About other nonlinear  effects]\label{othernonlineareffects} The model
  \eqref{eq:main} does not see the  
interactions that occur between the different modes $ u_n $ inside \eqref{ondenonlinscainterm}. Just a quick comment on  
this. In view of \eqref{condtauxiR}, the set $ \cV $ is symmetric with respect to the $ \tau $-axis. Consider a phase $ \psi (t,0,0,x_3) $ 
satisfying the eikonal equation associated to R-waves
\begin{equation} \label{Reikonal}
\part_t \psi  = p \bigl( \part_{x_3} \psi (t,0,0,x_3) \bigr) \, , \qquad \forall \, x_3 \in \RR .
\end{equation}
The function $ \psi(t,0,0,-x_3) $ is also a solution to \eqref{Reikonal}. On the other hand, the direction $ (\tau,\xi) $ is 
subject to \eqref{condtauxiR} if and only if its opposite $ (-\tau, -\xi) $, and $(-\tau,\xi) $, is satisfying \eqref{condtauxiL}. 
If a phase $ \psi (t,0,0,x_3) $
is as in \eqref{Reikonal}, the function $ - \psi (t,0,0,x_3) $ satisfies the eikonal equation associated to L-waves. 

\smallskip

\noindent The same would apply for  extended functions $ \psi (t,x) $ which could be issued from  $ \psi(t,0,0,x_3) $ by 
the resolution of  the complete eikonal equation, related to some $ \xi \in \RR^d \setminus \{ 0 \} $. 
Now, due to possible nonlinear interactions, a non oscillating term can be produced by combining the phases 
$ - \psi$ and $ + \psi $. Since the value $ \tau = 0 $ is still characteristic when $ \xi_3 = 0 $ (the zero eigenvalue is not 
completely eliminated, see  \cite[Lemma~6.1]{CheFon}), this term may propagate and be amplified. There would be at 
the level of the full system \eqref{mhdeq} a three-wave resonance to study.
\end{rem}

\noindent The two conditions $ N = 1 $ and $ d = 1 $ are of course quite restrictive. They do not allow to take into account 
a number of multidimensional and nonlinear aspects. But again, the focus here is on resonances, intermittencies and related
nonlinear effects, in a framework as accessible as possible.
 

\subsubsection{The source term $ F_{\! L} $}  \label{thecontentofFL} After adequate gauge transformations, see Remarks 
\ref{eliminationdeF} and \ref{disphumantissues} as well as the comment about (c) in the preceding Subsection \ref{disprela}, 
we look at the equation (\ref{eq:main}) under Assumption \ref{resoa}. To simplify  matters, we can suppose that $ F_{\! L} 
\equiv F_{\! L} ^0 $. On the other hand, to take into account the cut-off by $ \underline\chi_c $ and the possible pseudo-differential 
action $ \zeta (-\eps D_x) $ introduced in (\ref{ondenonlinsca}), we extend below the choice made in (\ref{firstchoiceofF}). 
Given $ M \in \NN^* $, we consider 
\begin{equation}\label{eq:source}
  F_{\! L} \equiv F_{\! L} ^0 = \eps^{3/2} \sum_{m \in [-M,M]}\ A_m (\eps \,
t,t,x,-i\eps\part_x)^* \ e^{i m \varphi(t,x) / \eps} , 
\end{equation}
where the phase $\varphi$ is given by \eqref{pharetain}, that is $ \varphi(t,x) =  t - x \, t + \gamma \, ( \cos \, t -1) $ with 
$ 0<\gamma <1/4 $, and the action of the adjoint $ A_m^* $ of the semi-classical pseudo-differential operator $ A_m 
(\eps \, t,t,x,-i\eps\part_x) $ corresponds  to the right quantization
\begin{equation}\label{rightquantization}
 A_m (\eps \, t,t,x,-i\eps\part_x)^* u (x) := \int \! \! \int
e^{i(x-y)\xi} A_m (\eps \, t,t,y,\eps\xi) u (y) dy d\xi .  
\end{equation}

\begin{ass}[Choice of the coefficients $A_m$]\label{choiopro}
  For all integers $ m \in [-M,M] $, we impose $ A_m(T,t,x,\xi) = \zeta_m(-\xi) a_m(T,t,x) $ where the functions $\zeta_m\in 
  \cC^\infty(\RR;\RR)$ and $ a_m \in \cC^\infty_b (\RR^3) $ satisfy the following conditions.
  \begin{itemize}
  \item  $\zeta_m$ goes to $1$ at infinity with
    \begin{equation*}
      \zeta_m(\xi) = 1+\cO\(\frac{1}{|\xi|}\),\quad |\xi|\to \infty \, ;
    \end{equation*}
  \item For $m=0$, $\zeta_0$ is identically zero near the origin,
    \begin{equation*}
      \exists \xi_0>0,\quad \zeta_{0\mid [-\xi_0,\xi_0]}\equiv 0 \, ;
    \end{equation*}
 \item There exists some $\cT >0 $ and some $ r \in ] 0, \gamma/2 [ $ such that
 \begin{equation} \label{suppdea} 
\forall \, m \in [-M,M] \, , \qquad \mathrm{supp} \, a_m 
\subset \, ]-\infty,\cT] \times [1,+\infty[ \, \times [-r,r] .  
\end{equation}
 \end{itemize}
\end{ass}

\noindent Since $\zeta_m$ is multiplied by $a_m$, the asymptotic value $\lim_{|\xi|\to \infty}\zeta_m(\xi)=1$ is 
somehow arbitrary. It could be replaced by any non-zero constant.  Also, some coefficients $a_m$ may
very well be identically zero. We will see that the most important coefficient is $a_1$. All other coefficients 
will have a negligible contribution in the regime we consider in this paper (this is Fact \ref{fact1}). 

\smallbreak

\noindent As explained in \cite{Whistler,MR3582249}, when dealing with strongly magnetized confined 
plasmas (SMP), the solution to the Vlasov equation is transported by an oscillating flow, see for instance (2.4) and 
(2.9) in \cite{editorial}. This induces special structures of the electric current $ {\bf J}_e $ which does enter in the 
composition of $ F$ at the level of (\ref{eq:main}). 

\smallskip

\noindent Again, due to the bouncing back and forth (of charged particles) 
between the mirror points at a frequency which has been normalized in (\ref{deterXrsati}) to the value $ 2 \, \pi $, 
it is expected that the profiles $ a_m $ inherit similar periodic structures with respect to $ t \in \RR $. The same 
applies in the case of nuclear magnetic resonance (NMR). According to (\ref{formpourUNRM}), the profiles $ a_m $ 
would be constant. But, in practice, they should be periodic due to relaxation phenomena between the repeated 
action of radio frequency (RF-)excitations. 

\smallskip

\noindent In fact, the periodicity property is important for our results only as far as the coefficient $a_1$ is concerned. 
That is why we just impose the following condition, which is a relaxed version of (\ref{asymptoticallypi}).

\begin{ass}[$ a_1$  is $2\pi$-periodic in $ t $ for large times]\label{persourcea} 
There exists $ {\rm t}_s >0 $ and a function $ \underline a  \in \cC^\infty (\RR \times \TT \times \RR ) $ such that
\begin{equation} \label{peratlar} 
\qquad \forall \, t \geq {\rm t}_s \, , \quad \forall \, n \in \NN \, , \quad a_1 (\cdot,t+ 2n \pi,\cdot) \equiv \underline a 
(\cdot,t+  2n \pi,\cdot) \equiv \underline a (\cdot,t,\cdot) . 
 \end{equation}
 \end{ass}


 \subsubsection{The nonlinearity $ F_{\! N \! L} $}  \label{contentofFNL}  The coupling between ``{\it particles}" and 
 ``{\it waves}" could also be described by nonlinear source terms. This induces an additional mixing, and provides a further 
 complication. In the same vein as (\ref{ajoutFLNL}) or (\ref{eq:NLedp}), the expression $ F_{\! N \! L} $ is chosen
 as a polynomial function in $ \eps^{-1} $, $ \eps $, $ u $ and $ \bar u $.

\begin{ass}[Choice of the nonlinearity]\label{choiononlinso} Given $ J \in \NN $ with $ 2 \leq J \in \NN $ and $ K \in \NN $,
as well as complex numbers $ \lambda_{j_1j_2\nu} \in \CC $, we impose
\begin{equation}\label{contofFNLNN}
     F_{\! N \! L} (\eps,t,x,u) := \sum_{2 \leq j_1 + j_2 \leq J} \ \sum^K_{\nu=2-j_1-j_2 } \! \lambda_{j_1j_2\nu} \, F_{j_1 j_2 
     \nu}(\eps,t,x,u) ,
 \end{equation}
where $ (j_1,j_2,\nu) \in \NN^2 \times \ZZ $, whereas 
 \begin{equation}\label{contofFNLNNstep2}
     F_{j_1 j_2 \nu} (\eps,t,x,u)= \eps^\nu 
      e^{i \omega_{j_1j_2\nu} t / \eps } \chi\( 3-2\frac{\eps t}{\cT}\) \chi\(\frac{x}{r\eps^\iota}\)u^{j_1} \bar u^{j_2} ,
\end{equation}
for a frequency $ \omega_{j_1j_2\nu} \in \RR $ and parameters $ \cT $, $ r $ and $ \iota $ satisfying as before 
the conditions $ 0 < \cT $, $ 0 < r < \gamma / 2 < 1/8 $ and $ \iota \in [0,1] $.
\end{ass}

\noindent We could assume more generally $\nu 
\in \RR$ (taking finitely many values), provided  that the size of the nonlinear coupling is at most critical, in the 
sense that $ j_1+j_2+\nu\le 2 $. Since the critical case corresponds to the choice $j_1+j_2+\nu=2$, assuming that 
$\nu \in \ZZ$ is not a strong restriction.

\smallskip

\noindent At the level of (\ref{contofFNLNNstep2}), the nonlinearity undergoes extra time and space localizations. 
The reasons for doing so will become clear in Section~\ref{sec:nonlinear}. 
They are related to the nature of the information established in the linear case (Section~\ref{sec:lineffect}). Indeed, 
we will get a precise  pointwise description of the linear solution only for sufficiently large time ($t \gtrsim 1/\eps$), 
and only in a small neighborhood of the origin ($|x| \le r$). 

\smallskip

\noindent The amplitude $ \eps^{3/2} $ of the source term $F_L$ and the 
nonlinearity (\ref{contofFNLNN}) are inspired by Subsection \ref{subsec:toymodel}. They are adjusted so that 
nonlinear effects can actually be critical in the limit $\eps\to 0$ on the long time scale $ T \sim 1 $ under consideration.


\subsection{The notion of quasi-rectification}\label{mildrecti} The term {\it rectification} has been first introduced in 
\cite{MR1687154} in the context of nonlinear diffractive geometric optics where it means the creation 
of non-oscillatory waves from highly oscillatory sources. A distinction is made between hyperplanes which are in the 
characteristic variety (contained in the section $ \cV $) and curved sheets. 

\smallskip

\noindent For wave vectors belonging to flat parts inside $ \cV $, the interaction 
cannot be ignored, while for wave vectors on curved parts, it is negligible at leading order. In the subsequent articles 
\cite{zbMATH01820804,MR1657485}, these ideas are extended to dispersive equations and to situations of "almost 
rectification" (when the resonance comes from the tangent space to the characteristic variety).

\smallskip

\noindent In what follows, the expression {\it quasi-rectification} will be used in reference to the pioneering work \cite{MR1687154}. 
As a matter of fact, as detailed in Paragraph \ref{analogies}, our approach presents certain similarities with that of \cite{MR1687154}. But there 
are also significant differences that will be emphasized in Paragraph \ref{differences}. In order to avoid confusion, it is important to 
explain clearly what the situation is. In Paragraph \ref{undermecha}, we provide an overview of what quasi-rectification is. The last 
Paragraph \ref{backsummary} is aimed to summarize the discussion.


\subsubsection{Analogies}\label{analogies} As in \cite{MR1687154}, we study a nonlinear hyperbolic equation for long time scales at which 
nonlinear effects are present. As in \cite{MR1687154}, the characteristic variety is a mix of curved and flat features. In Figure \ref{graofg}
the red graph is curved while its magenta asymptote is flat. As in \cite{MR1687154}, amplification phenomena can occur (on small sets). 
Moreover, in the spirit of \cite{MR1657485}, we deal with a kind of almost (or quasi) rectification" (at infinity in our case). 
But here is where the comparisons stop.


\subsubsection{Differences}\label{differences} The first change involves (\ref{mhdeq}). In \cite{MR1687154}, the authors impose 
the condition $ A = 0 $, and they investigate {\it diffractive} effects. On the contrary, we work here with $ A \not = 0 $, 
and we consider {\it dispersive} effects. When $ A = 0 $ and $ d = 1 $, the characteristic variety is a union of lines. By contrast, 
when $ A \not = 0 $ and $ d = 1 $, the section $ \cV $ does contain curved parts.

\smallskip

\noindent Since $ N = 1 $ and $ d = 1 $, many aspects of \cite{MR1687154} are not present here. For instance, we will not discuss 
problems related to the interaction of different modes ($ N > 1 $) of propagation (the interaction between the $ u_n $). Nor are we 
going to manage the multidimensional ($ d > 1 $) spreading of waves. However, there will be many new difficulties to deal with.

\smallskip

\noindent There is another distinction. In the article \cite{MR1687154}, the oscillations of the source term come from the oscillations  
of the initial data after a selection process implying the properties of the differential operator $L (\eps  D_x) $. That is not the case 
here. As a matter of fact, at the level of (\ref{ondenonlin}), the Cauchy data are simply zero. 

\smallskip

\noindent Here, the oscillations are imposed from outside, as a part of the source term $ F $. They are issued from the concrete 
considerations exposed in Section \ref{sec:model}. They do not at all involve the differential operator $L (\eps  D_x) $. As a matter 
of fact, the phase $ \varphi $ has no link with $L (\eps  D_x) $. In particular, the function $ \varphi $ is not solution to the eikonal 
equation that is associated with $L (\eps  D_x) $.  

\smallskip

\noindent In \cite{MR1687154}, there is a clear dichotomy between the flat and curved sheets contained in the characteristic variety 
$ \cV $. In Figure \ref{graofg}, there is no such strict separation. 
Instead, in view of (\ref{hypp2}), the function $ p$ is concave for large values of $ \xi $. Hence, there is no flat part. But there is 
a progressive transition between a curved dispersion relation and its flat asymptote (for $ \xi $ large). While the rectification would refer, 
among other things, to the presence inside $ \cV $ of branches without curvature, the {\it quasi}-rectification exploits the property that 
the curvature of the section $ \cV $ asymptotically approaches zero. 


\subsubsection{Underlying mechanisms}\label{undermecha}
\noindent The eikonal equation associated to $ p $, the one which could be obtained at the first step of a WKB analysis, 
would give the value of $ \part_t \psi $ in terms of $ \part_x \psi $ through
\begin{equation} \label{disprelrel}
\part_t \psi (t,x) = p \bigl( \part_x \psi (t,x) \bigr) \quad \Longleftrightarrow \quad (\part_t \psi,\part_x \psi)(t,x) \in \cV .
\end{equation}
Given a generic position $ (t,x) $, the phase $ \varphi $ that is involved in the source term $ F $ will not satisfy (\ref{disprelrel}). 
But remark that the spatial derivative $ \part_x \varphi (t,x) = - t $ becomes large when $ t $ is growing, while the time derivative 
$ \part_t \varphi (t,x) = 1-x - \gamma \, \sin t $ remains close to $ 1 $ (at least for $ \vert x \vert \ll 1 $ and $ t \simeq 0 $ modulo 
$ \pi $). Taking into account (\ref{resobis}), it follows that $ p \bigl( \part_x \varphi (t,x) \bigr) = p(-t) = p(t) $ is not far from $ \part_t 
\varphi (t,x) $ when $ t = n \, \pi $ with $ n $ going to infinity. This is the type of resonance which has been illustrated in the introduction 
through (\ref{third-si-devtt}). This is the reason why the asymptotic direction $ \tau = 1  $ of $ \cV $ is physically so important. 

\smallskip

\noindent Now, let us compare the position of $ \nabla_{\! t,x} \varphi $ relative to $ \cV $ more precisely. When $ \varphi  $ is as 
in (\ref{pharetain}), the gradient of $ \varphi$ gives rise to a folded Lagrangian manifold $ \cG (\varphi) $ 
(see Figure 3 in \cite{editorial}), which is 
\begin{equation*} 
\quad \ \cG (\varphi) := \bigl \lbrace (t,x,\part_t \varphi ,\part_x \varphi) (t,x) = ( t,x,1-x - \gamma \, \sin t , - t ) \, ; \, (t,x) \in \RR 
\times \RR \bigr \rbrace . 
\end{equation*}
In general, the direction $ (\part_t \varphi,\part_x \varphi) (t,x) $ is away from $ \cV $. But, near $ x = 0 $ and for large values 
of $ t $, it will repeatedly cross the section $ \cV $ in the course of time. Given some small $ x $, denote by $ t_k \equiv t_k (x) $ 
with $ t_k < t_{k+1} $ and $ k \in \NN^* $ the successive intersection points. The $ t_k $'s form a countably infinite set. By this 
way, wave packets $ u_k $ may be generated. As in (\ref{third-si-dev}), they look like
\[
  u_k(t,x) = \sqrt\eps\ a_k(t,x) \ e^{i \,
    \psi_k(t,x)/\eps} + o \( \sqrt\eps \) , \qquad k
  \in \NN \, ,
  \]
with $ \psi_k$ subject to the eikonal equation (\ref{disprelrel}). 

\smallskip

\noindent The phase $ \psi_k $ and the profile $ a_k $ are determined by the {\it local} 
geometrical properties of $ \cV $ near the position $ (\part_t \varphi, \part_x \varphi) \bigl(t_k(x),x \bigr) \in \cV $. Since 
$ p'(\xi) $ with 
$ \xi = - t $ is very small for large values of $ t $, especially when $ t \sim \eps^{-1} $, the waves $ u_k  $ become almost stationary 
for large integers $ k $. Their group velocity is not zero, but it tends to zero. It follows that the emitted waves $ u_k $ with $  k $ large
can strongly interact and produce important local effects during long times $ t \sim \eps^{-1} $. 

\smallskip

\noindent Since the accumulation process of the $ u_k $'s is related to the asymptotic shape of the set $ \cV $ (for large values of 
$ \xi $) and of the set $ \cG (\varphi) $ (for large values of $ t $), what happens ultimately depends on the \emph{global} geometrical
properties of $ \cV $ and $ \cG (\varphi) $. As will be seen, amplification phenomena can occur, but not everywhere.

\smallskip

\noindent The creation, propagation, accumulation and nonlinear interaction at the level of the evolution equation (\ref{ondenonlin}) of 
almost standing waves generated near a resonance by highly oscillatory sources like (\ref{nonlinform}), with $ \varphi$ 
as in (\ref{pharetain}), is what is called here {\it quasi-rectification}. As alluded to above, the study of quasi-rectification requires to combine 
local and global geometrical features of $ \cG (\varphi) $ and $ \cV $. From a physics viewpoint, the notion of quasi-rectification is well adapted 
to describe the observed production of quasi-electrostatic waves in SMP \cite{chorus} or NMR \cite{MR2127862}, and to measure the relative 
impacts. 

\smallskip

\noindent Note that certain mechanisms which are involved show also similarities with what is observed about surface plasmons 
\cite{zbMATH06629302}.

\smallskip

\noindent The same applies for vortex filaments \cite{zbMATH06313197,MR4105743,zbMATH06383460} with Talbot effect. In
this case, specific spatial structures appear at special times, while in the present context, we will obtain specific spatial structures
which remain over long time intervals. In Theorems~\ref{theo:resumeNL} and \ref{theo:resumeNLbis}, the wave function $\cU$ is of size
$\cO(1)$ at integer points only, over large time intervals.
 
\smallskip

\noindent We would also like to cite the recent works of Y. Colin de Verdi\`ere and L. Saint-Raymond, see \cite{MR4054361} and 
references therein. The motivations (fluid mechanics/SMP and NMR), the mechanisms (periodic medium/oscillating phase), the
structures (PDEs involving variable/constant coefficients, flat/corrugated Lagrangian) and the tools (semiclassical/WKB methods) 
are distinct. But still, there are similarities and deep connections. The two viewpoints are complementary.


\subsubsection{Back to the physical models, and summary} \label{backsummary}

In SMP, the motion of charged particles generates an electric current $ {\bf J}_e $. In NMR, the precession of the magnetic 
moment $ \bf M $ creates a magnetization current $ {\bf J}_m $. These two sorts of currents oscillate according to the phase 
$ \varphi $, and they both appear as a source term inside equations of 
\href{https://en.wikipedia.org/wiki/Maxwell\%27s_equations}{Maxwell's} type. But since the plasmas (in SMP) and the human 
tissues (in NMR) are inhomogeneous media, dispersion phenomena occur. 

\smallskip

\noindent The characteristic variety 
$ \operatorname{Char} L := \bigl \lbrace \bigl( t,x, p(\xi) , \xi \bigr) \, ; \, (t,x,\xi) \in \RR^3 \bigr \rbrace \subset \RR^2 \times \RR^2 $
is more  complicated than hyperplanes such as $ \tau = p(\xi) = 1 $. In real situations, there are  dispersive effects which are
encoded in the variations of the symbol $ p  $. As we have seen in Subsection~\ref{pseudo-differential}, this happens generically. 

\smallskip

\noindent In the context of SMP, the dispersive effects can be specified in details. Indeed, the dielectric tensor of magnetized 
plasmas can be computed, both in the cold case \cite{CheFon} and in the hot case \cite{CheFon2}. In SMP, the pertinent function 
$ p  $ is available, and it is such that $ p' (\xi) \not \equiv 0 $. Less information exists concerning NMR, but the situation 
should be similar.

\smallskip

\noindent The reason why the model $ p \equiv 1 $ is so important is
the following. On the one hand, after normalization, the function $p$  
will converge to $ 1 $ when $ \vert \xi \vert $ goes to infinity; thus, for large values of $ \vert \xi \vert $, the {\it dispersion relation} $ \tau = p(\xi) $ 
mimics the choice $ p \equiv 1 $ of \eqref{third-si}. On the other
hand, from \eqref{pharetain}, we can deduce that the derivative
$ \part_t \varphi $  
remains close to $ 1 $ and that $ \part_x \varphi = -t $; thus, for large values of $ t $ and especially during long times $ t \sim \eps^{-1} $, the 
eikonal equation 
 is almost satisfied. Due to the periodic part inside $ \varphi $, it is in fact repeatedly verified. By this way, two-dimensional oscillating waves 
$ u_k $ may be emitted.

\smallskip

\noindent Vlasov and Bloch equations are completely distinct from Maxwell's equations. The two objects $ \varphi $ and $ p $ are issued 
from different physical laws. As a consequence, the phase $ \varphi $ has nothing to do with the symbol $ p $. But they can intersect 
incidentally, in the sense of $ \cG \cap \cV $. For the foregoing reasons, they can even cross again and again. Such a configuration is 
a facet of what can be a {\it resonance}.

\smallbreak

\noindent As explained in \cite{editorial}, the cyclotron resonances and the internal repeated emissions of electromagnetic signals are 
both important phenomena which are involved in all collisionless magnetized plasmas. Applied in the context of SMP, our work helps 
to better understand the underlying mechanisms which are known to  generate in coronas, magnetospheres and fusion devices some 
heating and some \href{http://fusionwiki.ciemat.es/wiki/Anomalous_transport}{anomalous transport}. It also sheds a new light on some 
aspects of NMR.


\section{Linear analysis}\label{sec:lineffect}

\noindent In this section, we look at the evolution equation (\ref{eq:main}) with a source term $ F $ not involving $ u $, that is when 
$ F \equiv F_{\! L} $ with $ F_{\! L}$ as in Paragraph~\ref{thecontentofFL} and $ \varphi $ as in (\ref{pharetain}). 
Since the problem is linear, we can study separately each term of the sum present in $F_{\! L}$.  Accordingly,
we consider here what happens when $ F $ reduces to 
\begin{equation} \label{followformFgeninimodem} 
F \( \eps, \eps \, t,t,x,\frac{\varphi (t,x)}{\eps},u \) = \eps^{3/2}
A_m ( \eps \, t,t,x,-i\eps\part_x)^* \ e^{i \, m \, \varphi(t,x) / \eps} .
\end{equation}
We denote by $u^m$ the solution to (\ref{eq:main}) issued from the choice (\ref{followformFgeninimodem}). As usual, it is referred as 
the {\it mode $ m $}.  The purpose is to study the subsequent oscillatory integral 
$ \cI  (\eps,t,x;m\varphi,a_m) $, which is given by (\ref{soloscoif}), with $\zeta=\zeta_m$.  In
Subsection~\ref{subsec:non-resoscint}, we examine the case of a non-resonant phase. The rest of Section~\ref{sec:lineffect} 
is devoted to the analysis of the resonant situation. This starts in Subsection~\ref{subsec:resoscint} with a presentation of the 
strategy and results. Then, Subsections~\ref{startregime}, \ref{transiregime} and \ref{longregime} give the details.
Given a phase $\psi\in \cC^\infty(\RR;\RR)$, consider a generalization of \eqref{soloscoif}--\eqref{phaseOIFgen} which is 
  \begin{equation} \label{soloscoif-gen}
\ \cI (\eps,t,x;\psi,\zeta,a) = \frac{\sqrt \eps}{2\pi} \ \int \left( \int_0^t \! \int e^{- i \, \Phi(t,x;s,y,\xi) /\eps} \zeta(\xi) \ a(\eps \, s,
  s,y) \ ds \,  dy \right)  d \xi, 
  \end{equation}
where  
  \begin{equation} \label{soloscoif-genpourphi}
\Phi (t,x;s,y,\xi) := (s - t ) \, p(\xi) + (x-y) \, \xi - \psi(s,y) .  
  \end{equation}
We assume that the phase $\psi$ is \emph{at most quadratic}, in the sense that (for all $n\ge 2$)
\begin{equation} \label{boubdedfctbyn} 
\sup_{2 \leq \vert \alpha \vert \leq n} \ \sup_{(t,x) \in \RR^2} \
\vert \part^\alpha_{t,x} \psi (t,x) \vert =: \mathfrak B^\psi_n <\infty .  
\end{equation}
Such a notion is quite standard (see e.g. \cite{Fujiwara,KiKu81}) in the construction of fundamental solutions which are 
issued from Schr\"odinger equations involving a potential that is at most quadratic in space. Then, Hamilton-Jacobi equations 
with phases which are at most quadratic in space come into play. Here, the right variable is $(t,x)$. In the case of $\varphi$, 
see \eqref{pharetain}, the phase is of the form {\it linear}+{\it quadratic}+{\it bounded}. The quadratic part corresponds to 
the factor $tx$, and the function $\varphi$ is indeed at most quadratic. The asymptotic behavior of $\cI$ when $ \eps $ goes to 
$ 0 $ depends heavily on the existence or not of stationary points when looking at the phase $ \Phi (t,x;\cdot) $ as a
function of the variables $ (s,y,\xi) $. In view of
(\ref{soloscoif-gen}),  
we have
\begin{subequations}\label{statonlysy} 
\begin{eqnarray} 
\qquad & &  \displaystyle \part_s \Phi (t,x;s,y,\xi) \equiv \part_s \Phi (s,y,\xi) := p(\xi) - \part_s \psi(s,y) \, , \label{statonlysy1} \\
\qquad & & \displaystyle \part_y \Phi (t,x;s,y,\xi) \equiv \part_y
           \Phi (s,y,\xi) := -\xi - \part_y \psi(s,y) \,
           . \label{statonlysy2}  
\end{eqnarray}
\end{subequations}


\subsection{Non-resonant oscillatory integrals}\label{subsec:non-resoscint}

Unlike the derivative $ \part_\xi \Phi (t,x;\cdot) $, the right hand sides of (\ref{statonlysy}) do not involve the parameter $ t $. 
This makes things easier in the perspective of nonstationary phase arguments (in $ s $ and $ y $) applied with $ t \sim \eps^{-1} $. 
For this reason, we introduce the following notion.

\begin{defi}[Non-resonant phase] \label{defnon-resonantphase} Let $ r \in ]0,\gamma/2[ $ be the number allowing at the level of 
(\ref{eq:source}) to control uniformly the spatial supports of the profiles $ A_m $ inside (\ref{eq:source}). Fix some domain $ \cN $ 
satisfying
\begin{equation} \label{posideN}  
    \cN \subset\{ (s,y,\xi)\in \, ]0,\infty[\times ]-r,r[ \times\RR\} .
\end{equation}
The phase $ \psi $ is said to be non-resonant on the domain $ \cN $ if there exists a positive constant $ \eta>0 $ such that
\begin{equation} \label{non-resonanty} 
\qquad \forall \, (s,y,\xi) \in \cN \, , \qquad  \vert p(\xi) - \part_s \psi (s,y) \vert + \vert \xi +\part_y \psi (s,y) \vert \ge \eta . 
\end{equation}
\end{defi}

\noindent As a subset of the characteristic variety $ \text {Char} (L) $, one can distinguish
\begin{equation*} 
\quad \ \text {Char} (L,\cN) := \bigl \lbrace \bigl( s,y,p(\xi),\xi \bigr) \,  ; \, (s,y,\xi) \in \cN \bigr \rbrace \subset 
T^*(\RR^2)  . 
\end{equation*}
As a subset of  the Lagrangian manifold $ \cG (\phi) \subset T^*(\RR^2) $, one can identify
\begin{equation*} 
\cG_r(\psi) := \bigl \lbrace ( s,y,\part_s \psi, \part_y \psi ) (s,y) \,  ; \, s \in ]0,+\infty[ \, , \, y \in \, ]-r,r[ \bigr \rbrace  .  
\end{equation*}
The geometrical interpretation of Definition \ref{defnon-resonantphase} is the following. The phase $ \psi $ is 
non-resonant on the domain $ \cN$ if and only if the two subsets $ \text {Char} (L,\cN) $ and $ \cG_r(\psi) $ of 
the cotangent bundle $ T^* (\RR^2) $ stay a positive distance $ \eta $ away from each other. We now address 
the various harmonics $\psi = m\varphi$, where $\varphi$ is given by \eqref{pharetain}. We will see that two 
values play a special role, namely $m=0$ and $m=1$.

\begin{lem}[Non-resonant harmonics $ m \, \varphi $] \label{lem-nonresonantharm} Take 
  \[
    \cN =\{ (s,y,\xi)\in ]0,\infty[\times ]-r,r[\times\RR\}.
    \]
 Let $ \varphi $ as in (\ref{pharetain}).
 For $ m \in \ZZ \setminus \{ 0,1 \} $, the phase $ m \varphi $ is non-resonant on $\cN $.
\end{lem}

\begin{proof} In the case of $ \psi = m \varphi $, the condition (\ref{non-resonanty}) becomes
  \[
    \forall \, (s,y,\xi) \in \cN_m \, , \qquad 0 < \eta \leq \vert p(\xi) - m + m \, y + \gamma \, m \, \sin s \vert + \vert \xi - m \, s 
    \vert .
    \]
For $ m <0 $, the properties $p\ge0$, and  $r<\gamma/2$ from
Assumption~\ref{choiopro}, yield 
\begin{align*}
\vert p(\xi) - m + m \, y + \gamma \, m \, \sin s \vert & \ge \vert p(\xi) - m \vert - \vert m \, y + \gamma \, m \, \sin s \vert \\
\ & \ge \vert m \vert \ (1-r-\gamma) > 5/8. 
\end{align*}
For $ m \geq 2 $, the property $0\le p(\xi)\le 1$ yields
\[
  \vert p(\xi) - m + m \, y + \gamma \, m \, \sin s \vert \ge m \
  (1-r-\gamma) -1 > 1/4 .
  \]
It suffices to take $ \eta = 1/4 $ to get the result.
\end{proof}

\noindent The phase $ \psi \equiv 0 $, which corresponds to the choice $ m \varphi $ with $ m = 0 $, requires a different treatment,
which explains the hypothesis on $\zeta_0$ in Assumption~\ref{choiopro}.

\begin{lem}[The non-resonant zero-phase]  \label{ex-zerophase}  For $\delta>0$, set
  \[
    \cN_0^\delta =\{ (s,y,\xi)\in ]0,\infty[\times ]-r,r[ \times\RR,\quad |\xi|\ge \delta\}.
    \]
 Then, for all $\delta>0$, the phase $\psi\equiv 0$   is non-resonant
 on $\cN_0^\delta$. 
\end{lem}

\begin{proof}  It suffices to remark that
\begin{equation*} 
 \forall \, (s,y,\xi)  \in \cN_0^\delta \, , \qquad 0 < \delta \le \vert \xi \vert \leq \vert p(\xi) \vert + \vert \xi \vert . 
\end{equation*}
This yields (\ref{non-resonanty}) in the case $ \psi \equiv 0 $.
\end{proof}
\begin{rem}
  The above lemma becomes wrong when $\delta=0$, even if $\xi_c=0$. Indeed, in view of Assumption~\ref{resoa}, the sum 
  $ \vert p(\xi) \vert + \vert \xi \vert $ vanishes at $\xi=0$. 
\end{rem}

\noindent Finally, the remaining case $m=1$ turns out to be the richest. It will be analyzed in details in Subsection~\ref{subsec:resoscint}.
  
\smallskip

\noindent The above notion of non-resonant phase is motivated by the following result.

\begin{prop}[Vanishing oscillatory integrals] \label{prop-vanoscint}
  Let $\psi\in \cC^\infty(\RR;\RR)$ be an at most quadratic phase in the sense of (\ref{boubdedfctbyn}). We also assume that $ \psi $ is 
  non-resonant on a domain $\cN$ satisfying (\ref{posideN}). Let $ \zeta  $ and $ a $ be two smooth functions which are 
  adjusted as in Assumption \ref{choiopro}. Then, for all time $ t \geq \eps^{-1} \, \cT $, the Fourier integral operator $ \cI (\eps,t,x;\psi,\zeta,a) $  
issued from \eqref{soloscoif-gen} is well-defined as an oscillatory integral. In addition, for all $ n \ge 2$, there exists a constant  
$ C_n >0 $ such that, for all $ t \le 2\cT/\eps $ and for all
$ x \in \RR $, we have 
\begin{equation*} 
\qquad \vert \cI (\eps,t,x;\psi,\zeta,a) \vert \leq C_n \ r \ \cT \
\eps^{n-\frac{1}{2}}\eta^{1-4n} \ (1 + \mathfrak B^\psi_n)^n 
\ \parallel a \parallel_{W^{n,\infty} ( \cN)} . 
\end{equation*}
\end{prop}

\noindent As a corollary, we can assert that $ \cI = \cO (\eps^\infty) $.

\begin{proof} The notion of non-resonant phase is designed to apply nonstationary phase arguments with respect to the variables $ s $
and $ y $. But the details remain to be worked out. The main problems when dealing with (\ref{soloscoif-gen}) are due to the domain of 
integration which, knowing that $ t \sim \eps^{-1} $, is of size $ \eps^{-1} $, as well as to the quadratic behavior of $ \psi $.

\smallskip

\noindent As a consequence of (\ref{statonlysy}) and (\ref{non-resonanty}), we have $ (\part_s 
\Phi)^2 + (\part_y \Phi)^2 \not = 0 $ on the domain $ \cN $. Thus,  we can 
introduce
\[ Q (s,y;\part_s , \part_y) := \, i \( (\part_s \Phi)^2 + (\part_y \Phi)^2 \)^{-1} \ (\part_s \Phi \ \part_s + \part_y \Phi \ \part_y ) . \]
By construction, we have
\begin{equation} \label{IPPfirst} 
\eps  Q(s,y;\part_s , \part_y) \, e^{-i \, \Phi/\eps} = e^{-i \, \Phi/\eps}  . 
\end{equation}
Select some $ n \ge 1 $. After $ n $ integrations by parts using the identity (\ref{IPPfirst}), the expression $ \cI $ of (\ref{soloscoif-gen}) 
is transformed into
\begin{equation} \label{FIOpostIPP} 
\cI = (-1)^n \ \frac{\eps^{n+ \frac{1}{2}}}{2\pi} \ \int \left( \int \! \! \int e^{- i \, \Phi /\eps} \ (Q^*)^n \, \(\zeta(\xi)a(\eps s, s,y) 
\) \ ds \, dy \right)  d \xi  , 
\end{equation}
where $ Q^* $ is the adjoint operator of $ Q $. In order to compute $ Q^* $, first remark that
\begin{equation} \label{derisys}  
\part^2_{ss} \Phi = - \part^2_{ss} \psi \, , \qquad \part^2_{sy} \Phi = - \part^2_{sy} \psi \, , \qquad \part^2_{yy} \Phi = - \part^2_{yy} 
\psi . 
\end{equation}
Recall the standard conventions
\[
  \alpha = (\alpha_1,\alpha_2) \in \NN^2 \, , \quad \vert \alpha \vert := \alpha_1 + \alpha_2 \, , \quad X = (X_1,X_2) \, , \quad 
  X^\alpha = X_1^{\alpha_1} \, X_2^{\alpha_2} .
  \]
Introduce the rational functions
\[
  R^1_{1,0} (X) := \frac{P^1_{1,0} (X)}{X_1^2 + X_2^2} , \quad P^1_{1,0} (X) := X_1  , \quad R^1_{0,1} (X) := \frac{P^1_{0,1} 
    (X)}{X_1^2 + X_2^2}  , \quad P^1_{0,1} (X) := X_2  ,
  \]
as well as
\[
  R^1_{0,0} (X) := \frac{P^1_{0,0} (X)}{(X_1^2 + X_2^2)^2}  , \qquad P^1_{0,0} (X) := (-\part^2_{ss} \psi + \part^2_{yy} \psi ) \ 
  (X_1^2 - X_2^2) + \, 4 \ \part^2_{sy} \psi \ X_1 X_2  .
  \]
Using (\ref{derisys}) and the above definitions, we get
\[- i \ Q^* = R^1_{1,0} (\part_s \Phi,\part_y \Phi) \ \part_s + R^1_{0,1} (\part_s \Phi,\part_y \Phi) \ \part_y + R^1_{0,0} 
  (\part_s \Phi,\part_y \Phi) .
  \]
Since $ \psi$ is at most quadratic, the functions $ P^1_\star
(X_1,X_2) $ are in the polynomial ring $ \cC^\infty_b (\RR^2)
[X_1,X_2] $, 
whose elements of degree less than $ j $ take the following form
\[ P(X) = \sum_{\vert \alpha \vert \leq j} \, C_\alpha (s,y) \
  X^\alpha \, , \qquad C_\alpha \in \cC^\infty_b (\RR^2)  .
  \]
The derivative $ \part_\star $, with $ \part_\star = \part_s $ or $ \part_\star = \part_y $, can act on the coefficients of $ P $. Define
\[ \part_\star P(X) = \sum_{\vert \alpha \vert \leq j} \, (\part_\star C_\alpha) (s,y) \ X^\alpha \, , \qquad \part_\star C_\alpha \in \cC^\infty_b 
  (\RR^2) \, , \qquad \mathrm{deg} \, (\part_\star P) \leq j  .
  \]
The two polynomials $ P^1_{1,0} $ and $ P^1_{0,1}$ are of degree $ 1 $, whereas $ P^1_{0,0}  $ is of degree $ 2 $. All the 
coefficients of these polynomials are bounded by a constant multiplied by $ \mathfrak B^\psi_2 $, with $ \mathfrak B^\psi_2 $ 
as in (\ref{boubdedfctbyn}). Introduce the graded ring
\[
  \cR := \bigoplus_{j= 1}^\infty \, \cR_j \, , \qquad \cR_j := \left \{ R (X) = \frac{P(X)}{(X_1^2+X_2^2)^j} \, ; \, P \in \cC^\infty_b 
    [X_1,X_2] \, , \, \operatorname{deg} P \leq j \, \right \} .
  \]
We have $ R^1_{1,0} \in \cR_1 $ and $ R^1_{0,1} \in \cR_1 $. On the other hand, we find $ R^1_{0,0} \in \cR_2 $. Now, given 
$ R \in \cR_j $ as above, we can compute
\[ \begin{array} {rl}
\displaystyle \part_\star \left \lbrack \frac{P(\part_s \Phi,\part_y \Phi)}{( (\part_s \Phi)^2 + (\part_y \Phi)^2 )^j } \right \rbrack = \! \! \! 
& \displaystyle - \, \frac{( \part^2_{\star s} \psi  \, \part_{X_1} \! P + \part^2_{\star y} \psi \, \part_{X_2} \! P) \ ( \part_s \Phi^2 
+ \part_y \Phi^2 )}{((\part_s \Phi)^2 + (\part_y \Phi)^2 )^{j+1} } \\
\ & \displaystyle + \, 2 \ j \ \frac{ P \ (\part^2_{\star s} \psi
    \, \part_s \Phi + \part^2_{\star y} \psi \, \part_y \Phi
    )}{((\part_s \Phi)^2  
+ (\part_y \Phi)^2 )^{j+1}} + \frac{\part_\star P (\part_s \Phi,\part_y \Phi)}{( (\part_s \Phi)^2 + (\part_y \Phi)^2 )^j} .
   \end{array} 
  \]
This implies that $ \part_\star \cR_j \subset \cR_j \oplus \cR_{j+1} $. Now, a simple induction on $ n $ shows that
\begin{equation} \label{qstarn}  
 (-i)^n \ Q^* (s,y;\part_s , \part_y)^n = \, \sum_{\vert \alpha \vert \leq n} \, R^n_\alpha (\part_s \Phi,\part_y \Phi) \ 
\part^\alpha_{s,y} , 
\end{equation}
with
\begin{equation} \label{qstarncontent}  
\qquad \ R^n_\alpha (X) = \! \! \! \! \sum_{j= n}^{2 n-\vert \alpha\vert} \frac{P^{n,j} _\alpha (X)}{(X_1^2+X_2^2)^j} \, , 
\quad \operatorname{deg}  P^{n,j} _\alpha \leq j \, , \quad \ R^n_\alpha \in \! \! \! \bigoplus_{j= n}^{2 n-\vert \alpha\vert} \, \cR_j . 
\end{equation}
In view of (\ref{FIOpostIPP}), (\ref{qstarn}) and (\ref{qstarncontent}), to get a control on $ \cI $, we have to estimate terms which 
look like
\begin{equation*} 
\qquad \eps^{n+ \frac{1}{2}} \ \int \left( \int_0^{2\cT/\eps}  \! \! \int_{-r}^{+r} \, \frac{\vert P^{n,j} _\alpha (\part_s \Phi,\part_y \Phi) \vert}
{((\part_s \Phi)^2+(\part_y \Phi)^2)^{j/2}} \ \frac{\parallel a \parallel_{W^{n,\infty} (\cN )}}{((\part_s \Phi)^2+(\part_y \Phi)^2
)^{j/2}} \ ds \, dy \right)  d \xi .
\end{equation*}
From (\ref{non-resonanty}), we can deduce that 
\begin{equation} \label{estimingradphi} 
0 < \frac{1}{(\part_s \Phi)^2 + (\part_y \Phi)^2} \leq \inf \ \Bigl( \, \frac{2}{\eta^2} \, ; \, \frac{1}{\vert \xi + \part_y \psi (s,y) \vert^2} \Bigr) . 
\end{equation}
On the other hand, the coefficients of $ P^{n,j} _\alpha $ can be roughly controlled by a constant multiplied by $ (1 + \mathfrak 
B^\psi_n)^n $. Note that $ \part_s \Phi $ and $ \part_y \Phi $ are not necessarily bounded in $ \eps $ for $ s \sim \eps^{-1} $. 
Actually, the expression $ \part_y \Phi $ is not at all bounded in the case of $ \varphi $. This is why the information 
$  \operatorname{deg}  P^{n,j} _\alpha \leq j $ is important. Exploiting (\ref{estimingradphi}), we can find a constant $ C^{n,j} $ such that
\[ \forall \, (s,y,\xi) \in \cN\, , \qquad \frac{\vert P^{n,j} _\alpha (\part_s \Phi,\part_y \Phi) \vert} {((\part_s \Phi)^2 
+ (\part_y \Phi)^2 )^{j/2}} \leq \frac{C^{n,j}}{\eta^j} \ (1 + \mathfrak B^\psi_n)^n .
\]
Now, the idea is to use the right hand side of (\ref{estimingradphi}) and the condition $ 2 \leq n \leq j $ inside (\ref{qstarncontent}) 
to recover some integrability in $ \xi $. As a matter of fact, we have
\begin{align*} 
\displaystyle \int \left( \int_0^{2\cT/\eps} \! \! \! \! \int_{-r}^{+r} \! \! \! \right. & \left. \frac{ds \, dy}{((\part_s \Phi)^2+(\part_y 
\Phi)^2)^{j/2}} \right) d \xi \\
\ &  \leq \int_0^{2\cT/\eps}  \! \! \int_{-r}^{+r} \left( \int \inf \ \Bigl( \, \frac{2^{\frac{j}{2}}}{\eta^{j}} \, ; \, \frac{1}{\vert \xi + 
\part_y \psi (s,y) \vert^j} \Bigr) \ d \xi \right) \, ds \, dy \\
\ & \leq \frac{4 \ r \ \cT}{\eps \ \eta^{j-1}} \, \int \inf \ \Bigl( \, 2^{\frac{j}{2}} \, ; \, \frac{1}{\vert \xi \vert^j} \Bigr) \ d \xi < + \infty . 
  \end{align*}
Combining all the above information, we get the expected result.
\end{proof}

\smallskip

\noindent \underline{Preliminar}y \underline{conclusion}. In view of Lemma~\ref{lem-nonresonantharm} together with 
Proposition~\ref{prop-vanoscint}, for all $m\in \ZZ\setminus\{0,1\}$, we get $u^m = \cO (\eps^\infty) $. In the case $m=0$, 
the property $\xi_0>0$ in Assumption~\ref{choiopro},  Lemma~\ref{ex-zerophase} (with $\delta = \xi_0$) as well as  
Proposition~\ref{prop-vanoscint} also yield $u^0 = \cO (\eps^\infty) $. This means that {\bf all modes $ u^m $ with $ m \not = 1 $ are negligible}. 

\smallskip

\noindent The harmonic $ m = 1 $ is the only choice which may give rise to interesting phenomena from the viewpoint  
of quasi-rectification. This is Fact \ref{fact1} in the PDE context (\ref{third-si}). With this in mind, in the rest of this
section, we focus on the case $ m = 1 $. We work with $ \psi \equiv \varphi $ and $ a \equiv a^1 $. Unless otherwise
specified, for $ m = 1 $, we will simply use the notations $ a \equiv a^1 $  and $ u \equiv u^1 $. Thus, in the next subsection 
\ref{subsec:resoscint}, the function $ u $ represents the solution to 
\begin{equation} \label{eq:mainajoutlinpourm}
  \part_t u -\frac{i}{\eps}p\(\eps D_x\) u + \eps^{3/2} A_1^* \ e^{i \, \varphi(t,x) / \eps} =0,
  \quad u_{\mid t=0}=0 .
\end{equation}


\subsection{Resonant oscillatory integrals}\label{subsec:resoscint} The analysis of the expression $ \cI $ defined by 
(\ref{soloscoif-gen})-(\ref{soloscoif-genpourphi}) with $ \psi = \varphi $ relies quite heavily on the explicit formula (\ref{pharetain}). The 
aim here is to give an overview of next Subsections~\ref{startregime}, \ref{transiregime} and \ref{longregime}, where precise 
results will be established. The reader who is not interested in the details of proofs can read this Subsection~\ref{subsec:resoscint}, 
and then go directly to Section \ref{sec:nonlinear}.
In Paragraph~\ref{sec:basicmechanisms}, we clarify what is new in comparison with the phenomena which have 
been exhibited in the introduction. In Paragraph~\ref{sec:decomposition}, we introduce a well-adapted partition of the long 
time interval $ [0,\cT /\eps ] $ with $ \cT >0 $. Then, in Paragraph~\ref{sec:vs}, this leads to a distinction between 
a dispersive and some almost stationary regime. What happens in these two regimes is described in Paragraphs 
\ref{subsubsec:dispreg} and \ref{subsubsec:stareg}, successively.


\subsubsection{Basic mechanisms} \label{sec:basicmechanisms} 

\noindent The toy model presented in Subsection~\ref{subsec:toymodel} explains how wave packets can be produced 
over large time. It shows that the creation of wave packets is basically due to the combination of two factors: the first is the 
presence of a resonance; the second is the introduction as a source term of well-adjusted oscillations. In the toy model of 
the introduction, the discussion involves basic choices of $L $ ($ p\equiv  1 $) and $ \varphi $ (which does not depend on $ x $). 
Now, we want to better understand what happens under the more realistic geometrical conditions, which are 
Assumptions~\ref{resoa} and~\ref{choiopro}, together with formulas \eqref{pharetain} and \eqref{eq:source}.

\smallskip

\noindent When the symbol $ p$ is not constant, the emitted signals do propagate spatially. The form of the wave packets 
is determined by the local geometrical characteristics of $ \cV $ at the intersection points between $ \text {Char} (L) $ and 
$ \cG(\varphi) $. But the way in which these wave packets propagate and can accumulate over time depends on the 
asymptotic properties of $ \cV $ when $ \vert \xi \vert $ goes to $ + \infty $ and of $ \cG(\varphi) $ when $ t $ goes to 
$ + \infty $. 

\smallskip

\noindent It follows that the whole picture is a combination of {\it local} and {\it global} geometrical features of $ \text {Char}
(L) $ and $ \cG(\varphi) $. The resulting effects are constructive and destructive interferences. They appear in the 
absence of nonlinearity. 


\subsubsection{Decomposition into wave packets} \label{sec:decomposition} 

\noindent The function $ \Phi $ of (\ref{phaseOIFgen}) with $ m = 1 $ - or equivalently the function $ \Phi $ of (\ref{soloscoif-genpourphi}) 
with $ \psi = \varphi $ - can be separated into two parts, according to $ \Phi = \phi + \gamma - t $. Since the symbol $ p$ 
is even, we find
\begin{equation*} 
\phi(t,x;s,y,\xi) := (s - t ) \, \bigl \lbrack p(\xi) -1 \bigr \rbrack + (x-y) \, \xi + y \, s - \gamma \, \cos s . 
\end{equation*}
Accordingly, another way to formulate (\ref{soloscoif-gen}) is to write
\begin{equation} \label{solosc} 
\qquad \cI := \frac{\sqrt \eps}{2\pi} \ e^{i (t-\gamma)/\eps} \, \int \left( \int_0^t \! \int e^{- i \, \phi(t,x;s,y,\xi) /\eps} \zeta(\xi)
a(\eps \, s, s,y) \ ds \, dy \right)  d \xi .
\end{equation}
From now on, it is implicitly assumed that $ t $ (or $ T=\eps t$) runs over a finite period of long times, in coherence with (\ref{suppdea}).
More precisely, we wait until the action of the perturbation through $ a  $ is completed, that is
\begin{equation} \label{adjusttT} 
\cT / \eps \leq t \leq 2 \, \cT / \eps  ,\quad \text{or equivalently}\quad\cT \leq T \leq 2 \, \cT . 
\end{equation}
Given $ t_0 >0$, with $ \chi $ as in (\ref{mollichi}), define
\begin{equation} \label{defchivar} 
\forall \, s \in \RR \, , \qquad \chi_{t_0} (s) := \chi (s / t_0 ) . 
\end{equation}
In addition, we can tune $\chi$ so that it generates a partition of unity with
\begin{equation} \label{partibuiltchino}  
\sum_{k \in \ZZ} \, \chi_{2 \pi/3} (x-k\pi) = 1 . 
\end{equation}
We introduce some truncation near the diagonal $ s = \xi $, namely
\begin{equation} \label{defvrv} 
v(t,x) := \int_0^t \! \int \! \! \int e^{- i \, \phi(t,x;s,y,\xi)  /\eps} \zeta(\xi) a (\eps \, s , s, y) \ \chi_{1/4} (s-\xi) \ ds \, dy \, d \xi .  
\end{equation}
The above expression $ v(t,x) $ is given by the integral in $ (s,y,\xi) $ on the compact domain $ [0,t] \times [-r,r] \times 
[- 1/4, 1/4+t] $ of functions depending smoothly on $ (\eps, t,x) $. It is therefore a well-defined smooth function of 
$ (\eps, t,x) $. This choice is motivated by the following lemma.
\begin{lem} \label{lem-nonresonantoutof} For all $\delta>0$, the phase $\varphi$ given by \eqref{pharetain}  is 
non-resonant on the set 
  \[
    \cN_1^\delta =\{ (s,y,\xi) \in \, ]0,\infty[\times \RR\times\RR,\quad | \xi- s |\ge \delta\} .
    \]
\end{lem}

\begin{proof} Requiring $ \varphi $ being non-resonant  on $ \cN_1^\delta $ amounts to finding $ \eta > 0 $ such that
\begin{equation*} 
\qquad \forall \, (s,y,\xi) \in \cN_1^\delta\, , \qquad 0 < \eta\leq
\vert p( \xi) - 1 + y + \gamma \, \sin s \vert + \vert \xi -  s \vert . 
\end{equation*}
Now, by definition, we have $ \delta\le  \vert \xi -  s \vert $ for all $ (s,y,\xi) \in \cN_1^\delta$. Thus, the choice $ \eta = \delta $ is suitable.
\end{proof}

\noindent In view of Lemma~\ref{lem-nonresonantoutof}  with $\delta=1/4$, Proposition~\ref{prop-vanoscint} yields
\begin{equation}\label{remindudevenk}
 u(t,x) = \frac{\sqrt \eps}{2\pi} \ e^{i \, (t-\gamma) / \eps} \ v (t,x) + \cO\(\eps^\infty\),\quad x\in \RR,\quad \frac{\cT}{\eps} \le t\le \frac{2\cT}{\eps}. 
\end{equation}
There remains to analyze $ v $. Using again \eqref{suppdea} with $ m= 1 $ and $ a \equiv a_1 $,
we can replace (\ref{defvrv}) by an integral on the whole domain $ \RR^3 $, which is
\begin{equation} \label{defvrvkensum} 
v(t,x) = \iiint e^{- i \, \phi(t,x;s,y,\xi) /\eps} \zeta(\xi) a (\eps \, s, s, y) \ \chi_{1/4} (s-\xi) \ ds \, dy \, d \xi .  
\end{equation}
Using (\ref{adjusttT}) and (\ref{partibuiltchino}), we can write
\begin{equation} \label{sizeofk} 
\ v(t,x) = \sum_{k \in \cK} \, v_k (t,x) \, , \qquad \cK := \Bigl \lbrace \, k \in \NN \, ; \, k \leq \frac{2}{3} + \frac{\cT}{\pi \, \eps} \, \Bigr \rbrace ,
\end{equation}
\noindent where, for $ k \in \cK $, we have introduced the signal
\begin{equation} \label{defvrvk} 
\begin{aligned}
 v_k (t,x) := \iiint &  e^{- i \, \phi(t,x;s,y,\xi) /\eps}     \zeta(\xi) a (\eps \, s , s, y) \\ 
 &  \qquad \times \, \chi_{1/4} (s-\xi) \ \chi_{2 \pi/3} (s-k\pi) \ ds \, dy \, d \xi  .
\end{aligned}
\end{equation}
Fix $ k \in \ZZ $, and change $ s $ into $ s - k \, \pi $ and $ \xi $ into $ \xi - k \, \pi $ in order to obtain
\begin{equation} \label{combi1} 
v_k (t,x) = e^{- i \, ( t - k \, \pi + k \, \pi  \, x ) / \eps} \ w_k (t,x) ,
\end{equation}
with
\begin{equation} \label{defidewk} 
\begin{aligned}
 w_k (t,x) := \iiint  &  e^{- i \, \Phi_k (t,x;s,y,\xi) /\eps} \ a_k (\eps, s,y,\xi) \\ 
 & \ \qquad \times \, \chi_{1/4} (s-\xi) \ \chi_{2\pi/3} (s) \ ds \, dy \, d \xi , 
\end{aligned}
 \end{equation}
and where
\begin{subequations}\label{defphikini} 
\begin{eqnarray} 
  \qquad & &  \displaystyle a_k (\eps, s,y,\xi) := \zeta(k\pi+\xi)
             a (\eps \, k \, \pi + \eps \, s , k \, \pi + s,y) , \label{defphikini1} \\
\qquad & & \displaystyle \Phi_k (t,x;s,y,\xi) := (k \, \pi -t) \ p (k \, \pi + \xi) + s \ \bigl \lbrack p (k \, \pi + \xi) - 1 \bigr \rbrack \qquad \label{defphikini2} \\
\qquad & & \qquad \qquad \qquad \qquad + (x-y) \, \xi + s \, y - 
(-1)^k \ \gamma \, \cos s . \nonumber
\end{eqnarray}
\end{subequations}

\noindent In view of (\ref{sizeofk}), the expression $ v$ is the sum of the wave packets $ v_k$. Now, the advantage 
when working with $ v_k$ (or $ w_k $) is that the domain of integration in $ (s, y,\xi) $ is compact and independent of
$ \eps $, $ k $ and $ t $. As a counterpart, at the level of line (\ref{defidewk}), the phase $ \Phi_k $ and the profile $ a_k $ 
involve $ \eps $, $ k $ and $ t $ as parameters. 

\smallskip 

\noindent The expression $ F_{\! L} $ of \eqref{eq:source} looks like the one of (\ref{ajoutFLNL}). By analogy with 
(\ref{third-expliu})-(\ref{third-si-devtt}), the expected amplitude of $ u $ is $ \cO (\eps ) $. In comparison to this 
reference threshold, all terms of smaller size $ o(\eps) $ will be considered negligible.

 
\subsubsection{Dispersive \emph{vs.} almost stationary regime} \label{sec:vs}

The asymptotic behavior of $ w_k  $ when $ \eps $ goes to $ 0 $ depends heavily on the size of $ k $. In view  of 
\eqref{mollichi} and \eqref{suppdea}, it suffices to deal with integers $ k $ such that $ k \in \cK $, with $ \cK $ as 
in (\ref{sizeofk}). Let $ q $ be the integer of Assumption \ref{resoa}. Given some $ c >0 $, we can decompose 
$ \cK $ into two separate parts $ \cK^c_d $ and $ \cK^c_ s $ with
\begin{subequations}\label{defdispki} 
\begin{eqnarray} 
\qquad & &  \displaystyle \cK^c_d \equiv \cK^c_d \Bigl( \frac{1}{q+1} \Bigr) := \Bigl \lbrace \, k \in \NN \, ; \, k \leq 
\frac{c}{\eps^{1/(q+1)}} \, \Bigr \rbrace , \label{defdispki1} \\
\qquad & & \displaystyle \cK^c_s \equiv \cK^c_s \Bigl( \frac{1}{q+1} \Bigr) := \Bigl \lbrace \, k \in \NN \, ; \, \frac{c}
{\eps^{1/(q+1)}} < k \leq \frac{2}{3} + \frac{\cT}{\pi \, \eps} \, \Bigr \rbrace . \label{defdispk2} 
\end{eqnarray}
\end{subequations}
In (\ref{defdispki}), the symbol $ \cK $ has exponent $ c $ and subscripts $ d $ or $ s $. The exponent $ c $ is 
aimed to specify the choice of the constant $ c $, while the subscripts $ d $ or $ s $ are used to refer respectively 
to the words {\it \underline{d}ispersion} and {\it almost \underline{s}tationary}.
\begin{itemize}
\item In Subsection \ref{startregime}, by adjusting $ c >0 $ small enough, we can ensure some strong 
\underline{d}ispersion of the wave packets $ w_k$ for values of $ k $ in $ \cK^c_d $ and for $t$ of order 
$1/\eps$. The contribution of these wave packets is negligible. They mainly emerge from the ball $ \vert x \vert < r $, 
with $ r $ as in Assumption \ref{choiopro}.
\item In Subsection~\ref{transiregime}, we consider the case $ k \in \cK^c_s $, for which signals may be emitted 
at a higher order. Such signals  are almost \underline{s}tationary, and therefore they can be detected 
during long times $ t \sim \eps^{-1} $ inside the ball $ \vert x \vert < r $. In Subsection \ref{longregime}, we study 
the local accumulation (near the position $ x = 0 $) of these wave packets.
\end{itemize}


\subsubsection{Dispersive regime ($ k \in \cK^c_d$)} \label{subsubsec:dispreg} Using a non-stationary phase argument in $\xi$, 
we will establish in Lemma~\ref{localizations} that for any given choice of $R>0$, we can find a constant $c>0$ such that
\begin{equation*}
  \sup \ \Bigl \lbrace \, \vert w_k (t,x) \vert \, ; \ 0 \leq k \leq \frac{c}{\eps^{1/(q+1)}} \ , \ \frac{\cT}{\eps} \leq t \leq \frac{2
    \, \cT}{\eps} \ , \ \vert x \vert \leq R \, \Bigr \rbrace  = \cO \(\eps^{D-1}\) , 
\end{equation*}
where $D \geq 2 $ is the integer appearing in Assumption~\ref{strengthenp}. Therefore, since $ q \geq 2 $, the terms $ w_k$
with $ k \in \cK^c_d $ enter into the composition of $ u$ through a contribution which can be estimated according to 
\begin{equation}\label{dispregesti}
\frac{\sqrt \eps}{2 \, \pi} \ \sum_{0 \leq k\le c \, \eps^{-1/(q+1)}}|w_k(t,x)| \, \lesssim \, \eps^{\frac{1}{2} + D-1-\frac{1}{q+1}} \, 
\lesssim \, \eps^{7/6} \ll \eps . 
\end{equation}


\subsubsection{Stationary regime ($ k \in \cK^c_s $)} \label{subsubsec:stareg}

\noindent The absence of dispersion may be revealed through the existence of critical points when looking at the phase $ \Phi_k $.
Accordingly, for $k\in \cK^c_s $, the asymptotic behavior of $w_k$ when $ \eps $ goes to $ 0 $ will be analyzed by stationary phase arguments: 
\begin{itemize}
\item For all $k\in \cK^c_s $, the phase $ \Phi_k$ involved in the
  definition of $w_k$ in \eqref{defidewk} has at most one critical point
  in the domain of integration (Lemma~\ref{statioponi}). 
\item For all $k\in \cK^c_s $, the above mentioned possible critical
  point is necessarily non-degenerate (Lemma~\ref{statioponidege}). 
\item Possible values of $ k $ such that $ c \leq \eps^{\frac{1}{q+1}}
  \, k < c_1 $ with $ c_1 $ large enough belong to a transition zone,
  treated as a black box. For all 
$ k\in \cK^{c_1}_s $, the phase $\Phi_k$ has indeed a unique critical point (Lemma~\ref{statioponialways}). 
  \item When $ D \geq 3 $, the solution $ u $ to (\ref{eq:mainajoutlinpourm}) can
    be viewed modulo $ \cO \(\eps^{5/3}\)$ as a sum of
    wave packets $ u_k$ with $ k \in \cK^{c_1}_s $ (Lemma
    \ref{decomposol}).
\end{itemize}


\subsubsection{Accumulation of the wave packets} \label{sec:accumwavepack}

\noindent At this stage, Assumption~\ref{persourcea} (about the property of periodicity for large times) on the amplitude $ a$ is needed. Then, a clear 
asymptotic dichotomy occurs regarding the order of magnitude of $ u $. According to the
position in time and space, different orders of magnitude  
are possible for the wave function $u$, with a precise expression when
$u$ reaches its maximal order of magnitude.  

\begin{itemize}
\item Constructive interference
  (Proposition~\ref{supnormamplifi}). For $x= 2j\eps$ with $j\in \ZZ$
  --  as forecast by (\ref{third-si-devtt})
with $ u^1 $ multiplied by $ \eps^{3/2} $ -- the solution $ u$ to (\ref{eq:mainajoutlinpourm}) is of
order $\eps$ exactly, see (\ref{increasenc}) and
(\ref{increasencbis}). 

\smallskip

\item Destructive interference (Proposition~\ref{supnormnonamplifi2}). On the contrary, for $x=\alpha \eps$ 
with $\alpha\not \in 2\ZZ$, we find that the amplitude of $ u $ is $
o(\eps) $, see (\ref{destructiveinter}). 
\end{itemize}


\subsection{The dispersive regime}\label{startregime}

\noindent This is when $ k \in \cK^c_d $ with $ \cK^c_d $ defined as in (\ref{defdispki1}). In the sum (\ref{defvrvkensum}), 
the associated waves $ v_k $ or $ w_k $ are emitted by the source term during relatively small times $ s $. Thus, for $ t $ 
large enough, that is for $ t $ as indicated at the level of (\ref{adjusttT}), these waves have enough time to (partially) disperse 
away from any ball $ \vert x \vert < R $. 

\smallskip

\noindent It is this idea that is developed and quantified in Lemma \ref{localizations} below. The notations have been set up 
in Subsection \ref{subsec:resoscint}. In particular, the wave packets $ v_k $ and $ w_k $ are given by (\ref{defvrvk}) and
(\ref{defidewk}), respectively.

\begin{lem} [Dispersion of waves when $ k \in \cK^c_d  $] \label{localizations} Let $ w_k $ be defined as in (\ref{defidewk}). 
With $ D $ as in Assumption \ref{strengthenp}, and for $R>0$ fixed, there exists a constant $ c >0$ (depending on $ R $) 
such that
\begin{equation} \label{remainderrv} 
\sup \left\lbrace \vert w_k (t,x) \vert  ; \ 0 \leq k \leq \frac{c}{\eps^{1/(q+1)}}  , \ \frac{\cT}{\eps} \leq t \leq \frac{2 
\, \cT}{\eps} , \ \vert x \vert \leq R \right\rbrace  = \cO \(\eps^{D-1}\) . 
\end{equation} 
\end{lem}

\noindent Note that since $ \vert v_k (t,x) \vert = \vert w_k (t,x)
\vert $, the same applies to $ v_k$. 

\begin{proof} Observe first that the integral defining $ w_k $ at the level of (\ref{defidewk}) is restricted to the compact set
\begin{equation} \label{compsetwk} 
 \Upsilon:=\left\lbrace \, (s,y,\xi) \in \RR^3  ; \, \vert s \vert \leq 2 \, \pi /3 , \ \vert y \vert \leq r  , \ \vert s - \xi \vert 
\leq {1/4}  \right\rbrace ,
\end{equation}
where $ r $ comes from Assumption \ref{choiopro}. 
We apply the principle of non-stationary phase, but this time through integrations by parts involving $ \xi $. To this end, with 
$ \Phi_k $ as in (\ref{defphikini2}), we have to compute 
\begin{equation} \label{sysPhi} 
\part_\xi \Phi_k (t,x;s,y,\xi) = (k \, \pi +s-t) \ p'(k \, \pi + \xi) +x-y . \quad
\end{equation}
The situation is the following. Knowing (\ref{hypp3ded}), for $k\lesssim \eps^{1/(q+1)}$ and
  $t\gtrsim 1/\eps$, since $s,x$  and $y$ are bounded, we find
  \begin{equation*}
    \part_\xi \Phi_k (t,x;s,y,\xi)\approx -t p'(k \, \pi + \xi)+\cO(1)\approx \frac{t}{k^{q+1}}+\cO(1)\gtrsim \frac{\eps t}{c^{q+1}} -C_0 .
  \end{equation*}
  For $ c $ small enough, the right hand side becomes positive. The rest of the proof consists
  in making this heuristical computation more quantitative.
  
\smallbreak
  
\noindent In view of (\ref{suppdea}), on the domain of integration giving rise to $ w_k $, we have  $ 1 \leq  k \, \pi + s $ as well 
as (in view of the support of $\chi_{1/4} $) the inequality $ - 1/4 \leq \xi - s $. It follows that $ \xi_c \leq 1/2 < 3/4 \leq k \, \pi + \xi $. 
Taking into account (\ref{increasing3}) and (\ref{hypp3ded}), knowing that $ \ell <0$, for such values of $ \xi $, we have
\begin{equation} \label{assertthat} 
\exists
\,\delta_0>0, \qquad \delta_0\le \sup_{3/4 \leq \eta} \, \eta^{q+1} \, p'(\eta) \leq (k \, \pi + \xi )^{q+1} \ p'(k \, \pi + \xi ) .  
\end{equation}
For $ t $ as in (\ref{adjusttT}), for $ (s,y,\xi) $ as in (\ref{compsetwk}), and for $ \vert x \vert \leq R $, we have
$$ \frac{\cT}{\eps} - k \, \pi - \frac{2 \, \pi}{3} \leq t - k \ \pi - s \, , \qquad k \, \pi + \xi \leq k \, \pi + \frac{2 \, \pi}{3} +1 \, ,
\qquad \vert x-y \vert \leq r+R . $$
It follows that
\begin{equation} \label{folminc} 
\quad \delta_0  \Bigl( \frac{\cT}{\eps} - k \, \pi - \frac{2 \, \pi}{3} \Bigr)  \Bigl( k \, \pi + \frac{2 \, \pi}{3} + 1 \Bigr)^{-q-1} \! -  r-R 
\leq \vert \part_\xi \Phi_k (t,x;s,y,\xi) \vert . 
\end{equation}
Define
\vskip -6mm
\begin{equation} \label{defdec} 
\quad \ \eps_0 := \frac{3^{q+1}}{(6+4 \, \pi)^{q+1}} \ \frac{\delta_0 \, \cT}{2 \, (\delta_0+ r+R)} \, , \qquad c := \frac{1}{2 \, \pi} \ 
\(\frac{\delta_0\, \cT}{2 \, (\delta_0+  r+R)} \)^{\frac{1}{q+1}} . 
\end{equation}
By this way, for $ \eps \leq \eps_0 $ and $ k \leq c \, \eps^{-1/(q+1)} $, we can assert that
\[
 \eps^{\frac{1}{q+1}} \, \left( \frac{2 \, \pi}{3} +1 \right) \leq \frac{1}{2} \, \left( \frac{\delta_0  \, \cT}{2 \, (\delta_0+ r+R)} 
\right)^{\frac{1}{q+1}} \! \! , \quad \eps^{\frac{1}{q+1}} \, k \, \pi \leq \frac{1}{2} \, \left( \frac{\delta_0 \, \cT}{2 \, (\delta_0+ r+R)} 
\right)^{\frac{1}{q+1}} \! . 
\]
Sum these two inequalities and rearrange the terms to get
\begin{equation} \label{ineginterm1} 
(\delta_0+ r+R) \ \left( k \, \pi  + \frac{2 \, \pi}{3} +1 \right)^{q+1} \leq
\frac{\delta_0\, \cT}{2 \, \eps} .  
\end{equation}
In particular
\[ \delta_0 \left( k \, \pi  +  \frac{2 \, \pi}{3} +1\right) \leq (\delta_0+r+R) \ \Bigl( k \, \pi  + \frac{2 \, \pi}{3} +1 \Bigr)^{q+1} \leq
\frac{\delta_0\, \cT}{2 \, \eps} ,
 \]
 and hence
 \[
 \frac{\cT}{2 \, \eps}  \ge k \, \pi + \frac{2 \, \pi}{3} \, , \qquad \frac{\cT}{2 \, \eps}  \leq \frac{\cT}{\eps} - k \, \pi - \frac{2 \, \pi}{3} .
\]
Exploiting (\ref{ineginterm1}), it follows that
$$ (\delta_0+ r+R) \ \Bigl( k \, \pi  + \frac{2 \, \pi}{3} +1 \Bigr)^{q+1} \leq \frac{\delta_0 \, \cT}{2 \, \eps} \,  \leq \delta_0 \, \Bigl( 
\frac{\cT}{\eps} - k \, \pi - \frac{2 \, \pi}{3} \Bigr) . $$
Coming back to \eqref{folminc}, this yields
\begin{equation} \label{inedephik} 
\quad \forall \, \eps \in \, (0,\eps_0] \, , \quad \forall \, k \in \NN \cap [0, c \, \eps^{-1/(q+1)} ] \, , \quad \delta_0 \leq \vert 
\part_\xi \Phi_k (t,x;s,y,\xi) \vert . 
\end{equation}
As long as $ \part_\xi \Phi_k \not = 0 $, the following identity may be used
\begin{equation} \label{appliLxi} 
 e^{- i \, \Phi_k / \eps} = \cD_\xi \, e^{- i \, \Phi_k / \eps} , \qquad \cD_\xi \equiv D(k)_\xi := \frac{\eps \, i}{\part_\xi \Phi_k} \ \part_\xi . 
\end{equation}
The differential operator $ \cD_\xi $ is not self-adjoint. We have to deal with
\begin{equation} \label{adjointdxi} 
\cD_\xi^* \equiv D(k)_\xi^* = \frac{\eps \, i}{\part_\xi \Phi_k (t,x;s,y,\xi)} \ \part_\xi - \frac{\eps \, i \ \part^2_{\xi \xi} \Phi_k (t,x;s,y,\xi)}{( 
\part_\xi \Phi_k)(t,x;s,y,\xi)^2 } .  
\end{equation}
Knowing that (\ref{inedephik}) is verified on the domain of integration giving rise to $ w_k $, an integration by parts (in the 
variable $ \xi $) using $ \cD_\xi $ yields
\[ \vert w_k (t,x) \vert := \Bigl \vert \iiint e^{- i \, \Phi_k (t,x;s,y,\xi) /\eps} \ \cD_\xi^* \bigl \lbrack a_k (\eps, s,y) \ \chi_{\frac{1}{4}} (s-\xi) \ 
\chi_{\frac{2 \pi}{3}} (s) \bigr \rbrack \ ds \, dy \, d \xi \Bigr \vert .
\]
Taking  \eqref{inedephik} into account, the application of $ \cD_\xi^* $ allows to gain a power of $ \eps $ on condition that 
\begin{equation} \label{divborn} 
\frac{\vert \part^2_{\xi \xi} \Phi_k (t,x;s,y,\xi) \vert }{\part_\xi \Phi_k (t,x;s,y,\xi)^2 } = \frac{\vert (k \, \pi +s-t) \ p''(k \, \pi + \xi) \vert}
{\vert (k \, \pi +s-t) \ p'(k \, \pi + \xi) +x-y \vert^2 } = \cO(1) . 
\end{equation}
The difficulty is that, at the level of (\ref{remainderrv}), neither $ k $ nor $ t $ can be bounded uniformly in $ \eps \in \, ]0,1] $. 
However, with $ X := (k \, \pi +s-t) \ p'(k \, \pi + \xi) $, remark that
\[\frac{\vert x-y \vert}{\vert X \vert} \leq \frac{1}{2} \quad \Longrightarrow \quad \displaystyle \frac{\vert \part^2_{\xi \xi} 
\Phi_k \vert} {\vert \part_\xi \Phi_k \vert} = \Bigl|1 + \frac{x-y}{X} \Bigr|^{-1} \ \frac{\vert p'' (k \, \pi + \xi) \vert }{p' (k \, \pi + \xi)} \leq 2 \ 
\sup_{s \ge 3/4} \ \frac{\vert p''(s) \vert}{p'(s)} .
\]
On the other hand, using (\ref{inedephik}) and then $ \vert x-y \vert \leq r+R $, we have
\[ \frac{\vert x-y \vert}{\vert X \vert} \geq \frac{1}{2} \quad \Longrightarrow \quad \displaystyle \frac{\vert \part^2_{\xi \xi} \Phi_k \vert }{\vert 
\part_\xi \Phi_k \vert} \leq \frac{|X|}{\delta_0} \ \frac{\vert p'' (k \, \pi + \xi) \vert }{p' (k \, \pi + \xi)} \leq
\frac{2(r+R)}{\delta_0} \sup_{s \ge 3/4} \ \frac{\vert p''(s) \vert}{p'(s) } . \]
From (\ref{hypp2}) and (\ref{hypp3ded}), we can deduce that
\[ C(p) := \sup_{s\ge 3/4}  \left\lbrace \, \vert p'' (s) \vert / p'(s) \right \rbrace \, < + \infty . \]
In short, we have (\ref{divborn}) with
$$ \frac{\vert \part^2_{\xi \xi} \Phi_k (t,x;s,y,\xi) \vert }{\part_\xi \Phi_k (t,x;s,y,\xi)^2 } \leq \frac{2}{\delta_0} \ \max \, \Bigl( \, 1 \, ; \, 
\frac{ r+R}{\delta_0} \Bigr) \, C(p) \, < + \infty . $$
By extension, the action of $ (\cD_\xi^*)^{n-1} $ with $ n \leq D $ involves the quotients $ \part_\xi^j \Phi_k / \part_\xi \Phi_k $ with 
$ 0 \leq j \leq n $. Under Assumption \ref{strengthenp}, the above argument can be repeated $ D-1 $ times, leading to the estimate 
(\ref{remainderrv}).
\end{proof}

\noindent Note that (\ref{remindudevenk}) allows to capture any position $ x \in \RR $. On the contrary, to 
obtain uniform $ \cO (\eps^N ) $ bounds and to measure more precisely the quantitative aspects of the 
dispersive effects, we need to restrict the size of the spatial domain, as in \eqref{remainderrv}. This is 
why we will work with $ R=r $, where $ r $ is fixed as in Assumption \ref{choiopro}. In agreement with 
Lemma~\ref{lem-nonresonantharm}, we restrict the discussion to the ball $ \vert x \vert < r $. From now 
on, the values of $ \eps_0 $ and $ c $ are adjusted as indicated in (\ref{defdec}) with $ R=r $.


\subsection{The regime of standing waves}\label{transiregime}

\noindent This is when $ k \in \cK^c_s $ with $ \cK^c_s $ as in (\ref{defdispk2}) and $ c $ as in (\ref{defdec}). 
The novelty when $ k \in \cK^c_s $ is that the oscillatory integral (\ref{defidewk}) defining $ w_k $ may involve 
stationary points, which are positions $ (s,y,\xi) $ satisfying 
$$ \nabla_{\! s,y,\xi} \Phi_k (t,x;s,y,\xi) = 0 , $$ 
or equivalently
\begin{subequations}\label{stastapoint} 
\begin{eqnarray} 
& &  \displaystyle \part_s \Phi_k (t,x;s,y,\xi) = p(k \, \pi + \xi) -1 + y + (-1)^k \ \gamma \, \sin s = 0  , \label{stastapoint1} \\
& & \displaystyle \part_y \Phi_k (t,x;s,y,\xi) = s- \xi = 0  , \label{stastapoint2} \\
& & \displaystyle \part_\xi \Phi_k (t,x;s,y,\xi) = (k \, \pi +s-t) \, p'(k \, \pi + \xi) +x-y = 0 . \label{stastapoint3} 
\end{eqnarray}
\end{subequations}
In Paragraph~\ref{repcripoints}, we show that there exists inside $ \Upsilon $, with $ \Upsilon $ as in (\ref{compsetwk}), at 
most one critical point which is denoted by $ (s_k,y_k,\xi_k) (t,x) $. This means that $ (s_k,y_k,\xi_k) $ is as in (\ref{stastapoint}),
together with
\begin{equation} \label{cdtstapoint} 
\qquad - 2 \, \pi/3 \leq s_k \leq 2 \, \pi/3 \, , \quad \ - r \leq y_k \leq r \, , \quad \ s_k - 1/4 \leq \xi_k \leq s_k + 1/4 .
\end{equation}
In Paragraph \ref{asyforcripoints}, we derive asymptotic formulas describing the behavior of $ s_k $ when $ k $ goes to infinity. 
In Paragraph \ref{cripointsnon-dege}, we remark that the critical points are all non-degenerate. In Paragraph~\ref{confirepcripoints}, 
we take $ k \in \cK^{c_1}_s $ with $ c_1 \geq c $ large enough to find that there exists indeed a critical point. In Paragraph 
\ref{towardsspr}, we consider conditions under which stationary phase arguments can be employed.


\subsubsection{Possible existence of critical points}\label{repcripoints} The phase $ \Phi_k (t,x;\cdot) $ depends on $ (k,t,x) $ 
and also (implicitly) on $ \eps \in ]0,\eps_0] $ through the condition $ 0 \leq t \leq 2 \cT/ \eps $. The same applies to 
$ (s_k,y_k,\xi_k) (t,x) $. 

\begin{lem}[At most one signal may be emitted from any value $ k \in \cK^c_s $]  \label{statioponi} Let us 

\noindent start by fixing $ \eps_0 $ as in (\ref{defdec}). Up to decreasing this value of $ \eps_0 $, for all $ \eps \in ]0,\eps_0 ] $, 
for all $ k \in \cK^c_s $ and for all $ (t,x) \in [0, 2 \cT/ \eps] \times [-r,r] $, there is inside 
$ \Upsilon $, in the sense of \eqref{stastapoint}, at most one critical point $ (s_k,y_k,\xi_k) (t,x) $ of the phase $ \Phi_k 
(t,x;\cdot) $, with positions $ y_k $ and $ \xi_k $ determined by
\begin{equation} \label{defskykxik} 
y_k = 1 - p(k \, \pi + s_k) - (-1)^k \ \gamma \, \sin s_k \, , \qquad \xi_k = s_k . 
\end{equation}
\end{lem}
 
\begin{proof} The constraint (\ref{stastapoint2}) reads $ s = \xi $. Thus, we can focus on the positions $ (s,y,s) $ satisfying 
(\ref{cdtstapoint}) together with (\ref{stastapoint1}) and (\ref{stastapoint3}), which become
\begin{subequations}\label{sysPhik} 
\begin{eqnarray} 
& &  \displaystyle y = 1 - p(k \, \pi + s)  - (-1)^k \ \gamma \, \sin s  , \label{sysPhik1} \\
& & \displaystyle x = h_k (t;s) := \, 1 - p(k \, \pi + s) - (-1)^k \ \gamma \, \sin s  \label{sysPhik2} \\
& &  \displaystyle \qquad \qquad \qquad \ \ \, - (k \, \pi + s - t) \ p' (k \, \pi + s) . \nonumber
\end{eqnarray}
\end{subequations}
This furnishes already (\ref{defskykxik}). Now, we can consider the determination of $ s $. The time $ s $ is {\it a priori} localized 
as indicated in (\ref{cdtstapoint}). But, using (\ref{sysPhik1}), it is possible to get a more precise information on $ s $. We know 
that $ \vert y \vert \leq r \leq 1/8 $. On the other hand, from (\ref{reso2}), we can infer that
\begin{equation} \label{minmaxp} 
\exists \, C_2>0, \quad \forall \, \xi \ge 1\, , \quad
0\le \xi^q \ \bigl \lbrack 1 - p(\xi)  
\bigr \rbrack \leq C_2 . 
\end{equation}
Using the condition $ c \, \eps^{-1/(q+1)} < k $ inside the definition of $ \cK^c_s $, as well as (\ref{minmaxp}), since $ r < \gamma /2 $, 
we have from (\ref{sysPhik1}) that 
\begin{equation} \label{easyinfer} 
\vert \sin s \vert \leq \frac{r}{\gamma} + \frac{C_2}{\gamma \, (k \, \pi + s)^q} \leq \frac{1}{2} + \cO \bigl( \eps^{q/(q+1)} \bigr) . 
\end{equation}
In view of (\ref{easyinfer}), for $ \eps \in \, ]0, \eps_0] $ with $ \eps_0 $ small enough, we have to deal with the necessary 
condition $ \vert \sin s \vert <\sqrt 3 / 2 $. Since $ \vert s \vert \leq 2 \pi/3 $, this means that $ 1/2 < \cos s $. Given this, we can 
deduce the preliminary information 
\begin{equation} \label{preliinform} 
- \pi/3 < s < \pi/3 . 
\end{equation}
Compute
$$ \quad \partial_s h_k (t;s) := \, - \, 2 \ p' (k \, \pi + s) - (k \, \pi + s - t) \ p'' (k \, \pi + s) - (-1)^k \ \gamma \, \cos s .  $$
Exploiting (\ref{hypp2}) and (\ref{hypp3ded}), as well as $ \vert t - k \, \pi - s \vert \leq 2 \, \cT / \eps $, we find
\[
\vert \partial_s h_k (t;s) \vert = \gamma \ \vert \cos s \vert + \cO \bigl( (k \, \pi + s)^{-q-1} \bigr) + \eps^{-1} \ \cO \bigl( 
(k \, \pi + s)^{-q-2} \bigr) .
\]
Knowing that $ s $ must satisfy (\ref{preliinform}) and that $ k \in \cK^c_s $ is bounded from below as indicated in 
(\ref{defdispk2}), there remains
\[
\vert \partial_s h_k (t;s) \vert = \gamma \ \vert \cos s \vert + \cO \bigl( \eps^{1/ (q+1)} \bigr) = \cO (1) .
\]
Thus, for $ \eps \in \, ]0, \eps_0] $ with $ \eps_0 $ small enough, we can assert that
\begin{equation} \label{minassert} 
\forall \, k \in \cK^c_s \, , \qquad \forall \, s \in \Bigl]-\frac{\pi}{3},\frac{\pi}{3}\Bigr[  , \qquad 0 < \frac{\gamma}{4} \leq \vert
\partial_s h_k (t;s) \vert .  
\end{equation}
It follows that the function $ h_k (t;\cdot) $ is one-to-one from the interval $ ] -\pi/3,\pi/3[$ onto its image, which may or may 
not contain the real number $ x $. At all events, there exists at most one position $ s_k (t,x) \in \, ]-\pi/3,\pi/3[$ such that 
\begin{equation} \label{defskykxikbisbs} 
h_k \bigl(t;s_k(t,x) \bigr) = x . 
\end{equation}
In short, any position $ (s,y,\xi) $ satisfying (\ref{stastapoint}) and (\ref{cdtstapoint}) is subject to $ \vert s \vert <\pi / 3 $. 
Knowing this, as claimed in Lemma \ref{statioponi}, there exists  inside $ \Upsilon $ at most one critical point $ (s_k,y_k,\xi_k) 
(t,x) $ and, if any, the value of $ s_k (t,x) $ is determined through the implicit relation (\ref{defskykxikbisbs}), while $ (y_k,\xi_k) 
(t,x) $ is given by (\ref{defskykxik}). 
\end{proof}

\noindent By construction, the function $ s_k(t,\cdot) $ is, for all $ t \in [0 , 2 \cT / \eps ] $, defined on the interval
\begin{equation} \label{intervaldefsk} 
\cI \! s_k (t) := \Bigl \lbrace h_k (t;s) \, ; \, - \frac{\pi}{3} < s < \frac{\pi}{3}  \Bigr \rbrace .
\end{equation}

\begin{lem} [Properties of $ s_k $] \label{Propertiesofsk} Let $ \eps_0 $ as in Lemma \ref{statioponi}, as well as $ \eps \in ]0,\eps_0] $
and $ k \in \cK^c_s $. The function $ s_k  $ determined by (\ref{defskykxikbisbs}) with $ h_k $ as in (\ref{sysPhik2}) is smooth on its 
domain of definition, which is
$$ \cD\!  s_k := \bigl \lbrace (t,x) \in [0,2 \cT / \eps] \times \RR \, ; \, x \in \cI \! s_k (t) \bigr \rbrace . $$
For all $ \alpha \in \NN^2 $, we can find $ C_\alpha >0$ giving rise to the uniform estimate 
\begin{equation} \label{unigliformskftx} 
\sup_{k \in \cK^c_s} \ \sup_{(t,x) \in \cD \! s_k} \
\vert \part^\alpha_{t,x} s_k(t,x) \vert \leq C_\alpha ,\quad \forall
\eps\in]0,\eps_0]. 
\end{equation}
For $ \alpha = (1,0) $, we get the following more precise estimate
\begin{equation} \label{unigliformskftxpour10} 
 \sup_{(t,x) \in \cD \! s_k} \ \vert \part_t s_k(t,x) \vert = \cO (k^{-q-1} ) = \cO(\eps)  .
\end{equation}
\end{lem}

\begin{proof} The bound (\ref{unigliformskftx}) is, for $ \alpha = (0,0) $, a direct consequence of (\ref{cdtstapoint}). Compute
\begin{equation} \label{parttxskkk} 
 \part_t s_k (t,x) = - \frac{p' \bigl( k \pi + s_k (t,x) \bigr)}{\part_s h \bigl( k,t; s_k (t,x) \bigr)} , \qquad 
\part_x s_k (t,x) = \frac{1}{\part_s h \bigl( k,t; s_k (t,x) \bigr)} . 
\end{equation}
In view of (\ref{hypp3ded}) and (\ref{minassert}), this furnishes (\ref{unigliformskftxpour10}) for $ \alpha = (1,0) $, and also
(\ref{unigliformskftx}) for $ \alpha = (0,1) $. The general case $ \vert \alpha \vert > 1 $ can be obtained by induction. As a 
matter of fact, for $ \vert \alpha \vert > 1 $, the expression
$$ \part_s h \bigl( k,t; s_k (t,x) \bigr)^{\vert \alpha \vert} \, \part^\alpha_{t,x} s_k(t,x)  $$
is a finite linear combination of products involving $ p^i ( k \pi + s_k ) $, $ \part^j_s h ( k,t; s_k ) $ and $ \part^\beta_{t,x} 
s_k $ with $ i \leq \vert \alpha \vert $, $ j \leq \vert \alpha \vert $ and $ \vert \beta \vert < \vert \alpha \vert $.  It suffices 
to remark that all these quantities are uniformly bounded. This comes from Assumption \ref{strengthenp} concerning 
$ p^i $ and $ \part^j_s h $. This is due to the inductive hypothesis regarding $ \part^\beta_{t,x} s_k $.
\end{proof}


\subsubsection{Asymptotic formulas related to the critical points}\label{asyforcripoints} In the rest of this section and 
in Section \ref{sec:nonlinear}, we need to identify the asymptotic behavior of $ s_k $ and $ y_k $ for large values of 
$ k $. To this end, introduce
\begin{equation} \label{defskykxikbisbsalsoform} 
\tau_k (t ; s) := 1 - p(k \, \pi + s) - (k \,\pi + s - t) \, p' (k \, \pi + s) , 
\end{equation}
and remark that (\ref{defskykxikbisbs}) can also be formulated as
\begin{equation} \label{defskykxikbisbsalsoformpoursk} 
s_k (t,x) = (-1)^{k+1} \, \arcsin \, \Bigl( \frac{x}{\gamma} - \frac{\tau_k \bigl(t ; s_k (t,x) \bigr)}{\gamma} \Bigr) .
\end{equation}

\begin{lem} [Asymptotic formulas] \label{Asymptoticetformulas}
Let $ \eps_0 $ as in Lemma \ref{statioponi} as well as $ \eps \in ]0,\eps_0] $. 
Uniformly in $ k \in \cK_s^c $ and $ t \in [0 , 2 \cT / \eps] $, the critical point $ (s_k,y_k,\xi_k) (t,x) $, if any, is such that
\begin{subequations}\label{defskxjhgezjbbkldepour} 
\begin{eqnarray} 
& &  \displaystyle s_k (t,x) = (-1)^{k+1} \ \arcsin \, \Bigl( \frac{x}{\gamma} + \frac{\tau_k^0 (t) }{\gamma} \Bigr) +  
\cO \Bigl(\frac{1}{\eps \, k^{q+2}} \Bigr) , \label{defskxjhgezjbbkldepour1} \\
& & y_k (t,x) = x + (k \pi -t) p' (k \pi) + \cO \Bigl(\frac{1}{\eps \, k^{q+2}} \Bigr) = \cO\Bigl( \frac{1}{\eps \, k^{q+1}} \Bigr) , \label{defskxjhgezjbbkldepour2}
\end{eqnarray}
\end{subequations}
where
\begin{equation} \label{defskxjhgezjbbkltauk} 
\tau_k^0 (t) := \tau_k (t;0) = 1 - p (k \pi) - (k \pi -t) p' (k \pi) = \cO\Bigl( \frac{1}{\eps \, k^{q+1}} \Bigr) .
\end{equation}
\end{lem}

\noindent Recall that 
$$ \forall \, k \in \cK_s^c  , \qquad \cO \Bigl(\frac{1}{\eps \, k^{q+2}} \Bigr) = \cO \Bigl(\frac{1}{k} \Bigr) = \cO 
\bigl( \eps^{\frac{1}{q+1}} \bigr) , $$
so that (\ref{defskxjhgezjbbkldepour}) furnishes indeed an explicit description of $ s_k (t,x) $ and $ y_k (t,x) $ modulo a 
small remainder. Remark that $ \tau_k^0(t) = \cO (1) $ as long as $ k \sim \eps^{-1 / (q+1)} $ and $ t \sim \eps^{-1} $. 
By contrast, when $ \eps \, k^{q+1} $ is large, we can exploit the following information
\begin{equation}\label{devskyk} 
s_k (t,x) = \xi_k (t,x) = s_k^x + \cO \Bigl(\frac{1}{\eps \, k^{q+1}} \Bigr)  , \qquad y_k (t,x)  = x + \cO \Bigl( \frac{1}{\eps \, k^{q+1}} \Bigr) 
\end{equation}
with
\begin{equation} \label{defskx} 
s_k^x := (-1)^{k+1} \ \arcsin  \Bigl( \frac{x}{\gamma} \Bigr) .
\end{equation}

\begin{proof} It suffices to write
\begin{equation} \label{delobhk} 
\begin{array}{rl}
 \tau_k (t;s) = \! \! \! & \tau^0_k (t) + \bigl \lbrack p (k \pi)- p (k \pi +s) \bigr \rbrack \\
 & - (k \pi -t) \bigl \lbrack p' (k \pi +s) - p' (k \pi ) \bigr \rbrack - s p' (k \pi +s) .
 \end{array} 
 \end{equation}
On the one hand, we have $ \vert s \vert \leq 2 \pi / 3 $. On the other hand, we can exploit (\ref{hypp2}) and (\ref{hypp3ded})
to obtain (since $ \eps k $ is bounded when $ k \in \cK^c_s $)
 \begin{subequations}\label{etilffautmon} 
\begin{eqnarray} 
& &  \displaystyle \qquad \tau_k (t;s) = \tau^0_k (t) + \cO \Bigl( \frac{1}{k^{q+1}} \Bigr) + \cO \Bigl( \frac{1}{\eps \, k^{q+2}} \Bigr) =
\tau^0_k (t) + \cO \Bigl( \frac{1}{\eps \, k^{q+2}} \Bigr) , \label{etilffautmon1} \\
& & \displaystyle \qquad \tau_k^0 (t) = \cO \Bigl( \frac{1}{k^{q}} \Bigr) + \cO \Bigl( \frac{1}{\eps \, k^{q+1}} \Bigr) = \cO \Bigl( \frac{1}
{\eps \, k^{q+1}} \Bigr). \label{etilffautmon2}
\end{eqnarray}
\end{subequations}
 
 \noindent We can deduce (\ref{defskxjhgezjbbkldepour1}) from (\ref{defskykxikbisbsalsoformpoursk}) and (\ref{etilffautmon1}). 
 Similarly, (\ref{defskxjhgezjbbkldepour2}) is a consequence of the relation (\ref{stastapoint3}).
\end{proof}

 
\subsubsection{All critical points are non-degenerate}\label{cripointsnon-dege} 
In view of stationary phase arguments, introduce  the Hessian matrix $ S_k (t,x) $ of the 
scalar function $ \Phi_k (t,x;\cdot) $, that is
\renewcommand\arraystretch{1.7}
\[ \begin{array} {rl}
 S_k (t,x) \! \! \! \! & := \operatorname{Hess} \, ( \Phi_k ) \bigl( t,x; s_k (t,x) , y_k (t,x) , \xi_k (t,x) \bigr) \\
 \ & \, = \text{$ \renewcommand\arraystretch{1}
 \left( \begin{array}{ccc}
\part^2_{ss} \Phi_k & \part^2_{sy} \Phi_k & \part^2_{s \xi} \Phi_k \\ 
\part^2_{ys} \Phi_k & \part^2_{yy} \Phi_k & \part^2_{y \xi} \Phi_k \\ 
\part^2_{\xi s} \Phi_k & \part^2_{\xi y} \Phi_k & \part^2_{\xi \xi} \Phi_k 
\end{array} \right)\bigl( t,x;s_k (t,x) ,y_k (t,x) ,\xi_k (t,x) \bigr) . 
$}
\end{array}
\]
\renewcommand\arraystretch{1} 

\noindent It is notable that a control on the invertibility of $ S_k $ turns out to be available for all $ k \in \cK^c_s $. 
What is even more remarkable is that such a control can be obtained with uniform bounds with respect to $ k \in \cK^c_s $
and $ t $ as in (\ref{adjusttT}).

\begin{lem} [The critical points are uniformly non-degenerate] \label{statioponidege} Up to decreasing again 
the value of  $ \eps_0 \in ]0,1] $, for all $ \eps \in ]0,\eps_0 ] $, for all $ k \in \cK^c_s $ as well as for all $ (t,x) \in [0, 2 \cT/ \eps] \times [-r,r] $, the 
possible critical point $ (s_k,y_k,\xi_k) (t,x) $ of $ \Phi_k (t,x;\cdot) $ is non-degenerate, such that $ \vert s_k (t,x) 
\vert \leq \pi / 3 $, and there exists $ C \in \RR_+^* $ such that
\begin{equation} \label{minunifdet} 
\forall \, (k,t,x) \in \cK^c_s \times [0, 2 \cT / \eps ] \times [-r,r] \, , \qquad 0 < C \leq \vert \mathrm{det} \, S_k (t,x) \vert .  
\end{equation}
In addition, the signature of the matrix $ S_k (t,x) $, that is the number of positive eigenvalues minus the number of 
negative eigenvalues, is given by
\begin{equation} \label{signSk} 
\operatorname{sign} \, \bigl( S_k (t,x) \bigr) = (-1)^k  . 
\end{equation}
\end{lem}

\begin{proof} We have already proven the condition $ \vert s_k \vert \leq \pi / 3 $, see (\ref{preliinform}). Now, 
taking into account (\ref{defphikini}), we find
\begin{equation} \label{devSkenvue}  
\begin{array}{rl} 
 S_k \! \! \! & = \begin{pmatrix}
(-1)^k \, \gamma \, \cos s_k & 1 & p' (k \, \pi + \xi_k) \\
1 & 0 & - 1 \\
p' (k \, \pi + \xi_k) & -1 & (k \, \pi + s_k - t) \, p'' (k \, \pi +
\xi_k)
\end{pmatrix} \\
\ & \displaystyle =  S^{\delta_k}_{(-1)^k} + \cO \Bigl( \frac{1}{\eps \, k ^{q+2}} \Bigr) 
\end{array} 
\end{equation}
with the conventions
\begin{equation} \label{defdelta} 
\quad S^{\delta_k}_\pm (t,x) := 
\begin{pmatrix}
\pm \, \delta_k & 1 & 0 \\
1 & 0 & - 1 \\
0 & -1 & 0
\end{pmatrix} , \qquad \frac{\gamma}{2} < \delta_k \equiv \delta_k (t,x) :=
\gamma \, \cos \, s_k < \frac{1}{4} .  
\end{equation}
For $ k \in \cK^c_s $, we find $ \eps^{-1} k^{-q-2} \lesssim \eps^{1/(q+1)} \ll 1 $. This furnishes 
\begin{equation} \label{calculdettt} 
\operatorname{det} S_k = \operatorname{det} \bigl( S^{\delta_k}_{(-1)^k} \bigr) + o(1) = - (-1)^k \delta_k + o(1) .
\end{equation}
In particular, this implies (\ref{minunifdet}). On the other hand, we have the algebraic property
\begin{equation} \label{algebraicproperty} 
  \operatorname{Tr} \, (S^{\delta_k}_\pm) = \pm \delta_k =  -\operatorname{det} \, (S^{\delta_k}_\pm) \not =0 .
\end{equation}
The trace is the sum of the eigenvalues, and the determinant is their product. In view of (\ref{algebraicproperty}), 
the eigenvalues cannot all have the same sign. Since $S^{\delta_k}_\pm$ is a
  $3\times3$ matrix, we have only two possibilities:
  \begin{itemize}
  \item The integer $ k $ is even. From (\ref{calculdettt}), the determinant must be negative. Two eigenvalues are positive 
  and one is negative. The signature is $ 1 $.
      \item The integer $ k $ is odd. From (\ref{calculdettt}), the determinant must be positive. One eigenvalue is positive and 
      two are negative. The signature is $ -1 $.
  \end{itemize}
Both results are consistent with (\ref{signSk}).
\end{proof}


\subsubsection{The existence for sure of critical points when $ k $ is large enough}\label{confirepcripoints} Let $ c_1\ge c $. Define 
$ \cK_s^{c_1} $ as in  (\ref{defdispk2}). The inequality $ c \le c_1$ implies that $ \cK_s^{c_1} \subset \cK_s^{c} $. For $ c_1 $ large 
enough, the content of Lemma \ref{statioponi} can be refined.

\begin{lem} [Signals from $ \cK^{c_1}_s $ with $ {c_1} $ large enough are always  detected] \label{statioponialways} Let $ \eps_0 $ 
as in Lemma \ref{statioponidege} as well as $ \eps \in ]0,\eps_0] $. There exists $ c_1 \geq c $ such that, for all $ k \in \cK_s^{c_1} $  
and for all $ (t,x) \in [0, 2 \cT/ \eps] \times [-r,r] $, there is exactly one position 
$ (s_k,y_k,\xi_k) (t,x) $ satisfying the two conditions \eqref{stastapoint} and \eqref{cdtstapoint}. 
\end{lem}

\begin{proof} First observe that, given $ \eps_0 $ as in Lemma \ref{statioponidege}, all preceding estimates remain valid if, like in 
the case of $ c_1 $, we fix the value of $ c $ above the one of (\ref{defdec}). Consider (\ref{sysPhik2}), and remark that
\begin{equation} \label{delobhk2} 
 h_k (t;s) = (-1)^{k+1} \ \gamma \, \sin s + \cO(k^{-q}) + \eps^{-1} \ \cO(k^{-q-1}) . 
 \end{equation}
Since $ c_1 \, \eps^{-1/(q+1)} \leq k $, it follows that
\[
 h_k (t;\pm \pi/3) = \pm (-1)^{k+1} \ \sqrt 3 \ \gamma / 2 
 +\cO \bigl( \eps^{\frac{q}{q+1}}\bigr)+\cO\bigl(c_1^{-q-1} \bigr) . 
\]
In particular, for $ c_1 $ large enough and $\eps$ sufficiently small, we find
\[ \pm (-1)^{k+1} \ h_k (t;\pm \pi/3) > \gamma / 2 . \]
Taking into account (\ref{intervaldefsk}), we have
$$ x \in [-r,r] \subset [- \gamma/2,\gamma/2] \subset \cI \! s_k (t) . $$ 
Since $ h_k (t;\cdot) $ is continuous, we can apply the intermediate value theorem. It says that we can find $ s_k (t,x) \in ]- \pi/3 , \pi/3[ $ 
satisfying (\ref{defskykxikbisbs}). Lemma \ref{statioponi} guarantees that such $ s_k (t,x) $ is unique in the interval $ ]- 2 \pi/3 , 2 \pi/3 [ $. 
\end{proof}


\subsubsection{Towards stationary phase results}\label{towardsspr}

\noindent Below, we recall a standard statement, which can be found e.g. in \cite[Proposition~5.2]{DimSjo99} or
\cite[Theorem~3.16]{Zworski}. It will be used in this section and in the nonlinear analysis of Section~\ref{sec:nonlineareffectg=1}.

\begin{theo}[From \cite{DimSjo99,Zworski}]\label{theo:phazstat} 
Select $\phi\in \cC^\infty(\RR^n;\RR)$ and $a\in \cC_c^\infty(\RR^n)$ satisfying  $\operatorname{supp} a =\Upsilon$. Let $ h > 0 $. Denote
\[ I_h = I_h(a,\phi) := \int_{\RR^n} e^{-i\phi(x)/h}a(x)dx.\]
  Suppose that
  \[ x_0\in \Upsilon , \quad \nabla_{\! x} \phi(x_0)=0 ,\quad \operatorname{det} \part^2 \phi (x_0)\not =0.
  \]
  Assume further that $ \nabla_{\! x}\phi (x)\not =0$ on  $\Upsilon\setminus\{x_0\}$. Then, for all $ N \in \NN^* $, there exist 
  differential operators $M_{2j}(x;D)$ of order less than or equal to $2j$ such that
  \[ \begin{array}{l} 
\displaystyle \Bigl \vert I_h - h^{n/2} \sum_{j=0}^{N-1} h^j  \bigl \lbrack M_{2j}(x;D) a (\cdot) \bigr \rbrack_{x =x_0} e^{-i\phi(x_0)/h} \Bigr \vert \\
\displaystyle \qquad \qquad\qquad\qquad \qquad \le C_N h^{n/2+N} \! \! \sum_{|\alpha| \le 2N+n+1} \! \! \|\part^\alpha a\|_{L^\infty} .
    \end{array}
    \]
    The constant $C_N$ depends on the compact $\Upsilon$ and also on the $L^\infty$ norm of $ \phi $ and its derivatives 
    on $\Upsilon$. In particular, denoting by $ \operatorname{sign} S $ the signature of $ S $, we find
    \begin{equation} \label{stapourA0} 
    M_0 =\frac{ (2\pi)^{n/2}}{| \operatorname{det} \part^2 \phi (x_0)|^{1/2}} \ e^{-i\frac{\pi}{4}\operatorname{sign} \part^2 \phi
    (x_0)}. 
    \end{equation}
\end{theo}

\noindent Theorem \ref{theo:phazstat} is aimed to be applied to the oscillatory integral defining $ w_k $. When doing this, it 
is important to get uniform estimates with respect to all parameters $ k $, $ t $ and $ x $. Lemma \ref{statioponidege} is a first 
indication that this works well. Another aspect is related to the uniform control of the constants $ C_N $. As mentioned above, this can 
be achieved by looking at the derivatives of $ \Phi_k  $ on $ \Upsilon $.

\begin{lem}[Estimates on the derivatives of $ \Phi_k $]  \label{controlofphhsggde} Let $ \eps_0 $ 
as in Lemma \ref{statioponidege} as well as $ \eps \in ]0,\eps_0] $. With the compact set $\Upsilon$ given by 
\eqref{compsetwk}, for all $ N \ge 2 $, there exists a constant $ C_N $ such that uniformly in $ t $ as in \eqref{adjusttT} and 
in $ x $ with $ \vert x \vert \leq r $, we have
\begin{equation} \label{boundunigradphi} 
\sup_{k \in \cK^c_s} \ \sup_{(s,y,\xi) \in \Upsilon} \ \sum_{1 \leq \vert \alpha \vert \leq N}  \vert \part^\alpha_{s,y,\xi} \Phi_k (t,x;s,y,\xi) 
\vert \leq C_N .
\end{equation}
\end{lem}

\begin{proof} Looking at (\ref{compsetwk}) and (\ref{stastapoint2}), we have indeed (\ref{boundunigradphi}) for the terms which are involving 
multi-indices $ \alpha = (\alpha_1 , \alpha_2, \alpha_3 ) $ with $ 1 \leq \alpha_2 $. For $ \alpha_2 = 0 $ and $ 1 \leq \alpha_1 $,
consider the line (\ref{stastapoint1}). When $ \alpha_1 = 1 $, just apply (\ref{resobis}). For $ 1 < \alpha_1 $, combine the property
(\ref{hypp3ded}) together with (\ref{controlderip}). Now, assume that $ \alpha_1 = \alpha_2 = 0 $ and $ 1 \leq \alpha_3 \leq N $. 
Then, exploiting (\ref{controlderip}), we find
$$ \vert k \, \pi + s -t \vert \ \vert p^{(\alpha_3)} (k \, \pi + \xi) \vert \, \lesssim \, \eps^{-1} \ p' (k \, \pi + \xi) = \eps^{-1} \, \cO (
k^{-q-1} ) . $$ 
Since $ k \in \cK^c_s $, the right hand side is bounded. Summing  the preceding upper bounds over multi-indices $\alpha$
yields  (\ref{boundunigradphi}). 
\end{proof}


\subsection{The accumulation of wave packets}\label{longregime} 

\noindent In Paragraph~\ref{strucwavepack}, the solution $ u $ to the equation (\ref{eq:mainajoutlinpourm}) is represented 
modulo some $ o(\eps) $ as a sum of wave packets $ u_k $ with $ k \in \cK^{c_1}_s $. Then, the purpose is to distinguish 
between situations where constructive  \href{https://en.wikipedia.org/wiki/Interference_(wave_propagation)}{interferences} occur 
(Paragraph \ref{consinterference}) from those where, on the contrary, destructive
\href{https://en.wikipedia.org/wiki/Interference_(wave_propagation)}{interferences} take place (Paragraph~\ref{desinterference}). 


\subsubsection{The solution as a sum of wave packets}\label{strucwavepack} 
\noindent  Combining (\ref{remindudevenk}), (\ref{defvrvkensum}) and (\ref{combi1}), the solution $ u$ to (\ref{eq:mainajoutlinpourm}) 
can be put in the form
\begin{equation}\label{sumofwavepackets1}  
u(t,x) = \sum_{k \in \NN} \, u_k (t,x) + \cO ( \eps^\infty) ,
\end{equation}
\vskip -2mm
\noindent with
\begin{equation}\label{sumofwavepackets2}   
u_k (t,x) := \frac{\sqrt \eps}{2 \, \pi} \ e^{i \, (-\gamma+ k \, \pi - k \, \pi \, x)/\eps} \ w_k(t,x) . 
\end{equation}

\begin{lem}\label{decomposol} Fix $ c_1 $ as in Lemma \ref{statioponialways}. Under Assumption \ref{strengthenp} with 
$ D \geq 3 $, we can expand $ u  $ according to
\begin{equation}\label{sumofwavepackets3}  
u(t,x) = \sum_{k \in \cK_s^{c_1}} \, u_k (t,x) + \, \cO \bigl( \eps^{2-\frac{1}{q+1}} \bigr) .
 \end{equation}
 The wave packets $ u_k $ are of size $ \cO (\eps^2) $. Assuming that $ D \geq 4 $, they have the form
 \begin{equation} \label{desofuk} 
u_k (t,x) = \eps^2 \ b_k (\eps,t,x) \ e^{i \, \Psi_k (t,x,s_k) / \eps} + \cO \bigl( \eps^{2 + \frac{1}{q+1}} \bigr) = \cO 
\bigl( \eps^2 \bigr) ,
 \end{equation}
with  
 phases $ \Psi_k  $ and profiles $ b_k $ given by
 \begin{subequations}\label{defpsikbk} 
\begin{eqnarray} 
& & \qquad \displaystyle \Psi_k (t,x,s_k) := - \, \gamma + t + (-1)^k \, \gamma \, \cos s_k \label{defpsikbk1} \\
& & \qquad \displaystyle \qquad \qquad \qquad \ \ + \bigl \lbrack 1 - p (k \, \pi + s_k) \bigr \rbrack \, (k \, \pi + s_k  -t) - 
(k \, \pi + s_k) \, x , \nonumber \\
& & \qquad \displaystyle b_k (\eps,t,x) := (2 \, \pi)^{1/2} \ \frac{e^{-i \, (-1)^k \, \frac{ \pi}{4}} }{\vert \operatorname{det}
S_k \vert^{1/2}} \ a (\eps \, k \, \pi + \eps \, s_k , k \, \pi + s_k,y_k ) . \label{defpsikbk2} 
\end{eqnarray}
\end{subequations}
\end{lem}

\begin{proof} The sum inside (\ref{sumofwavepackets1}) can be split into 
\[  u(t,x) = \underbrace{ \vphantom{ \sum_{\frac{c}{\eps^{1/(q+1)}}<k < \frac{c_1}{\eps^{1/(q+1)}}} }  
\sum_{k \in \cK_d^c} u_k (t,x)}_{\text{\large \ding{202}}} + \underbrace{\sum_{\frac{c}{\eps^{1/(q+1)}}
<k \le \frac{c_1}{\eps^{1/(q+1)}}} \, u_k (t,x)}_{\text{\large \ding{203}}} + \underbrace{\vphantom{ 
\sum_{\frac{c}{\eps^{1/(q+1)}}<k < \frac{c_1}{\eps^{1/(q+1)}}} } \sum_{k \in \cK_s^{c_1}} \, u_k (t,x)}_{
  \text{\large \ding{204}}}  \, + \, \cO ( \eps^\infty) . \]
We may recognize here the dispersive part $ \text{\large \ding{202}} $, the transitional part $ \text{\large \ding{203}} $
which is possibly absent (when $c=c_1$), and the cumulative part $ \text{\large \ding{204}} $. Not all $ k \in \NN $ 
have a leading order contribution, and not all with the same size. We will explain separately how to estimate each part. 

\medskip

\noindent \text{\large \ding{202}} For $ k \in \cK_d^c $, it suffices to apply Lemma \ref{localizations} to get
$$ \Bigl \vert \sum_{k \in \cK_d^c} u_k (t,x) \Bigr \vert \leq \frac{\sqrt \eps}{2 \, \pi} \, \sum_{k \in \cK_d^c} \vert w_k (t,x) \vert 
\leq C \ \sqrt \eps \, \sum_{k \in \cK_d^c} \eps^{D-1} = \cO \bigl( \eps^{D-\frac{1}{2}-\frac{1}{q+1}} \bigr) . $$

\medskip

\noindent \text{\large \ding{203}} For $ k \in \cK_s^c $ with $ c $ as
in (\ref{defdec}), the idea is to exploit Lemma~\ref{statioponidege}
to implement Theorem~\ref{theo:phazstat} at the level of  
the oscillatory integral (\ref{defidewk}). To do this, all assumptions
must be checked: 

\smallskip

\noindent - The first, and most important, is (\ref{compsetwk}) which guarantees that the integration 
is on a compact set (independent of $ k $, $ t $ or $ \eps $).

\smallskip

\noindent - The second is (\ref{boundunigradphi}) which enables, away from $ (s_k,y_k,\xi_k)$, to perform $ D-1 $ 
integrations by parts, and still to obtain some $ \cO (\eps^{D-1} ) $ error term. When doing this, a major difficulty is 
that the phase $ \Phi_k  $ still depends on $ (t , x) $. And therefore, according to (\ref{adjusttT}), since $ t $
may be of size $ \eps^{-1} $, it does depend on $ \eps $. The aim of the control (\ref{boundunigradphi}) is precisely 
to overcome this difficulty.

\smallskip

\noindent - The main contribution is provided by a small neighborhood of $ (s_k,y_k,\xi_k) $. The implementation 
of Morse Lemma (usually used when proving Theorem~\ref{theo:phazstat}) is made possible by Lemma~\ref{statioponidege}. 
It requires three derivatives of $ \Phi_k$ to obtain a $ \cC^1 -$diffeomorphism. This implies that $ D $ must be at
least equal to $ 3 $. 

\smallskip

\noindent - The phase $ \Phi_k (t,x;\cdot) $ depends on the three variables $ (s,y,\xi) \in \RR^3 $, and therefore
the leading-order term is of amplitude $ \eps^{3/2} $ modulo some small $ o (\eps^{3/2}) $. Then, any 
extra derivative on $ \Phi_k  $ allows to gain a power of $ \eps $ in the asymptotic expansion. We must 
take $ D \geq 4 $ to be sure of some $ \cO (\eps^{5/2}) $ precision.

\smallskip

\noindent Now, the expression $ w_k$ of (\ref{defidewk}) can be
expanded in powers of $ \eps $ through  
Theorem~\ref{theo:phazstat}. To this end, taking into account the definitions (\ref{mollichi}), (\ref{defchivar}) and (\ref{defskykxik}) 
together with Lemma \ref{statioponidege} which implies $ \vert s_k \vert \leq \pi/3 $, first remark that
\begin{equation} \label{preliphasta} 
\chi_{1/4} (s_k-\xi_k) \ \chi_{2 \, \pi/3} (s_k) = \chi_{1/4} (0) \ \chi_{2 \, \pi/3} (s_k) = 1 .
\end{equation}
By Assumption~\ref{choiopro}, $\zeta(\xi) =1+\cO(1/|\xi|)$ as
$|\xi|\to \infty$, so
\[ \forall \, k \in \cK^c_s \, , \qquad \zeta (k \pi + s_k) = 1 + \cO \bigl( 1/k \bigr) \, , \qquad 1 - \chi (k \pi + s_k) = 1 . \]
Therefore, for all $ k \in \cK^c_s $, we find
\[
  A(\eps \, k \, \pi + \eps \, s_k , k \, \pi + s_k,y_k , k \pi + s_k) = a (\eps \, k \, \pi + \eps \, s_k , k \, \pi + s_k,y_k ) +  
  \cO \bigl( \eps^{1/ (q+1)} \bigr) .
  \]
On the other hand, the signature is given by (\ref{signSk}). Combining
all the above information, Theorem~\ref{theo:phazstat} yields (with $N=1$)
\begin{equation} \label{Phitracedev} 
\begin{aligned}
\quad \, w_k (t,x) & =   \left( 2 \, \pi \, \eps \right)^{3/2}
\ \frac{e^{-i \, (-1)^k \, \frac{ \pi}{4}} }{\vert \mathrm{det} \ S_k \vert^{1/2}} \ e^{- i \, \Phi_k (t,x;s_k,y_k,\xi_k) /\eps} \\
& \quad \ \times a (\eps \, k \, \pi + \eps \, s_k , k \, \pi + s_k,y_k ) \, + \, \cO \bigl( \eps^{\frac{3}{2}+ \frac{1}{q+1}} \bigr) ,
\end{aligned}
\end{equation}
where the remainder term, larger than the one provided by
Theorem~\ref{theo:phazstat}, stems from the above approximation of $A$. 
In (\ref{Phitracedev}), the $ \cO $ is uniform with respect to $ k \in \cK_s^c $ or $ t $ as in (\ref{adjusttT}). Whether 
there exists a stationary point or not, we have $w_k=\cO(\eps^{3/2})$, hence $u_k=\cO(\eps^2)$. This rough estimation
gives rise to
$$ \Bigl \vert \sum_{c \, \eps^{-1/(q+1)} < k < c_1 \, \eps^{-1/(q+1)}} u_k (t,x) \Bigr \vert = \cO \bigl( \eps^{2 - 
\frac{1}{q+1}} \bigr) . $$

\smallskip

\noindent \text{\large \ding{204}} For $ k \in \cK_s^{c_1} $ with $ c_1 $ as in Lemma \ref{statioponialways}, the content 
of $ u_k $ can be specified. Using the definition of $ u_k $ at the level of (\ref{sumofwavepackets2}) together 
with (\ref{Phitracedev}), we find \eqref{desofuk}, with
\eqref{defpsikbk}. 
 Integers $ k \in \cK_s^{c_1} $ are the most numerous; they may provide the main contribution; and therefore they are set 
aside at the level of (\ref{sumofwavepackets3}). Since $ D - 1/2 \geq 2 $ when $ D = 3 $, we can retain (\ref{sumofwavepackets3}).
\end{proof} 

\noindent In a similar way to the elementary model of Section~\ref{subsec:toymodel}, the superposition of the wave packets $ u_k $ can
induce a time growth of the solution $ u $ to (\ref{eq:mainajoutlinpourm}). The source term of (\ref{eq:main}) is of size $ \eps^{3/2}$; 
in view of (\ref{desofuk}),  it can trigger signals $ u_k  $ of amplitude $ \eps^2 $; at first sight, it can produce during long times 
$ t \sim \eps^{-1} \, T $ a contribution which may be of size $ \eps^2 \, t \sim \eps \, T $. 

\smallskip

\noindent That being said, this cumulative effect is only likely but not certain to occur, due to possible cancellations. The aim of the 
next Paragraphs~\ref{consinterference} and \ref{desinterference} is to check what is actually happening. 


\subsubsection{Constructive interferences}\label{consinterference} In this paragraph, we show that the amplification 
phenomenon of the preamble does apply at special positions. 

\begin{prop}[Asymptotic behavior of the solution on some moving lattice]  \label{supnormamplifi} $\, $

\noindent Under Assumptions~\ref{resoa},~\ref{normareso} and \ref{strengthenp} (with $ D \geq 4 $) on the symbol 
$ p $, as well as Assumption~\ref{persourcea} on the profile $ a \equiv a^1 $, for all $ T \in [\cT , 2 \, \cT] $ and all $ j \in \ZZ $, 
the solution $ u $ to (\ref{eq:mainajoutlinpourm}) is such that
\begin{itemize}
\item If $ q = 2 $, 
\begin{equation} \label{increasenc} 
  \begin{aligned}
\ u \( \frac{T}{\eps} , 2  j \eps \) =  o (\eps) + 
\frac{\eps}{\sqrt{2\pi \gamma}} &
e^{i \frac{T}{\eps^2} }\Bigl(  e^{ -i \frac{\pi}{4} } \, \int_0^{+\infty} e^{-i \, \frac{ \ell}{6} \, 
(\frac{1}{s} - \frac{T}{s^2})} \aaa (s,0,0) \ ds \\
& + e^{ - i \, (\frac{2\gamma}{\eps} - \frac{\pi}{4} )}  \, \int_0^{+\infty} e^{-i \, \frac{ \ell}{6} \, 
(\frac{1}{s} - \frac{T}{s^2})} \aaa (s,\pi,0) \ ds \Bigr) .
\end{aligned}
\end{equation}
\item If $q>2$,
  \begin{equation} \label{increasencbis} 
  \begin{aligned}
    \ u \( \frac{T}{\eps} , 2 \, j \, \eps \) = \, \cO \( \eps^{1+ \frac{q-2}{q+1}} \) + 
\frac{\eps}{\sqrt{2\pi \gamma}} &
e^{i \frac{T}{\eps^2} }\Bigl(  e^{ -i \frac{\pi}{4} } \, \int_0^{+\infty}  \aaa (s,0,0) \ ds \\
& + e^{ - i \, (\frac{2\gamma}{\eps} - \frac{\pi}{4} )}  \, \int_0^{+\infty} \aaa (s,\pi,0) \ ds \Bigr) .
\end{aligned}
 \end{equation}
\end{itemize}
  \end{prop}

\noindent The leading term in the right hand side of (\ref{increasenc}) and (\ref{increasencbis}) does not depend 
on $ j $. On the other hand, for $ j = 0 $, the formula (\ref{increasenc}) provides the asymptotic behavior of $ u(T/\eps,\cdot) $ 
at a fixed position, which is the origin $ x = 0 $. 

\smallskip

\noindent Now, compare (\ref{third-si-devtt}) multiplied by $ \eps^{3/2} $ with (\ref{increasencbis}). When $ q > 2 $ and in the 
(extended) situation where $ a (\cdot,0,0) \equiv \mathbf 1_{[0,\cT]} $, the two formulas coincide. However, in the critical case 
$ q = 2 $, there are some differences. The wave packets $ u_k $ have
larger group velocities; their wave front sets can  
mix; they can interact meaningfully. As a matter of fact, the identity (\ref{increasenc}) is more complicated, and the amplification 
effect can be altered by the oscillatory factor in front of
$ \aaa$.
\smallbreak

\noindent We also note that if $\aaa$ is not only $2\pi$-periodic in its second
argument, but $\pi$-periodic, then the above formula boils down to the
one stated in Theorem~\ref{theo:resumeNL}.

\begin{proof} The starting point is (\ref{sumofwavepackets3}) together with (\ref{desofuk}). Select some $ \alpha \in \RR $.
Since the $ \cO(\eps^2) $ inside (\ref{desofuk}) is uniform with respect to $ k $ and $ t $, a rough estimate yields
\begin{equation} \label{devdepbkbobo} 
\eps^{-1} \, u (t,\eps \, \alpha) = \sum_{k \in \cK_s^{c_1}} \eps \ b_k (\eps,t,\eps \, \alpha) \ e^{i \, \Psi_k(t,\eps \, \alpha,s_k) 
/ \eps} + \, \cO \left( \eps^{\frac{q}{q+1}} \right) .
\end{equation}
Recall the definitions inside (\ref{defpsikbk}). The ingredients $ \Psi_k $ and $ b_k $ of (\ref{devdepbkbobo}) 
are not free from a dependence on $ \eps $ which may arise when specifying the choice of $ k $, when replacing $ t $ by 
$ T / \eps $, or when substituting $ x $ with $ \eps \, \alpha $. A first step in the analysis is to simplify modulo small error terms 
the content of $ \Psi_k $ and $ b_k $. Let us start by reducing $ \Psi_k $. With $ s^x_k $ as in (\ref{defskx}), coming back to 
(\ref{defpsikbk1}), compute 
\[
  \begin{array}{l}
\Psi_k (t,x,s_k) - \, \Psi_k (t,x,s_k^x) = (-1)^k \, \gamma \, (\cos s_k - \cos s_k^x) - (s_k - s_k^x) \ x \\
\qquad \qquad + \bigl \lbrack 1 - p (k \, \pi + s^x_k) \bigr \rbrack \, (s_k  - s_k^x) + (k \, \pi + s_k - t) \, 
\bigl \lbrack p (k \, \pi + s^x_k) - p (k \, \pi + s_k) \bigr \rbrack .
   \end{array} 
   \]
Combine the mean value theorem with (\ref{reso2}) and (\ref{hypp3ded}). For large values of $ k $, this gives rise to
\[ \Psi_k (t,x,s_k) - \, \Psi_k (t,x,s_k^x) = \left \lbrack 1 + \cO \left(\frac{1}{k^q} \right)  + \frac{1}{\eps} \, \cO\left(\frac{1}{k^{q+1}} \right) 
\right \rbrack \ \cO \left( \vert s_k - s_k^x \vert \right) .
\]
Then, knowing that $ k \in \cK_s^{c_1} $, we can apply (\ref{devskyk}) to just retain
\begin{equation} \label{devdepsik} 
\Psi_k (t,x,s_k) = \Psi_k (t,x,s_k^x) + \cO\left(\frac{1}{\eps \, k^{q+1}} \right) .
\end{equation}
For $ k \in \cK_s^{c_1} $ and $ \eps $ small enough to be sure that $ {\rm t}_s \leq k \, \pi + s_k $ with $ {\rm t}_s $ as in (\ref{peratlar}), 
using (\ref{peratlar}) and (\ref{devskyk}), we find
\begin{equation} \label{devdepbkinterm} 
\qquad a (\eps \, k \, \pi + \eps \, s_k, k \, \pi + s_k,y_k ) = \aaa (\eps \, k \, \pi , k \, \pi + s^x_k,x ) + \cO (\eps) + \cO \left(\frac{1}{\eps \, 
 k^{q+1}} \right)  .
\end{equation}
Examine (\ref{defpsikbk2}). To interpret the quantity $ \vert
\operatorname{det} S_k \vert $, exploit Lemma~\ref{statioponidege} (and its
proof). There remains 
\begin{equation} \label{devdepbk} 
\begin{array} {rl}
\displaystyle b_k (\eps,t,x) = \! \! \! & \displaystyle (2 \, \pi)^{1/2} \ e^{-i \, (-1)^k \, \frac{ \pi}{4}} \ (\gamma \, \cos s_k^x)^{-1/2} \\
\ & \displaystyle \times \, \aaa (\eps \, k \, \pi , k \, \pi + s^x_k,x ) + \cO(\eps) + \cO \left(\frac{1}{\eps \, k^{q+1}} \right) . \qquad \quad
\end{array}
\end{equation}
Replace $ x $ by $ \eps \, \alpha $. Coming back to (\ref{defskx}), this yields
\[
\qquad s_k^{\eps \, \alpha} = (-1)^{k+1} \ \arcsin \ \left( \frac{\eps \, \alpha}{\gamma} \right) = (-1)^{k+1} \ \frac{\eps \, \alpha}{\gamma} 
+ \cO (\eps^2) . 
\]
It follows that
\begin{subequations}\label{devdepsiken0} 
\begin{eqnarray} 
& & \displaystyle \Psi_k (t,\eps \, \alpha,s_k) = \Psi_k^0 (t) - k \, \pi \, \alpha \, \eps + \cO\left( \eps^2 \right) + \cO \left( \eps^{-1} \, 
k^{-q-1} \right) , 
\label{devdepsiken01} \\
& & \displaystyle b_k (\eps,t,\eps \, \alpha) = b_k^0 + \cO\left(\eps \right) + \cO\left(\eps^{-1} \, k^{-q-1} \right) , \label{devdepsiken02} 
\end{eqnarray}
\end{subequations}
with:
\begin{subequations}\label{defpsikbkeno0pb} 
\begin{eqnarray} 
\quad & & \displaystyle \ 
\begin{array}{ll}
\Psi_k^0 (t) := \! \! \! \! & - \, \gamma + t + (-1)^k \, \gamma + \bigl \lbrack 1 - p (k \, \pi) \bigr \rbrack \, (k \, \pi -t) , 
\end{array}
\label{defpsikbkeno0pb1} \\
& & \displaystyle \quad b_k^0 := (2 \, \pi)^{1/2} \ e^{-i \, (-1)^k \, \frac{ \pi}{4}} \ \gamma^{-1/2} \ \aaa (\eps \, k \, \pi , k \, \pi ,0 ) . 
\label{defpsikbkeno0pb2}  
\end{eqnarray}
\end{subequations}

\noindent For $ k \in \cK_s^{c_1} $, a precision like $ \cO (1 / \eps \, k^{q+1}) $ is not enough, since  for $ k \sim \eps^{-1/(q+1) } $, 
it is not necessarily small. By contrast, for larger $k$'s, assuming that $ \eps^{-1} \, \eta \, \leq \, k $ for some $ \eta \in \, (0,1] $, 
since $ q > 1 $, we have $ \cO (1 / \eps \, k^{q+1}) = \cO ( \eps^q / \eta^{q+1} ) $, and therefore
\begin{equation} \label{aepseta}
e^{i \, \Psi_k(t, \eps \, \alpha,s_k) / \eps} = e^{i \, \Psi_k^0 (t) / \eps} \ e^{- i \, k \, \pi \, \alpha} + \cO (\eps) + \cO 
\left(\frac{\eps^{q-1}}{\eta^{q+1}} \right) . 
\end{equation}
Note the loss of precision by the power $ \eps^{-1} $ when dividing $
\Psi_k $ by $ \eps $, as well as a bad dependence  
upon $ \eta $ near $ \eta = 0 $ inside the last term above. For the moment, we fix some $ \eta \in \, ]0,1] $. 
Back to (\ref{devdepbkbobo}), for $ k \in \cK_s^{c_1} $ with $ k \leq \eps^{-1} \, \eta $, just apply (\ref{desofuk}) in the rough 
form $ u_k = \cO(\eps^2) $ to get
\begin{equation} \label{oubliincorp}
\sum_{\cK_s^{c_1} \ni k \, \leq \eps^{-1} \, \eta} \eps \ b_k (\eps,t,\eps \, \alpha) \ e^{i \, \Psi_k(t,\eps \, \alpha,s_k) 
/ \eps} = \, \cO \left( \eta \right) .
\end{equation}
For $ k \in \cK_s^{c_1} $ with $ \eps^{-1} \, \eta \leq k $, we can separate even numbers $ k $ from odd numbers $ k $. In 
other words, we can split $ \cK_s^c $ into $ \cK_s^c (e) \cup \cK_s^c (o) $ with
\[
\cK_s^c (e) := \lbrace k \in \cK_s^c \, ; \, k \ \text{is even} \rbrace \, , \qquad \cK_s^c (o) := \lbrace k \in \cK_s^c \, ; \, k \ \text{is odd}
\rbrace \, . 
\]
By this way, using (\ref{devdepbkbobo}), (\ref{devdepsiken02}), (\ref{aepseta}) and (\ref{oubliincorp}), we get
\begin{equation} \label{ffirstsum} 
 \begin{aligned}
\eps^{-1} \, u (t,\eps \, \alpha) &=  \sum_{{\rm par} \in \{ e,o \} }
\ \sum_{\eps^{-1} \eta \, \leq \, k \in \cK_s^{c_1}({\rm par})} \eps \ b_k^0 \ 
e^{i \, \Psi_k^0 (t) / \eps} \ e^{- i \, k \, \pi \, \alpha}  \\
 & \quad + \cO(\eta) + \cO \left(\frac{\eps^{q-1}}{\eta^{q+1}} \right) + \cO\left(\eps^{\frac{q}{q+1}}\right)  . 
\end{aligned} 
\end{equation}
Now, we consider the dependence on $ k $ when computing the two sums inside (\ref{ffirstsum}):

\smallskip

\noindent i) For $ \alpha = 2 \, j $ as required in (\ref{increasenc}), we have to deal with $ e^{- i \, k \, \pi \, \alpha} = e^{- i \, 2 \, 
(k \, j) \, \pi} =1 $. The phase shift induced by the spatial inhomogeneities of $ \varphi $ is not detected;

\smallskip

\noindent ii) For $ {\rm par} = e $ or $ {\rm par} = o $, the power $ (-1)^k $ inside (\ref{defpsikbkeno0pb}) is simply $ 1 $ or $ -1 \, $;

\smallskip

\noindent iii) According to $ k \in \cK_s^{c_1} (e) $ or $ k \in \cK_s^{c_1} (o) $, we can replace $ \aaa (\eps \, k \, \pi , k \, \pi ,0 ) $
by $ \aaa (\eps \, k \, \pi , 0,0 ) $, or by $ \aaa (\eps \, k \, \pi
, \pi ,0 ) $, respectively.

\smallskip

\noindent After that, a dependence upon $ k $ remains inside
(\ref{ffirstsum}). It is examined in detail below. In view of  
\eqref{reso2}, remark that
\begin{equation} \label{partinPsik} 
 \begin{aligned} \bigl \lbrack 1 - p (k \, \pi) \bigr \rbrack \ (k \, \pi -t) & = - \, \frac{\ell}{q \, (q+1)}  \left( \frac{1}{(k \pi)^{q-1}} - 
 \frac{t}{(k \pi)^q} \right) \\
 \ & \quad \, +  \left( \frac{t}{k^q} - \frac{\pi}{k^{q-1}} \right) o(1) . 
 \end{aligned} 
\end{equation}
For $ k $ with $ \eps^{-1} \, \eta \, \leq \, k \in \cK_s^{c_1} (e) $, since $ 2 \leq q $, exploiting (\ref{defpsikbkeno0pb}) and (\ref{partinPsik})
together with ii) and iii), we can deduce that
\begin{equation} \label{devb0kpsi0k}  
\begin{aligned}
b_k^0 \ e^{\frac{i}{\eps} \, \Psi_k^0(\frac{T}{\eps})} &= \sqrt{\frac{2 \pi}{ \gamma}} \aaa (\eps \, k \, \pi , 0 ,0 ) 
\ e^{- i \, \frac{\pi}{4}+\frac{i \, T}{\eps^2}}  e^{-i \frac{\ell}{q(q+1)} \left(\frac{1}{\eps ( k \pi)^{q-1}} -  
\frac{T}{\eps^2( k \pi)^q }\right)}  \\
&\quad+ \frac{T}{\eta^q} \ o( \eps^{q-2} ) + \frac{1}{\eta^{q-1}} \ o( \eps^{q-2} ) . 
\end{aligned} 
\end{equation}
Introduce the symbols $ \cO_\eta (\eps^k) $ and $ o_\eta (\eps^k) $ to mean respectively $ C(\eta) \, \cO(\eps^k) $ and $ C(\eta) \, 
o(\eps^k) $ for some constant $ C(\eta) $ which may go to $ + \infty $ when $ \eta $ goes to zero. When summing even $ k $ at the 
level of the first sum inside (\ref{ffirstsum}), with $ {\rm par} = e $, we  recognize a Riemann sum with small width $ \eps \, 2 \, \pi $. 
Since the regularity of the integrand degenerates at $ s= 0 $, the rate of convergence is simply $ O_\eta (\eps) $. By this way, in view 
of (\ref{devb0kpsi0k}), when $ q=2 $, we obtain
\[
\begin{array}{rl}
\displaystyle \sum_{\eps^{-1} \eta \, \leq \, k \in \cK_s^c(e)} \! \! \eps \ b_k^0 \ e^{\frac{i}{\eps} \Psi_k^0 (\frac{T}{\eps})} = \! \! \! & 
\displaystyle \frac{e^{i \, ( \frac{T}{\eps^2} - \frac{\pi}{4}) }}{\sqrt{2  \pi \gamma}} \ \int_{\pi \, \eta}^{+\infty} \! \! e^{-i \, 
\frac{\ell}{6} \, (\frac{1}{s} - \frac{T}{s^2 })} \ \aaa(s, 0 ,0 ) \ ds \\
\ & \displaystyle + \,  o_\eta (\eps^0) + \cO_\eta (\eps) ,
\end{array}
\]
where we have used the property that $k\pi\eps $ goes up to $\cT$, so the integral carries over the whole 
support of $\aaa (\cdot,0,0)$. 
When $q>2$, since $ \eta \leq \eps \, k \leq \cT $, observe that (\ref{devb0kpsi0k}) involves the factor
$$ e^{-i \frac{\ell}{q(q+1)} \left(\frac{1}{\eps ( k \pi)^{q-1}} - \frac{T}{\eps^2( k \pi)^q }\right)} = e^{-i \frac{\ell}{q(q+1)} \, 
\left(\frac{1}{(\pi \eps k)^{q-1}} - \frac{T}{(\pi \eps k)^q} \right) \, \eps^{q-2}} = 1 + \cO \( \frac{\eps^{q-2}}{\eta^q} \) . $$
Hence, when $q>2$, the Riemann sum argument together with (\ref{devb0kpsi0k}) yields simply
\[
 \sum_{\eps^{-1} \eta \, \leq \, k \in \cK_s^c(e)} \! \! \! \! \! \eps \ b_k^0 \ e^{\frac{i}{\eps} \, \Psi_k^0 (\frac{T}{\eps})} = 
 \frac{e^{i \, (\frac{T}{\eps^2} - \frac{\pi}{4})}}{(2 \, \pi)^{1/2} \ \gamma^{1/2}} \ \int_{\pi \, \eta}^{+\infty} \!  \aaa(s, 0 ,0 ) 
 \ ds + \cO (\eps) + \cO \( \frac{\eps^{q-2}}{\eta^q} \) ,
\]
where we have used the fact that for a smooth integrand (as it is the case with $\aaa$ only, that is, without the above singular phase
term), a convergence rate (of the order of the discretization parameter) is available in Riemann sums. In the two preceding integrals, 
the integrand is bounded near $ s = 0 $. Modulo some $ \cO(\eta) $, we can integrate from $ 0 $ to $ + \infty $. Similar considerations 
apply when dealing with odd values of $k$. Summing up, we find the leading-order term of (\ref{increasenc}). 

\smallskip

\noindent Now, come back to (\ref{ffirstsum}). From the preceding estimates, when $ q =2 $, the error term is of the type  
$$ \eps \ \Bigl \lbrack \, \cO(\eta) + \cO \left(\frac{\eps^{q-1}}{\eta^{q+1}} \right) + \cO\left(\eps^{\frac{q}{q+1}}\right) + o_\eta (\eps^0) 
+ \cO_\eta (\eps) \, \Bigr \rbrack =  \eps \ \bigl \lbrack \, \cO(\eta) + o_\eta (\eps^0) \, \bigr \rbrack . $$
This is valid for all $ \eta \in \, ]0,1] $. By fixing $ \eta $ increasingly smaller and then letting $ \eps $ go to zero, this implies the bound 
$ o(\eps) $, as expected in (\ref{increasenc}). On the contrary, when $ q > 2 $, we have to deal with an error term like
\[ \eps \(\cO(\eta) + \cO \left(\frac{\eps^{q-1}}{\eta^{q+1}} \right) + \cO\left(\eps^{\frac{q}{q+1}}\right) + \cO (\eps) +  
\cO \( \frac{\eps^{q-2}}{\eta^q} \) \) . \]
Setting $ \eta = \eps^{\frac{q-2}{q+1}} $ then yields (\ref{increasencbis}).
 \end{proof}

\noindent The set $ \cC^0_c (\RR_+ ) $ of all continuous functions having a compact support is a Banach space when it is equipped 
with the sup-norm. Define $ \lambda (q) = \ell $ if $ q = 2 $ (with $ \ell $ as in Assumption \ref{resoa}) and $ \lambda (q) = 0 $ if $ q > 2 $. Given $ T\ge 0 $ and $ q \geq 2 $, 
consider the nontrivial continuous linear form
$$ \begin{array}{rcl}
\cL(T) \, : \, \cC^0_c (\RR_+ ) & \longrightarrow & \RR \\
\aaa  & \longmapsto & \displaystyle \cL(T)(\aaa) := \int_0^{+\infty} \! e^{-i \, \frac{\lambda(q)}{6} \, (\frac{1}{s} - \frac{T}{s^2 })} \ 
\aaa (s) \ ds .
\end{array}  
$$
Its kernel $ \mathrm{ker} \, \cL(T) $ is a closed vector space of codimension one. Obviously, the complement $ \bigl( \mathrm{ker} \, 
\cL(T) \bigr)^c $ of $ \mathrm{ker} \, \cL(T) $ is dense so that, generically, $ \aaa\in \bigl( \mathrm{ker} \, \cL(T) \bigr)^c $.

\begin{cor}[Constructive interference]  \label{supnormamplificor} Fix any $ T \in [\cT , 2 \, \cT] $. Select $ a$ as in Assumption 
\ref{persourcea}, with moreover $ \aaa (\cdot,0,0) $ or $ \aaa (\cdot,\pi,0) $ in $  \bigl( \mathrm{ker} \, \cL(T) \bigr)^c $. Looking 
at the solution $ u  $ to (\ref{subsec:resoscint}) at the time $ T /\eps $ and at well chosen positions (which may depend on the 
parameter $ \eps $), one can observe some amplification of the sup norm. As a matter of fact, for all $ j \in \ZZ $, we have:
\begin{equation} \label{increasenccor} 
\limsup_{\eps \rightarrow 0+} \quad \Bigl \vert \frac{1}{\eps} \ u \Bigl( \frac{T}{\eps} , 2 \, \eps \, j \Bigr) \Bigr \vert = 
\ell^{\aaa}_{s} \not = 0 
 \end{equation}
 with
 $$ \ell^{\aaa}_{s} := \frac{1}{\sqrt{2 \pi \gamma}} \ \left \lbrack \, \vert \cL(T)( \aaa (\cdot,0,0) ) \vert + \vert \cL(T)( \aaa 
 (\cdot,\pi,0) ) \vert \, \right \rbrack . $$
\end{cor}

\smallskip

\noindent Assume moreover that $ \aaa (\cdot,0,0) $ is $ \pi $-periodic. Then, (\ref{increasenc}) gives rise to
\begin{equation} \label{increasenccorreal} 
\quad \Bigl \vert \frac{1}{\eps} \ u \Bigl( \frac{T}{\eps} , 2 \, \eps \, j \Bigr) \Bigr \vert = \frac{2}{\sqrt{2 \pi \gamma}} 
\vert \cL(T)( \aaa (\cdot,0,0) ) \vert \ \Bigl \vert \cos \, ( \gamma - \frac{\pi}{4} ) \, \frac{1}{\eps} \Bigr \vert + o(1) . 
\end{equation}
Therefore, any number contained in the interval $ \lbrack 0, \ell^{\aaa}_{s} \rbrack  $ is an adherent point of the family $ \bigl \lbrace 
\eps^{-1} \, \vert u (\eps^{-1} \, T , 2 \, \eps \, j) \vert  \bigr \rbrace_\eps $. This is typical of a highly oscillating behavior. As mentioned 
before, the formula (\ref{third-si-devtt}) looks like (\ref{increasenc}) and (\ref{increasencbis}). But, as will be seen in the next paragraph, 
outside the moving lattice $ \{ 2 \, \eps \, j ; j \in \ZZ \} $, the situation is completely different.


\subsubsection{Destructive interferences}\label{desinterference} In this paragraph, we consider the situation where $ x = \alpha \, \eps $ with 
$ \alpha \in \RR \setminus \{ 2 \, \ZZ \} $. Then, the property i) in the proof of Proposition \ref{supnormamplifi} no longer applies. 
The definition (\ref{sumofwavepackets2}) of $ u_k $ does contain the  factor $ e^{- i \, k \, \pi \, \alpha} $ which comes from 
the spatial inhomogeneities of the phase $ \varphi$ and which, after summation, can induce additional cancellations. 

\begin{prop} [Destructive interference]  \label{supnormnonamplifi2}
  Select any $ \alpha \in \RR \setminus \{ 2 \, \ZZ \} $. Suppose that
  Assumptions~\ref{resoa},~\ref{normareso},~\ref{strengthenp} (with $
  D \geq 4 $) and~\ref{persourcea} are satisfied. Then for all 
$ T \in [\cT , 2 \, \cT] $, the solution $ u $ to (\ref{eq:mainajoutlinpourm}) is such that
\renewcommand\arraystretch{1.5}
\begin{equation}\label{destructiveinter} 
u \( \frac{T}{\eps} , \alpha \, \eps \) = \left \{ \begin{array}{lcl} o(\eps) & \text{if} & q = 2 , \\
\cO \left( \eps^{\frac{6 q-2}{5 \, q}} \right) & \text{if} & q > 2  . 
\end{array} \right.
\end{equation}
\renewcommand\arraystretch{1}
\end{prop}
 
\begin{proof}
We resume \eqref{ffirstsum}, which holds for all $\eta>0$. Exploiting (\ref{partinPsik}) and (\ref{devb0kpsi0k}), this becomes
  \begin{align*}
\frac{1}{\eps} \ u \( \frac{T}{\eps} , \alpha \, \eps \) = & \, (2 \, \pi)^{1/2} \ \gamma^{-1/2} \ e^{i \, (\frac{T}{\eps^2} - \frac{\pi}{4})} \ 
\sum_{\frac{\eta}{\eps} \le k \,  \text{even} \le \frac{\cT}{\pi\eps}} \eps \ e^{- i k \pi \alpha} \, G^e_q (\eps, \eps k \pi) \\
  \ & + (2 \, \pi)^{1/2} \ \gamma^{-1/2} \ e^{i \, (\frac{T}{\eps^2} - \frac{\pi}{4} - \frac{2 \, \gamma}{\eps})} \ 
  \sum_{\frac{\eta}{\eps} \le k \,  \text{odd} \le \frac{\cT}{\pi\eps}} \eps \ e^{- i k \pi \alpha} \, G^o_q (\eps, \eps k \pi) \\
& + \cO(\eta) + \cO \left(\frac{\eps^{q-1}}{\eta^{q+1}} \right) + \cO\left(\eps^{\frac{q}{q+1}}\right) +  o \( \frac{\eps^{q-2}}{\eta^q} \) ,
\end{align*}
where by definition 
\[
G^e_q(\eps,s) := e^{-i \, \frac{\ell}{6} \, (\frac{1}{s} - \frac{T}{s^2 }) \, \eps^{q-2}} \ \aaa (s, 0 ,0 ) \, , \qquad 
G^o_q(\eps,s) := e^{-i \, \frac{\ell}{6} \, (\frac{1}{s} - \frac{T}{s^2 }) \, \eps^{q-2}} \ \aaa (s, \pi ,0 ) .
\]
We consider separately the two above sums. We discuss the case $ k $ even, the case $ k $ odd being similar.
The idea is to use Abel's summation formula. To this end, given some $ \delta \in \, ]0,1] $, we interpret the sum 
as follows
$$ \sum_{\frac{\eta}{\eps} \le k \,  \text{even} \le \frac{\cT}{\pi\eps}} \eps \ e^{- i k \pi \alpha} \, G^e_q (\eps, \eps k \pi) 
= \sum_{\frac{\eta}{\delta} \leq j \leq \frac{\cT}{2 \pi \delta}} \ \ \sum_{\frac{j \delta}{\eps} \leq k \, \text{even} \leq \frac{(j+1) 
\delta}{\eps}} \eps \ e^{- i k \pi \alpha} \, G^e_q (\eps, \eps k \pi) . $$
For all $ j $, fix some $ k_j $ even inside $ \bigl \lbrack \frac{j \delta}{\eps} , \frac{(j+1) \delta}{\eps} \bigr \rbrack $. 
For all $ k $ in this interval, Taylor's formula gives rise to
\begin{equation*}
  |G^e_q(\eps,\eps k \pi)-G^e_q(\eps,\eps k_j \pi) | \le \pi \, \delta \, \sup_{s\ge \pi \eta}|\part_s G^e_q(\eps,s)|  =
  \delta \ \cO\(\frac{\eps^{q-2}}{\eta^3}\) + \delta \ \cO (1) .
\end{equation*}
It follows that
\[ \sum_{\frac{\eta}{\eps} \le k \,  \text{even} \le \frac{\cT}{\pi\eps}} \eps \ e^{- i k \pi \alpha} \, G^e_q (\eps, \eps k_j \pi) 
= \cE_r^e \, + \! \! \! \sum_{\frac{\eta}{\delta} \leq j \leq \frac{\cT}{2 \pi \delta}} \! \! \! \eps \ G^e_q (\eps, \eps k_j \pi) \times 
\! \! \! \! \! \! \sum_{\frac{j \delta}{\eps} \leq k \, \text{even} \leq \frac{(j+1) \delta}{\eps}} \! \! \!  \! \! \! e^{- i k \pi \alpha} . \]
The error term $ \cE_r^e $ can be estimated according to
$$ \cE_r^e = \frac{\cT}{2 \pi \delta} \ \frac{\delta}{\eps} \ \eps \ \left \lbrack \delta \ \cO\(\frac{\eps^{q-2}}{\eta^3}\) + \delta \ 
\cO (1) \right \rbrack = \delta \ \cO\(\frac{\eps^{q-2}}{\eta^3}\) + \delta \ \cO (1) . $$ 
Since $ e^{- i \, \pi \, \alpha} \not = 1 $, we have
\begin{equation*}
  \Bigl \vert \sum_{\frac{\eta}{\delta} \leq j \leq \frac{\cT}{2 \pi \delta}} \! \! \! \eps \ G^e_q (\eps, \eps k_j \pi) \times 
\! \! \! \! \! \! \sum_{\frac{j \delta}{\eps} \leq k \, \text{even} \leq \frac{(j+1) \delta}{\eps}} \! \! \!  \! \! \! e^{- i k \pi \alpha}  
\, \Bigr \vert = \frac{1}{\delta} \ \cO(\eps ) .
\end{equation*}
In short, we have
\renewcommand\arraystretch{2.5}
\begin{equation*}
\begin{array}{rl}
  \frac{1}{\eps} \ \left |u \Bigl( \frac{T}{\eps} , \alpha \, \eps \Bigr) \right| = \! \! \! & \displaystyle \cO(\eta) + \cO \left(\frac{\eps^{q-1}}{\eta^{q+1}} 
  \right) + \cO\left(\eps^{\frac{q}{q+1}}\right) + o \( \frac{\eps^{q-2}}{\eta^q} \) \\
  \ & \displaystyle + \, \delta \ \cO\(\frac{\eps^{q-2}}{\eta^3}\) + \delta \ \cO (1) + \frac{1}{\delta} \   \cO(\eps ) .
  \end{array}
\end{equation*}
\renewcommand\arraystretch{1}

\noindent This is valid for all $ (\eta , \delta) \in \, ]0,1]^2 $. We fix $ \delta = \eta^4 $ so that
$$ \cO(\eta) + \delta \ \cO\(\frac{\eps^{q-2}}{\eta^3}\) + \delta \ \cO (1) + \frac{1}{\delta} \ \cO(\eps) = \cO(\eta) + \frac{1}{\eta^4} \ \cO(\eps ) . $$
By fixing $ \eta $ increasingly smaller and then letting $ \eps $ goes to zero, we can recover some $ o (\eps^0) $ or, after multiplication by 
$ \eps $, some $ o (\eps) $ as announced in (\ref{destructiveinter}). When $ q > 2 $, a better estimate is available by optimizing the choice 
of $ \eta $. Just take $ \eta = \eps^{(q-2) / 5 q} $ to obtain
(\ref{destructiveinter}).
\end{proof}

\begin{rem}[Contrast between constructive and destructive interferences]\label{nonoptimal}  The controls of the error terms inside 
\eqref{increasenc}, \eqref{increasencbis} and \eqref{destructiveinter} are not claimed to be sharp. For instance, by specifying a rate 
of convergence at the level of \eqref{hypp2relax}, the precision $ o(\eps) $ in \eqref{increasenc} and \eqref{destructiveinter} could be 
improved into $ \cO(\eps^{1+\kappa}) $ for some  $ \kappa>0 $. At all events, the amplitude of the solution  $ u $ to (\ref{subsec:resoscint})
is asymptotically maximal on a set of Lebesgue measure zero, which is the lattice $ \eps \ZZ $ moving with $ \eps \in ]0,1] $. Everywhere 
else, it is smaller.
\end{rem}


\section{Nonlinear analysis} \label{sec:nonlinear}

\noindent In this chapter, we prove the nonlinear information (2) of Theorem~\ref{theo:resumeNL}, as well as Theorem~\ref{theo:resumeNLbis}.
In Section \ref{sec:settingNL}, we precise the framework, and we collect various estimates about the solution $ u^{(0)} $ of (\ref{eq:Picard0}). 
In Section~\ref{sec:sorting-gauge}, we measure the influence of different types of nonlinearity according to gauge parameters 
 $ \textfrak {g}$  that characterize them. We prove that nonlinear effects are not detected at leading order as long as $ \textfrak {g}  
 \not = 1 $. This is Fact \ref{fact2} in the PDE context. As a consequence, when $ \textfrak {g}  \not = 1 $, the distinction between 
 constructive and destructive interferences remains in the same state as in the linear case. This dichotomy does persist when 
 $ \textfrak {g}  = 1 $. But, as will be seen in Section \ref{sec:nonlineareffectg=1}, the profiles exhibited in (\ref{increasenc}) must 
 be modified accordingly, in order to take into account the nontrivial effects of nonlinearity.
 

\subsection{General setting}\label{sec:settingNL} In Paragraph \ref{subsec:mainassreca}, we recall the main assumptions, and 
we start the discussion about nonlinear effects. In Paragraph \ref{singularintegraloperator}, we study the kernel of a singular operator, 
which appears when seeking sup norm estimates. In Paragraph \ref{subsec:Classigaugeparameters}, we classify the different sorts 
of gauge parameters, and we illustrate them by examples. In Paragraph~\ref{subsec:roughestimaref}, we establish various estimates 
concerning the solution $ u^{(0)} $ of \eqref{eq:Picard0}. 


\subsubsection{Main assumptions}\label{subsec:mainassreca} We work under the hypotheses of Theorems~\ref{theo:resumeNL}
and~\ref{theo:resumeNLbis} concerning $ p $ and $\varphi$. In particular, we suppose that $ q= 2 $ and $ D \geq 4 $. The phase 
$\varphi$ is subject to Assumption~\ref{choiofa}. The expression $ u^{(0)} $ is obtained by solving the linear equation \eqref{eq:Picard0}, 
with $ F_{\! L} $ as in (\ref{eq:source}). In (\ref{eq:source}), the sum is assumed to be finite,  to avoid extra 
discussions about the convergence of infinite sums which can appear in the approximating process. 

\smallskip

\noindent In view of Propositions~\ref{supnormamplifi} and \ref{supnormnonamplifi2}, the function $u^{(0)}$ is of size $\eps $ in 
$L^\infty$. It follows that the quadratic nonlinearity of (\ref{eq:Picard1}) can be expected to play a role at leading order for long 
times $ t \sim \eps^{-1} $. The right hand side of (\ref{eq:Picard1}) may seem quite specific. It is adjusted in order to generate 
through (\ref{eq:Picard1}) a solution $ u^{(1)} $ of size comparable to $ u^{(0)} $. To understand why, and also to discern the 
possible effects of other nonlinearities, it is interesting to generalize (\ref{eq:Picard1}) up to some extent. With this in mind, we 
replace (\ref{eq:Picard1})  by (\ref{eq:Picard1NL}) with $ F_{\!N\!L} $ satisfying Assumption \ref{choiononlinso}. Following 
\eqref{eq:u/U}, we introduce
\begin{equation} \label{passutocU}
 u^{(j)} (t,x) = \eps  e^{i t/\eps}\, \cU^{(j)} \Bigl( \eps t,\frac{x}{\eps}\Bigr) , \qquad  \cU^{(j)}(T,z) = \frac{1}{\eps} e^{-iT/\eps^2} u^{(j)}
  \Bigl( \frac{T}{\eps},\eps z \Bigr). 
\end{equation}
The solution $ u^{(1)} $ to (\ref{eq:Picard1NL}) is a superposition of the contributions brought by the different terms $ F_{j_1 j_2 \nu} $
composing $ F_{\!N\!L} $, see (\ref{contofFNLNNstep2}). Thus, we can study separately what happens for a fixed choice of 
$ (j_1,j_2,\nu) \in \NN^2 \times \ZZ $. With this in mind, we focus our attention on a single monomial having the form
\begin{equation}\label{singlemonomial}
F_{\! N \! L} \equiv F_{j_1 j_2 \nu} = \eps^\nu e^{i \omega t/\eps} \chi \Bigl( 3-2\frac{\eps t}{\cT}\Bigr)\chi \Bigl(
   \frac{x}{r\eps^{\iota}} \Bigr) u^{j_1}\bar  u^{j_2}.
\end{equation}
We have seen in Subsection~\ref{subsec:toymodel} that the gauge parameter $ \textfrak {g} $ is a good indicator of the time 
oscillations which remain in the source term of equation (\ref{eq:EDONLajout}) after filtering out of the equation (\ref{third-si}) 
through the change (\ref{changeofuu}). When dealing with (\ref{singlemonomial}), a similar definition applies.

\begin{defi}[Gauge parameter] \label{defidegaugeparameter} The gauge parameter
  associated with $F_{j_1j_2\nu}$ is the real number $ \textfrak {g}_{j_1 j_2 \nu} $ defined by $
\textfrak {g}_{j_1 j_2 \nu} := \omega +j_1-j_2 $.
\end{defi}

\noindent From now on, we fix $ F_{\! N \! L}  $ as in (\ref{singlemonomial}) with indices $ \nu $, $ j_1 $ and $ j_2 $ adjusted 
in such a way that $\nu+j_1+j_2 \geq 2$. We will sometimes simply note $ \textfrak {g} \equiv \textfrak {g}_{j_1 j_2 \nu} \in \RR $. 
In (\ref{singlemonomial}), the coefficient which appears in front of $ u^{j_1}\bar  u^{j_2} $ is the product of three factors.

\smallskip

\noindent In the light of the first factor $ e^{i \omega t/\eps}  $, in the case of a non-zero frequency $ \omega \not = 0 $, 
the source term $ F_{\! N \! L} $ does involve time oscillations. Reasons for introducing $ e^{i \omega t/\eps}  $ have 
been explained in Remark~\ref{eliminationdeF}, and also in Paragraph \ref{disprela} when adjusting $ p  $ in 
order to recover (\ref{reso1}).

\smallskip

\noindent Looking at the second factor, the source term $ F_{\! N \! L} $ is switched on after all signals have been emitted, 
that is during the long time interval $ [\cT / \eps , 2 \cT / \eps ] $, which could be replaced by $ [\eta / \eps , 1 / (\eta \eps) ] $
for any $ \eta \in ] 0,1] $. But a positive gap ($ \eta > 0 $) seems to be needed. Indeed, Lemma~\ref{localizations} makes a 
first group of wave packets which, due to a dispersive phenomenon, is negligible in the limit $\eps\to 0$. It requires to consider 
sufficiently large times, so the phase $\Phi_k$ could be uniformly non-stationary in $\xi$.

\smallskip

\noindent In the light of the third factor, the source term $ F_{\! N \! L} $ is spatially localized in a ball of size
$ r \eps^\iota $. The impact of $ F_{\! N \! L} $ is potentially all the more stronger that $ \iota $ is small. The 
choice of a large negative parameter $ \iota $, with $ \iota \ll -1 $, involves almost no spatial localization. The 
case $ \iota = 0 $ corresponds to a diluted source which acts on the domain where Propositions \ref{supnormamplifi} 
and \ref{supnormnonamplifi2} furnish some refined information. Finally, the selection of the limiting value $ \iota = 1 $ 
implies a concentrated source which, for convenience, is placed here at the origin. Larger values of $ \iota $, with 
$ \iota \geq 1 $, will not be investigated because they have little interest. 

\smallskip

\noindent The impact of the nonlinearity (\ref{singlemonomial}) can be measured by looking at the difference 
$ \cW:=\cU^{(1)} - \cU^{(0)} $. From (\ref{eq:Picard0}) and  (\ref{eq:Picard1NL}), it is easy to deduce that
\begin{equation} \label{eq:detercD}
\qquad \part_T \cW- \frac{i}{\eps^2} \(p (- i \, \part_z) -1 \)  \cW =\eps^{\nu +j_1+j_2-2}
e^{i (\textfrak {g}-1) T/ \eps^2} \cG^\eps , \qquad \cW_{\mid t=0} \equiv 0 ,
\end{equation}
where $\textfrak {g} \equiv \textfrak {g}_{j_1 j_2 \nu} $ is as in Definition~\ref{defidegaugeparameter}, and the source 
term $ \cG^\eps  $ is determined by
\begin{equation}\label{eq:defcG}
\cG^\eps (T,z) =\chi \Bigl( 3-2\frac{T}{\cT} \Bigr) \, \chi \Bigl(\frac{z}{r\eps^{\iota-1}} \Bigr) \, \cU^{(0)}(T,z)^{j_1} \,
\bar\cU^{(0)}(T,z)^{j_2} .
\end{equation}
By construction, the function $ \cG^\eps  $ is smooth and compactly supported in $ (T,z) $. We have seen in 
Section~\ref{sec:lineffect} that $ \cU^{(0)} (T,z) $, and therefore $ \cG^\eps (T,z) $, is some $ \cO(1)$ as long as
$ (T,z) $ is such that $\cT\le T\le 2\cT$ and $|z|\le r/\eps$. Moreover, this control is sharp when $z\in \ZZ$. 
Coming back to (\ref{eq:detercD}), Duhamel's formula reads
\begin{equation} \label{eq:duhamelNL}
\begin{array}{rl}
\quad \ \cW(T,z)= & \displaystyle \! \! \! \eps^{\nu +j_1+j_2-2} \ (2 \pi)^{-1} \\
 & \displaystyle \times \int_0^T \! \! \iint e^{-i(z-y)\xi +i\frac{T-s}{\eps^2}(p(\xi)-1) +i(\textfrak {g}-1)\frac{s}{\eps^2}} 
 \, \cG^\eps(s,y) \, dsdyd\xi .
 \end{array} 
\end{equation}
Our aim is to study $ \cW(T,z) $ through (\ref{eq:duhamelNL}). As a first step, we would like to establish that,
for $ \cG^\eps $ as in (\ref{eq:defcG}), we have
\begin{equation}\label{eq:defcGintpetit}
\quad \int_0^T \! \! \iint e^{-i(z-y)\xi +i\frac{T-s}{\eps^2}(p(\xi)-1)
    +i(\textfrak {g}-1)\frac{s}{\eps^2}} \, \cG^\eps(s,y) \,
  dsdyd\xi =\cO(1)\text{ in } L^\infty .
\end{equation}
We know already that the function $ \cG^\eps  $ is of size $ 1 $ at integer points, and that it is of smaller amplitude 
at all other spatial positions. Thus, the matter is to understand how the integral operator inside (\ref{eq:defcGintpetit}) 
acts on $ L^\infty $. The main problem when dealing with (\ref{eq:defcGintpetit}) is the global domain of integration 
in $ \xi $ and (for $ \iota < 1 $) the large domain (of size $ \eps^{\iota -1} $) of integration in $ y $. This difficulty is 
examined, and partly solved, in the next paragraph.


\subsubsection{A singular integral operator}\label{singularintegraloperator} 
Following the convention (\ref{defFouriertrans}), denote by $ \cF_\star $ the partial Fourier 
transform with respect to the variable $ \star \in \{y,\xi \} $. Given $ \tau \in \RR $ and $ \Lambda \in L^\infty (\RR) $, 
define the operator $ B^\Lambda_\tau  $ by
\begin{align}
B^\Lambda_\tau \cG^\eps (z) & := \cF_\xi \bigl( 2 \pi \bigl( e^{i \tau (p(\xi)-1)} - 1 \bigr) \Lambda(\xi)  (\cF^{-1}_y \cG^\eps) (\xi) \bigr) (z) \label{defBktau} \\
\ & \, = \iint e^{-i (z-y) \xi} \bigl( e^{i \tau (p(\xi)-1)} - 1 \bigr) \Lambda(\xi) \cG^\eps (y) dyd\xi  . \nonumber
\end{align}
When $ \Lambda \equiv \mathds{1}_{\mathbb R} $, the operator $ B^\Lambda_\tau $ is simply denoted by $ B_\tau := 
B^{\mathds{1}_{\mathbb R}}_\tau $. Looking at (\ref{defBktau}), it is clear that $ B^\Lambda_\tau  : H^\sigma(\RR) 
\rightarrow H^\sigma(\RR) $ is a bounded operator for all $\sigma\in \RR$, with
\begin{equation}\label{Aktauv}
 \| B^\Lambda_\tau \cG^\eps \|_{H^\sigma(\RR)} \leq 2 \| \Lambda \|_{L^\infty(\RR) } \| \cG^\eps \|_{H^\sigma(\RR)} .
 \end{equation}
In \eqref{defBktau}, we first integrate in $ y $ and then in $ \xi $. Another viewpoint, which is more adapted to get $ L^\infty $-estimates, is 
to first integrate in $ \xi $ and then in $ y $. By this way, we find $ B^\Lambda_\tau \cG^\eps = K^\Lambda_\tau * \cG^\eps $ with a kernel $ K^\Lambda_\tau  $ 
given by
\begin{equation} \label{defdeffeKt}
\quad \ K^\Lambda_\tau (y) :=\int e^{-i y \xi} \bigl( e^{i \tau (p(\xi) -1)} -1 \bigr) \Lambda(\xi) d \xi , \qquad K_\tau := K^{\mathds{1}_{\mathbb R}}_\tau  . 
\end{equation}
In view of (\ref{reso2}), where $ \omega^\infty_+ = 1 $ and $ q = 2 $,
the integrand inside (\ref{defdeffeKt}) is integrable, and it depends
smoothly on the parameters $ \tau $ and $ y $.  
The expression $ K^\Lambda_\tau (y) $ is therefore well defined. It is
continuous with respect to $ (\tau,y) $, and 
in view of (\ref{reso2}), 
\begin{equation}\label{assertrat}
\vert K^\Lambda_\tau (y) \vert \leq C \, \vert \tau \vert \, \( \int \frac{ d \xi}{1+\xi^2} \) \, \Vert \Lambda \rVert_{L^\infty
(\RR)} .
\end{equation}
Now, the solution $ \cW $ to (\ref{eq:detercD}) can be decomposed into
$ \cW_l + \cW_{nl} $  with
\begin{align}
\cW_l (T,z) & := \eps^{\nu +j_1+j_2-2} \int_0^T e^{i (\textfrak {g}-1) s / \eps^2} \cG^\eps (s,z) ds , \label{partcDl} \\
\cW_{nl} (T,z) & := \frac{\eps^{\nu +j_1+j_2-2}}{2 \pi} \int_0^T e^{i (\textfrak {g}-1) s / \eps^2}  B_{(T-s)/\eps^2} \cG^\eps (s,z) ds . 
\label{partcDnl}
\end{align}
By this way, the improper integral inside (\ref{eq:duhamelNL}) is defined without ambiguity. Indeed, both (\ref{partcDl})
and (\ref{partcDnl}) involve local integrals with respect to $ s $ of bounded functions. For (\ref{partcDl}), this is obvious since
$ \cG^\eps (s,z) = \cO (1) $. Concerning (\ref{partcDnl}), this results from the pointwise estimate
$$ \vert B_{(T-s)/\eps^2} \cG^\eps (s,z) \vert \leq \| K_{(T-s)/\eps^2} \|_{L^\infty(\RR)} \ \| \cG^\eps (s,\cdot) 
\|_{L^1(\RR)} \lesssim \eps^{\iota -3} . $$
Let us look more closely at $ \cW_l (T,z) $. To get $ \cW_l (T,z) $, it suffices to know $ \cG^\eps (\cdot,z) $, that is $ \cU^{(0)}(\cdot,z) $. 
In this sense, the action on $ \cG^\eps $ leading to $ \cW_l $ is \underline{\it l$\, $}ocal in space, and therefore it is consistent with the 
dichotomy between constructive and destructive interferences exhibited in Propositions \ref{supnormamplifi} and \ref{supnormnonamplifi2}. 
In fact, a precise asymptotic description of $ \cW_l $ is available.

\begin{lem} [Description of the part $ \cW_l $] \label{descripofcdl} We work under Assumption \ref{persourcea}, with moreover 
$ \aaa (T,\cdot,x) $ periodic of period $ \pi $. Then, for all $ z \in \RR \setminus \{2 \ZZ \} $, we find that $ \cW_l (T,z) = o(1) $. When 
$ \nu+j_1+j_2 > 2 $ or when $ \textfrak {g} \not = 1 $, for all $ z = 2 j $ with $ j \in \ZZ $, we have again $ \cW_l (T,2j) = o(1) $. On the 
contrary, when $ \iota \in [0,1[ $, $ \nu+j_1+j_2 = 2 $ and $ \textfrak {g} = 1 $, we obtain that
{\small
\begin{align} 
& \ \cW_l (T,2j) = o(1) + \, \displaystyle \Bigl \lbrack \sqrt{\frac{2}{\pi \gamma}} \, \cos \Bigl( \frac{\gamma}{\eps} - \frac{\pi}{4} \Bigr) 
\Bigr \rbrack^{j_1+j_2} e^{i (j_2-j_1) \gamma / \eps} \int_0^T \chi \Bigl( 3-2\frac{s}{\cT} \Bigr) \times \label{estnonlincasg} \\
& \quad \Bigl(  \int_0^{+\infty} e^{-i \, \frac{ \ell}{6} \, (\frac{1}{\sigma_1} - \frac{s}{\sigma_1^2})} 
\, \aaa (\sigma_1,0,0) \, d\sigma_1 \Bigr)^{j_1} \Bigl(  \int_0^{+\infty} e^{i \, \frac{ \ell}{6} \, (\frac{1}{\sigma_2} - \frac{s}{\sigma_2^2})} 
\, \aaa (\sigma_2,0,0) \, d\sigma_2 \Bigr)^{j_2} \, ds . \nonumber
\end{align}}
\end{lem} 

\noindent Lemma \ref{descripofcdl} is instructive. It indicates, among other things, that the constructive interferences do not
impact $ \cW_l $ when $ \textfrak {g} \not = 1 $. As will be seen, this principle also applies to $ \cW_{nl} $. 

\begin{proof} First, recall that $ \cW_l (T,z) = 0 $ when $ 0 \leq T \leq \cT $. For $ \cT \leq T $, observe that
\begin{equation} \label{boundofDl} 
\vert \cW_l (T,z) \vert \leq \eps^{\nu+j_1+j_2 - 2} \int_{\cT}^T \vert \cU^{(0)} (s,z) \vert^{j_1+j_2} ds = \cO \bigl( \eps^{\nu+j_1+j_2 - 2} \bigr) . 
\end{equation}
When $ \nu+j_1+j_2 > 2 $, the smallness of $ \cW_l (T,z) $ follows directly from (\ref{boundofDl}). Now, assume that 
$ \nu+j_1+j_2 = 2 $. The first assertion of Lemma \ref{descripofcdl}, the one implying positions $ z \in \RR \setminus \{2 \ZZ \} $, 
is a direct consequence of (\ref{boundofDl}), Proposition \ref{supnormnonamplifi2} and  Lebesgue's dominated convergence 
theorem. Finally, consider the case $ z = 2 j $ with $ j \in \ZZ $. From the forthcoming bound (\ref{majounidederiofU0}), that will 
be derived independently in Paragraph \ref{subsec:roughestimaref}, we know that
$$ \vert \part_s \cG^\eps (s,z) \vert = \cO \bigl( \vert \part_s \cU^{(0)} (s,z) \vert \bigr) = \cO \bigl( \eps^{-2/3} \bigr) .$$
When $ \textfrak {g} \not = 1 $, an integration by parts in $ s $ performed at the level of (\ref{partcDl}) indicates that $ \vert \cW_l 
(s,z) \vert = \cO (\eps^{4/3}) $. When $ \textfrak {g} = 1 $, the time oscillating factor disappears from (\ref{partcDl}). Plug 
(\ref{increasenc}) into (\ref{passutocU}) to recover an asymptotic description of $ \cU^{(0)} $. When $ \iota \in [0,1[ $, for 
small values of $ \eps $, we find $ \chi (2j / r \eps^{\iota -1}) = 1 $, yielding (\ref{estnonlincasg}). 
\end{proof}

\noindent It should be noticed that the formula (\ref{estnonlincasg}) with $ (j_1,j_2) = (2,0) $ differs from (\ref{desintertheoprinbis}).
In (\ref{estnonlincasg}), the two integrals in $ d \sigma_1 $ and $ d \sigma_2 $ are separated while, at the level of (\ref{desintertheoprinbis}),
they are correlated through a nontrivial factor. The reason of this difference is that the contribution $ \cW_{nl} $ is not at all a small 
perturbation of $ \cW_l $. As will be seen, the decomposition of $ \cW $ into $ \cW_l $ and $ \cW_{nl} $ is suitable to show (at least 
when $ \textfrak {g} \not = 1 $) the sup norm decreasing of $ \cW $. But it is not sufficiently precise to obtain (\ref{desintertheoprinbis}). 
When $ \textfrak {g} = 1 $, the two terms $ \cW_l $ and $ \cW_{nl} $ combine  asymptotically to form (\ref{desintertheoprinbis}), which 
provides with the correct prediction.

\smallskip

\noindent There remains to study $ \cW_{nl} $. The access to $ \cW_{nl} $ is more complicated than for $ \cW_l $. Indeed, in 
(\ref{partcDnl}), the action of $ B_\tau $ is \underline{\it n$\, $}on \underline{\it l$\, $}ocal in space, and it is also singular in terms of 
$ \eps $ when $ \eps $ goes to zero. Let us examine this in more detail. From Young's convolution inequality, we know that
\begin{equation} \label{youngineq}
\| B^\Lambda_\tau \cG^\eps \|_{L^\infty (\RR)} \leq \| K^\Lambda_\tau \|_{L^p (\RR)} \| \cG^\eps \|_{L^{p/(p-1)} (\RR)} , 
\qquad \forall \, p \in [1,+\infty] . 
\end{equation}
Come back to (\ref{partcDnl}). Since $ \tau $ is aimed to be replaced by $ (T-s) / \eps^2 $, the access to $ \cW_{nl} $ needs to 
consider large values of $ \tau $, say $ \tau \in [1,+\infty[ $.

\begin{lem}[Estimates on the $ L^2$ and $ L^\infty $ norms of the kernel $ K^\Lambda_\tau $] \label{Lpestikernel} Fix $ \rho\ge 0 $, 
and assume that $ \Lambda \in L^\infty (\RR) $ is such that $ \Lambda (\xi) = \cO (\vert \xi \vert^{-\rho} ) $ when $ \vert \xi \vert $ 
goes to $ + \infty $. Denote by $ \tilde q := \tilde p / (\tilde p-1) $ the H\" older conjugate of $ \tilde p $. By convention, we have 
$ \tilde q = 1 $ when $ \tilde p = + \infty $, and $ \tilde q = 2 $ when $ \tilde p = 2 $. Then, for large values of $ \tau $, we find
\begin{equation} \label{kernelestiaa1} 
\forall \tilde p \in \{+ \infty,2\} , \qquad \| K^\Lambda_\tau \|_{L^{\tilde p} (\RR)} \lesssim 1 + \tau^{(1/q \tilde q) - (\rho/q)} , 
\end{equation}
where $ q \geq 2 $ is the number stemming from \eqref{hypp2}. 
\end{lem} 

\noindent When $ \Lambda  $ is just bounded ($ \rho = 0 $), the estimate (\ref{kernelestiaa1}) helps control the explosion when 
$ \tau \rightarrow + \infty $ of the $ L^{\tilde p} (\RR)$-norm 
of $ K^\Lambda_\tau $. The situation is improving when $ \rho > 0 $. In particular, when $ \tilde q = 1 $ and $ \rho = 1 $, the family 
$ (K^\Lambda_\tau)_\tau $ is bounded in $ L^\infty (\RR) $.

\begin{proof} We can assert that
$$ \| K^\Lambda_\tau  \|_{L^{\tilde p} (\RR)} \lesssim \Bigl( 1 + \int_1^{+\infty} \bigl \vert 1 - \cos  \bigl( \tau (1 - p(\xi)) \bigr) \bigr 
\vert^{\tilde q/2}  \, \vert \Lambda(\xi) \vert^{\tilde q} \, d \xi \Bigr)^{1/ \tilde q} . $$
This is obvious when $ \tilde p = + \infty $. This is a consequence of (\ref{defdeffeKt}) and Plancherel theorem when $ \tilde p = 2 $.
The change of variables $ \eta = 1 - p(\xi) $ sends $ \xi = 1 $ to the positive value $ \eta_1 := 1 - p(1) $, and $ \xi = + \infty $ to 
$ \eta_\infty = 0 $. It gives rise to
$$ \| K^\Lambda_\tau  \|_{L^{\tilde p} (\RR)} \lesssim \Bigl( 1 + \int_0^{\eta_1} \frac{\vert 1 - \cos (\tau \eta) \vert^{\tilde q/2}}
{p' \circ (1-p)^{-1} (\eta)} \, \vert \Lambda \circ (1-p)^{-1} (\eta) \vert^{\tilde q} \, d \eta \Bigr)^{1/ \tilde q} , $$
where $ (1-p)^{-1} :]0,\eta_1] \rightarrow [1,+\infty[ $ is the inverse function of $ 1-p $. 
From (\ref{hypp3ded}) and (\ref{reso2}), it is easy to infer that 
\begin{align*}
\exists C >0 , & \qquad  p' \circ (1-p)^{-1} (\eta) \ge C \eta^{(q+1)/q} ,  \quad \forall \eta \in ]0,\eta_1] , \\  
\exists C >0 , & \qquad \vert \Lambda \circ (1-p)^{-1} (\eta) \vert^{\tilde q} \leq C \eta^{\rho \tilde q/q} ,  \quad 
\forall \eta \in ]0,\eta_1] .
\end{align*}
It follows that
\begin{equation} \label{itfollowsthat1} 
\| K^\Lambda_\tau  \|_{L^{\tilde p} (\RR)} \lesssim \Bigl( 1 + \, \tau^{(1 - \rho \tilde q)/q} \int_0^{\tau \eta_1} \vert 1 - 
\cos \eta \vert^{\tilde q/2} \eta^{(- q -1 + \rho \tilde q) /q } \, d \eta \Bigr)^{1/ \tilde q} . 
\end{equation}
The integral on the right hand side of (\ref{itfollowsthat1}) is convergent near $ \eta_\infty = 0 $ because
$$ \forall \tilde q \in \{1,2 \} , \qquad \tilde q -1 - (1/q) + (\rho \tilde q / q ) > -1 . $$
When $ \rho \tilde q = 1 $, (\ref{kernelestiaa1}) is a direct consequence of (\ref{itfollowsthat1}). Otherwise, remark that
\begin{equation} \label{itfollowsthat2}
 \int_{\eta_1}^{\tau \eta_1} \vert 1 - \cos \eta \vert^{\tilde q/2} \eta^{(- q -1 + \rho \tilde q) /q } \, d \eta \leq \frac{2 q}{1- 
 \rho \tilde q} \eta_1^{(\rho \tilde q -1)/q} \bigl(1 - \tau^{- (1-\rho \tilde q)/q} \bigr) . 
\end{equation}
From (\ref{itfollowsthat1}) and (\ref{itfollowsthat2}), we can deduce (\ref{kernelestiaa1}).
 \end{proof}
 
\begin{lem} [Pointwise estimates on the kernel $ K_\tau $] \label{Lpestikernel2} In the case $ \Lambda \equiv \mathds{1}_{\mathbb R} $ 
and $ q = 2 $, we find (for some $ c \in \CC $) that 
\begin{equation} 
K_\tau (0) \sim c \tau^{1/2} , \quad \text{whereas:} \quad \forall \,
y \not =0 , \quad K_\tau (y) = \cO (\tau^{1/6}) . \label{kernelestiaasupnorm}
\end{equation}
\end{lem} 

\begin{proof} For large values of $ \vert \xi \vert $, we know that $ 1 -p(\xi) \sim c \xi^{-2} $ for some positive constant $ c $, say $ c = 1 $. 
In what follows, to simplify the discussion, we directly replace $ 1 -p(\xi) $ by $ \xi^{-2} $. Then, the change of variables $ \tau \xi^{-2} = \eta $
gives rise to 
\begin{align} 
 K_\tau (0) & = \tau^{1/2} \int_0^{+\infty} (e^{- i \eta}-1) \eta^{-3/2} d \eta \sim c \tau^{1/2} . \nonumber 
  \end{align}
This furnishes the left part of (\ref{kernelestiaasupnorm}), and this indicates that (\ref{kernelestiaa1}) is sharp (at least when $ \tilde p = + \infty $ 
and $ \rho = 0 $). Now, fix some $ y \not = 0 $, and decompose $ K_\tau (y) $ into
\renewcommand\arraystretch{2}
$$ \begin{array}{rl}
K_\tau (y) = \! \! \! & \displaystyle \int_{\vert \xi \vert \leq 1} e^{-i y \xi} ( e^{-i \tau \xi^{-2}}-1 ) d\xi + \frac{1}{i y} \int_{1 < \vert \xi \vert \leq c \tau^{1/3}} 
\part_\xi \bigl( e^{-i y \xi} \bigr) d\xi  \\
\ & \displaystyle + \int_{1 < \vert \xi \vert \leq c \tau^{1/3}} e^{-i (y \xi + \tau \xi^{-2} )} d\xi + \int_{c \tau^{1/3} \leq \vert \xi \vert } e^{-i y \xi}
( e^{-i \tau \xi^{-2}}-1 ) d\xi . 
\end{array} $$
\renewcommand\arraystretch{1}

\noindent The first line is clearly some $ \cO (1) $. An integration by parts in the last term yields, modulo $ \cO (1) $, a better 
decreasing in $ \xi $, namely
$$ \int_{c \tau^{1/3} \leq \vert \xi \vert } e^{-i y \xi} ( e^{-i \tau \xi^{-2}}-1 ) d\xi = \cO (1) + \frac{2 \tau}{y} 
\int_{c \tau^{1/3} \leq \vert \xi \vert } \xi^{-3} e^{-i (y \xi + \tau \xi^{-2})} d\xi . $$
Then, apply the change of variables $ \xi = \tau^{1/3} \eta $ to obtain
$$ K_\tau (y) = \cO (1) + \tau^{1/3} \int_{1 < \vert \eta \vert \leq c} \! \! \! e^{-i \tau^{1/3} (y \eta + \eta^{-2} )} d\eta + \frac{2 \tau^{1/3}}{y} 
\int_{c \leq \vert \eta \vert } \! \! \!  \eta^{-3}  e^{-i \tau^{1/3} (y \eta + \eta^{-2} )} d\eta . $$
Use the principle of non-stationary phase to restrict the domain of integration near the (unique) critical point $ \eta = (2/y)^{1/3} $. When doing 
this, note that the boundary terms can be avoided by smoothing the above localizations.  After stationary phase approximation, there remains 
some $ \cO (\tau^{1/6}) $ as expected.
 \end{proof}

\noindent Lemma \ref{Lpestikernel2} indicates that, when $ \vert \tau \vert $ goes to $ + \infty $, the function $ K_\tau  $ may explode 
more rapidly near the origin than elsewhere. In view of (\ref{kernelestiaasupnorm}), the $ L^2 $-information contained in (\ref{kernelestiaa1}) 
appears as an intermediate information between the two extreme behaviors at $ y=0$ and $ y \not = 0 $. It is more precise than the local 
$ L^2 $-estimate that could be deduced from (\ref{kernelestiaa1}) when $ \tilde p = + \infty $. 

\begin{cor}\label{corollairefac} Fix $ \rho \ge 0 $, and assume that $ \Lambda \in L^\infty (\RR) $ is such that $ \Lambda (\xi) = \cO (\vert \xi 
\vert^{-\rho} ) $ when $ \vert \xi \vert $ goes to $ + \infty $. Let $ h \in L^\infty (\RR) $. Define
$$ \cG^\eps (z) := \chi \Bigl(\frac{\eps z}{r \eps^{\iota}} \Bigr) \, h(z) . $$
Then, for all $ \iota \in [0,1] $, we have
\begin{equation}\label{minoarrarvzdfl} 
\| B^\Lambda_{(T-s)/\eps^2} \cG^\eps \|_{L^\infty (\RR)} \lesssim \| h \|_{L^\infty (\RR)} \( \eps^{(\iota-1)/2} + \eps^{(\iota/2)
+ \rho-1}\) .
\end{equation}
\end{cor}

\noindent The loss in the right hand side of (\ref{minoarrarvzdfl}) is decreasing when $ \rho $ is growing to $ \rho = 1/2 $, and then 
it is saturated for $ \rho = 1/2 $ at the value $ \eps^{(\iota-1)/2} $. As will be seen in the proof below, this residual loss (when $ \iota \in 
[0,1] $) is coming from the $ L^2 $-impact of the spatial localization in a domain of size $ \eps^{\iota-1} $.

\begin{proof} Exploit (\ref{youngineq}) with $ p=2 $, and then
  (\ref{kernelestiaa1}) with $ \tilde p = 2 $ (and $ q=2 $) to get
  \begin{align*}
\| B^\Lambda_{(T-s)/\eps^2} \cG^\eps \|_{L^\infty (\RR)}   & \lesssim \( 1 + \frac{T-s}{\eps^2} \)^{(1/4)-(\rho/2)}
\Bigl \lbrack \int \chi \Bigl( \frac{z}{r\eps^{\iota-1}}\Bigr)^2 h(z)^2 dz \Bigr \rbrack^{1/2} \\
  &  \lesssim \| h \|_{L^\infty (\RR)} \( 1 + \eps^{-(1/2)+\rho}\) \, \eps^{(\iota - 1)/2} ,
\end{align*} 
which yields (\ref{minoarrarvzdfl}) since $ \iota \leq 1 $.
 \end{proof}

\noindent Applied in the context of $ \cW_{nl} $, this furnishes
\begin{equation} \label{exploitatthelevel1}
\| \cW_{nl} (T,\cdot) \|_{L^\infty (\RR)} \lesssim \eps^{\nu+j_1+j_2-2} \| \cU^{(0)} \|_{L^\infty ([0,T] \times 
\RR)}^{j_1+j_2} ( 1 + \eps^{(\iota/2)-1}) .
\end{equation}
The influence of nonlinearities is clearly stronger when $ \nu+j_1+j_2 $ is small. In view of (\ref{exploitatthelevel1}), for 
$ \nu+j_1+j_2+\frac{\iota}{2}-3 > 0 $, it may be neglected. On the contrary, when $ \nu+j_1+j_2+\frac{\iota}{2}-3 < 0 $, the control 
(\ref{exploitatthelevel1}) does not help to extract some uniform bound in $ L^\infty $. For instance, it is clearly  insufficient when 
$ \nu+j_1+j_2 = 2 $, even in the most favorable case $ \iota = 1 $. 

\smallskip

\noindent The preceding analysis does not take into account the time oscillations (with respect to the variable $ s $) 
which can lead to further cancellations when computing the integral term inside (\ref{partcDnl}). Observe that the time 
derivative (in $ s $) of the phase involved in \eqref{eq:duhamelNL} is $ \textfrak {g} - p(\xi) $. This is why the discussion 
in the next paragraph is organized around the zeroes of this function. 


\subsubsection{Classification of gauge parameters}\label{subsec:Classigaugeparameters} Recall that the gauge parameter
$ \textfrak {g} $ has been introduced at the level of Definition \ref{gaugeparadefi} (or \ref{defidegaugeparameter}).
Compare the oscillating integral (\ref{soloscoif}) with (\ref{eq:duhamelNL}). For the choice $ \varphi (s,y) = \textfrak {g} \, s $ 
in (\ref{phaseOIFgen}), the phase $ \Phi $ of (\ref{soloscoif}) coincides with the one which is coming from (\ref{eq:duhamelNL}). But, in comparison with 
Section \ref{sec:setting}, the novelty is that the expression $ \cG^\eps (T,z) $, in contrast to $ a(T,s,x) $, is not strictly speaking
a ``profile". The spatial variables $ z $ and $ x $ are not the same. The support of $ \cG^\eps (T,\cdot) $ is (at least when 
$ \iota < 1 $) of size $ \eps^{\iota-1} \gg 1 $, while the support of $ a(T,s,\cdot) $ is of size one. Dealing with $ \cG^\eps  $ 
in the original variable $ x $ (instead of $ z $) would mean to involve an expression that is expected to be rapidly oscillating 
in $ x $, and therefore that is not compatible with integrations by parts in $ x $.

\smallskip

\noindent Much less information is available on $ \cG^\eps  $ than on $ a  $. However, due to the filtering procedure 
(\ref{passutocU}), it could be expected that $ \part_s \cG^\eps  $ is (to some extent) under control. This forecast, that 
will be confirmed in what follows, explains why integrations by parts in $ s $ should remain effective. But, to this end, 
the criterion (\ref{non-resonanty}) must be restricted. We must now focus on the role of $ \part_s \Phi \equiv p(\xi) - 
\textfrak {g} $. In place of $ \eta $ in (\ref{non-resonanty}), define the threshold
\begin{equation}\label{minocpg} 
c_{\textfrak {g}} := \inf \ \bigl \lbrace \vert p(\xi) - \textfrak {g} \vert \, ; \, \xi \in \RR \bigr \rbrace  . 
\end{equation}
Different situations can occur. 

\begin{defi} \label{classigaugedefi} The gauge parameter $ \textfrak {g}  $ is said:
\begin{itemize}
\item {\rm non resonant} when $ \textfrak {g} \not \in [0,1] $ so that $ c_{\textfrak {g}} >0$;
\item {\rm transitionally resonant} when $ \textfrak {g} =0 $ so that  $ c_{\textfrak {g}} = 0 $. Then, $ p \equiv \textfrak {g}  = 0 $ on the whole
interval $ [-\xi_c,\xi_c] $, and it becomes non zero near $ \pm \xi_c \, $;
\item {\rm pointwise resonant} when $ \textfrak {g} \in \, ]0,1[ $ so that  $ c_{\textfrak {g}} = 0 $, and we 
can find a unique position $ \xi_{\textfrak {g}} \in ]\xi_c,+\infty[ $ such that $ p(\xi_{\textfrak {g}} ) = p(- \xi_{\textfrak {g}} ) =\textfrak {g} \, $;
\item {\rm completely resonant} when $ \textfrak {g} = 1 $ so that  $ c_{\textfrak {g}} = 0 $. Then, the function $ p(\xi) $ can become 
arbitrarily closed to $  \textfrak {g} = 1 $ when $ \vert \xi \vert $ goes to $ + \infty $. 
\end{itemize}
\end{defi}

\noindent Below, we list some examples of $ F_{\! N\! L} $, where for the simplicity of the
presentation, we leave out the localizations involving the function
$\chi$. 
\begin{exam}[A standard choice]\label{excasecommon} Just take $
  F_{\! N\! L}= u^2 $ so that $ (j_1,j_2,\nu) = (2,0,0) $ and $ \omega = 0 $. We 
find $ \nu+j_1+j_2=2$ (critical size). The gauge parameter  is  $
\textfrak {g} = 2 $. It is non resonant.  
\end{exam}

\begin{exam}[Quadratic nonlinearity in $ \vert u \vert $]\label{excasecommonbis} For the selection of $ F_{\! N\! L}  = 
\vert u \vert^2 = u \bar u $, we find $(j_1,j_2,\nu)  = (1,1,0) $ and  $ \omega = 0 $. We still have $\nu+j_1+j_2=2$ 
(critical size), but this time, the gauge parameter is $ \textfrak {g} = 0 $. It is transitionally resonant.  
\end{exam}

\begin{exam}[Presence of time oscillations]\label{implitimeosc} For $ F_{\! N\! L}  = e^{i \omega t/\eps} \vert u \vert^2 $ 
with $ \omega \in  ]0,1[ $, we find $ \textfrak {g} \in  ]0,1[ $. The gauge parameter is pointwise resonant. 
\end{exam}

\begin{exam}[The nonlinearity investigated in Theorem~\ref{theo:resumeNL} and~\ref{theo:resumeNLbis}]
\label{excaseofintronl} The choice made in equation \eqref{eq:Picard1}, that is $ F_{\! N\! L}  = e^{- i t/\eps} u^2 $, 
is built with $ (j_1,j_2,\nu) = (2,0,0) $ and $ \omega = -1 $, so that $ \nu+j_1+j_2=2 $ and $ \textfrak {g} = 1 $. The 
size is critical and the gauge parameter  is  completely resonant. 
\end{exam}

\begin{exam}[Critical cubic nonlinearity]\label{excaseofcubictronl} The critical size can be achieved for a cubic nonlinearity 
like $ F_{\! N\! L}= \eps^{-1} e^{i\omega t/\eps} \vert u \vert^2 u $, in which case $ (j_1,j_2,\nu) = (2,1,-1) $.  When
$ \omega = 0 $, the gauge parameter is completely non resonant. This situation is expected to involve leading order 
nonlinear effects, like in Example~\ref{excaseofintronl}.
\end{exam}


\subsubsection{Various estimates involving $ \cU^{(0)} $}\label{subsec:roughestimaref} The purpose of this paragraph is 
to list estimates that are accessible concerning $ \cU^{(0)} $, and therefore that could be used when dealing with $ \cG^\eps  $
at the level of (\ref{eq:duhamelNL}). Propositions~\ref{supnormamplifi} and~\ref{supnormnonamplifi2} already furnish 
the following optimal local (for $ \vert x \vert \leq r $ or $ |z|\le r/\eps $) sup norm estimate
\begin{equation} \label{majounideU0}
\exists \, C >0, \quad \forall \, T \in [\cT , 2\cT] , \quad \forall |z|\le r/\eps, \quad \vert \cU^{(0)} (T,z) \vert \leq C .
\end{equation}
Global $ L^2 $ and $ L^\infty $ controls are also available. Unlike (\ref{majounideU0}), they may not be sharp.

\begin{lem} [Global control in $ L^2 $ and $ L^\infty $-norm] \label{GlobalcontrolintheL2} For all $ (j,n) \in \NN^2 $, we have
\begin{align} 
\forall \, T \in [0 , 2\cT] , \quad & \| \part^j_T \part^n_z \cU^{(0)} (T,\cdot) \|_{L^2(\RR)} = \cO 
(\eps^{-2j-n-1}) , \label{L2majounideU0mn} \\
\forall \, T \in [0 , 2\cT] , \quad & \| \part^j_T \part^n_z \cU^{(0)} (T,\cdot) \|_{L^\infty(\RR)} = 
\cO (\eps^{-2j-n- \frac{3}{2}}) . \label{LinftymajounideU0t} 
\end{align}
\end{lem} 

\noindent  We emphasize the discrepancy between the optimal uniform $L^\infty$-bound inside \eqref{majounideU0} 
(for data localized in space), and the bound \eqref{LinftymajounideU0t} with $j=n=0$, which holds globally in space. 
This  loss of an $\eps^{-3/2}$ factor is most likely ``only'' technical. It explains why in the forthcoming analysis, the 
presence at the level of \eqref{eq:defcG} of some spatial cut-off (driven by $ \iota $ with $ 0 \leq \iota \leq 1 $) is needed. 

\smallskip

\noindent The derivatives $ \part_t $  and $ \part_x $ applied to oscillations of the form (\ref{eq:source}) with $ \varphi  $ 
as in (\ref{pharetain}) produce respectively the factors $ \eps^{-1} \part_t \varphi \sim \eps^{-1} $ and $ \eps^{-1} \part_x \varphi 
\sim \eps^{-1} t $. Thus, for long times $ t \sim \eps^{-1} $, it could be expected that the action of $ \part^\alpha_{Tz} $ takes 
the form of a loss similar to
\begin{equation}\label{couldbeexpected}
\qquad \forall \alpha = (\alpha_1,\alpha_2) \in \NN^2 , \qquad \part^\alpha_{Tz} \sim \eps^{-\alpha_1+\alpha_2} \part^\alpha_{tx} 
\sim \eps^{-2 \alpha_1} t^{\alpha_2} \sim \eps^{-2 \alpha_1- \alpha_2} . 
\end{equation}
The bounds (\ref{L2majounideU0mn}) and (\ref{LinftymajounideU0t}) are both in agreement with this prediction since the 
application of $ \part_T = \eps^{-1} \part_t $ and $ \part_z = \eps \part_x $ cost respectively $ \eps^{-2} $ and $ \eps^{-1} $.

\begin{proof} Consider the equation (\ref{eq:Picard0}) with $ F_{\! L} \equiv F $ as in (\ref{followformFgeninimodem}).
Since $ p(\eps D_x) $ is a pseudo-differential operator with constant coefficients, it does commute with the derivative 
$ \part_t^j \part_x^n $. Thus, through usual $ L^2 $-energy estimates, we can infer that
\[
  \part_t \| \part_t^j \part_x^n u^{(0)} \|^2_{L^2(\RR)} \lesssim \eps^{3/2} \! \! \sum_{\vert m\vert \leq M} \! \! 
\| \part_t^j \part_x^n u^{(0)} \|_{L^2(\RR)} \| \part_t^j \part_x^n
A_m^* e^{i m\varphi/\eps}  
\|_{L^2(\RR)} .
\]
Remark that 
\[ \part_t^j \part_x^n A_m^* = \sum_{i_n=0}^n \sum_{i_j=0}^j 
\left( \begin{array}{c}
\! \! n \! \! \\
\! \! i_n \! \!
\end{array} \right) \left( \begin{array}{c}
\! \! j \! \! \\
\! \! i_j \! \!
\end{array} \right) (\part_t^{i_j} \part^{i_n}_x
A_m)^* \part_t^{j-i_j} \part_x^{n-i_n} .\] 
By assumption, the symbol $ \part_t^{i_j} \part^{i_n}_x A_m(\eps t,t,,\cdot) $ is smooth, and its derivatives with 
respect to $ x $ and $ \xi $ are uniformly bounded in $ \eps $. From the Calder\'on--Vaillancourt Theorem, we know that 
$ (\part_t^{i_j} \part^{i_n}_x A_m)^*$ acts continuously on $ L^2(\RR) $. On the other hand, the function $ a_m$, and 
therefore $ \part_t^{i_j} \part^{i_n}_x A_m $, is spatially supported in the ball $ \vert x \vert < r $. Thus, we can replace 
$ e^{i m \varphi/\eps} $ by the $ L^2 $-function $ \chi ( \vert x \vert / \tilde r) e^{i m\varphi/\eps} $ where $ \tilde r := 8r/ 5 $. 
And thereby, we have to estimate a sum of terms similar to 
\[ \| \part_t^{j-i_j} \part_x^{n-i_n} \bigl( \chi (  x / \tilde r) e^{i m\varphi/\eps} \bigr) \|_{L^2(\RR)} .
\]
The derivatives $ \part_t $ and $ \part_x $, when they are applied to the oscillation $ e^{i m\varphi/\eps} $, produce 
respectively the singular factors $ \eps^{-1} \part_t \varphi \sim \eps^{-1} $ and $ \eps^{-1} \part_x \varphi \sim t \eps^{-1} $. 
The worst term arises when $ (i_j,i_n) = (0,0) $. As a consequence, we find that
\[ \part_t \| \part_t^j \part_x^n u^{(0)} \|^2_{L^2(\RR)} \lesssim \eps^{3/2} \eps^{-j} t^n \eps^{-n} \| \part_t^j \part_x^n u^{(0)} 
\|_{L^2(\RR)} , \]
and therefore, by Gr\"onwall's lemma, that
\begin{equation*} 
\| \part_t^j \part_x^n u^{(0)} (t,\cdot) \|_{L^2(\RR)} \lesssim t^{n+1} \eps^{(3/2)-j-n} . 
\end{equation*}
Then, in line with (\ref{L2majounideU0mn}), we get that
\begin{align*}  
\| \part_T^j \part_z^n \cU^{(0)} (T,\cdot) & \|_{L^2(\RR)} = \frac{1}{\eps}  \left\|\sum_{i_j=0}^j
 \begin{pmatrix}
    j\\
  i_j
  \end{pmatrix}
 \part_T^{i_j} \bigl \lbrack e^{-i T/ \eps^2} \bigr \rbrack  \times \part_T^{j-i_j} \part^n_z \bigl \lbrack u^{(0)} \(
  \frac{T}{\eps} , \eps \cdot \) \bigr \rbrack\right\|_{L^2(\RR)}  \\
  \ \lesssim& \, \eps^{-1+n-j} \sum_{i_j=0}^j
          \eps^{-i_j} \left\| (\part_t^{j-i_j} \part^n_x u^{(0)} ) \( \frac{T}{\eps} , \eps \cdot \) \right\|_{L^2(\RR)}   \\
  \ \lesssim& \, \eps^{-1+n-j} \sum_{i_j=0}^j
          \eps^{-i_j} \frac{1}{\sqrt \eps} \left\| (\part_t^{j-i_j} \part^n_x u^{(0)} ) \( \frac{T}{\eps} , \cdot \) \right\|_{L^2(\RR)}   \\
\ \lesssim & \, \eps^{-1+n-j} \sum_{i_j=0}^j \eps^{-i_j} \frac{1}{\sqrt \eps} \frac{T}{\eps} \eps^{3/2} \eps^{-(j-i_j)} \( \frac{T}{\eps} \)^n
\eps^{-n} = \cO (\eps^{-2j-n-1}) .
\end{align*}
This furnishes (\ref{L2majounideU0mn}). The sup norm control (\ref{LinftymajounideU0t}) is then a consequence of the standard
one-dimensional Gagliardo-Nirenberg  inequality
\begin{equation}\label{Gagliardo-Ni}
 \| \cV \|_{L^\infty (\RR)} \leq \sqrt 2\| \cV \|_{L^2 (\RR)}^{1/2} \| \part_z \cV \|_{L^2 (\RR)}^{1/2} .
\end{equation}

\vskip -4mm

\end{proof}

\noindent The interest of using $ \cU^{(0)} $ instead of $ u^{(0)}$ is twofold. First, as noted in (\ref{majounideU0}),  
the amplitude of the wave becomes of size one. Secondly, when passing from $ u^{(0)}$ to $ \cU^{(0)} $, the main 
temporal oscillations are locally filtered out in the following sense.

\begin{lem} [Local sup norm controls involving derivatives of $ \cU^{(0)} $] \label{lem:fine-estimates} 
\begin{align} 
\exists \, C >0, \quad \forall \, T \in [\cT , 2\cT] , \quad \forall
  |z|\le r/\eps, \quad & \vert \part_z \cU^{(0)} (T,z) \vert \leq  
C \eps^{-1} \, , \label{majounidederiofUdebut} \\
\exists \, C >0, \quad \forall \, T \in [\cT, 2\cT] , \quad \forall
  |z|\le r/\eps, \quad & \vert \part_T \cU^{(0)} (T,z) \vert \leq  
C \eps^{-2/3} \, , \label{majounidederiofU0} \\
\exists \, C >0,\quad \forall \, T \in [\cT , 2\cT] , \quad \forall |z|\le r/\eps, \quad & \vert \part^2_{Tz} \cU^{(0)} (T,z) \vert \leq 
C \eps^{-3/2} \, . \label{majounidederiofUfin}
\end{align}
\end{lem} 

\noindent Proposition~\ref{supnormamplifi} implies that the information (\ref{majounideU0}) is sharp. Starting from 
(\ref{majounideU0}), the two controls (\ref{majounidederiofU0}) and (\ref{majounidederiofUfin}), which both involve 
time derivatives, represent improvements in comparison to what would be provided by (\ref{couldbeexpected}). This 
means that the time oscillations contained in $ u^{(0)} $ have indeed been somewhat filtered out when passing from 
$ u^{(0)} $ to $ \cU^{(0)} $.

\begin{proof} Denote by  $u^{(j)}_m$ with $ j \in \{0,1\} $ the $m^{th} $ harmonic  of $u^{(j)}$. In particular, $ u^{(0)}_m $ can be 
obtained by solving
\begin{equation}\label{eq:u0mhar}
\part_t u^{(0)}_m - \frac{i}{\eps} p (- i  \eps \part_x) \,
u^{(0)}_m = \eps^{3/2}A_m ( \eps \, t,t,x,-i\eps\part_x)^*
\, e^{i \,m \varphi(t,x) / \eps} ,
\end{equation} 
with initial data $ {u^{(0)}_m}_{\mid t=0} \equiv 0 $. The situation is as in Paragraph \ref{solvescalar}, see (\ref{firstchoiceofF})
and (\ref{phaseOIFgen}), with
\begin{equation*} 
 \Phi (t,x;s,y,\xi) := (s - t ) \, p(\xi) + (x-y) \, \xi - m \varphi (s,y)  . 
\end{equation*}
In order to lighten the notations, we drop the dependence of the phase upon $m$. As in Paragraph \ref{sec:decomposition}, 
we can separate $ \Phi $ according to $ \Phi= \phi + m (\gamma -t) $ to deal with
\begin{equation*}
\phi (t,x;s,y,\xi) := (s - t ) \, \bigl \lbrack p(\xi) -m \bigr \rbrack + (x-y) \, \xi +m y \, s - m\gamma \, \cos s .  
\end{equation*}
The function $ u_m^{(0)} $ looks like $ u $ in (\ref{solosc}), that is
\begin{equation*}
 u_m^{(0)} (t,x) = \frac{\sqrt \eps}{2\pi} \, e^{i m (t-\gamma) / \eps} \! \int_0^t \! \! \iint e^{- i \phi(t,x;s,y,\xi) /\eps}  
\zeta_m(\xi)a_m (\eps  s , s, y ) ds dy d \xi .
\end{equation*}
Apply the derivative $ \part_x $ to the above relation. This introduces a factor $ \xi/\eps $ in the integral. Since 
$ \xi $ is like $ \lesssim \eps^{-1} $ at the critical points, a rough estimate furnishes
\begin{equation} \label{boundprior} 
 \left|\part_x u_m^{(0)} (t,x)  \right|\lesssim \sqrt \eps \eps^{3/2} \eps^{-2} \eps^{-1} = \eps^{-1} , 
 \end{equation}
where the products of powers of $ \eps $ follows, one after another, from the amplitude, the stationary phase approximation 
(in dimension $ 3$), the term $\xi/\eps \sim \eps^{-2} $, and the number $ \eps^{-1} $ of critical points. The bound (\ref{boundprior}) 
is equivalent to (\ref{majounidederiofUdebut}).

\smallskip

\noindent From (\ref{eq:u0mhar}), we can deduce that
\[ \part_t \bigl( e^{-it/\eps} u^{(0)}_m  \bigr) = \eps^{3/2} A^*_m e^{i(m \varphi -t)/\eps} + \frac{i}{\eps} e^{-it/\eps} 
\( p (- i \, \eps \, \part_x) -1 \) u^{(0)}_m .\]
Then, from the above integral representation of $u^{(0)}_m$, we get
\begin{align}
& \part_t \bigl( e^{-it/\eps} u^{(0)}_m (t,x) \bigr) = \eps^{3/2} A^*_m e^{i(m \varphi -t)/\eps} + \frac{i}{2\pi\sqrt\eps} \, 
e^{i (mt-t-m \gamma)/\eps} \label{areutiliser} \\
& \qquad \qquad \quad \times \underbrace{\int_0^t \! \! \iint e^{-i\phi(t,x;s,y,\xi)/\eps} \(p(\xi)-1\)\zeta_m(\xi)a_m 
(\eps s,s,y)dsdyd\xi}_{=: \cP_m (t,x)} . \nonumber
\end{align}
Coming back to (\ref{rightquantization}) where $ u $ is replaced adequately, the first term in \eqref{areutiliser} can be expressed as
\[
  \vert A^*_m e^{i(m \varphi -t)/\eps} \vert = \frac{1}{\eps} \Bigl \vert \iint e^{i \, (x \xi - y \xi -myt)  /\eps}\zeta_m(\xi) a_m (\eps\, 
  t,t,y) dy d\xi \, \Bigr \vert .
  \]
The phase involved has only one critical point $ (y,\xi) = (x,-mt) $ which is non degenerate. Since the dimension is two, this
allows to gain the factor $ \eps $ so that
\begin{equation}\label{cen'estpasfait}
\eps^{3/2} A_m (\eps\, t,t,x,- i \eps\part_x)^* e^{i(m
  \varphi(t,x)-t)/\eps} = \cO ( \eps^{3/2} ) .  
\end{equation}
Let us now consider the expression $ \cP_m $ emphasized in (\ref{areutiliser}). In view of Lemma~\ref{lem-nonresonantharm} 
and Lemma~\ref{ex-zerophase}, and since $\zeta_0=0$ near the origin (Assumption~\ref{choiopro}),  Proposition~\ref{prop-vanoscint}
implies that $ \cP_m =\cO(\eps^\infty) $ when $ m \not = 1 $, so that
\begin{equation}\label{eq:d_t ugen}
\forall m \not = 1 , \quad \part_t \bigl( e^{-it/\eps} u^{(0)}_m(t,x) \bigr) = \cO\bigl( \eps^{3/2} \bigr)  ,\quad |x|\le r,\quad
 \frac{\cT}{\eps} \le t \le \frac{2\cT}{\eps}  .
\end{equation}
We thus focus on the resonant case $m=1$. We decompose $ \cP_1 $ as we did concerning $ v $ at the 
level of (\ref{sizeofk}) and (\ref{defvrvk}), to get
\begin{equation} \label{sizeofkfortaum} 
\ \cP_1 (t,x) = \sum_{k \in \cK} \, \cP_{1,k} (t,x), \qquad \cK = \Bigl \lbrace \, k \in \NN \, ; \, k \leq \frac{2}{3} + \frac{\cT}{\pi 
\, \eps} \, \Bigr \rbrace \, .
\end{equation}
The analysis of Section~\ref{startregime} readily shows that we can find a constant $c>0$ such that (recall that here, $q=2$)
\begin{equation}\label{esticleconcerfpart}
\sum_{0 \leq k \leq c \eps^{-1/3}} \, | \cP_{1,k}(t,x) | \lesssim \sum_{0 \le k \leq c \eps^{-1/3}} \eps^{D-1} \lesssim \eps^{D-4/3}  .
\end{equation}
For $ k \in \cK $ with $ c \eps^{-1/3} \leq k $, as in Section~\ref{transiregime}, we can rely on a stationary phase argument. 
The only difference in the treatment of $\cP_1$ compared to the preceding analysis of $u^{(0)}$ is the presence of the factor 
$ p(\xi)-1 $ in the integral. Remark that
\[
  \vert p(k\pi+\xi)-1 \vert=
  \cO (k^{-2}) ,\quad \text{uniformly in }|\xi|\le 1.
\]
This property allows to improve the convergence of the sum of wave packets. As a matter of fact, resuming the stationary 
phase argument in $ (s,y,\xi) $ and relying on the decay in $k$ which is provided by the factor to gain the convergence of 
the series in $k$, we come up with
\begin{equation}\label{esticleconcerrpart}
\sum_{c \eps^{-1/3} \leq  k \in \cK} \, | \cP_{1,k}(t,x) | \lesssim \eps^{3/2} \sum_{c \eps^{-1/3} \leq  k \in \cK} \, \frac{1}{k^2} 
\lesssim \eps^{3/2} \eps^{1/3} = \eps^{11/6}  . 
\end{equation}
As prescribed at the level of (\ref{areutiliser}), divide (\ref{esticleconcerfpart}) and (\ref{esticleconcerrpart}) by $ \sqrt \eps $. Since 
$ D \geq 4 $, we can retain that
\begin{equation}\label{eq:d_t u}
 \part_t \bigl( e^{-it/\eps} u^{(0)}_1(t,x) \bigr) = \cO \bigl( \eps^{4/3} \bigr)  ,\quad |x|\le r ,\quad 0\le t \le \frac{2\cT}{\eps}.
\end{equation}
Compute
\begin{equation*}
\part_T \cU^{(0)} (T,z) = \frac{1}{\eps^2}\part_t \(e^{-it/\eps} u^{(0)} \(t,\eps z\)\)\Big|_{t = T/\eps} .
\end{equation*}
To estimate $ \part_T \cU^{(0)} $, it suffices to multiply (\ref{eq:d_t ugen}) and (\ref{eq:d_t u}) by $ \eps^{-2} $, and to sum 
on the finite number of integers $ m $ satisfying $ \vert m \vert \leq M $. This yields (\ref{majounidederiofU0}).

\smallskip

\noindent Now, in order to get \eqref{majounidederiofUfin}, just take the derivative of (\ref{areutiliser}) with respect to 
$ x $. Since $ \part_x \varphi/\eps \sim t/\eps \sim 1/\eps^{2} $, 
we have
\begin{equation*}
\eps^{3/2} \part_x \( A_m (\eps\, t,t,x,- i \eps\part_x)^* e^{i(m \varphi(t,x)-t)/\eps} \) = \cO ( \eps^{-1/2} ) . 
\end{equation*}
This term turns out to bring the largest contribution. As a matter of fact, at the level of the oscillating integral in the 
second line of (\ref{areutiliser}), the $x$-derivative produces the factor $ \part_x \phi/\eps= \xi/\eps $. The multiplication 
by $ \xi $ is compensated by the decreasing of $ p -1 $. Due to the assumptions on $ p  $, the symbol 
$ \xi \bigl( p(\xi) -1 \bigr) \zeta_m (\xi) $ remains in a convenient class. We can still apply Lemma \ref{localizations}, 
except that the control inside (\ref{esticleconcerfpart}) must be replaced by some $ \cO (\eps^{D -7/3})$. On the other 
hand, at the critical points, $\xi/\eps$ behaves like $ k/\eps $, so the estimate (\ref{esticleconcerrpart}) becomes 
\[
  \eps^{3/2} \sum_{c \eps^{-1/3} \leq  k \in \cK} \, \frac{1}{k^2} \frac{k}{\eps} = \eps^{1/2} \sum_{c \eps^{-1/3} \leq  k \in \cK} \, 
  \frac{1}{k} \lesssim \sqrt\eps \ln \frac{1}{\eps}  .
  \]
  Gathering the above three estimates and recalling that $ \part_z =  \eps \part_x $, we can easily infer the content 
  of \eqref{majounidederiofUfin}.
  \end{proof}
  
\noindent So far, we have not exploited the fact that (\ref{majounideU0}) is achieved on a set of zero Lebesgue measure.
This property is useful for what follows.

\begin{lem} [Local vanishing properties] \label{vanishingprop} Let $ (m,n) \in \NN^* \times \NN $. Given a function $ w  $
in the Schwartz space $ \cS (\RR) $, define
$$ \cJ_\eps \equiv \cJ_\eps(T) \equiv \cJ_\eps(m,n,w;T) := \int \vert w (y) \vert  \vert\cU^{(0)} (T,y) \vert^m \vert \eps^{2/3} \part_T 
\cU^{(0)} (T,y) \vert^n dy . $$
Then, for all $ T \in [\cT , 2\cT] $, we have $ \cJ_\eps = o(1) $.
\end{lem} 

\begin{proof} Decompose $ \cJ_\eps $ into $ \cJ_\eps = \cJ^+_\eps + \cJ^-_\eps $ with
$$ \cJ^\pm_\eps := \int_{ \pm \eps \vert y \vert \leq \pm r} \vert w (y) \vert  \vert\cU^{(0)} (T,y) \vert^m  \vert \eps^{2/3} \part_T 
\cU^{(0)} (T,y) \vert^n  dy . $$
Exploit the global estimate (\ref{LinftymajounideU0t}) and the decreasing of $ w \in \cS (\RR) $ to obtain
\begin{align} 
\vert \cJ^-_\eps \vert & \leq C \eps^{- (3m/2)- (17 n/6)} \| w \|_{L^1 (r \leq \eps \vert y \vert)}= \cO(\eps^\infty ) . \nonumber
\end{align}
On the other hand, due to (\ref{majounideU0}) and (\ref{majounidederiofU0}), the family
$$ \mathbbm {1}_{[-r/\eps,r/ \eps]} (y) \vert w (y) \vert  \vert\cU^{(0)} (T,y) \vert^m \vert \eps^{2/3} \part_T \cU^{(0)} (T,y) \vert^n , 
\quad \eps \in \, ]0,1] $$
is uniformly bounded by $ C \vert w \vert \in L^1 (\RR) $. Applying Proposition \ref{supnormnonamplifi2}, it converges to 
zero out of the set $ 2 \ZZ $, which is of Lebesgue measure $ 0 $ in $ \RR $. Under such hypotheses, the Lebesgue dominated 
convergence theorem guarantees that $ \cJ^+_\eps = o (1) $.
\end{proof}


\subsubsection{General estimates involving $ \cW = \cU^{(1)}-\cU^{(0)} $} \label{subsec:roughestimadife} Depending on 
the choice of the parameters $ \nu $, $ j_1 $, $ j_2 $, $\omega $ and $ \iota $, the source term $ F_{\! N \! L}  $ can bring 
a contribution which is of the same size of $ \cU^{(0)} $, or not. To understand what happens, it is interesting to first investigate 
a situation implying no condition on $ \nu $, $ j_1 $, $ j_2 $, $\omega $, and no particular assumptions (through $ \iota $) on 
the spatial localization. To this end, we could directly exploit (\ref{LinftymajounideU0t}) with $  j = n =0 $ at the level of 
(\ref{exploitatthelevel1}) to obtain a preliminary sup norm control on $ \cW $. But, knowing (\ref{majounidederiofUdebut}), 
it is possible to improve this first bound.

\begin{lem} [Global sup norm control on $ \cU^{(1)} - \cU^{(0)} $] Fix any $ \iota \in \, ]-\infty,1] $. Then
\begin{equation}\label{roughdiffUcon}
\| (\cU^{(1)} - \cU^{(0)} ) (T,\cdot) \|_{L^\infty (\RR)} = \cO (\eps^{\nu - \frac{j_1}{2} -\frac{j_2}{2} - 2}) . 
\end{equation}
\end{lem} 

\begin{proof} Recall (\ref{Gagliardo-Ni}) and (\ref{Aktauv}) which allow to control $ B^\Lambda_\tau \cV $ in sup norm uniformly 
in $ \tau $ through the $ L^2 $-norms of $ \cV $ and $ \partial_z \cV $ as indicated below
\begin{align}
\| B^\Lambda_\tau \cV \|_{L^\infty (\RR)} & \lesssim \| B^\Lambda_\tau \cV \|_{L^2 (\RR)}^{1/2} \| \part_z (B^\Lambda_\tau \cV) \|_{L^2 (\RR)}^{1/2} 
\label{inclusioninin} \\
 & \lesssim \| \Lambda \|_{L^\infty (\RR)} \| \cV \|_{L^2 (\RR)}^{1/2} \|\part_z \cV \|_{L^2 (\RR) }^{1/2} .  \nonumber
\end{align}
On the other hand, using (\ref{L2majounideU0mn}) and (\ref{LinftymajounideU0t}) with $ j = 0 $ and $ n \in \{0,1\} $, we can infer that,
for all $ k \in \{0,1\} $, we have
\begin{align} 
& \displaystyle \| \part_y^k \bigl \lbrack \chi ( \eps^{1-\iota} r^{-1} \vert \cdot \vert ) \, \cU^{(0)}(s,\cdot)^{j_1} \, \bar 
\cU^{(0)}(s,\cdot)^{j_2} \bigr \rbrack \|_{L^2 (\RR)} \label{inferpourfkl2} \\
& \qquad \lesssim \| \cU^{(0)}(s,\cdot) \|^{j_1+j_2-1}_{L^\infty (\RR)}  \bigl( \| \cU^{(0)}(s,\cdot) \|_{L^2 (\RR)} + \| \part_y^k
\cU^{(0)}(s,\cdot) \|_{L^2 (\RR)} \bigr) \nonumber  \\
& \qquad \lesssim  (\eps^{-\frac{3}{2}})^{j_1+j_2-1} (\eps^{-1} + \eps^{-k-1} ) . \nonumber
\end{align}
Recall that $ \cW = \cU^{(1)} - \cU^{(0)} = \cW_l + \cW_{nl} $ with $ \cW_l $ and $ \cW_{nl} $ as in (\ref{partcDl}) and (\ref{partcDnl}).
The part $ \cW_l $ can be controlled according to
$$ \| \cW_l \|_{L^\infty (\RR)} \lesssim \eps^{\nu + j_1 + j_2-2}  \| \cG^\eps \|_{L^\infty (\RR)} . $$
From (\ref{eq:defcG}) and (\ref{LinftymajounideU0t}) with $ j=n=0 $, we can easily deduce (\ref{roughdiffUcon}). On the other hand,
combine (\ref{inclusioninin}) and (\ref{inferpourfkl2}) at the level of (\ref{partcDnl}) to get
$$ 
\begin{array} {rl}
\| \cW_{nl} \|_{L^\infty (\RR)} \lesssim \! \! \! & \eps^{\nu + j_1 + j_2-2} \, \| \cG^\eps \|^{1/2}_{L^2 (\RR)} \, \| \partial_z \cG^\eps 
\|^{1/2}_{L^2 (\RR)} \\
\lesssim \! \! \!  & \displaystyle \eps^{\nu + j_1 + j_2-2} \, (\eps^{-\frac{3}{4}})^{j_1+j_2-1} \eps^{- \frac{1}{2}} \, 
(\eps^{-\frac{3}{4}})^{j_1+j_2-1} \, \eps^{-1} ,
\end{array} 
 $$
which leads directly to (\ref{roughdiffUcon}).
\end{proof}

\noindent The preliminary estimate (\ref{roughdiffUcon}) is far from enough to reach some $ \cO(1) $ or less, under the sole 
condition $ \nu + j_1 + j_2 \geq 2 $. More specific arguments (involving especially the spatial localization) are needed to 
improve the above analysis. 


\subsection{Sorting of gauge parameters} \label{sec:sorting-gauge} The constructive interferences of Theorem \ref{theo:resumeNL}
occur on a set of Lebesgue measure zero. From this viewpoint, comparisons in $ L^p $-norms with $ p < + \infty $ cannot be relevant. 
We must stick to the use of the sup norm. This motivates the following definition, which is inspired by a notion of linearizability introduced in \cite{PG96}.
\begin{defi}[Linearizability during long times]\label{def:linearizable}
  We say that the nonlinearity plays no 
role at leading order during long times when
\begin{equation} \label{Linearizabilitydu} 
\sup_{0\le T\le 2\cT} \| ( \cU^{(1)} - \cU^{(0)})(T,\cdot)\|_{L^\infty(\RR)} =o(1) \ \text{ as } \ \eps \to 0.
\end{equation}
\end{defi}

\noindent In this subsection, we show that when $ \textfrak {g} \not = 1 $,  nonlinear effects are absent at leading order during long 
times (in the sense of Definition~\ref{def:linearizable}), provided that $ \iota = 1 $. When $\textfrak{g}\not \in \{0,1\} $, the 
assumption $\iota=1$ may be relaxed to $\iota\in [0,1]$. In the next paragraphs \ref{nonresonantgaugeparameters},  
\ref{strictlyresonantgaugeparameters} and \ref{transitionalgaugeparameters}, following the distinctions which have been 
made in Definition~\ref{classigaugedefi}, we examine successively the cases $ \textfrak {g} \not \in [0,1] $, $ \textfrak {g} 
\in \, ]0,1[ $, and $ \textfrak {g} = 0 $.

\begin{rem}
  The results of this  subsection,  Propositions~\ref{sortnonlinaer}, \ref{sortnonlinearstrictgauge} and 
  \ref{sortnonlinaergaugeenzero}, rely on the estimates \eqref{majounideU0}, \eqref{majounidederiofUdebut}, 
  \eqref{majounidederiofU0} and \eqref{majounidederiofUfin}, which have been established only for $T\ge
  \cT$, hence the time localizing factor $\chi(3-2\eps t/\cT)$  in $\cG^\eps$. We will see later that these estimates 
  could be  adapted for $T\ge \eta $ with $ \eta > 0 $. But the case of  smaller times $t\ll \eta /\eps$ is not straightforward. 
\end{rem}

\subsubsection{The case of non resonant gauge parameters.} \label{nonresonantgaugeparameters} This is when 
$ \textfrak {g} \not \in [0,1] $. Then, the distance from $ p(\xi) $ to $ \textfrak {g} $ remains bounded below by a positive 
constant. In other words, the function
\begin{equation*}
\begin{array} {rcl}
\Gamma \, : \, \RR & \! \!  \longrightarrow \!  \! & \RR \\
\xi & \! \! \longmapsto \! \!  & \Gamma (\xi ) :=  \bigl( p(\xi)- \textfrak {g} \bigr)^{-1} 
\end{array}
\end{equation*}
is bounded, that is $ \Gamma \in L^\infty (\RR) $.
 
\begin{prop}\label{sortnonlinaer} Assume that $ \textfrak {g} \not \in [0,1] $ and $ \iota \in [0,1] $. Then, the nonlinearity 
plays no role at leading order during long times. More precisely
\begin{equation}\label{abovethesamenr}
\forall \, T \in \lbrack 0 , 2 \cT \bigr \rbrack , \quad \| (\cU^{(1)} - \cU^{(0)}) (T,\cdot) \|_{L^\infty (\RR) } \, = 
\cO \bigl( \eps^{\iota/2+1/3} \bigr) . 
\end{equation}
\end{prop}

\begin{proof} It suffices to examine the critical size case, where $ \nu+j_1+j_2=2 $. In the case $ \nu+j_1+j_2>2 $, 
the above $\cO \( \eps^{\iota/2+1/3} \)$ is readily improved to $\cO \( \eps^{\iota/2+1/3+\nu+j_1+j_-2} \)$. The idea 
is to come back to Duhamel's formula \eqref{eq:duhamelNL}, and to exploit the oscillations occurring with respect 
to the time variable $ s $. Integrating by parts in $ s $, we find
\begin{align*}
  \cW (T,z) & =
  \frac{1}{2 \pi}
  \int_0^T \! \! \iint e^{-i(z-y)\xi +i\frac{T-s}{\eps^2}(p(\xi)-1)
                +i(\textfrak {g}-1)\frac{s}{\eps^2}} \cG^\eps(s,y)
                dsdyd\xi\\
& = \frac{i\eps^2}{2\pi} \iint e^{-i(z-y)\xi +i\frac{T-s}{\eps^2}(p(\xi)-1)
                +i(\textfrak {g}-1)\frac{s}{\eps^2}} \Gamma(\xi)\cG^\eps(s,y)
                dyd\xi\Big|_{s=0}^T\\
& - \frac{i\eps^2}{2\pi}\int_0^T \! \! \iint e^{-i(z-y)\xi +i\frac{T-s}{\eps^2}(p(\xi)-1)
                +i(\textfrak {g}-1)\frac{s}{\eps^2}} \Gamma(\xi)\part_s\cG^\eps(s,y) 
                dsdyd\xi.
\end{align*}
Given $ s \in [0,T] $, for $j\in\{0,1\}$, write the integral in $(y,\xi)$ in the more concise form
\begin{align*}
  \frac{1}{2\pi}\iint e^{-i(z-y)\xi +i\frac{\sigma}{\eps^2}(p(\xi)-1)
                } \Gamma(\xi)\part_s^j\cG^\eps(s,y)
                dyd\xi = \cF_{\xi} \bigl(e^{i\frac{\sigma}{\eps^2}(p(\xi)-1)
                }\Gamma(\xi) \cF^{-1}_{y} (\part_s^j\cG^\eps )\bigr).
\end{align*}
Set apart the weight
$$ \tilde\Gamma(\sigma,\xi) \equiv \tilde\Gamma_\eps (\sigma,\xi):= e^{i\frac{\sigma}{\eps^2}(p(\xi)-1)
                }\Gamma(\xi) \, , \qquad \sup_{\eps \in ]0,1]} \ \| \tilde\Gamma_\eps \|_{L^\infty (\RR^2) } 
                < + \infty . $$
To estimate such a term in $L^\infty$, we use the $L^2$-norms as intermediary norms, 
 so we lose as little information as possible at the
level of Fourier transforms, thanks to Plancherel identity. To do so,
we invoke Gagliardo--Nirenberg 
inequality,
\begin{align*}
  \left \| \cF_{\xi} \bigl(  \tilde\Gamma \cF^{-1}_{y}  ( \part_s^j \cG^\eps ) \bigr) \right\|_{L^\infty(\RR)}&\lesssim
 \left \| \cF_{\xi}\bigl(  \tilde\Gamma\cF^{-1}_{y}  (\part_s^j \cG^\eps ) \bigr)\right\|_{L^2}^{1/2} \, 
    \left\|\part_z \cF_{\xi} \bigl(  \tilde\Gamma\cF^{-1}_{y} (\part_s^j\cG^\eps ) \bigr) \right\|_{L^2}^{1/2}\\
 &\lesssim
 \left\|\part_s^j\cG^\eps\right\|_{L^2}^{1/2}
   \left\| \part_y\part_s^j\cG^\eps\right\|_{L^2}^{1/2}. 
\end{align*}
The assumption $\iota\ge 0$ is needed to later invoke the $L^\infty$ estimates of Lemma \ref{lem:fine-estimates}, 
concerning $\cU^{(0)}$. Below, to simplify notations, we can drop the exponent in $\cU^{(0)}$. For $j\in \{0,1 \} $, 
since $\iota\in [0,1]$, we can assert that
\begin{align*} 
\Bigl\| \, \part_y^j & \left[ \chi \( \frac{\cdot}{\eps^{\iota-1} r} \) \part_s ( \cU^{j_1}\bar\cU^{j_2}) \right] \, 
\Bigr\|_{L^2 (\RR)} \lesssim \\
\  & \quad \eps^{(1-\iota)/2} \|
    \cU(s)\|^{j_1+j_2-1}_{L^\infty(|z|\le r/\eps)}
    \| \part_s \cU(s) \|_{L^\infty(|z|\le r/\eps)} \\
\  &   + \eps^{(\iota-1)/2}
    \| \cU(s) \|^{j_1+j_2-2}_{L^\infty(|z|\le r/\eps)}
    \| \part^j_y \cU(s) \|_{L^\infty(|z|\le r/\eps)} \| 
\part_s \cU(s) \|_{L^\infty(|z|\le r/\eps)} \\
\  &  + \eps^{(\iota-1)/2}
    \| \cU(s) \|^{j_1+j_2-1}_{L^\infty(|z|\le r/\eps)}
    \| \part^{(1,j)}_{sy} \cU(s) \|_{L^\infty(|z|\le r/\eps)}  . 
\end{align*}
Since $ \iota \in [0,1] $, we can exploit the local sup norm estimates (\ref{majounideU0}), (\ref{majounidederiofUdebut}), (\ref{majounidederiofU0}) 
and (\ref{majounidederiofUfin}), so the above estimate yields
\begin{align*}
   &\left\| \part_y^j\left[\chi \( \frac{\cdot}{\eps^{\iota-1}
                 r}\)  \part_s \( \cU^{j_1}\bar\cU^{j_2}\) ) \right]
     \right\|_{L^2 (\RR)} \\
   &\qquad\lesssim \eps^{(1-\iota)/2} \eps^{-2/3} +
     \eps^{(\iota-1)/2} \eps^{-j} \eps^{-2/3} +
     \eps^{(\iota-1)/2} \eps^{-2/3-5 j/6} 
\lesssim\eps^{\iota/2} \eps^{-j} \eps^{-7/6} . 
\end{align*}
We  conclude that
\begin{equation} \label{recober}
\vert (\cU^{(1)} - \cU^{(0)} ) (T,z) \vert \lesssim \eps^2 \times
\eps^{\iota/4} \eps^{-7/12} \times \eps^{\iota/4}\eps^{-1/2} \eps^{-7/12} ,
\end{equation}
which is exactly (\ref{abovethesamenr}).
\end{proof}


\subsubsection{The case of pointwise resonant gauge parameters.} \label{strictlyresonantgaugeparameters} 
This is when $  \textfrak {g} \in \, ] 0,1[ $. In view of (\ref{increasing}), we can assert that 
\begin{equation} \label{assertexistxig}
\exists \, ! \, \xi_{\textfrak {g}} \in \, ]\xi_c,+\infty[ \, ; \qquad p(\xi_{\textfrak {g}}) = \textfrak {g} , \qquad 0 < 
p' (\xi_{\textfrak {g}}) .
\end{equation}

\begin{prop}\label{sortnonlinearstrictgauge} Assume that $ \textfrak {g} \in \, ] 0,1[ $ and that $ \iota \in [0,1] $. 
Then, the nonlinearity plays no role at leading order during long times. More precisely, for all $ \mu < 1/6 $, 
we have
\begin{equation}\label{abovethesamepointwise}
\forall \, T \in \lbrack 0 , 2 \cT \bigr \rbrack , \quad \| (\cU^{(1)} - \cU^{(0)} ) (T,\cdot) \|_{L^\infty(\RR) } \, = 
\cO (\eps^{3\iota/4+\mu}) . 
\end{equation}
\end{prop}

\noindent This furnishes again some $o(1)$ in line with Definition~\ref{def:linearizable}. When $ \iota = 0 $, 
this bound is weaker than (\ref{abovethesamenr}). 

\begin{proof} We can still work with $ \nu+j_1+j_2=2$. Fix $ \eta \in [0,1[ $. We perform a frequency localization 
of size $ \eps^\eta $ near the two problematic positions $ \pm \xi_{\textfrak {g}} $. In practice, we insert in the integral 
\eqref{eq:duhamelNL} defining the error $\cU^{(1)}-\cU^{(0)}$ the decomposition
\begin{equation}\label{fctindicatricesum}
 1 = (1-\chi) \(\frac{\xi^2 -\xi_{\textfrak {g}}^2}{\eps^{\eta}}\)+ \chi
 \(\frac{\xi^2 -\xi_{\textfrak {g}}^2}{\eps^{\eta}}\).  
\end{equation}
Concerning the left part of (\ref{fctindicatricesum}), that is away
from the values $ \xi = \pm \xi_{\textfrak {g}} $, the proof of 
Proposition~\ref{sortnonlinaer} can be repeated with $ \Gamma $ replaced by 
\[
  \Gamma_\eta (\xi) := \frac{1}{p(\xi)- \textfrak {g}}  (1-\chi)
  \(\frac{\xi^2 -\xi_{\textfrak {g}}^2}{\eps^{\eta}}\) .
  \]
By construction, the function $ \Gamma_\eta  $ is zero on some
set of size $ \eps^\eta $ containing $ \pm \xi_{\textfrak {g}}$.  
It follows from (\ref{increasing}) and \eqref{assertexistxig} that $
\Gamma_\eta  $ is globally 
bounded by $ C \eps^{-\eta} $. The integration by parts with respect
to  
the time variable $ s $ can still be performed, but now we have to
take into account this singular estimate for 
$ \|\Gamma_\eta\|_{L^\infty} $. As a consequence, the gain is $ \eps^{2-\eta} $ instead of $ \eps^2 $. The corresponding contribution is therefore of 
size $ \eps^{\frac{\iota}{2}-\eta+\frac{1}{3}} $ instead of being of size $ \eps^{\frac{\iota}{2}+ \frac{1}{3}} $ as in (\ref{abovethesamenr}). 
\smallskip

\noindent From now on, we fix  $ \eta \in
[0,\frac{\iota}{2}+\frac{1}{3}[$ (so the above estimate yields a small
contribution), and we study the contribution coming from the right part of (\ref{fctindicatricesum}). 
The idea is to exploit at the level of \eqref{eq:duhamelNL} the
oscillations with respect to $ \xi $. To this end, the identity
\eqref{eq:duhamelNL} may be reformulated as
\begin{equation} \label{eq:duhamelpintwise}
\begin{aligned}
\cW (T,z) = \frac{1}{2\pi} \int_0^T \! \! \int e^{ i (\textfrak {g} s-T)/\eps^2} \cJ(\eps,T-s,y,z) \cG^\eps (s,y) \, dsdy \, ,
  \end{aligned}
\end{equation}
where we have put aside the oscillatory integral
\begin{equation} \label{oscillaintI}
\cJ(\eps,s,y,z) := \int e^{i \psi (\xi)/ \eps^2} \chi
 \(\frac{\xi^2 -\xi_{\textfrak {g}}^2}{\eps^{\eta}}\) d \xi , 
\end{equation}
built with the phase $ \psi (\xi) \equiv \psi(\eps, s,y,z;\xi) $ given by
\begin{equation} \label{psidefdef}
 \psi (\xi) := s p(\xi) - \eps^2 (z-y) \xi , \qquad \psi' (\xi) = s p'
 (\xi) - \eps^2 (z-y) .  
\end{equation}
Due to the presence of $\chi$, we have the obvious estimate $ |\cJ|\lesssim \eps^\eta $.
Fix  $ \delta \in [\eta,1[ $. Since $ p' (\xi_{\textfrak {g}})>0 $,
for all time $ s \gtrsim \eps^\delta $, we have $ 
\psi'(\xi) \gtrsim \eps^\delta $ for all $ \xi $ at a distance $ \sim \eps^\eta $ from $ \xi_{\textfrak {g}} $, 
and for all $ y $ and $ z $ located at a distance 
less than $ \sim \eps^{-1} $ from the origin. For $ s \gtrsim
\eps^\delta $, an integration by parts with respect to $ \xi $ yields
\[\cJ = i \eps^2 \int e^{i \psi (\xi)/ \eps^2} \left[- \frac{\psi''(\xi)}{\psi'(\xi)^2} \chi
 \(\frac{\xi^2 -\xi_{\textfrak {g}}^2}{\eps^{\eta}}\) 
 + \frac{2 \xi}{\eps^{\eta} \psi'(\xi)} \chi'
 \(\frac{\xi^2 -\xi_{\textfrak {g}}^2}{\eps^{\eta}}\) \right] d \xi . \]
This indicates a gain of $ \eps^{2-2 \delta} $ when computing $
\cJ$. Since this operation may be repeated indefinitely,
we deduce that $ \cJ = \cO (\eps^\infty ) $ for $ s\gtrsim \eps^\delta $.

\smallskip

\noindent There remains to control the contribution which, in (\ref{eq:duhamelpintwise}), is brought by the times $ s $ 
satisfying $ T- C\eps^\delta\le s\le T $. For such $ s $, a rough
estimate based on (\ref{majounideU0}) yields, since $|\cJ|\lesssim
\eps^\eta$, 
some $ \cO(\eps^{\delta + \eta+\iota-1}) $ error term. By optimizing the smallness of $ \eps^{\frac{\iota}{2}-\eta+\frac{1}{3}} $ 
and $ \eps^{\delta + \eta+\iota-1} $ through the selection of $ \eta = \frac{2}{3}-\frac{\delta}{2}-\frac{\iota}{4} $, we get some 
$ \eps^{\frac{3 \iota}{4}- \frac{1}{3} 
+ \frac{\delta}{2}} $ estimate. Since $ \delta < 1 $ can be chosen
arbitrarily closed to $ 1 $, we obtain
(\ref{abovethesamepointwise}). 
\end{proof}


\subsubsection{The case of the transitional gauge parameter.} \label{transitionalgaugeparameters}
This is when $ \textfrak {g} = 0 $. In this case, for all $ \xi $ in the interval $ [-\xi_c, \xi_c] $, we
have $ p(\xi) = \textfrak {g} = 0 $.  
On the other hand, in the case $\xi_c>0$ (which we shall assume in this
paragraph), the transitional region near 
the extreme positions $ \pm\xi_c$ is much more degenerate  
than in Paragraph~\ref{strictlyresonantgaugeparameters}. The function
$ p $ is flat near $ \pm\xi_c $.  
Consequently, there is no way to exploit as before, in the vicinity of
$ \pm\xi_c $, the oscillations with respect  
to $ \xi $. Still, we can show the following result, by restricting
the order of the spatial localization.

\begin{prop}\label{sortnonlinaergaugeenzero} Assume that $ \textfrak {g} = 0 $ and that $ \iota = 1$. 
Then, the nonlinearity plays no role at leading order during long
times, in the sense of Definition~\ref{def:linearizable}.
\end{prop}

\begin{proof} Again, we can suppose that the order of magnitude of the
  nonlinearity is critical, $ \nu+j_1+j_2=2 $. We recall that we
  assume here $\xi_c>0$. If  $\xi_c=0$, the proof of
  Proposition~\ref{sortnonlinearstrictgauge} can be repeated. The argument of
  Paragraph~\ref{nonresonantgaugeparameters},  
that is an integration by parts with respect to the time variable, does apply for frequencies $ \xi $ located away from $ [-2\xi_c,2\xi_c] $. 
Then, it suffices to consider
\begin{equation} \label{eq:duhameltransi}
\begin{aligned}
\cW (T,z) = \frac{1}{2\pi} \int_0^T \! \! \int e^{- i T/\eps^2} \tilde \cJ(\eps,T-s,y,z) \cG^\eps(s,y) \, dsdy ,
  \end{aligned}
\end{equation}
where, with $ \psi  $ as in (\ref{psidefdef}), we have put aside 
\begin{equation} \label{oscillainttransi}
\tilde \cJ(\eps,s,y,z) := \int e^{i \psi (\xi)/ \eps^2} \chi \(\frac{\xi}{4\xi_c}\) d \xi = \cO(1) .
\end{equation}
By this way, using the uniform boundedness of $\cU^{(0)}$ and \eqref{majounideU0}, we find  
\[ \vert (\cU^{(1)}-\cU^{(0)})(T,z) \vert \lesssim \int \chi \(\frac{y}{r}\) \vert \cU^{(0)}(s,y) \vert dy . \]
From Lemma \ref{vanishingprop} with $ m = 1 $ and $ n = 0 $, the integral on the right hand side 
goes to $ 0 $ with $\eps$. This argument based on the Dominated Convergence Theorem breaks 
down when $0\le \iota<1$ because the localizing function $\chi(y/(\eps^{\iota-1}r))$ is no longer 
integrable \emph{uniformly in $\eps$}.  
\end{proof}

\noindent The condition $ \iota = 1 $ is quite restrictive because it requires a concentrated source. In fact, 
the difficulties raised by the value $ \textfrak {g} = 0 $ are somewhat artificial. They are induced by the localization
procedure of Paragraph \ref{redscalar}. At the level of (\ref{multipnparchi}), the symbol is multiplied by $ 1-\chi(\xi) $.
This operation does not correspond to a physical phenomenon but rather to a technical simplification.

 \smallskip
 
 \noindent In the next subsection, we examine the remaining situation $ \textfrak {g} =1 $, especially in 
 the interesting and representative framework of equation (\ref{eq:Picard1}).


\subsection{The completely resonant situation} \label{sec:nonlineareffectg=1} 

\noindent We now prove Theorem~\ref{theo:resumeNLbis}, along with some generalizations. When $ \textfrak {g} = 1 $,
there exists no $ \xi \in \RR $ such that $ p(\xi) = \textfrak {g} $. But, due to (\ref{resobis}), the quantity $ p(\xi) $ becomes 
arbitrarily close to the limiting value $ \textfrak {g} =1 $ when $ \vert \xi \vert $ goes to infinity. Since large values of $ \xi $ 
are addressed when dealing with (\ref{eq:duhamelNL}), the effects of this approximated resonance are enhanced in the 
actual context. The study of (\ref{eq:Picard1}) corresponds to the choice $ (j_1,j_2,\nu)=(2,0,0) $ with $\omega = -1$, so 
that indeed $ \textfrak {g} =1 $. The expression $ \cG^\eps $ of (\ref{eq:defcG}) reduces to
\begin{equation}\label{duhamelcassomipdef}
\cG^\eps (T,z) := \chi\(3-2\frac{T}{\cT}\) \chi\(\frac{\eps z}{r\eps^{\iota}}\) \cU^{(0)}(T,z)^2 .
 \end{equation}
And, in this context, Duhamel's formula (\ref{eq:duhamelNL}) simply reads
\begin{equation}\label{duhamelcassomip}
\cW \(T,z\)  = \operatorname{Op} (\cG) (T,z) ,
 \end{equation}
where we have introduced the integral operator
\begin{equation}\label{integrduhamelcassomip}
\operatorname{Op} (\cG) (T,z) := \frac{1}{2\pi} \int_0^T \! \! \int \! \! \int e^{-i(z-y)\xi + i \frac{T-s}{\eps^2}(p(\xi)-1) } \cG(s,y) \, dsdyd\xi . 
 \end{equation}
 To lighten the notations, we will often write $ u $ for $ u^{(0)} $ and $\cU$ for $\cU^{(0)}$. The above integral involves the 
 extended spatial cut-off $ \vert y \vert \leq r \eps^{\iota-1} $,
 through the introduction of $ \chi (\eps \cdot  / r\eps^{\iota} ) $ 
 inside $ \cG^\eps  $. Like in the previous subsection, we impose $
 \iota\in [ 0,1] $. This condition is needed because precise 
 information regarding $\cU(s,y)$ is available on condition that $|y|\le r \eps^{-1} $. Indeed, Section~\ref{sec:lineffect} 
 provides a description of the solution $u (t,x) $ to (\ref{eq:Picard0}) only for $|x|\le r$, that is only for $|y|\le r \eps^{-1} $. 
 
 \smallskip
 
 \noindent The formula (\ref{duhamelcassomipdef}) also involves, through the implementation of $\chi(3-2s/\cT)$, the  time cut-off 
 $ \cT \leq s \leq 2 \cT $. The choice of $ [\cT,2 \cT ] $ is inherited from \eqref{majounideU0}  and Lemma~\ref{lem:fine-estimates}. It is introduced for convenience.  It could be relaxed by expanding the time domain of integration to 
 any compact set inside $ ]0,\infty[$. For instance, it could be adapted to any interval of the form $ [\eta \cT, \eta^{-1} \cT ] $ 
 with $ \eta \in ]0,1] $. 
 
\smallbreak

\noindent When $ \textfrak {g} =1 $, the estimates of Lemma \ref{lem:fine-estimates} do not suffice to show the smallness of 
$ \cW$. And for good reason: the nonlinearity plays a role at leading
order, and modifies the asymptotic behavior by a 
nontrivial $ \cO(1) $-effect that is revealed by (\ref{desintertheoprinbis}). To see why, the idea is to use the fine description of 
the function $\cU \equiv \cU^{(0)} $ which is provided by Section~\ref{sec:lineffect}, and to inject it into \eqref{duhamelcassomip}. 
Thus, like in Section~\ref{sec:lineffect}, we impose $q\ge2$ and $D\ge 4$.  

\smallbreak

\noindent In Paragraph \ref{reparatoryw}, we explain how to exchange $ \cU^2 $ inside 
(\ref{duhamelcassomipdef})-(\ref{duhamelcassomip})-(\ref{duhamelcassomip}) with a more tractable expression made of a sum 
of wave packets, without changing the content of $ \cW $ modulo $ o(1) $. In Paragraph \ref{bilinear}, we simplify the content of
these wave packets, and we exploit their specific structure in order to replace the triple integral (\ref{integrduhamelcassomip}) by 
a simple integral in time. In the last Paragraph \ref{asymptoticanalysis}, we perform the asymptotic analysis, showing Theorem 
\ref{theo:resumeNLbis}.


\subsubsection{Reduction to a sum of oscillating waves}\label{reparatoryw} 

\noindent The strategy to analyze $ \cW $ when $ \textfrak {g} =1 $ is to approximate $ \cU $ by a sum of oscillating waves 
indexed by $ k \in \cK^c_s $. Looking at $ \cU^2 $, this yields a bilinear form indexed by $ (k_1, k_2) \in \cK^c_s  \times \cK^c_s  $. 
 Given $ \beta \in [0,1] $, define
\begin{equation}\label{defdispkiforbeta} 
\cD \! \cK^c_s (\beta) := \bigl \lbrace (k_1,k_2) \in \cK^c_s \times \cK^c_s \, ; \, c \, \eps^{- \beta} < k_1 + k_2 \bigr \rbrace . 
\end{equation}

\noindent In Section~\ref{transiregime}, we have seen that, for all $k\in \cK^{c}_s$, the function $\Phi_k (t,x;\cdot) $ 
has at most one critical point $ (s_k,y_k, \xi_k) $ which satisfies (\ref{cdtstapoint}), which is non-degenerate, and 
which is such that $ \xi_k = s_k $. Recall  that $ s_k $ and $ y_k $ depend smoothly on $ (t , x) $. Using the function 
$ s_k (t,x) $ issued from Lemma \ref{statioponi}, we define the auxiliary function
\renewcommand\arraystretch{1.3}
\begin{equation}\label{defipsiepskajout}
\begin{array}{rl}
\bm{\uppsi}_k (t,x) :=  \! \! \! \! & - x k \pi - x s_k (t,x) + (-1)^k  \gamma \cos s_k (t , x )  \\
\ & + \bigl \lbrack 1 - p \bigl( k \pi + s_k ( t , x ) \bigr) \bigr \rbrack 
\bigl( k \, \pi + s_k ( t , x )  - t \bigr) .
\end{array}
\end{equation}
   \renewcommand\arraystretch{1}

\smallskip

\noindent In the statement below, we eliminate from (\ref{duhamelcassomip}) a number of terms which seem difficult 
to identify precisely, but which are small enough. 

\begin{prop}[The difference $ \cW =\cU^{(1)} -
  \cU^{(0)} $ as a sum of interacting  
  terms]\label{reduceprop}
  Fix $ \iota \in [0,1] $, and $ \beta $ such that
\begin{equation}\label{fixbeta}
\frac{1}{q+1} \leq \beta < \frac{3+\iota}{5} \leq 1 .
\end{equation}
Then, the difference $ \cW = \cU^{(1)} - \cU^{(0)} $ is such that 
\begin{equation}\label{duhamelcassomipsum}
 \cW (T,z) = o(1) + \sum_{ (k_1,k_2) \in \cD \! \cK^c_s (\beta)} \cW^\eps_{k_1,k_2} (T,z) ,
  \end{equation}
where $ \cD \! \cK^c_s (\beta) $ is as in \eqref{defdispkiforbeta}, whereas
\begin{equation}\label{cWk1k2def}
\cW^\eps_{k_1,k_2} (T,z) := \eps^2 \, e^{-2i \gamma / \eps} \operatorname{Op}(\cG^\eps_{k_1,k_2}) (T,z) .
 \end{equation}
With $ \bm{\uppsi}_k  $ as in \eqref{defipsiepskajout}, the bilinear interaction term $ \cG^\eps_{k_1,k_2} $ 
of \eqref{cWk1k2def} is given by
\begin{equation}\label{lienentreGetH}
\quad \cG^\eps_{k_1,k_2} (s,y) := \chi\(3-2\frac{s}{\cT}\) \chi\(\frac{\eps y}{r\eps^{\iota}}\) e^{i (\bm{\uppsi}_{k_1} 
+ \bm{\uppsi}_{k_2}) (s/\eps, \eps y)/ \eps} \cB^\eps_{k_1,k_2} (s/\eps, \eps y) .
 \end{equation}
The functions $ \cB^\eps_{k_1,k_2} (t,x) $ are defined on $ [0, 2 \cT ] \times [-r,r] $. They take the form
\begin{equation} \label{taketheformBkk}
\cB^\eps_{k_1,k_2} := (2 \pi)^{-2} \cA_{k_1}^\eps \cA_{k_2}^\eps , \qquad \cA_k^\eps = \cA_{k,0}^\eps + \eps 
\cA_{k,1}^\eps , 
\end{equation}
with 
\begin{equation} \label{inparticc}
\begin{array}{rl}
\cA_{k,0}^\eps (t,x) := \! \! \! & (2\pi)^{3/2} \, e^{-i(-1)^k\frac{\pi}{4}} \, | \operatorname{det} S_k 
(t,x) |^{-1/2} \\
\ & \times \tilde a_k \bigl( \eps,s_k (t,x)  ,y_k (t,x)  , s_k (t,x)  \bigr) ,
\end{array}
\end{equation}
where the matrix $ S_k $ is defined at the beginning of Paragraph~\ref{cripointsnon-dege}, with 
$ \operatorname{det} S_k (t,x) $ as in \eqref{minunifdet}, whereas $ \tilde a_k  $ is as in 
\eqref{tildeak}. For all $  i \in \{ 0,1 \} $, we can find a positive constant $ C_i $ such that
\begin{equation} \label{uniformestpourt}
\sup_{\eps \in ]0,1]} \ \sup_{k \in \cK^c_s} \ \sup_{t \in  [0, 2 \cT]} \ \sup_{x \in [-r,r]} \ \vert \cA_{k,i}^\eps  (t,x) \vert \leq C_i .
\end{equation}
Moreover, for all $ \alpha \in \NN^2 $, we can find a constant $ C_\alpha >0 $ such that
 \begin{equation}\label{euniformboundedspder} 
\sup_{\eps \in ]0,1]} \ \sup_{(k_1,k_2) \in \cD \! \cK^c_s (\beta)} \ \sup_{t \in  [0, 2 \cT]} \ \sup_{x \in  
[-r,r]} \ \vert \part_{t,x}^\alpha \cB^\eps_{k_1,k_2} (t,x) \vert \leq C_\alpha .
\end{equation}
\end{prop}

\noindent The double sum inside \eqref{duhamelcassomipsum} is actually finite since $ k = \cO (\eps^{-1} ) $ 
when $ k \in  \cK_s^c $: it runs over at most $ \cO ( \eps^{-2} ) $ terms. It involves fewer terms when 
$ \beta $ becomes close to the upper bound $ (3+\iota)/5 $ which is important because, as will be seen later, 
other conditions will be needed on $ \beta $. In comparison to $ \cU^2 $, the advantage of working with an 
expression like $ \cG $ in (\ref{lienentreGetH}) is a clear separation between an ``explicit phase''  $ \bm{\uppsi} $ 
and, in view of (\ref{euniformboundedspder}), a ``generalized profile'' $ \cB $. Knowing $ \bm{\uppsi} $, this will 
allow us to compute more precisely the content of $ \cW $.

\begin{proof} To prepare the analysis of $\cW $, we have first to resume the stationary phase arguments playing a central 
role in Section~\ref{sec:lineffect}. As already explained, see Lemmas~\ref{lem-nonresonantharm} and \ref{ex-zerophase} 
together with Proposition~\ref{prop-vanoscint}, the harmonic $m = 1$ inside (\ref{eq:source}) is the only one which may 
contribute to $ u \equiv u^{(0)} $ or $ \cU \equiv \cU^{(0)} $ modulo $\cO (\eps^\infty )$. Thus, it suffices to deal with  
$ a\equiv a_1 $. In coherence with \eqref{defphikini}, we work with
\begin{align}
 & a_k (\eps, s,y,\xi) = \zeta(k\pi+ \xi) a (\eps \, k \, \pi + \eps \, s , k \, \pi + s,y)  , \nonumber \\
 & \tilde a_k (\eps, s,y,\xi) := a_k (\eps, s,y,\xi) \, \chi_{1/4} (s-\xi) \, \chi_{2\pi/3} (s) , \label{tildeak} \\
&\Phi_k (t,x;s,y,\xi) = (k \, \pi -t) \ p (k \, \pi + \xi) + s \ \bigl \lbrack p (k \, \pi + \xi) - 1 \bigr \rbrack  
+ (x-y)  \xi \nonumber \\
 & \qquad \qquad \qquad \quad \ + s \, y - (-1)^k \ \gamma \, \cos s  .  \nonumber
\end{align}
Taking into account (\ref{tildeak}), the wave packet $ w_k $ of \eqref{defidewk} becomes
\begin{equation}\label{reinterdewk} 
 w_k (t,x) := \iiint e^{- i \, \Phi_k (t,x;s,y,\xi) /\eps} \, \tilde a_k (\eps, s,y,\xi) \, ds dy d \xi .
\end{equation}
Back to $u$ through (\ref{remindudevenk}), (\ref{sizeofk}) and (\ref{combi1}), we find
\begin{equation} \label{bacckkk} 
  u(t,x) = \cO ( \eps^\infty) + \frac{\sqrt \eps}{2 \, \pi} \sum_{0\le k \le 2/3+\cT/(\pi\eps)}  \, \ e^{i \, (-\gamma+ k \, \pi - k \, 
  \pi \, x)/\eps} \ w_k(t,x) .
\end{equation} 
Recall the distinction \eqref{defdispki} between $ \cK^c_d $ and $ \cK^c_s $. Then, use (\ref{passutocU}) to translate 
(\ref{bacckkk}) in terms of $\cU$ according to
\begin{equation}\label{eq:decompUdudepart}
\quad  \begin{aligned}
    \cU(T,z) = &\frac{1}{\eps} e^{-iT/\eps^2} u\Bigl( \frac{T}{\eps},\eps z \Bigr) = \cO ( \eps^\infty) + \frac{1}{ 2 \pi \sqrt \eps} \, 
    e^{-iT/\eps^2} \\
    & \times \Bigl( \sum_{k \in \cK^c_d} + \sum_{k \in \cK^c_s}  \Bigr)  e^{i (-\gamma + 
    k \pi - k \pi \eps z)/\eps} w_k \Bigl( \frac{T}{\eps},\eps z \Bigr) .
      \end{aligned}
 \end{equation}
 We have proved in Section~\ref{startregime}, see Lemma~\ref{localizations}, that for $c>0$ sufficiently small and 
 for $\cT/\eps\le t\le 2\cT/\eps$, we have $w_k = \cO(\eps^{D-1})$ uniformly on  $\cK^c_d$. Since $ q \geq 2 $ and 
 $ D \geq 4 $, in (\ref{eq:decompUdudepart}), the sum on $ \cK^c_d $ accounts for  some
$$ \cO \bigl( \eps^{- \frac{1}{2} - \frac{1}{q+1} +D-1} \bigr) = \cO ( \eps^2 ) .$$
Rearranging the terms, we can retain that
 \begin{equation}\label{eq:decompUdudepartremains}
\qquad \qquad \cU(T,z) = \cO ( \eps^2) + \frac{e^{-i \gamma/\eps}}{ 2 \pi \sqrt \eps} \sum_{k \in \cK^c_s} e^{-iT/\eps^2} e^{i 
( k \pi - k \pi \eps z)/\eps} w_k \Bigl( \frac{T}{\eps},\eps z \Bigr) .
 \end{equation}
When $ k \in \cK^c_s $, the content of the $ w_k $'s can be more detailed through asymptotic expansions in powers of $ \eps $. 
In the absence of a critical point, by nonstationary 
phase arguments, we just find $ w_k = \cO (\eps^\infty ) $. Otherwise, we can apply Theorem~\ref{theo:phazstat} 
to (\ref{reinterdewk}) in space dimension $ n = 3 $, with variables $ (s,y,\xi) \in \RR^3 $. This time, by stationary 
phase arguments, there exist differential operators denoted by $ M_{2j}^k (s,y,\xi ; D_{s,y,\xi}) $, giving rise to functions
 \begin{equation}\label{formuleAkj}
\cA_{k,j}^\eps (t,x) := \bigl \lbrack M_{2j}^k(s,y,\xi,D_{s,y,\xi}) \tilde a_k (\eps,\cdot) \bigr \rbrack_{\mid (s,y,\xi) = 
(s_k,y_k,\xi_k)(t,x)}  
\end{equation}
such that 
$$ \Bigl \vert w_k(t,x) - \eps^{3/2} \sum_{j=0}^{N-1} \eps^j \cA^\eps_{k,j} (t,x) e^{-i\Phi_k (t,x;s_k (t,x) ,y_k (t,x) , 
\xi_k (t,x) ) /\eps} \Bigr \vert = \cO ( \eps^{3/2+N} ) . $$
The expressions $ \cA_{k,j}^\eps  $ depend smoothly on $ \eps \in [0,1] $ through $ \tilde a_k (\eps,\cdot) $. 
In fact, taking into account (\ref{defskykxik}), they can be viewed as smooth functions of $ \eps \in [0,1] $ and $ s_k $. 
Then, the smoothness of $ \tilde a_k   $ and $ s_k   $  is communicated to $ \cA_{k,j}^{\eps}  $. 
The expressions $ \cA_{k,0}^\eps $ and $ \cA_{k,1}^\eps $ of (\ref{taketheformBkk}) are defined by (\ref{formuleAkj}).
In particular, combining (\ref{signSk}) and (\ref{stapourA0}), we find (\ref{inparticc}).

\smallskip

\noindent  As a result of Lemma \ref{controlofphhsggde}, the coefficients of the differential operator $ M_{2j}^k $ 
are uniformly bounded with respect to $ k\in \cK^{c}_s $. Due to Assumption \ref{choiopro}, the same applies to 
all derivatives of $ \tilde a_k  $ and to the preceding $ \cO ( \eps^{3/2+N} ) $. Applying Lemma~\ref{Propertiesofsk}, 
we can assert that the quantities $ \part^\alpha_{t,x} \cA_{k,j}^{\eps} (t,x) $ are, for all $ \alpha \in \NN^2 $,  uniformly 
bounded with respect to $\eps \in [0,1] $, $ k \in \cK_s^c $, $ t \in [0,2 \cT ] $ and $ x \in [-r,r] $. And therefore the 
functions $ \cA_{k,j}^{\eps}  $ can be viewed as ``generalized profiles'' satisfying the condition (\ref{euniformboundedspder}), 
which defines an algebra. As a consequence, coming back to the definition (\ref{taketheformBkk}) of $ \cB^\eps_{k_1,k_2} $, 
we have (\ref{euniformboundedspder}). 

\smallskip

\noindent  Plug the above expansion of $ w_k $ with $ N = 2 $ inside (\ref{eq:decompUdudepartremains}). The different 
phases combine to produce in coherence with the definition (\ref{defipsiepskajout}) the new phase
$$ \bm{\uppsi}_k (t,x) =  - t+ k \pi - k  \pi  x - \Phi_k \bigl( t,x; s_k (t,x) , y_k (t,x) , \xi_k (t,x) \bigr) . $$
When $ N = 2 $, the remainder is of size
$$ \frac{1}{\sqrt \eps} \sum_{k \in \cK^c_s} \cO ( \eps^{3/2+2} ) = \eps^{-1/2} \eps^{-1} \cO ( \eps^{3/2+2} ) = \cO 
( \eps^2 ) . $$
Recalling the definition $\cA_k^\eps = \cA_{k,0}+\eps\cA_{k,1}$, there remains
 \begin{equation}\label{eq:decompUdudepartremainsorzero}
\qquad \qquad \cU(T,z) = \cO ( \eps^2 ) + \eps \, \frac{e^{-i \gamma/\eps}}{ 2 \pi} \sum_{k \in \cK^c_s} e^{i \bm{\uppsi}_k 
(T/\eps,\eps z) /\eps} \cA_k^\eps (T/\eps,\eps z) ,
 \end{equation}
where, by convention, we set $ \cA_k^\eps \equiv 0 $ when there is no critical point. In (\ref{eq:decompUdudepartremainsorzero}), 
the sum on $ \cK^c_s $ runs over $\cO(1/\eps)$ terms which are all uniformly of size $\cO(1)$. Since $\eps$ is in factor 
of the sum, this may furnish some $ \cO (1) $. Now, we can compute 
$$ \cU(T,z)^2 = \eps^2 \, \cR^\eps (T,z) + \eps^2 \, e^{-2 i\gamma/\eps} \! \! \sum_{k_1 \in \cK^c_s} \sum_{k_2 \in \cK^c_s} 
\cB_{k_1,k_2}^\eps (T/\eps,\eps z) e^{i ( \bm{\uppsi}_{k_1} + \bm{\uppsi}_{k_2} )  (T/\eps,\eps z) / \eps} , $$
where$ \cR^\eps = \cO (1) $ and  $\cB_{k_1,k_2} =
(2\pi)^{-2}\cA_{k_1}\cA_{k_2}$. Define  
$$ \cR \! \cG^\eps (T,z) := \chi\(3-2\frac{T}{\cT}\) \chi\(\frac{\eps z}{r\eps^{\iota}}\) \cR^\eps (T,z) . $$
Replace $ \cU^2 $ as given by the above representation inside (\ref{duhamelcassomipdef}) and (\ref{duhamelcassomip}).
This explains the origin of the bilinear interaction term $ \cG^\eps_{k_1,k_2} = \cO (1) $ of (\ref{lienentreGetH}). Also, with 
$ \cW^\eps_{k_1,k_2} $ as in (\ref{cWk1k2def}), this gives rise to
\begin{equation}\label{remainscUquadratic} 
\cW = \eps^2 \operatorname{Op} (\cR \! \cG^\eps) + \sum_{k_1 \in \cK^c_s} \sum_{k_2 \in \cK^c_s} \cW^\eps_{k_1,k_2}  . 
\end{equation}
To go further, we have now to evaluate the amplitude of $ \cW^\eps_{k_1,k_2} $. To this end, we interpret 
(\ref{cWk1k2def}) as we did with $ \cW $ at the level of (\ref{partcDl}) and (\ref{partcDnl}). In other words, we write
$ \cW^\eps_{k_1,k_2} = \cW^\eps_{k_1,k_2,l}  + \cW^\eps_{k_1,k_2,nl}  $ with
\begin{align}
\cW^\eps_{k_1,k_2,l} (T,z) & := \eps^2 \, (2 \pi) \, e^{-2i \gamma /\eps} \int_\cT^T \cG^\eps_{k_1,k_2}(s,z) \, ds = \cO 
(\eps^2 ) , \label{partck1k2Dl} \\
\cW^\eps_{k_1,k_2,nl} (T,z) & := \eps^2 \, e^{-2i \gamma / \eps} \int_\cT^T B_{(T-s)/\eps^2} \cG^\eps_{k_1,k_2}(s,z) \, ds . 
\label{partck1k2Dnl}
\end{align}
At the level of (\ref{partck1k2Dnl}), we perform integrations by parts in $ y $, based on the identity
$$  \Lambda(\xi) (1-\part^2_y)^{1/4} e^{i y \xi} =
e^{i y \xi} , \qquad\Lambda(\xi):= (1 + \vert \xi \vert^2)^{-1/4} = \cO (\vert \xi
\vert^{-1/2} ) , $$ 
to get
$$ \cW_{k_1,k_2,nl} (T,z) := \eps^2 \, e^{-2i \gamma / \eps} \int_\cT^T B^\Lambda_{(T-s)/\eps^2} 
\bigl( (1-\part^2_y)^{1/4} \cG^\eps_{k_1,k_2} \bigr) (s,z) \, ds, $$
where we recall that the operator $B^\Lambda_\tau$ is defined in \eqref{defBktau}. 
Coming back to (\ref{defipsiepskajout}), (\ref{lienentreGetH}) and (\ref{euniformboundedspder}), we see that a spatial 
derivative $ \part_y $ applied on $ \cG^\eps_{k_1,k_2} (s,y) $ induces some loss of size $ \cO (k_1+k_2) $, where the 
$ \cO $ is uniform in $ \eps $ whereas its argument $ k_1+k_2 $, which
comes from the plane wave part $ x k \pi $ inside $ \bm{\uppsi}_k ( 
\cdot) $, is not, since $k_1,k_2\in \cK_s^c$. By interpolation, 
$$ \bigl( (1-\part^2_y)^{1/4} \cG^\eps_{k_1,k_2} \bigr) (s,y) = \cO \bigl( (k_1+k_2)^{1/2} \bigr) .  $$ 
Then, from Corollary \ref{corollairefac} applied with the optimal choice $ \rho = 1/2 $, we obtain that
\begin{equation}\label{estimlinftywk1nl}
k_1 + k_2 \leq \eps^{-\beta} \quad \Longrightarrow \quad \cW_{k_1,k_2,nl} (T,z) = \cO \bigl( \eps^{2+ (\iota-1 - \beta)/2 } \bigr) .
\end{equation}
Concerning $ \cR \! \cG^\eps $, the same type of arguments (this time involving Corollary \ref{corollairefac} with simply $ \rho = 0 $)
yields a single contribution of the type
\begin{equation}\label{ragoutt} 
\eps^2 \operatorname{Op} (\cR \! \cG^\eps) = \cO(\eps^2) + \cO (\eps^{1 + (\iota/2)} ) = \cO (\eps^{1 + (\iota/2)} ) . 
\end{equation}
In (\ref{remainscUquadratic}), given $ \beta $ as in (\ref{fixbeta}), we split the double sum into
$$ \sum_{k_1 \in \cK^c_s} \sum_{k_2 \in \cK^c_s}  = \sum_{k_1 + k_2 \leq \eps^{-\beta}} + \sum_{(k_1,k_2) 
\in \cD \! \cK^c_s (\beta)}. $$
In the above right hand side, the first sum involves at most $ \cO (\eps^{-2\beta} ) $ terms. Thus, using successively 
(\ref{ragoutt}), (\ref{partck1k2Dl}) and (\ref{estimlinftywk1nl}), we obtain
\begin{equation}\label{remainscUquadratic2} 
\begin{array}{rl}
\cW (T,z) = \! \! \! & \cO \bigl( \eps^{1+ (\iota/2) } \bigr) + \eps^{-2 \beta} \cO \bigl( \eps^2 \bigr) + \eps^{-2 \beta} \cO \bigl( 
\eps^{2+ (\iota-1-\beta)/2) } \bigr) \\
\ & \displaystyle + \sum_{(k_1,k_2) \in \cD \! \cK^c_s (\beta)} \cW^\eps_{k_1,k_2} (T,z) . \end{array}
\end{equation}
Computing
$$ 2+ (\iota-1 - \beta)/2 -2\beta = (3 + \iota - 5 \beta)/2  , $$
we obtain a positive number provided \eqref{fixbeta} is satisfied, hence \eqref{duhamelcassomipsum}.
\end{proof}

\noindent At this stage, it is instructive to revisit the proof of Lemma \ref{lem:fine-estimates} on the basis of the 
representation (\ref{eq:decompUdudepartremainsorzero}) of $ \cU $. When computing $ \part_T \cU$ 
through (\ref{eq:decompUdudepartremainsorzero}), we can focus on the sum involving $ k \in \cK^c_s $, and neglect 
the $ \cO (\eps^2) $ term which is presumed to be negligible. Coming back to the definition (\ref{defipsiepskajout}) of 
$ \bm{\uppsi}_k $ and because $ \cA^\eps_k $ can be viewed as a smooth function of $ \eps \in \lbrack 0,1 \rbrack $ 
and $ s_k $, the most significant contributions are brought by 
\begin{align*}
\part_T \bigl \lbrack \bm{\uppsi}_k (T/\eps, \eps z)/\eps \bigr \rbrack & =  \eps^{-2} (\part_t \bm{\uppsi}_k) 
(T/\eps, \eps z) \\
 & \lesssim \eps^{-2} \bigl\lbrack \bigl( 1 + \cO(\eps^{-1} k^{-q-1} + k^{-q}) \bigr) \vert \part_t s_k \vert + 
\vert 1 - p ( k \pi + s_k ) \vert \bigr \rbrack \\
 & \lesssim \cO (\eps^{-2}) \vert \part_t s_k \vert + \cO (\eps^{-2} k^{-q}) , \\
\part_T \bigl \lbrack \cA^\eps_k (T/\eps, \eps z)/\eps \bigr \rbrack  & \lesssim \eps^{-1} \vert \part_t s_k \vert .
\end{align*} 
Using (\ref{unigliformskftxpour10}), we can see that the right hand sides are $ \cO(\eps^{-2} k^{-q}) $. 
This means that the wave packets composing $ \cU$ inside (\ref{eq:decompUdudepartremainsorzero}) contain 
less and less time oscillations as $ k $ becomes large. This also implies that the derivative $ \part_T \cU$ appears 
as a sum of terms which may be controlled according to (with $ q = 2 $)
$$ \vert  \part_T \cU (T,z) \vert \lesssim \eps \sum_{k \in \cK^c_s} \eps^{-2} k^{-q} \lesssim \eps^{-1} 
( \eps^{-1/(q+1)} )^{-q+1} \lesssim \eps^{-1} \eps^{1/3} \lesssim \eps^{-2/3} . $$
This corresponds exactly to (\ref{majounidederiofU0}), which should be therefore optimal. In the same way, starting 
from (\ref{defipsiepskajout}) and exploiting (\ref{unigliformskftx}) with $ \alpha = (0,1) $, we find that 
$$ \part_z \bigl \lbrack \bm{\uppsi}_k (T/\eps, \eps z)/\eps \bigr \rbrack = \cO (k) . $$
This yields
$$ \vert  \part_z \cU (T,z) \vert \lesssim \eps \sum_{k \in \cK^c_s} k \lesssim \eps^{-1} , $$
which indicates that (\ref{majounidederiofUdebut}) could not be
improved any further.


\subsubsection{Analysis of the bilinear interaction term}\label{bilinear} We now come back to the 
study of the operator $\operatorname{Op}  $ given by (\ref{integrduhamelcassomip}). Since $ p(\xi) - 1 \sim 0 $ 
when $ \vert \xi \vert $ goes to $ + \infty $, for large values of $ \xi $, there is no way to gain some 
smallness in $ \eps $ by performing integrations by parts with respect to $ s $. Thus, the strategy 
is to fix $ s $, to integrate in $ (y,\xi) $, and to exploit the special form of $ \cG^\eps_{k_1,k_2}  $
in order to get a more tractable expression. The formula (\ref{lienentreGetH}) reveals the role of the phase
$$ \mathfrak p_{k_1,k_2} (t,x) := \bm{\uppsi}_{k_1}  (t,x) +  \bm{\uppsi}_{k_2}  (t,x) . $$
On the other hand, since $ |x| = \eps |y| \lesssim \eps^\iota $ on the domain of integration, for $ \iota \in ]0,1] $, 
only small values of $ x $ are involved. This remark indicates that the impact of $ \mathfrak p_{k_1,k_2}$ 
should be mainly driven by its Taylor expansion near $ x = 0 $, which may be written 
\begin{equation}\label{taylorexpansion} 
\mathfrak p_{k_1,k_2} (t,x) = \mathfrak p^0_{k_1,k_2} (t) + x \, \mathfrak p^1_{k_1,k_2} (t) + x^2 \,
\mathfrak {r}_{k_1,k_2} (t,x) , 
\end{equation}
where $ \mathfrak {r}_{k_1,k_2}  $ is the smooth function which is issued from the Lagrange remainder
at second order.

\smallskip

\noindent On the other hand, we find
\begin{subequations}\label{secondoredrtaylor} 
\begin{eqnarray} 
\quad & & \displaystyle \mathfrak p^0_{k_1,k_2} (t) := \bm{\uppsi}_{k_1}  (t,0) +  \bm{\uppsi}_{k_2}  
(t,0) = \sum_{i=1}^2 \Bigl \lbrace (-1)^{k_i}  \gamma \cos s_{k_i} (t , 0) \label{secondoredrtaylor0}  \\
\quad & & \displaystyle \qquad \qquad \quad \quad + \bigl \lbrack 1 - p \bigl( {k_i} \pi + s_{k_i} 
( t , 0 ) \bigr) \bigr \rbrack \bigl( {k_i} \, \pi + s_{k_i} ( t , 0 )  - t \bigr) \Bigr \rbrace , \nonumber \\
\quad & & \displaystyle \mathfrak p^1_{k_1,k_2} (t) := \part_x \bm{\uppsi}_{k_1}  (t,0) + \part_x \bm{\uppsi}_{k_2}  
(t,0) = - (k_1+k_2) \pi \label{secondoredrtaylor1}  \\
\quad & & \displaystyle \qquad \qquad \quad \ + \sum_{i=1}^2 \Bigl \lbrace - s_{k_i} (t,0) - (-1)^{k_i}  
\gamma \part_x s_{k_i} (t , 0 ) \sin s_{k_i} (t , 0 ) \nonumber \\
\quad & & \displaystyle \qquad \qquad \quad \quad \ + \bigl \lbrack 1 - p \bigl( {k_i} \pi + s_{k_i} ( t , 0 ) 
\bigr) \bigr \rbrack \part_x s_{k_i} ( t , 0 )  \nonumber \\
\quad & & \displaystyle \qquad \qquad \quad \quad \ - p' \bigl( {k_i} \pi + s_{k_i} ( t , 0 ) \bigr) \part_x s_{k_i} 
( t , 0 ) \bigl( {k_i} \, \pi + s_{k_i} ( t , 0)  - t \bigr) \Bigr \rbrace .  \nonumber 
\end{eqnarray}
\end{subequations}

\noindent As stated below, both $ \mathfrak p^0_{k_1,k_2} (t) $ and $ \mathfrak p^1_{k_1,k_2} (t) $ are involved 
in the asymptotic behavior of $ \cW^\eps_{k_1,k_2} $ at leading order.

\begin{prop}[Simplification of the bilinear interaction
  terms]\label{simplyfprop}
  Fix $ \iota \in ]\iota_-,1] $ with 
$ \iota_- := (13 - \sqrt{89}) / 8 < 1/2 $, and select $ \beta $ such that 
\renewcommand\arraystretch{1.8}
\begin{equation}\label{doubleineqagah}
\left \lbrace \begin{array} {lcl}
\displaystyle \frac{3}{2 (q+1)} \leq \frac{3 \iota + (3/\iota) (1-2 \iota)}{q+1} < \beta < \frac{3+\iota}{5} < 1 & 
\text{if} & \iota \in ]\iota_- , 1/2] \, , \\
\displaystyle \frac{3}{2 (q+1)} \leq \frac{1+ \iota}{q+1} < \beta < \frac{3+\iota}{5} < 1 & \text{if} & \iota \in [1/2,1] .
\end{array} \right.
\end{equation}
 \renewcommand\arraystretch{1}

\noindent Then, uniformly in $ (k_1,k_2) \in  \cD \! \cK^c_s (\beta) $, with $ \cA_{k,0}^\eps  $ as in \eqref{inparticc}, we have
\begin{equation}\label{Wepsk12entemps}
 \begin{aligned}
\cW^\eps_{k_1,k_2} (T,z) &= o(\eps^2)  \\
+  \eps^2 & (2 \pi)^{-2} e^{-2i \gamma / \eps} \int_0^T \chi\(3-2\frac{s}{\cT}\) 
e^{i \mathfrak p^0_{k_1,k_2} (s/\eps)/ \eps} \\
&\times e^{i z \mathfrak p^1_{k_1,k_2} (s/\eps)}  e^{ i (T-s) \lbrack p \circ \mathfrak p^1_{k_1,k_2} 
(s/\eps) -1 \rbrack / \eps^2} (\cA_{k_1,0}^\eps \cA_{k_2,0}^\eps) (s/\eps,0) \, ds .
 \end{aligned}
   \end{equation}
\end{prop}

\noindent Note that the reduction of (\ref{cWk1k2def})-(\ref{lienentreGetH}) to (\ref{Wepsk12entemps}) is quite 
striking. As a matter of fact, we got rid of the nonlocal aspect in $ dy d\xi $ that is involved by the operator $ \operatorname{Op} $ 
of (\ref{integrduhamelcassomip}). Indeed, knowing that $ \cG $ is as in (\ref{lienentreGetH}), the integral in $ dy d\xi $
reduces to a multiplication by the factor exhibited in (\ref{Wepsk12entemps}). This new presentation has many 
advantages. The action of $ \operatorname{Op} $ on $ L^\infty $ is not uniformly controlled in $ \eps \in ]0,1] $ with apparently, 
in view of Corollary \ref{corollairefac}, an optimal loss of the type $ \eps^{(\iota-1)/2} $ when $ \rho = 1/2 $.
Looking at (\ref{cWk1k2def}), we can only say that $ \cW^\eps_{k_1,k_2} = \cO (\eps^{(3+\iota)/2}  )$. But, looking 
at (\ref{Wepsk12entemps}), as a direct consequence of (\ref{uniformestpourt}), we can assert that
\begin{equation}\label{Wepsk12entempstotal}
\forall (k_1,k_2) \in  \cD \! \cK^c_s (\beta) , \qquad \cW^\eps_{k_1,k_2} (T,z) = \cO (\eps^2) .
\end{equation}

\begin{proof} The oscillating part inside (\ref{lienentreGetH}) can be decomposed according to
\renewcommand\arraystretch{2}
\begin{equation}\label{decompopoentrois} 
\begin{array}{rll}
e^{i (\bm{\uppsi}_{k_1} + \bm{\uppsi}_{k_2}) (s/\eps, \eps y)/ \eps} \! \! \! & = e^{i \mathfrak p^0_{k_1,k_2} (s/\eps) / \eps} &
\Bigl \rbrace \ \text{\scriptsize $ \circled{1} $} \\
\ & \times e^{i \mathfrak p^1_{k_1,k_2} (s/\eps) y} & \Bigl \rbrace \ \text{\scriptsize $ \circled{2} $} \\
\ & \times e^{i r^2 \eps^{2 \iota-1} \mathfrak {r}_{k_1,k_2} (s/\eps, \eps y) (r^{-1} \eps^{1-\iota} y)^2} & \Bigl \rbrace \ 
\text{\scriptsize $ \circled{3} $} 
\end{array} 
\end{equation}
\renewcommand\arraystretch{1}

\noindent The general idea of the proof is the following. At $ s $ fixed, the contribution {\scriptsize $ \circled{1} $} does not 
participate to the integration in $ (y,\xi) $, and hence appears as a simple factor at the level of (\ref{Wepsk12entemps}).
The oscillation {\scriptsize $ \circled{2} $} may be combined with $ e^{i y \xi} $ to produce after integration with respect to 
$ y $ a phase shift of size $ \mathfrak p^1_{k_1,k_2} (s/\eps) $ on the Fourier side, that is in $ \xi $. Knowing that 
$ \eps^{1-\iota} |y| \lesssim 1 $, the contribution {\scriptsize $ \circled{3} $} may be associated with the cut-off 
$ \chi (r^{-1} \eps^{1-\iota} y ) $ to produce, for $ \iota \geq 1/2 $ in the spatial variable $ \tilde y := r^{-1} \eps^{1-\iota} y $, a 
localized non oscillating term that may be viewed as  a profile
($\iota=1/2$), or a
some modulation ($\iota>1/2$). The other case $ \iota < 1/2 $ is
slightly different, 
more difficult, and will be considered separately. Finally, the integration in $ \xi $ operates as an inverse Fourier transform, 
which is reminiscent of a Dirac mass at $ \xi = \mathfrak p^1_{k_1,k_2} (s/\eps) $. 

\smallskip

\noindent Now, let us get into the specifics. Using Fubini's theorem together with (\ref{decompopoentrois}), we can interpret 
(\ref{integrduhamelcassomip}), (\ref{cWk1k2def}) and (\ref{lienentreGetH}) according to 
\begin{equation}\label{express561}
\begin{aligned}
& \cW^\eps_{k_1,k_2} (T,z) = \frac{\eps^2 \, e^{-2i \gamma /\eps}}{2\pi} \int_0^T \chi\(3-2\frac{s}{\cT}\) 
e^{i \mathfrak p^0_{k_1,k_2} (s/\eps) / \eps} \\
&\qquad  \times \bigg \lbrace \int e^{-i z \xi + i (T-s) (p(\xi)-1) / \eps^2 } \\
&\qquad \qquad \times \Bigl \lbrace \int 
e^{i y (\xi + \mathfrak p^1_{k_1,k_2} (s/\eps) )}
\cE^\eps_{k_1,k_2} (s/\eps, r^{-1} \eps^{1-\iota} y) \, d y \Bigr \rbrace \, d \xi \bigg \rbrace \, ds ,
 \end{aligned}
  \end{equation}
 where, taking $ \tilde y = r^{-1} \eps^{1-\iota} y $ as a new variable to work with a spatial localization of size one,
we have introduced
\begin{equation}\label{c'estceliuirce}
 \cE^\eps_{k_1,k_2} (t, \tilde y) := \chi (\tilde y) e^{i r^2 \eps^{2
     \iota -1} \mathfrak {r}_{k_1,k_2} (t, r  
\eps^\iota \tilde y) \tilde y^2} \cB_{k_1,k_2}^\eps (t,r \eps^\iota \tilde y) . 
\end{equation}
Obviously, we have
$$ \forall t  \ge 0 , \qquad \text{supp} \ \cE^\eps_{k_1,k_2} (t, \cdot) \subset \text{supp} \ \chi  \subset [-1,1] \, . $$
As long as $ \iota \in [1/2,1] $, the expression $ \cE^\eps_{k_1,k_2} (t,\cdot) $ does not oscillate in $ \tilde y $. Otherwise,
when $ \iota \in [0,1/2[ $, each derivative in $ \tilde y $ causes a loss of $ \eps^{2 \iota -1} $. On the other hand, the Lagrange
remainder $ \mathfrak {r}_{k_1,k_2} $ is built with second order derivatives of $ \bm{\uppsi}_k (t,\cdot) $ which, 
in view of (\ref{unigliformskftx}), may be uniformly bounded in $ (k_1,k_2) \in \cD \! \cK^c_s (\beta) $. We can combine 
this information with (\ref{euniformboundedspder}) to see that, for all $ j \in \NN $, we can find a constant $ C_j >0 $ 
such that 
$$ \sup_{(k_1,k_2) \in \cD \! \cK^c_s (\beta)} \ \sup_{t \ge 0} \ \sup_{\eps \in ]0,1]} \ \sup_{\tilde y \in \RR} \ \vert 
\part_{\tilde y}^j \cE^\eps_{k_1,k_2} (t, \tilde y) \vert \leq C_j \, \bigl( 1 + \eps^{(2\iota-1)j} \bigr) .  $$
By integration by parts in $ y $ and then interpolation, it follows
that for all $ \rho \ge 0 $ we can find a constant $ C_\rho 
>0 $ such that, uniformly in $ (k_1,k_2) \in \cD \! \cK^c_s (\beta) $, we have 
\begin{equation}\label{rendrerigoureux}
\quad \sup_{t \ge 0} \ \sup_{\eps \in ]0,1]} \ \sup_{\tilde \xi \in \RR} \ \vert ( 1 + \vert \tilde \xi \vert^2)^{\rho/2} 
\cF \bigl( \cE^\eps_{k_1,k_2}  (t, \cdot) \bigr) (\tilde \xi) \vert \leq C_\rho  \, \bigl( 1 + \eps^{(2\iota-1) \rho} \bigr) .  
\end{equation}
In (\ref{express561}), the integral with respect to $ y $ is exactly
\[ r \eps^{\iota-1} \cF \bigl( \cE^\eps_{k_1,k_2} (s/\eps, \cdot ) \bigr) (\tilde \xi) , \qquad \tilde \xi := - r \eps^{\iota-1} \bigl( \xi + 
\mathfrak p^1_{k_1,k_2} (s/\eps) \bigr) . \]
In the interest of simplifying notation, we sometimes note $ \mathfrak p^1 $ in place of $ \mathfrak p^1_{k_1,k_2} (s/\eps) $. 
Then, the  last two lines of (\ref{express561}) become
\[ \int e^{i z (\mathfrak p^1 + r^{-1} \eps^{1-\iota} \tilde \xi ) } e^{i (T-s) \lbrack p( - \mathfrak p^1 - r^{-1} \eps^{1-\iota} \tilde \xi  ) 
-1 \rbrack / \eps^2} \cF \bigl( \cE^\eps_{k_1,k_2} (s/\eps, \cdot ) \bigr) (\tilde \xi) \, d \tilde \xi = \, \text{\scriptsize $ \circled{4} $} + 
 \text{\scriptsize $ \circled{5} $} \, ,\]
with, since the function $ p $ is even
 \begin{align*}
\text{\scriptsize $ \circled{4} $}  & := e^{i z \mathfrak p^1} e^{i (T-s) ( p( - \mathfrak p^1) -1 ) / \eps^2}  \int e^{i z r^{-1} 
\eps^{1-\iota} \tilde \xi} \cF \bigl( \cE^\eps_{k_1,k_2} (s/\eps, \cdot ) \bigr) (\tilde \xi) \, d \tilde \xi ,  \\
\ & \ = e^{i z \mathfrak p^1} e^{i (T-s) ( p( - \mathfrak p^1) -1 ) / \eps^2} (2 \pi) \, \cF^{-1} \circ \cF \bigl( \cE^\eps_{k_1,k_2} 
(s/\eps, \cdot ) \bigr) ( \eps^{1-\iota}  r^{-1} z ) \\ 
\ & \ = e^{i z \mathfrak p^1} e^{i (T-s) ( p(\mathfrak p^1) -1 ) / \eps^2} (2 \pi) \, \cE^\eps_{k_1,k_2} (s/\eps, \eps^{1-\iota}  r^{-1} z ) .
\end{align*} 
In view of the definition (\ref{c'estceliuirce}) together with (\ref{euniformboundedspder}), we have 
\begin{align*}
\cE^\eps_{k_1,k_2} (t, \eps^{1-\iota}  r^{-1} z ) & = \chi (\eps^{1-\iota} r^{-1} z ) e^{i \eps z^2 \mathfrak {r}_{k_1,k_2} 
(t, \eps z)} \cB_{k_1,k_2}^\eps (t,\eps z) \\
 &  =  \chi (\eps^{1-\iota} r^{-1} z)  \bigl( 1 + \cO(\eps) \bigr) \bigl( \cB_{k_1,k_2}^\eps (t,0) + \cO(\eps) \bigr) \\
 &  = \chi (\eps^{1-\iota} r^{-1} z) \cB_{k_1,k_2}^\eps (t,0) + \cO(\eps) .
\end{align*} 
With (\ref{taketheformBkk}) and (\ref{uniformestpourt}), this becomes 
\[ \cE^\eps_{k_1,k_2} (t, \eps^{1-\iota}  r^{-1} z ) = \chi (\eps^{1-\iota} r^{-1} z) (2 \pi)^{-2} ( \cA_{k_1,0}^\eps \cA_{k_2,0}^\eps) 
(t,0) + \cO(\eps) .\]
Plug $ \text{\scriptsize $ \circled{4} $} $ with $ \cE^\eps_{k_1,k_2}  $ as above in place of the integral in $ dy d\xi $ 
(the two last lines) of (\ref{express561}). This furnishes the leading-order term of (\ref{Wepsk12entemps}). Thus, it remains 
to control the part $ \text{\scriptsize $ \circled{5} $} $, which is
\begin{align*}
\text{\scriptsize $ \circled{5} $} &  := \int e^{i z (\mathfrak p^1 + r^{-1} \eps^{1-\iota} \tilde \xi ) } 
\cF \bigl( \cE^\eps_{k_1,k_2} (s/\eps, \cdot ) \bigr) (\tilde \xi) \\
 & \qquad \ \, \times \bigl \lbrack e^{i (T-s) ( p( - \mathfrak p^1 - r^{-1} \eps^{1-\iota} \tilde \xi  ) -1 ) / \eps^2} - e^{i (T-s) ( p( - 
\mathfrak p^1) -1 ) / \eps^2}\bigr \rbrack  \, d \tilde \xi .
\end{align*} 
Since $ p $ is even, we have 
\begin{equation}\label{evenwehave}
\begin{aligned}
& \vert e^{i (T-s) ( p( - \mathfrak p^1 - r^{-1} \eps^{1-\iota} \tilde \xi  ) -1 ) / \eps^2} - e^{i (T-s) ( p( - \mathfrak p^1) -1 ) 
/ \eps^2} \vert \\
& \qquad \qquad \qquad \leq 2 \, \Bigl \vert \, \sin \Bigl( \frac{T-s}{2 \eps^2} \, \bigl( p(\mathfrak p^1 + r^{-1} \eps^{1-\iota} 
\tilde \xi  ) - p(\mathfrak p^1)  \bigr) \Bigr) \, \Bigr \vert \\
&  \qquad \qquad \qquad \leq  \frac{\vert T-s \vert }{\eps^2} \, \bigl \vert p(\mathfrak p^1 + r^{-1} \eps^{1-\iota} \tilde \xi  ) 
- p(\mathfrak p^1) \bigr \vert . 
\end{aligned}
\end{equation}
For the moment, we assume that $ 0 < \iota < 1 $. In the integral defining $ \text{\scriptsize $ \circled{5} $} \, $, we 
can distinguish a part where $ \eps^{1-\iota} \vert \tilde \xi \vert \leq 1 $ to take advantage of the smallness of the difference 
$ \vert p(\mathfrak p^1 + r^{-1} \eps^{1-\iota} \tilde \xi  ) - p(\mathfrak p^1) \vert $, and a part where $ 1 \leq \eps^{1-\iota} \vert 
\tilde \xi \vert $ to benefit from the rapid decreasing of $ \cF \bigl( \cE^\eps_{k_1,k_2} (s/\eps, \cdot ) \bigr)$. In other words, 
we use the fact that
\begin{equation}\label{evenwehavedecoup5}
\left \vert \, \text{\scriptsize $ \circled{5} $} \, \right \vert \leq \text{\scriptsize $ \circled{5} $}_-^{\, \iota} + \text{\scriptsize 
$ \circled{5} $}_+^{\, \iota}  ,
\end{equation}
with (for $ 0 < \iota < 1 $):
\[ \text{\scriptsize $ \circled{5} $}_\pm^{\, \iota} := \frac{\vert T-s \vert }{\eps^2} \int_{\pm \eps^{1-\iota} \vert \tilde \xi \vert \leq \pm 1}  \bigl 
\vert p(\mathfrak p^1 + r^{-1} \eps^{1-\iota} \tilde \xi  ) - p(\mathfrak p^1) \bigr \vert \, \vert \cF \bigl( \cE^\eps_{k_1,k_2} (s/\eps, \cdot ) 
\bigr) (\tilde \xi) \vert \, d \tilde \xi .\]
Since $ \vert p \vert $ is bounded by $ 1 $, exploiting \eqref{rendrerigoureux}, we find
\begin{align}
\text{\scriptsize $ \circled{5} $}_-^{\, \iota} & \leq \frac{\vert T-s \vert}{\eps^2} \int_{1 \leq \eps^{1-\iota} \vert \tilde \xi \vert} 2 \, C_\rho \, 
\bigl( 1 + \eps^{(2\iota-1) \rho} \bigr) \, (1 + \vert \tilde \xi \vert^2)^{-\rho/2} \, 
d \tilde \xi \nonumber \\
& \leq 2 \, C_\rho \, \frac{\vert T-s \vert }{\eps^2} \Bigl( \int_{1
                            \leq\eps^{1-\iota}
                             \vert \tilde \xi \vert} \vert \eps^{1-\iota} \tilde \xi 
\vert^{-\rho} d ( \eps^{1-\iota} \tilde  \xi ) \Bigr) \, \eps^{(1-\iota) (\rho-1)} \, \bigl( 1 + \eps^{(2\iota-1) \rho} \bigr) \nonumber \\
& \leq 2 \, C_\rho \, \vert T-s \vert \Bigl( \int_{1 \leq \vert \xi \vert} \vert \xi \vert^{-\rho} d \xi \Bigr) \, 
\eps^{(1-\iota) (\rho-1)-2} \, \bigl( 1 + \eps^{(2\iota-1) \rho} \bigr) . \nonumber 
\end{align}
From now on, we fix some $ \rho $ satisfying
\begin{equation}\label{fixsomegamma}
\left \lbrace \begin{array}{lcll}
\rho > -1 + 3/\iota & \geq 5 & \text{if} & \iota \in \, ]0,1/2] , \\
\rho > 1 + 2 /(1-\iota) \! \! \! \! \! & \geq 5  & \text{if} & \iota \in [1/2,1[ .
\end{array} \right. 
\end{equation}
By this way, we can recover $ \text{\scriptsize $ \circled{5} $}_-^{\, \iota} = o(1) $. 

\smallskip

\noindent To control  $ \text{\scriptsize $ \circled{5} $}_+^{\, \iota} $, we remark that
\begin{align*}
\text{\scriptsize $ \circled{5} $}_+^{\, \iota}  &  \leq \frac{\vert T-s \vert}{\eps^2} \ \Bigl( \sup_{\vert \xi \vert \leq r^{-1}} \, 
\vert p' ( \mathfrak p^1 + \xi ) \vert \Bigr) \int_{\vert \tilde \xi \vert \leq \eps^{\iota -1 }} r^{-1} \eps^{1-\iota} \vert \tilde \xi \vert \, 
C_\rho \frac{\bigl( 1 + \eps^{(2\iota-1) \rho} \bigr)}{(1 + \vert \tilde \xi \vert^2)^{+\rho/2}} \, d \tilde \xi \\
&  \leq \frac{C_\rho}{r} \, \vert T-s \vert \, \Bigl( \int (1 + \vert \tilde \xi \vert^2)^{(1-\rho)/2} \, d \tilde \xi \Bigr) \, 
\Bigl( \sup_{\vert \xi \vert \leq r^{-1}} \, \vert p' ( \mathfrak p^1 + \xi ) \vert \Bigr) \, \frac{1 + \eps^{(2\iota-1) \rho}}{\eps^{1+\iota}} \, . 
\end{align*} 
Since $ (1-\rho)/2 \leq -2 $, the above integral in $ \tilde \xi $ is finite. Now, we want to extract some additional smallness
from the sup term. To this end, we come back to the definition (\ref{secondoredrtaylor1}) of $ \mathfrak p^1 $. Knowing that 
$ (k_1,k_2) \in  \cD \! \cK^c_s (\beta) $, we get
\begin{align*} 
\vert \mathfrak p^1_{k_1,k_2} (t) \vert \geq c \pi \eps^{-\beta} - \sum_{i=1}^2 &\Bigl \lbrace \vert s_{k_i} 
(t,0) \vert + \gamma \vert \part_x s_{k_i} (t , 0 ) \vert  \\
&\quad+ \bigl \lbrack 1 - p \( {k_i} \pi + s_{k_i} ( t , 0 ) \) 
\bigr \rbrack \vert \part_x s_{k_i} ( t , 0 ) \vert  \\
&\quad+ \, 2 \cT \eps^{-1} \, p' \bigl( {k_i} \pi + s_{k_i} ( t , 0 ) \bigr) 
\vert \part_x s_{k_i} ( t , 0 ) \vert \Bigr \rbrace.  
\end{align*}
Since $ k_i \in \cK^c_s $, we can exploit (\ref{unigliformskftx}) and (\ref{hypp3ded}) to obtain
\begin{equation}\label{aaaaaaaaa}
\begin{aligned} 
\vert \mathfrak p^1_{k_1,k_2} (t) \vert  & \geq c \pi \eps^{-\beta} - 2 C_{(0,0)} -2 (\gamma + 1) C_{(0,1)}  - \sum_{i=1}^2 2 \cT \eps^{-1} \, \cO (k_i^{-q-1}) C_{(0,1)} \\
 & \geq c \pi \eps^{-\beta} - \cO (1) \gtrsim \eps^{-\beta} .
\end{aligned} 
 \end{equation}
Then, from (\ref{hypp3ded}), we get
\begin{equation}\label{minotauree}
\sup_{\vert \xi \vert \leq r^{-1}} \, \vert p' ( \mathfrak p^1 + \xi ) \vert \lesssim \eps^{\beta (q+1)}  .
\end{equation}
It follows that we can find $\rho$ satisfying \eqref{fixsomegamma} and
so that $ \text{\scriptsize $ \circled{5} $}_+^{\, \iota} = o(1) $ if
$$ \left \lbrace \begin{array}{lcl}
- 3 \iota + (3/ \iota) (2 \iota -1) + \beta (q+1) > 0 & \text{if} & \iota \in \, ]0,1/2] , \\
-1 - \iota + \beta (q+1) >0 & \text{if} & \iota \in [1/2,1[ .
\end{array} \right. $$
This provides with a lower bound for $ \beta $ which must be
compatible with (\ref{fixbeta}). When $ \iota \in [1/2,1[ $, we demand
\begin{equation*}
  -1-\iota +(q+1)\frac{3+\iota}{5}>0,
\end{equation*}
a condition which is always satisfied for $q\ge 2$ and $\iota\le 1$. 
When $ \iota \in ]0,1/2 ] $, we require
\begin{equation*}
  -3\iota+ (3/ \iota) (2 \iota -1) +(q+1)\frac{3+\iota}{5}>0,
\end{equation*}
a condition which boils down, in the less
favorable case $ q = 2 $, to
\begin{equation*}
  12\iota^2-39\iota+15<0.
\end{equation*}
Recalling that $\iota<1$, this condition implies that $ \iota_- < \iota $. This is where the specific value $\iota_-$ appears.

\smallskip

\noindent There remains to discuss the limiting case $ \iota = 1 $. The definition of $ \text{\scriptsize $ \circled{5} $}_-^{\, \iota} $
for $ 0 < \iota < 1 $ could be extended in the case $ \iota = 1 $. But the preceding argument does not work when $ \iota = 1 $, 
because there is no finite choice of $ \rho $ satisfying (\ref{fixsomegamma}). For this reason, we adopt the following alternative 
definition 
\[ \text{\scriptsize $ \circled{5} $}_\pm^{\, 1} := \frac{\vert T-s \vert }{\eps^2} \int_{\pm \vert \tilde \xi \vert \leq \pm \eps^{-\mu}}  \bigl 
\vert p(\mathfrak p^1 + r^{-1} \tilde \xi ) - p(\mathfrak p^1) \bigr \vert \, \vert \cF \bigl( \cE^\eps_{k_1,k_2} (s/\eps, \cdot ) \bigr) (\tilde \xi) 
\vert \, d \tilde \xi . \]
Use (\ref{rendrerigoureux}) with $ \rho > 1 + (2/\mu) > 3 $. When dealing with the sign $ - $, the above shift toward high frequencies 
$ \eps^{-\mu} \leq \vert \tilde \xi \vert $ allows to recover some smallness. Indeed:
\begin{align}
\text{\scriptsize $ \circled{5} $}_-^{\, 1} & \leq \frac{\vert T-s \vert}{\eps^2} \int_{\eps^{-\mu} \leq \vert \tilde \xi \vert} 2 \, C_\rho \, 
\bigl( 1 + \eps^{(2\iota-1) \rho} \bigr) \, (1 + \vert \tilde \xi \vert^2)^{-\rho/2} \, 
d \tilde \xi \nonumber \\
& \lesssim \frac{\vert T-s \vert }{\eps^2} \Bigl( \int_{1 \leq \vert \eps^\mu \tilde \xi \vert} \vert \eps^\mu \tilde \xi 
\vert^{-\rho} d ( \eps^\mu \tilde  \xi ) \Bigr) \, \eps^{\mu (\rho-1)} \nonumber \\
& \lesssim \vert T-s \vert \Bigl( \int_{1 \leq \vert \xi \vert} \vert \xi \vert^{-\rho} d \xi \Bigr) \, 
\eps^{\mu (\rho-1)-2} = o(1)  . \nonumber 
\end{align}
On the other hand, this does not affect the control of the part with the sign $ + $. Taking into account 
(\ref{aaaaaaaaa}), we find that
\begin{equation}\label{minotaureedepl}
\forall \, \vert \tilde \xi \vert \leq r^{-1} \eps^{-\mu} , \qquad \vert \mathfrak p^1 + r^{-1} \tilde \xi \vert \gtrsim \eps^{-\beta} ,
\end{equation}
and therefore, as before, we have
\begin{align*}
\text{\scriptsize $ \circled{5} $}_+^{\, 1} &  \leq \frac{\vert T-s \vert}{\eps^2} \ \Bigl( 
\sup_{\vert \xi \vert \leq r^{-1} \eps^{-\mu}} \, 
\vert p' ( \mathfrak p^1 + \xi ) \vert \Bigr) \int_{\vert \tilde \xi \vert \leq \eps^{-\mu }} r^{-1} \vert \tilde \xi \vert \, 
C_\rho \frac{2}{(1 + \vert \tilde \xi \vert^2)^{\rho/2}} \, d \tilde \xi \\
 &  \leq 2 C_\rho r^{-1} \, \vert T-s \vert \, \Bigl( \int (1 + \vert \tilde \xi \vert^2)^{(1-\rho)/2} \, 
d \tilde \xi \Bigr) \, \eps^{\beta (q+1)-2} ,
\end{align*} 
which is some $ o(1) $ for $ \beta $ as in (\ref{doubleineqagah}). Note that the above argument applied in the case $ 0 < \iota < 1$
would not improve (\ref{doubleineqagah}). As a matter of fact, the condition (\ref{doubleineqagah}) is issued from the analysis of 
$ \text{\scriptsize $ \circled{5} $}_+^{\, 1} $ which is not modified
by using (\ref{minotaureedepl}).

At this stage, we have proved
 \begin{align*}
\cW^\eps_{k_1,k_2} (T,z) &= o(\eps^2)  \\
+  \eps^2 & (2 \pi)^{-2} e^{-2i \gamma / \eps} \chi (\eps^{1-\iota} r^{-1} z) \int_0^T \chi\(3-2\frac{s}{\cT}\) 
e^{i \mathfrak p^0_{k_1,k_2} (s/\eps)/ \eps} \\
&\times e^{i z \mathfrak p^1_{k_1,k_2} (s/\eps)}  e^{ i (T-s) \lbrack p \circ \mathfrak p^1_{k_1,k_2} 
(s/\eps) -1 \rbrack / \eps^2} (\cA_{k_1,0}^\eps \cA_{k_2,0}^\eps) (s/\eps,0) \, ds .
 \end{align*}
  Since $ \iota >0$, we can substitute inside the term $ \chi (\eps^{1-\iota} 
r^{-1} z) $ with $ \chi (0) = 1 $, hence \eqref{Wepsk12entemps}. 
\end{proof}


\subsubsection{The asymptotic analysis}\label{asymptoticanalysis} From now on, we fix $ \beta $ as indicated 
in (\ref{doubleineqagah}). As a consequence of
(\ref{Wepsk12entempstotal}),
for $ j = 1 $ or $ j = 2 $,
we have
$$ \sum_{\substack{k_j \leq \eps^{-\beta} \\ (k_1,k_2) \in \cD \! \cK^c_s (\beta)}} \vert \cW^\eps_{k_1,k_2} (T,z) 
\vert \lesssim  \eps^{-\beta} \eps^{-1} \eps^2 = \cO (\eps^{1 - \beta} ) = o(1) . $$
Knowing this, we can replace (\ref{duhamelcassomipsum}) by 
\begin{equation}\label{duhamelcassomipsumred}
 \cW (T,z) = o(1) + \sum_{ \eps^{-\beta} \leq k_1 \in \cK^c_s} \ \sum_{ \eps^{-\beta} \leq k_2 \in \cK^c_s} 
 \cW^\eps_{k_1,k_2} (T,z) .
  \end{equation}
The final stage is to exploit the tools and the arguments of Section \ref{sec:lineffect} in order to pass to the limit 
at the level of (\ref{duhamelcassomipsumred}) when $ \eps $ goes to zero.

\begin{prop}[Proof of Theorem \ref{theo:resumeNLbis}]\label{passtothelimitWbeta}  Fix $ \iota \in ] \iota_-,1] $ 
and $ \beta $ as in \eqref{doubleineqagah}. Then, the limit of $ \cW (T,z) $ when $ \eps $ goes to zero is given 
by \eqref{desintertheoprinbis} when $ z = 2 j $ with $ j \in \ZZ $, and by \eqref{destrasyana} otherwise.
\end{prop}

\begin{proof} The starting point is (\ref{Wepsk12entemps}) together with (\ref{duhamelcassomipsumred}). The 
information $ \eps^{-\beta} \leq k_j $ inside (\ref{duhamelcassomipsumred}) is crucial because it allows to simplify
the content of $ \mathfrak p^0 $, $ \mathfrak p^1 $ and $ \cA^\eps_{k_j,0} $ at the level of (\ref{Wepsk12entemps}).
All these expressions depend on $ s_k $ and $ y_k $. But, knowing that $ \eps^{-\beta} \leq k $, from (\ref{defskxjhgezjbbkldepour1}) 
(\ref{defskxjhgezjbbkldepour2}) and (\ref{defskxjhgezjbbkltauk}), we have
\begin{equation}\label{ilfaiuut}
\qquad \quad s_k (t,0) = \cO(\eps^{(q+1)\beta-1}) ,\qquad y_k = \cO(\eps^{(q+1)\beta-1}) ,
\end{equation}
where, taking into account (\ref{doubleineqagah}), we are sure that 
\begin{equation}\label{ilfaiutlemettre} 
2 (q+1)\beta - 2 > 1 , \qquad q \beta > (q+1) \beta -1 .  
\end{equation}
Now, look at (\ref{secondoredrtaylor0}) to extract
\begin{align*}
\mathfrak p^0_{k_1,k_2} (t) &= \sum_{i=1}^2 \Bigl \lbrace (-1)^{k_i}  \gamma + \cO \( 
s_{k_i} (t , 0)^2 \) \\
&\qquad \quad \quad + \bigl \lbrack 1 - p \bigl( {k_i} \pi + s_{k_i} 
( t , 0 )\bigr) \bigr \rbrack \bigl( {k_i} \, \pi - t \bigr) + \cO (k_i^{-q} ) \vert s_{k_i} ( t , 0 )\vert \Bigr \rbrace . 
\end{align*}
From (\ref{reso2}) and (\ref{ilfaiuut}), we can deduce that
\begin{align*}
  \mathfrak p^0_{k_1,k_2} (t) &= \sum_{i=1}^2 \Bigl \lbrace (-1)^{k_i} \gamma + \cO(\eps^{2(q+1)
\beta-2})  \\
& \qquad \quad \quad - \, \frac{\ell}{q \, (q+1)} \Bigl( \frac{1}{(k_i \pi)^{q-1}} -   \frac{t}{(k_i \pi)^q} 
\Bigr)  \bigl( 1 + o(1) \bigr) + \cO(\eps^{q \beta + (q+1)\beta-1}) \Bigr \rbrace . 
\end{align*}
Using (\ref{ilfaiutlemettre}), there remains
 \begin{equation}\label{envoila1} 
 \begin{aligned}
\mathfrak p^0_{k_1,k_2} (s/\eps) / \eps & =  o(1)  + \( (-1)^{k_1} + (-1)^{k_2} \) (\gamma / \eps) \\
& \quad - \sum_{i=1}^2 \Biggl \lbrace \frac{\ell \, \eps^{q-2}}{q \, (q+1)} \( \frac{1}{(\eps k_i \pi)^{q-1}} -   
\frac{s}{(\eps k_i \pi)^q} \)  \( 1 + o(1) \) \Biggr \rbrace . 
\end{aligned} 
\end{equation}  
Next, consider (\ref{secondoredrtaylor1}). Taking into account (\ref{unigliformskftx}) with again (\ref{ilfaiuut}) and (\ref{ilfaiutlemettre}), 
this gives rise to
\begin{align*}
  \mathfrak p^1_{k_1,k_2} (t) &= - (k_1+k_2) \pi + \sum_{i=1}^2 \Bigl \lbrace \cO \bigl( \vert 
s_{k_i} (t , 0) \vert \bigr) + \cO (k_i^{-q} )  + \cO (\eps^{-1} k_i^{-q-1} ) \Bigr \rbrace \\
& = - (k_1+k_2) \pi + \cO(\eps^{(q+1)\beta-1}) = - (k_1+k_2) \pi + o(1) . 
\end{align*}
With (\ref{hypp3ded}), this implies that
\[ p \circ \mathfrak p^1_{k_1,k_2} (t) = p \( - (k_1 + k_2) \pi \) + \cO \( (k_1+k_2)^{-q-1} \)  \times\cO\(\eps^{(q+1)\beta-1}\) ,\]
and therefore
\begin{equation}\label{envoila2}  
\begin{aligned}
  \frac{\(p \circ \mathfrak p^1_{k_1,k_2} (s/\eps) -1\)}{\eps^2} & =
  \frac{\bigl( p (k_1 \pi + k_2 \pi) -1)}{\eps^2} 
+ \cO (\eps^{2 (q+1)\beta-3}) \\
&   = \frac{\bigl( p (k_1 \pi + k_2 \pi) -1)}{\eps^2} + o(1) \\
&  = - \, \frac{\ell \eps^{q-2}}{q \, (q+1)} \frac{1}{(\eps k_1 \pi + \eps k_2 \pi)^q } \bigl( 1 + o(1) \bigr) + o(1) .
\end{aligned} 
\end{equation}  
Finally, we examine (\ref{inparticc}). From (\ref{devSkenvue}), (\ref{ilfaiuut}) and (\ref{ilfaiutlemettre}), we get easily
 \begin{align*}
\operatorname{det} S_k &  =  (-1)^{k+1} \gamma \cos s_k+ \cO \( \frac{1}{\eps \,  k ^{q+2}} \) \\
 & =  (-1)^{k+1} \gamma + \cO(\eps^{2(q+1)\beta-2}) + \cO\(\eps^{-1 + \beta (q+2)}\) \\
 & = (-1)^{k+1} \gamma + \cO( \sqrt \eps ) . 
 \end{align*}
 On the other hand, combining (\ref{tildeak}) and (\ref{ilfaiuut}), we find that
$$ \tilde a_k \bigl( \eps,s_k (t,0)  ,y_k (t,0)  , s_k (t,0)  \bigr) = \zeta (k\pi) a(\eps k\pi,k\pi,0) + \cO( \sqrt \eps ) . $$
From Assumption~\ref{choiopro}, we know that $ \zeta(k\pi) = 1 + \cO ( k^{-1} ) = 1 + \cO(\eps^{\beta}) $. 
To make things easier, we replace Assumption~\ref{persourcea} by the more restrictive condition 
(\ref{asymptoticallypi}). Then, we can infer that
\begin{equation}\label{remindakk} 
\begin{aligned}
\tilde a_k \( \eps,s_k (t,0)  ,y_k (t,0)  , s_k (t,0)  \)  & = a (\eps
k\pi,k \pi,0) + o (1) \\  
& = \aaa (\eps k\pi,0,0) + o (1) .
\end{aligned}
\end{equation}  
From now on, we work with the case $ q = 2 $, which is the most interesting and also the most difficult situation,
because (\ref{envoila1}) and (\ref{envoila2}) contain supplementary contributions modulo $ o(1) $. 
We can decompose the sum inside (\ref{duhamelcassomipsumred}) into
\[  \sum_{ \eps^{-\beta} \leq k_1 \in \cK^c_s} \ \sum_{ \eps^{-\beta} \leq k_2 \in \cK^c_s}  = \sum_{
\substack{ \{ (\rhd,\lhd) ; \rhd \, \text{is even or odd} \\ \qquad \quad \ \lhd \, \text{is even or odd} \} } } \ \ 
\sum_{\substack{\eps^{-\beta} \leq k_1 \in \cK^c_s \\ k_1 \, \text{is of type} \, \rhd}} \ \ 
\sum_{\substack{\eps^{-\beta} \leq k_2 \in \cK^c_s \\ k_2 \, \text{is of type} \, \lhd}}  .\]
We deal below with the sum corresponding to the choice $ (\rhd,\lhd) = (\text{even}, \text{even}) $, the other cases 
being completely similar. When $ k_1 $ and $ k_2 $ are even, by combining (\ref{envoila1}), (\ref{envoila2}) and 
(\ref{remindakk}), the product inside (\ref{Wepsk12entemps}) can be reworded into
\begin{align*}
e^{i \mathfrak p^0_{k_1,k_2} (s/\eps) / \eps} e^{i z \mathfrak p^1_{k_1,k_2} (s/\eps)} e^{ i (T-s)(p \circ 
\mathfrak p^1_{k_1,k_2} (s/\eps) -1)/ \eps^2} (\cA_{k_1,0}^\eps \cA_{k_2,0}^\eps) (s/\eps,0) \\
\qquad \qquad \qquad = o(1) - \displaystyle \, i \, (2 \pi)^3 \,
  \gamma^{-1} \, e^{2 i \gamma / \eps}  
e^{-i z (k_1+k_2) \pi} \, d_\eps ( \eps \pi k_1 , \eps \pi k_2 ,s) ,
 \end{align*} 
with 
\begin{equation}\label{disappearsd}  
 \begin{aligned}
 d_\eps (\sigma_1,\sigma_2,s) := &  e^{- i \frac{\ell}{6} (\frac{1}{\sigma_1}-\frac{s}{\sigma_1^2} )
(1+o(1))} e^{- i \frac{\ell}{6} (\frac{1}{\sigma_2}-\frac{s}{\sigma_2^2} ) (1+o(1))} \\
 & \times e^{-i \frac{\ell}{6} \frac{T-s}{(\sigma_1 + \sigma_2)^2} (1+o(1))} 
\aaa(\sigma_1,0,0) \aaa(\sigma_2,0,0) ,
 \end{aligned} 
 \end{equation}  
where the dependence on $ \eps $ is hidden in the $ o(1) $. Passing to the limit when $ \eps $ goes to $ 0 $, the $ o(1) $
disappears from (\ref{disappearsd}). There remains
\begin{equation}\label{disappearsd0}   
d_0 (\sigma_1,\sigma_2,s) := e^{-i \frac{\ell}{6} \frac{T-s}{(\sigma_1 + \sigma_2)^2}} b(\sigma_1,s)  b(\sigma_2,s) , 
 \end{equation} 
where, in coherence with the introduction, we have introduced
\[ b(\sigma,s) := e^{-i\frac{\ell}{6}\(\frac{1}{\sigma}-\frac{s}{\sigma^2}\)}\aaa(\sigma,0,0).  \]
When $z\not \in 2 \ZZ$, the Abel sum argument can be readily repeated. As before, the terms $ e^{-i z k_1 \pi} $  or 
$ e^{-i z k_2 \pi} $ compensate (locally in $ k_1 $ or in $ k_2 $) after summation. By this way, we can recover (\ref{destrasyana}). 
Otherwise, we can recognize a double Riemann sum with a width of $ \eps \pi $, which is
\[ - \frac{2 i}{\gamma \pi} \int_0^T \chi\(3-2\frac{s}{\cT}\) 
\sum_{\substack{\eps^{-\beta} \leq k_1 \in \cK^c_s \\ k_1 \, \text{is even}}} \ \ 
\sum_{\substack{\eps^{-\beta} \leq k_2 \in \cK^c_s \\ k_2 \, \text{is even}}} (\eps \pi)^2 d_\eps ( \eps \pi k_1 , \eps \pi k_2 ,s) 
 \, ds . \]
 For all $ s \in [0,T] $, the function $ d_\eps (\cdot,s) $ is defined on the quadrant $ \cQ := \RR_+^* \times \RR_+^* $.
 It is smooth and bounded on the open domain $ \cQ $. The singularities of the exponents near $ \sigma_1 = 0 $ or 
 $ \sigma_2 = 0 $ translate only into fast oscillations. Moreover, due to (\ref{suppdea}), the support of the function 
 $ d_\eps  (\cdot,s) $ is uniformly bounded. In particular, the function $ d_0 (\cdot,s) $ is integrable on $ \cQ $. As 
 a consequence, the double Riemann sum does converge (when $ \iota < 1 $) towards 
$$ - \frac{2 i}{\gamma \pi} \int_0^T \chi\(3-2\frac{s}{\cT}\) \( \int_0^{+\infty} \! \int_0^{+\infty} d_0 ( \sigma_1 , \sigma_2 ,s) 
d \sigma_1 d \sigma_2 \) ds . $$
The other choices of $ (\rhd,\lhd) $ combine to form (\ref{desintertheoprinbis}).
\end{proof}

\noindent We conclude with a series of comments.

\begin{rem}[The origin of the correlation coefficient] As is well known, the weak limit of a product is in general different 
from the product of the weak limits. Comparing \eqref{consintertheoprin} and \eqref{desintertheoprinbis}, this principle applies 
in the present context. Indeed, the function $ d_0$ of \eqref{disappearsd0} is not the product of $ b (\sigma_1,s) $ 
and $ b (\sigma_2,s)$. There is a correlation coefficient which is issued from the multiplication inside \eqref{Wepsk12entemps}
by
\[ e^{ i (T-s)(p \circ \mathfrak p^1_{k_1,k_2} (s/\eps) -1)/ \eps^2} . \]
Looking at \eqref{envoila2}, we see that this nonlinear effect depends on dispersive properties through the asymptotic 
behavior of the symbol $ p $ when $ \vert \xi \vert $ goes to $ + \infty $. It measures how the various frequencies
$ k_1 $ and $ k_2 $ interact asymptotically (through their sum) in order to affect the profile.
   \end{rem}

\begin{rem}[About the spatial localization] The limiting case $ \iota = 1 $ could be incorporated just by multiplying 
\eqref{desintertheoprinbis} by $ \chi (r^{-1} z) $. By contrast, the
case $ \iota < \iota_- $ seems more difficult to assess.  
By pushing the Taylor expansion \eqref{taylorexpansion} up to the next order $ 3 $, for $ \iota \in [1/3,1/2[ $, it would 
be still possible to separate some explicit ``oscillating part" from some ``generalized profile". But then, explicit formulas 
are no more available, and the presence of some extra oscillations can really change the asymptotic behavior 
\eqref{desintertheoprinbis}.
   \end{rem}

\begin{rem}[About the critical cubic nonlinearity] \label{criticalcubicrem} Come back to Example~\ref{excaseofcubictronl}. This corresponds 
to the study of
  \begin{align*}
     \cU^{(1)}\(T,z\) & = \cU^{(0)}(T,z) \\
   +\frac{1}{2\pi} &\int_0^T \! \! \int \! \! \int e^{-i(z-y)\xi + i \frac{T-s}{\eps^2}(p(\xi)-1) } \chi\(\frac{y}{r\eps^{\iota-1}}\) 
   | \cU^{(0)}(s,y)|^2\cU^{(0)}(s,y) \, dsdyd\xi .
  \end{align*}
  The above trilinear interaction involves the phase
$$ \begin{array}{rl}
\mathfrak p_{k_1,k_2,k_3} (t,x) \! \! \! & := \bm{\uppsi}_{k_1}  (t,x) -  \bm{\uppsi}_{k_2}  (t,x) + \bm{\uppsi}_{k_3}  (t,x)  \\
\ & = \mathfrak p^0_{k_1,k_2,k_3} (t,x) - (k_1-k_2+k_3) \pi + \cdots 
\end{array} $$
and thereby, the Dirac mass argument should select the position $ (k_1-k_2+k_3) \pi $. It could be expected to obtain 
triple integrals of the form
  \[
    \int_0^T \chi\(3-2\frac{s}{\cT}\)\!\!\!\int_0^{+\infty}\!\!\!\! \int_0^{+\infty} \!\!\!\!
    \int_0^{+\infty} e^{-i\frac{\ell}{6} 
\frac{T-s}{(\sigma_1\pm \sigma_2 \pm \sigma_3)^2}} b(\sigma_1,s)\bar b(\sigma_2,s)b(\sigma_3,s)
dsd\sigma_1d\sigma_2d\sigma_3.
\]
But the presence of two wave-numbers ($ k_1 $ and $ k_2 $) with opposite signs may change the situation. Indeed,
for large values of $ k_1 $ and $ k_2 $ with $ k_1 -k_2 = \cO(1) $, the asymptotic behavior of $ p  $ is no more
involved when computing $ p \circ \mathfrak p $. The analysis is apparently different. It should require further 
development. 
 \end{rem}

\begin{rem}[About the full nonlinear case] \label{nonlieafinrem} The description of the solution $ u$ to \eqref{third-si-gene} 
with $ F_{\! N\! L} $ as in \eqref{eq:NLedp} is a far more complicated task for a number of reasons. Consider 
for instance the non completely resonant case , when $\textfrak {g} \not = 1 $. The linearizability has been 
established by nonstationary phase arguments relying on Lemma~\ref{lem:fine-estimates}. From this perspective, 
the global controls provided by Lemma~\ref{GlobalcontrolintheL2} do not suffice. As a consequence, to prove 
Theorem~\ref{theo:resumeNLbis} in the case of the complete nonlinear equation (as opposed to the first two Picard 
iterates), we would need estimates similar to those from Lemma~\ref{lem:fine-estimates}. However, those do not 
seem to be propagated in an easy way by an iterative scheme. Since the proof of Lemma~\ref{lem:fine-estimates} 
actually relies on the wave packets decomposition of Section~\ref{sec:lineffect}, extending this wave packets 
decomposition to a nonlinear framework (like in e.g. \cite{CoHaLu08,GaHaLu16}, or
\cite{GeMaSh09,GeMaSh12a,GeMaSh12b}, which may be understood as
generalizations of WKB methods) might 
be a way to treat the full nonlinear equation. As evoked in the introduction, we will 
not pursue this question here. 
\end{rem}


\bibliographystyle{abbrv}
\bibliography{biblio-draft}

\end{document}

%% file: macros.tex
\def\eps{\varepsilon }

\def\({\left(}
\def\){\right)}
\def\<{\left\langle}
\def\>{\right\rangle}

\newcommand\br{\begin{remark}}
\newcommand\er{\end{remark}}
\newcommand\bp{\begin{pmatrix}}
\newcommand\ep{\end{pmatrix}}
\newcommand{\be}{\begin{equation}}
\newcommand{\ee}{\end{equation}}
\newcommand{\ba}[1]{\begin{array}{#1}}
\newcommand{\ea}{\end{array}}

\newcommand{\beg}{\begin{example}}
\newcommand{\eeg}{\end{exaplem}}
\newcommand{\bpr}{\begin{proposition}}
\newcommand{\epr}{\end{proposition}}
\newcommand{\bt}{\begin{theorem}}
\newcommand{\et}{\end{theorem}}
\newcommand{\bc}{\begin{corollary}}
\newcommand{\ec}{\end{corollary}}
\newcommand{\bl}{\begin{lemma}}
\newcommand{\el}{\end{lemma}}
\newcommand{\bd}{\begin{definition}}
\newcommand{\ed}{\end{definition}}
\newcommand{\brs}{\begin{remarks}}
\newcommand{\ers}{\end{remarks}}

\newtheorem{theo}{Theorem}[section]

\newtheorem{lem}[theo]{Lemma}
\newtheorem{cor}[theo]{Corollary}
\newtheorem{prop}[theo]{Proposition}
\newtheorem{defi}[theo]{Definition}
\newtheorem{ass}[theo]{Assumption}
\newtheorem{rem}[theo]{Remark}
\numberwithin{equation}{section}
\newtheorem{defi-pro}[theo]{Definition-Proposition}

\newtheorem{exam}[theo]{Example}

\def\eps {\varepsilon}

\def\part{\partial}

\newcommand{\aaa}{\underline a}

\newcommand{\CC}{{\mathbb C}}

\newcommand{\NN}{{\mathbb N}}

\newcommand{\RR}{{\mathbb R}}
\newcommand{\TT}{{\mathbb T}}

\newcommand{\ZZ}{{\mathbb Z}}

\newcommand\cA{{\mathcal A}}
\newcommand\cB{{\mathcal B}}
\newcommand\cC{{\mathcal C}}
\newcommand\cD{{\mathcal D}}
\newcommand\cE{{\mathcal E}}
\newcommand\cF{{\mathcal F}}
\newcommand\cG{{\mathcal G}}

\newcommand\cI{{\mathcal I}}
\newcommand\cJ{{\mathcal J}}
\newcommand\cK{{\mathcal K}}
\newcommand\cL{{\mathcal L}}
\newcommand\cM{{\mathcal M}}
\newcommand\cN{{\mathcal N}}
\newcommand\cO{{\mathcal O}}
\newcommand\cP{{\mathcal P}}
\newcommand\cQ{{\mathcal Q}}
\newcommand\cR{{\mathcal R}}
\newcommand\cS{{\mathcal S}}
\newcommand\cT{{\mathcal T}}
\newcommand\cU{{\mathcal U}}
\newcommand\cV{{\mathcal V}}
\newcommand\cW{{\mathcal W}}

